\documentclass[11pt,a4paper]{amsart}

\usepackage{euler}
\usepackage{bbm,amsfonts,amsmath,amssymb,mathrsfs,tikz,hyperref,dsfont,
csquotes,enumerate,subfigure} 
\hypersetup{pdfborder={0000}, colorlinks=true, linkcolor=blue,citecolor=citegreen}
\definecolor{citegreen}{rgb}{0.2,0.2,0.6}
\usepackage[
kerning=true]{microtype}

\hypersetup{
 colorlinks=true,
 citecolor=darkred,
 linkcolor=blue,
 urlcolor=blue}

\newcommand\Kre{{Krej{\v{c}}i{\v{r}}{\'\i}k}}

\usepackage[shortlabels]{enumitem}

\usepackage[font=small]{caption}

\usepackage{palatino}
\usepackage{eucal}
\usepackage{verbatim}

\usepackage{amsmath,amssymb,amsthm,graphicx,epstopdf,mathrsfs,url}

\hypersetup{
 colorlinks=true,
 citecolor=darkred,
 linkcolor=blue,
 urlcolor=blue}

\ifx\figforTeXisloaded\relax \else\global\let\figforTeXisloaded=\relax\fi
\catcode`\@=11
\ifx\ctr@ln@m\undefined\else%
    \immediate\write16{*** Fig4TeX WARNING : \string\ctr@ln@m\space already defined.}\fi
\def\ctr@ln@m#1{\ifx#1\undefined\else%
    \immediate\write16{*** Fig4TeX WARNING : \string#1 already defined.}\fi}
\ctr@ln@m\ctr@ld@f
\def\ctr@ld@f#1#2{\ctr@ln@m#2#1#2}
\ctr@ld@f\def\ctr@ln@w#1#2{\ctr@ln@m#2\csname#1\endcsname#2}
{\catcode`\/=0 \catcode`/\=12 /ctr@ld@f/gdef/BS@{\}}
\ctr@ld@f\def\ctr@lcsn@m#1{\expandafter\ifx\csname#1\endcsname\relax\else%
    \immediate\write16{*** Fig4TeX WARNING : \BS@\expandafter\string#1\space already defined.}\fi}
\ctr@ld@f\edef\colonc@tcode{\the\catcode`\:}
\ctr@ld@f\edef\semicolonc@tcode{\the\catcode`\;}
\ctr@ld@f\def\t@stc@tcodech@nge{{\let\c@tcodech@nged=\z@%
    \ifnum\colonc@tcode=\the\catcode`\:\else\let\c@tcodech@nged=\@ne\fi%
    \ifnum\semicolonc@tcode=\the\catcode`\;\else\let\c@tcodech@nged=\@ne\fi%
    \ifx\c@tcodech@nged\@ne%
    \immediate\write16{}
    \immediate\write16{!!!=============================================================!!!}
    \immediate\write16{ Fig4TeX WARNING:}
    \immediate\write16{ The category code of some characters has been changed, which will}
    \immediate\write16{ result in an error (message "Runaway argument?").}
    \immediate\write16{ This probably comes from another package that changed the category}
    \immediate\write16{ code after Fig4TeX was loaded. If that proves to be exact, the}
    \immediate\write16{ solution is to exchange the loading commands on top of your file}
    \immediate\write16{ so that Fig4TeX is loaded last. For example, in LaTeX, we should}
    \immediate\write16{ say :}
    \immediate\write16{\BS@ usepackage[french]{babel}}
    \immediate\write16{\BS@ usepackage{fig4tex}}
    \immediate\write16{!!!=============================================================!!!}
    \immediate\write16{}
    \fi}}
\ctr@ld@f\def\FigforTeX{F\kern-.05em i\kern-.05em g\kern-.1em\raise-.14em\hbox{4}\kern-.19em\TeX}
\ctr@ld@f\def\W@rnmesoldA#1{\W@rnmesold}
\ctr@ld@f\def\W@rnmesoldAB#1(#2){\W@rnmesold}
\ctr@ld@f\def\W@rnmesold{%
    \immediate\write16{}
    \immediate\write16{!!!=============================================================!!!}
    \immediate\write16{ Fig4TeX WARNING:}
    \immediate\write16{ The file to be compiled is not compatible with the current version}
    \immediate\write16{ of Fig4TeX. To fix that, upgrade the source file (mainly change \BS@ ps*}
    \immediate\write16{ macros by \BS@ fig* macros), or use fig4tex184.tex instead (\BS@ input fig4tex184}
    \immediate\write16{ or \BS@ usepackage{fig4tex184}).}
    \immediate\write16{!!!=============================================================!!!}
    \immediate\write16{}}
\ctr@ln@m\psbeginfig\let\psbeginfig\W@rnmesoldA
\ctr@ln@m\psset\let\psset\W@rnmesoldAB
\ctr@ln@m\pssetdefault\let\pssetdefault\W@rnmesoldAB
\ctr@ln@m\pssetupdate\let\pssetupdate\W@rnmesoldA
\ctr@ln@w{newdimen}\epsil@n\epsil@n=0.00005pt
\ctr@ln@w{newdimen}\Cepsil@n\Cepsil@n=0.005pt
\ctr@ln@w{newdimen}\dcq@\dcq@=254pt
\ctr@ln@w{newdimen}\PI@\PI@=3.141592pt
\ctr@ln@w{newdimen}\DemiPI@deg\DemiPI@deg=90pt
\ctr@ln@w{newdimen}\PI@deg\PI@deg=180pt
\ctr@ln@w{newdimen}\DePI@deg\DePI@deg=360pt
\ctr@ld@f\chardef\t@n=10
\ctr@ld@f\chardef\c@nt=100
\ctr@ld@f\chardef\@lxxiv=74
\ctr@ld@f\chardef\@xci=91
\ctr@ld@f\mathchardef\@nMnCQn=9949
\ctr@ld@f\chardef\@vi=6
\ctr@ld@f\chardef\@xxx=30
\ctr@ld@f\chardef\@lvi=56
\ctr@ld@f\chardef\@@lxxi=71
\ctr@ld@f\chardef\@lxxxv=85
\ctr@ld@f\mathchardef\@@mmmmlxviii=4068
\ctr@ld@f\mathchardef\@ccclx=360
\ctr@ld@f\mathchardef\@dccxx=720
\ctr@ln@w{newcount}\p@rtent \ctr@ln@w{newcount}\f@ctech \ctr@ln@w{newcount}\result@tent
\ctr@ln@w{newdimen}\v@lmin \ctr@ln@w{newdimen}\v@lmax \ctr@ln@w{newdimen}\v@leur
\ctr@ln@w{newdimen}\result@t\ctr@ln@w{newdimen}\result@@t
\ctr@ln@w{newdimen}\mili@u \ctr@ln@w{newdimen}\c@rre \ctr@ln@w{newdimen}\delt@
\ctr@ld@f\def\degT@rd{0.017453 }  
\ctr@ld@f\def\rdT@deg{57.295779 } 
\ctr@ln@m\v@leurseule
{\catcode`p=12 \catcode`t=12 \gdef\v@leurseule#1pt{#1}}
\ctr@ld@f\def\repdecn@mb#1{\expandafter\v@leurseule\the#1\space}
\ctr@ld@f\def\arct@n#1(#2,#3){{\v@lmin=#2\v@lmax=#3%
    \maxim@m{\mili@u}{-\v@lmin}{\v@lmin}\maxim@m{\c@rre}{-\v@lmax}{\v@lmax}%
    \delt@=\mili@u\m@ech\mili@u%
    \ifdim\c@rre>\@nMnCQn\mili@u\divide\v@lmax\tw@\c@lATAN\v@leur(\z@,\v@lmax)
    \else%
    \maxim@m{\mili@u}{-\v@lmin}{\v@lmin}\maxim@m{\c@rre}{-\v@lmax}{\v@lmax}%
    \m@ech\c@rre%
    \ifdim\mili@u>\@nMnCQn\c@rre\divide\v@lmin\tw@
    \maxim@m{\mili@u}{-\v@lmin}{\v@lmin}\c@lATAN\v@leur(\mili@u,\z@)%
    \else\c@lATAN\v@leur(\delt@,\v@lmax)\fi\fi%
    \ifdim\v@lmin<\z@\v@leur=-\v@leur\ifdim\v@lmax<\z@\advance\v@leur-\PI@%
    \else\advance\v@leur\PI@\fi\fi%
    \global\result@t=\v@leur}#1=\result@t}
\ctr@ld@f\def\m@ech#1{\ifdim#1>1.646pt\divide\mili@u\t@n\divide\c@rre\t@n\m@ech#1\fi}
\ctr@ld@f\def\c@lATAN#1(#2,#3){{\v@lmin=#2\v@lmax=#3\v@leur=\z@\delt@=\tw@ pt%
    \un@iter{0.785398}{\v@lmax<}%
    \un@iter{0.463648}{\v@lmax<}%
    \un@iter{0.244979}{\v@lmax<}%
    \un@iter{0.124355}{\v@lmax<}%
    \un@iter{0.062419}{\v@lmax<}%
    \un@iter{0.031240}{\v@lmax<}%
    \un@iter{0.015624}{\v@lmax<}%
    \un@iter{0.007812}{\v@lmax<}%
    \un@iter{0.003906}{\v@lmax<}%
    \un@iter{0.001953}{\v@lmax<}%
    \un@iter{0.000976}{\v@lmax<}%
    \un@iter{0.000488}{\v@lmax<}%
    \un@iter{0.000244}{\v@lmax<}%
    \un@iter{0.000122}{\v@lmax<}%
    \un@iter{0.000061}{\v@lmax<}%
    \un@iter{0.000030}{\v@lmax<}%
    \un@iter{0.000015}{\v@lmax<}%
    \global\result@t=\v@leur}#1=\result@t}
\ctr@ld@f\def\un@iter#1#2{%
    \divide\delt@\tw@\edef\dpmn@{\repdecn@mb{\delt@}}%
    \mili@u=\v@lmin%
    \ifdim#2\z@%
      \advance\v@lmin-\dpmn@\v@lmax\advance\v@lmax\dpmn@\mili@u%
      \advance\v@leur-#1pt%
    \else%
      \advance\v@lmin\dpmn@\v@lmax\advance\v@lmax-\dpmn@\mili@u%
      \advance\v@leur#1pt%
    \fi}
\ctr@ld@f\def\c@ssin#1#2#3{\expandafter\ifx\csname COS@\number#3\endcsname\relax\c@lCS{#3pt}%
    \expandafter\xdef\csname COS@\number#3\endcsname{\repdecn@mb\result@t}%
    \expandafter\xdef\csname SIN@\number#3\endcsname{\repdecn@mb\result@@t}\fi%
    \edef#1{\csname COS@\number#3\endcsname}\edef#2{\csname SIN@\number#3\endcsname}}
\ctr@ld@f\def\c@lCS#1{{\mili@u=#1\p@rtent=\@ne%
    \relax\ifdim\mili@u<\z@\red@ng<-\else\red@ng>+\fi\f@ctech=\p@rtent%
    \relax\ifdim\mili@u<\z@\mili@u=-\mili@u\f@ctech=-\f@ctech\fi\c@@lCS}}
\ctr@ld@f\def\c@@lCS{\v@lmin=\mili@u\c@rre=-\mili@u\advance\c@rre\DemiPI@deg\v@lmax=\c@rre%
    \mili@u\@@lxxi\mili@u\divide\mili@u\@@mmmmlxviii%
    \edef\v@larg{\repdecn@mb{\mili@u}}\mili@u=-\v@larg\mili@u%
    \edef\v@lmxde{\repdecn@mb{\mili@u}}%
    \c@rre\@@lxxi\c@rre\divide\c@rre\@@mmmmlxviii%
    \edef\v@largC{\repdecn@mb{\c@rre}}\c@rre=-\v@largC\c@rre%
    \edef\v@lmxdeC{\repdecn@mb{\c@rre}}%
    \fctc@s\mili@u\v@lmin\global\result@t\p@rtent\v@leur%
    \let\t@mp=\v@larg\let\v@larg=\v@largC\let\v@largC=\t@mp%
    \let\t@mp=\v@lmxde\let\v@lmxde=\v@lmxdeC\let\v@lmxdeC=\t@mp%
    \fctc@s\c@rre\v@lmax\global\result@@t\f@ctech\v@leur}
\ctr@ld@f\def\fctc@s#1#2{\v@leur=#1\relax\ifdim#2<\@lxxxv\p@\cosser@h\else\sinser@t\fi}
\ctr@ld@f\def\cosser@h{\advance\v@leur\@lvi\p@\divide\v@leur\@lvi%
    \v@leur=\v@lmxde\v@leur\advance\v@leur\@xxx\p@%
    \v@leur=\v@lmxde\v@leur\advance\v@leur\@ccclx\p@%
    \v@leur=\v@lmxde\v@leur\advance\v@leur\@dccxx\p@\divide\v@leur\@dccxx}
\ctr@ld@f\def\sinser@t{\v@leur=\v@lmxdeC\p@\advance\v@leur\@vi\p@%
    \v@leur=\v@largC\v@leur\divide\v@leur\@vi}
\ctr@ld@f\def\red@ng#1#2{\relax\ifdim\mili@u#1#2\DemiPI@deg\advance\mili@u#2-\PI@deg%
    \p@rtent=-\p@rtent\red@ng#1#2\fi}
\ctr@ld@f\def\pr@c@lCS#1#2#3{\ctr@lcsn@m{COS@\number#3 }%
    \expandafter\xdef\csname COS@\number#3\endcsname{#1}%
    \expandafter\xdef\csname SIN@\number#3\endcsname{#2}}
\pr@c@lCS{1}{0}{0}
\pr@c@lCS{0.7071}{0.7071}{45}\pr@c@lCS{0.7071}{-0.7071}{-45}
\pr@c@lCS{0}{1}{90}          \pr@c@lCS{0}{-1}{-90}
\pr@c@lCS{-1}{0}{180}        \pr@c@lCS{-1}{0}{-180}
\pr@c@lCS{0}{-1}{270}        \pr@c@lCS{0}{1}{-270}
\ctr@ld@f\def\invers@#1#2{{\v@leur=#2\maxim@m{\v@lmax}{-\v@leur}{\v@leur}%
    \f@ctech=\@ne\m@inv@rs%
    \multiply\v@leur\f@ctech\edef\v@lv@leur{\repdecn@mb{\v@leur}}%
    \p@rtentiere{\p@rtent}{\v@leur}\v@lmin=\p@\divide\v@lmin\p@rtent%
    \inv@rs@\multiply\v@lmax\f@ctech\global\result@t=\v@lmax}#1=\result@t}
\ctr@ld@f\def\m@inv@rs{\ifdim\v@lmax<\p@\multiply\v@lmax\t@n\multiply\f@ctech\t@n\m@inv@rs\fi}
\ctr@ld@f\def\inv@rs@{\v@lmax=-\v@lmin\v@lmax=\v@lv@leur\v@lmax%
    \advance\v@lmax\tw@ pt\v@lmax=\repdecn@mb{\v@lmin}\v@lmax%
    \delt@=\v@lmax\advance\delt@-\v@lmin\ifdim\delt@<\z@\delt@=-\delt@\fi%
    \ifdim\delt@>\epsil@n\v@lmin=\v@lmax\inv@rs@\fi}
\ctr@ld@f\def\minim@m#1#2#3{\relax\ifdim#2<#3#1=#2\else#1=#3\fi}
\ctr@ld@f\def\maxim@m#1#2#3{\relax\ifdim#2>#3#1=#2\else#1=#3\fi}
\ctr@ld@f\def\p@rtentiere#1#2{#1=#2\divide#1by65536 }
\ctr@ld@f\def\r@undint#1#2{{\v@leur=#2\divide\v@leur\t@n\p@rtentiere{\p@rtent}{\v@leur}%
    \v@leur=\p@rtent pt\global\result@t=\t@n\v@leur}#1=\result@t}
\ctr@ld@f\def\sqrt@#1#2{{\v@leur=#2%
    \minim@m{\v@lmin}{\p@}{\v@leur}\maxim@m{\v@lmax}{\p@}{\v@leur}%
    \f@ctech=\@ne\m@sqrt@\sqrt@@%
    \mili@u=\v@lmin\advance\mili@u\v@lmax\divide\mili@u\tw@\multiply\mili@u\f@ctech%
    \global\result@t=\mili@u}#1=\result@t}
\ctr@ld@f\def\m@sqrt@{\ifdim\v@leur>\dcq@\divide\v@leur\c@nt\v@lmax=\v@leur%
    \multiply\f@ctech\t@n\m@sqrt@\fi}
\ctr@ld@f\def\sqrt@@{\mili@u=\v@lmin\advance\mili@u\v@lmax\divide\mili@u\tw@%
    \c@rre=\repdecn@mb{\mili@u}\mili@u%
    \ifdim\c@rre<\v@leur\v@lmin=\mili@u\else\v@lmax=\mili@u\fi%
    \delt@=\v@lmax\advance\delt@-\v@lmin\ifdim\delt@>\epsil@n\sqrt@@\fi}
\ctr@ld@f\def\extrairelepremi@r#1\de#2{\expandafter\lepremi@r#2@#1#2}
\ctr@ld@f\def\lepremi@r#1,#2@#3#4{\def#3{#1}\def#4{#2}\ignorespaces}
\ctr@ld@f\def\@cfor#1:=#2\do#3{%
  \edef\@fortemp{#2}%
  \ifx\@fortemp\empty\else\@cforloop#2,\@nil,\@nil\@@#1{#3}\fi}
\ctr@ln@m\@nextwhile
\ctr@ld@f\def\@cforloop#1,#2\@@#3#4{%
  \def#3{#1}%
  \ifx#3\Fig@nnil\let\@nextwhile=\Fig@fornoop\else#4\relax\let\@nextwhile=\@cforloop\fi%
  \@nextwhile#2\@@#3{#4}}

\ctr@ld@f\def\@ecfor#1:=#2\do#3{%
  \def\@@cfor{\@cfor#1:=}%
  \edef\@@@cfor{#2}%
  \expandafter\@@cfor\@@@cfor\do{#3}}
\ctr@ld@f\def\Fig@nnil{\@nil}
\ctr@ld@f\def\Fig@fornoop#1\@@#2#3{}
\ctr@ln@m\list@@rg
\ctr@ld@f\def\trtlis@rg#1#2{\def\list@@rg{#1}%
    \@ecfor\p@rv@l:=\list@@rg\do{\expandafter#2\p@rv@l|}}
\ctr@ld@f\def\trtlis@rgtok#1{\let@xte={}\let\n@xt\addt@t@xt\addt@t@xt #1}
\ctr@ln@m\M@cro
\ctr@ln@m\n@xt
\ctr@ld@f\def\addt@t@xt#1{\if#1|\let\n@xt\relax\else%
    \if#1,\expandafter\M@cro\the\let@xte|\let@xte={}%
    \else\let@xte=\expandafter{\the\let@xte #1}\fi\fi\n@xt}
\ctr@ln@w{newbox}\b@xvisu
\ctr@ln@w{newtoks}\let@xte
\ctr@ln@w{newif}\ifitis@K
\ctr@ln@w{newcount}\s@mme
\ctr@ln@w{newcount}\l@mbd@un \ctr@ln@w{newcount}\l@mbd@de
\ctr@ln@w{newcount}\superc@ntr@l\superc@ntr@l=\@ne        
\ctr@ln@w{newcount}\typec@ntr@l\typec@ntr@l=\superc@ntr@l 
\ctr@ln@w{newdimen}\v@lX  \ctr@ln@w{newdimen}\v@lY  \ctr@ln@w{newdimen}\v@lZ
\ctr@ln@w{newdimen}\v@lXa \ctr@ln@w{newdimen}\v@lYa \ctr@ln@w{newdimen}\v@lZa
\ctr@ln@w{newdimen}\unit@\unit@=\p@ 
\ctr@ld@f\def\unit@util{pt}
\ctr@ld@f\def\ptT@ptps{0.996264 }
\ctr@ld@f\def\ptpsT@pt{1.00375 }
\ctr@ld@f\def\ptT@unit@{1} 
\ctr@ld@f\def\setunit@#1{\def\unit@util{#1}\setunit@@#1:\invers@{\result@t}{\unit@}%
    \edef\ptT@unit@{\repdecn@mb\result@t}}
\ctr@ld@f\def\setunit@@#1#2:{\ifcat#1a\unit@=\@ne#1#2\else\unit@=#1#2\fi}
\ctr@ld@f\def\d@fm@cdim#1#2{{\v@leur=#2\v@leur=\ptT@unit@\v@leur\xdef#1{\repdecn@mb\v@leur}}}
\ctr@ln@w{newif}\ifBdingB@x\BdingB@xtrue
\ctr@ln@w{newdimen}\c@@rdXmin \ctr@ln@w{newdimen}\c@@rdYmin  
\ctr@ln@w{newdimen}\c@@rdXmax \ctr@ln@w{newdimen}\c@@rdYmax
\ctr@ld@f\def\b@undb@x#1#2{\ifBdingB@x%
    \relax\ifdim#1<\c@@rdXmin\global\c@@rdXmin=#1\fi%
    \relax\ifdim#2<\c@@rdYmin\global\c@@rdYmin=#2\fi%
    \relax\ifdim#1>\c@@rdXmax\global\c@@rdXmax=#1\fi%
    \relax\ifdim#2>\c@@rdYmax\global\c@@rdYmax=#2\fi\fi}
\ctr@ld@f\def\b@undb@xP#1{{\Figg@tXY{#1}\b@undb@x{\v@lX}{\v@lY}}}
\ctr@ld@f\def\ellBB@x#1;#2,#3(#4,#5,#6){{\s@uvc@ntr@l\et@tellBB@x%
    \setc@ntr@l{2}\figptell-2::#1;#2,#3(#4,#6)\b@undb@xP{-2}%
    \figptell-2::#1;#2,#3(#5,#6)\b@undb@xP{-2}%
    \c@ssin{\C@}{\S@}{#6}\v@lmin=\C@ pt\v@lmax=\S@ pt%
    \mili@u=#3\v@lmin\delt@=#2\v@lmax\arct@n\v@leur(\delt@,\mili@u)%
    \mili@u=-#3\v@lmax\delt@=#2\v@lmin\arct@n\c@rre(\delt@,\mili@u)%
    \v@leur=\rdT@deg\v@leur\advance\v@leur-\DePI@deg%
    \c@rre=\rdT@deg\c@rre\advance\c@rre-\DePI@deg%
    \v@lmin=#4pt\v@lmax=#5pt%
    \loop\ifdim\v@leur<\v@lmax\ifdim\v@leur>\v@lmin%
    \edef\@ngle{\repdecn@mb\v@leur}\figptell-2::#1;#2,#3(\@ngle,#6)%
    \b@undb@xP{-2}\fi\advance\v@leur\PI@deg\repeat%
    \loop\ifdim\c@rre<\v@lmax\ifdim\c@rre>\v@lmin%
    \edef\@ngle{\repdecn@mb\c@rre}\figptell-2::#1;#2,#3(\@ngle,#6)%
    \b@undb@xP{-2}\fi\advance\c@rre\PI@deg\repeat%
    \resetc@ntr@l\et@tellBB@x}\ignorespaces}
\ctr@ld@f\def\initb@undb@x{\c@@rdXmin=\maxdimen\c@@rdYmin=\maxdimen%
    \c@@rdXmax=-\maxdimen\c@@rdYmax=-\maxdimen}
\ctr@ld@f\def\c@ntr@lnum#1{%
    \relax\ifnum\typec@ntr@l=\@ne%
    \ifnum#1<\z@%
    \immediate\write16{*** Forbidden point number (#1). Abort.}\end\fi\fi%
    \set@bjc@de{#1}}
\ctr@ln@m\objc@de
\ctr@ld@f\def\set@bjc@de#1{\edef\objc@de{@BJ\ifnum#1<\z@ M\romannumeral-#1\else\romannumeral#1\fi}}
\s@mme=\m@ne\loop\ifnum\s@mme>-19
  \set@bjc@de{\s@mme}\ctr@lcsn@m\objc@de\ctr@lcsn@m{\objc@de T}
\advance\s@mme\m@ne\repeat
\s@mme=\@ne\loop\ifnum\s@mme<6
  \set@bjc@de{\s@mme}\ctr@lcsn@m\objc@de\ctr@lcsn@m{\objc@de T}
\advance\s@mme\@ne\repeat
\ctr@ld@f\def\setc@ntr@l#1{\ifnum\superc@ntr@l>#1\typec@ntr@l=\superc@ntr@l%
    \else\typec@ntr@l=#1\fi}
\ctr@ld@f\def\resetc@ntr@l#1{\global\superc@ntr@l=#1\setc@ntr@l{#1}}
\ctr@ld@f\def\s@uvc@ntr@l#1{\edef#1{\the\superc@ntr@l}}
\ctr@ln@m\c@lproscal
\ctr@ld@f\def\c@lproscalDD#1[#2,#3]{{\Figg@tXY{#2}%
    \edef\Xu@{\repdecn@mb{\v@lX}}\edef\Yu@{\repdecn@mb{\v@lY}}\Figg@tXY{#3}%
    \global\result@t=\Xu@\v@lX\global\advance\result@t\Yu@\v@lY}#1=\result@t}
\ctr@ld@f\def\c@lproscalTD#1[#2,#3]{{\Figg@tXY{#2}\edef\Xu@{\repdecn@mb{\v@lX}}%
    \edef\Yu@{\repdecn@mb{\v@lY}}\edef\Zu@{\repdecn@mb{\v@lZ}}%
    \Figg@tXY{#3}\global\result@t=\Xu@\v@lX\global\advance\result@t\Yu@\v@lY%
    \global\advance\result@t\Zu@\v@lZ}#1=\result@t}
\ctr@ld@f\def\c@lprovec#1{%
    \det@rmC\v@lZa(\v@lX,\v@lY,\v@lmin,\v@lmax)%
    \det@rmC\v@lXa(\v@lY,\v@lZ,\v@lmax,\v@leur)%
    \det@rmC\v@lYa(\v@lZ,\v@lX,\v@leur,\v@lmin)%
    \Figv@ctCreg#1(\v@lXa,\v@lYa,\v@lZa)}
\ctr@ld@f\def\det@rm#1[#2,#3]{{\Figg@tXY{#2}\Figg@tXYa{#3}%
    \delt@=\repdecn@mb{\v@lX}\v@lYa\advance\delt@-\repdecn@mb{\v@lY}\v@lXa%
    \global\result@t=\delt@}#1=\result@t}
\ctr@ld@f\def\det@rmC#1(#2,#3,#4,#5){{\global\result@t=\repdecn@mb{#2}#5%
    \global\advance\result@t-\repdecn@mb{#3}#4}#1=\result@t}
\ctr@ld@f\def\getredf@ctDD#1(#2,#3){{\maxim@m{\v@lXa}{-#2}{#2}\maxim@m{\v@lYa}{-#3}{#3}%
    \maxim@m{\v@lXa}{\v@lXa}{\v@lYa}
    \ifdim\v@lXa>\@xci pt\divide\v@lXa\@xci%
    \p@rtentiere{\p@rtent}{\v@lXa}\advance\p@rtent\@ne\else\p@rtent=\@ne\fi%
    \global\result@tent=\p@rtent}#1=\result@tent\ignorespaces}
\ctr@ld@f\def\getredf@ctTD#1(#2,#3,#4){{\maxim@m{\v@lXa}{-#2}{#2}\maxim@m{\v@lYa}{-#3}{#3}%
    \maxim@m{\v@lZa}{-#4}{#4}\maxim@m{\v@lXa}{\v@lXa}{\v@lYa}%
    \maxim@m{\v@lXa}{\v@lXa}{\v@lZa}
    \ifdim\v@lXa>\@lxxiv pt\divide\v@lXa\@lxxiv%
    \p@rtentiere{\p@rtent}{\v@lXa}\advance\p@rtent\@ne\else\p@rtent=\@ne\fi%
    \global\result@tent=\p@rtent}#1=\result@tent\ignorespaces}
\ctr@ln@m\getredf@ctB
\ctr@ld@f\def\getredf@ctBDD#1{\getredf@ctDD#1(\v@lX,\v@lY)}
\ctr@ld@f\def\getredf@ctBTD#1{\getredf@ctTD#1(\v@lX,\v@lY,\v@lZ)}
\ctr@ld@f\def\FigptintercircB@zDD#1:#2:#3,#4[#5,#6,#7,#8]{{\s@uvc@ntr@l\et@tfigptintercircB@zDD%
    \setc@ntr@l{2}\figvectPDD-1[#5,#8]\Figg@tXY{-1}\getredf@ctDD\f@ctech(\v@lX,\v@lY)%
    \mili@u=#4\unit@\divide\mili@u\f@ctech\c@rre=\repdecn@mb{\mili@u}\mili@u%
    \figptBezierDD-5::#3[#5,#6,#7,#8]%
    \v@lmin=#3\p@\v@lmax=\v@lmin\advance\v@lmax0.1\p@%
    \loop\edef\T@{\repdecn@mb{\v@lmax}}\figptBezierDD-2::\T@[#5,#6,#7,#8]%
    \figvectPDD-1[-5,-2]\n@rmeucCDD{\delt@}{-1}\ifdim\delt@<\c@rre\v@lmin=\v@lmax%
    \advance\v@lmax0.1\p@\repeat%
    \loop\mili@u=\v@lmin\advance\mili@u\v@lmax%
    \divide\mili@u\tw@\edef\T@{\repdecn@mb{\mili@u}}\figptBezierDD-2::\T@[#5,#6,#7,#8]%
    \figvectPDD-1[-5,-2]\n@rmeucCDD{\delt@}{-1}\ifdim\delt@>\c@rre\v@lmax=\mili@u%
    \else\v@lmin=\mili@u\fi\v@leur=\v@lmax\advance\v@leur-\v@lmin%
    \ifdim\v@leur>\epsil@n\repeat\figptcopyDD#1:#2/-2/%
    \resetc@ntr@l\et@tfigptintercircB@zDD}\ignorespaces}
\ctr@ln@m\figptinterlines
\ctr@ld@f\def\inters@cDD#1:#2[#3,#4;#5,#6]{{\s@uvc@ntr@l\et@tinters@cDD%
    \setc@ntr@l{2}\vecunit@{-1}{#4}\vecunit@{-2}{#6}%
    \Figg@tXY{-1}\setc@ntr@l{1}\Figg@tXYa{#3}%
    \edef\A@{\repdecn@mb{\v@lX}}\edef\B@{\repdecn@mb{\v@lY}}%
    \v@lmin=\B@\v@lXa\advance\v@lmin-\A@\v@lYa%
    \Figg@tXYa{#5}\setc@ntr@l{2}\Figg@tXY{-2}%
    \edef\C@{\repdecn@mb{\v@lX}}\edef\D@{\repdecn@mb{\v@lY}}%
    \v@lmax=\D@\v@lXa\advance\v@lmax-\C@\v@lYa%
    \delt@=\A@\v@lY\advance\delt@-\B@\v@lX%
    \invers@{\v@leur}{\delt@}\edef\v@ldelta{\repdecn@mb{\v@leur}}%
    \v@lXa=\A@\v@lmax\advance\v@lXa-\C@\v@lmin%
    \v@lYa=\B@\v@lmax\advance\v@lYa-\D@\v@lmin%
    \v@lXa=\v@ldelta\v@lXa\v@lYa=\v@ldelta\v@lYa%
    \setc@ntr@l{1}\Figp@intregDD#1:{#2}(\v@lXa,\v@lYa)%
    \resetc@ntr@l\et@tinters@cDD}\ignorespaces}
\ctr@ld@f\def\inters@cTD#1:#2[#3,#4;#5,#6]{{\s@uvc@ntr@l\et@tinters@cTD%
    \setc@ntr@l{2}\figvectNVTD-1[#4,#6]\figvectNVTD-2[#6,-1]\figvectPTD-1[#3,#5]%
    \r@pPSTD\v@leur[-2,-1,#4]\edef\v@lcoef{\repdecn@mb{\v@leur}}%
    \figpttraTD#1:{#2}=#3/\v@lcoef,#4/\resetc@ntr@l\et@tinters@cTD}\ignorespaces}
\ctr@ld@f\def\r@pPSTD#1[#2,#3,#4]{{\Figg@tXY{#2}\edef\Xu@{\repdecn@mb{\v@lX}}%
    \edef\Yu@{\repdecn@mb{\v@lY}}\edef\Zu@{\repdecn@mb{\v@lZ}}%
    \Figg@tXY{#3}\v@lmin=\Xu@\v@lX\advance\v@lmin\Yu@\v@lY\advance\v@lmin\Zu@\v@lZ%
    \Figg@tXY{#4}\v@lmax=\Xu@\v@lX\advance\v@lmax\Yu@\v@lY\advance\v@lmax\Zu@\v@lZ%
    \invers@{\v@leur}{\v@lmax}\global\result@t=\repdecn@mb{\v@leur}\v@lmin}%
    #1=\result@t}
\ctr@ln@m\n@rminf
\ctr@ld@f\def\n@rminfDD#1#2{{\Figg@tXY{#2}\maxim@m{\v@lX}{\v@lX}{-\v@lX}%
    \maxim@m{\v@lY}{\v@lY}{-\v@lY}\maxim@m{\global\result@t}{\v@lX}{\v@lY}}%
    #1=\result@t}
\ctr@ld@f\def\n@rminfTD#1#2{{\Figg@tXY{#2}\maxim@m{\v@lX}{\v@lX}{-\v@lX}%
    \maxim@m{\v@lY}{\v@lY}{-\v@lY}\maxim@m{\v@lZ}{\v@lZ}{-\v@lZ}%
    \maxim@m{\v@lX}{\v@lX}{\v@lY}\maxim@m{\global\result@t}{\v@lX}{\v@lZ}}%
    #1=\result@t}
\ctr@ln@m\n@rmeucC
\ctr@ld@f\def\n@rmeucCDD#1#2{\Figg@tXY{#2}\divide\v@lX\f@ctech\divide\v@lY\f@ctech%
    #1=\repdecn@mb{\v@lX}\v@lX\v@lX=\repdecn@mb{\v@lY}\v@lY\advance#1\v@lX}
\ctr@ld@f\def\n@rmeucCTD#1#2{\Figg@tXY{#2}%
    \divide\v@lX\f@ctech\divide\v@lY\f@ctech\divide\v@lZ\f@ctech%
    #1=\repdecn@mb{\v@lX}\v@lX\v@lX=\repdecn@mb{\v@lY}\v@lY\advance#1\v@lX%
    \v@lX=\repdecn@mb{\v@lZ}\v@lZ\advance#1\v@lX}
\ctr@ln@m\n@rmeucSV
\ctr@ld@f\def\n@rmeucSVDD#1#2{{\Figg@tXY{#2}%
    \v@lXa=\repdecn@mb{\v@lX}\v@lX\v@lYa=\repdecn@mb{\v@lY}\v@lY%
    \advance\v@lXa\v@lYa\sqrt@{\global\result@t}{\v@lXa}}#1=\result@t}
\ctr@ld@f\def\n@rmeucSVTD#1#2{{\Figg@tXY{#2}\v@lXa=\repdecn@mb{\v@lX}\v@lX%
    \v@lYa=\repdecn@mb{\v@lY}\v@lY\v@lZa=\repdecn@mb{\v@lZ}\v@lZ%
    \advance\v@lXa\v@lYa\advance\v@lXa\v@lZa\sqrt@{\global\result@t}{\v@lXa}}#1=\result@t}
\ctr@ln@m\n@rmeuc
\ctr@ld@f\def\n@rmeucDD#1#2{{\Figg@tXY{#2}\getredf@ctDD\f@ctech(\v@lX,\v@lY)%
    \divide\v@lX\f@ctech\divide\v@lY\f@ctech%
    \v@lXa=\repdecn@mb{\v@lX}\v@lX\v@lYa=\repdecn@mb{\v@lY}\v@lY%
    \advance\v@lXa\v@lYa\sqrt@{\global\result@t}{\v@lXa}%
    \global\multiply\result@t\f@ctech}#1=\result@t}
\ctr@ld@f\def\n@rmeucTD#1#2{{\Figg@tXY{#2}\getredf@ctTD\f@ctech(\v@lX,\v@lY,\v@lZ)%
    \divide\v@lX\f@ctech\divide\v@lY\f@ctech\divide\v@lZ\f@ctech%
    \v@lXa=\repdecn@mb{\v@lX}\v@lX%
    \v@lYa=\repdecn@mb{\v@lY}\v@lY\v@lZa=\repdecn@mb{\v@lZ}\v@lZ%
    \advance\v@lXa\v@lYa\advance\v@lXa\v@lZa\sqrt@{\global\result@t}{\v@lXa}%
    \global\multiply\result@t\f@ctech}#1=\result@t}
\ctr@ln@m\vecunit@
\ctr@ld@f\def\vecunit@DD#1#2{{\Figg@tXY{#2}\getredf@ctDD\f@ctech(\v@lX,\v@lY)%
    \divide\v@lX\f@ctech\divide\v@lY\f@ctech%
    \Figv@ctCreg#1(\v@lX,\v@lY)\n@rmeucSV{\v@lYa}{#1}%
    \invers@{\v@lXa}{\v@lYa}\edef\v@lv@lXa{\repdecn@mb{\v@lXa}}%
    \v@lX=\v@lv@lXa\v@lX\v@lY=\v@lv@lXa\v@lY%
    \Figv@ctCreg#1(\v@lX,\v@lY)\multiply\v@lYa\f@ctech\global\result@t=\v@lYa}}
\ctr@ld@f\def\vecunit@TD#1#2{{\Figg@tXY{#2}\getredf@ctTD\f@ctech(\v@lX,\v@lY,\v@lZ)%
    \divide\v@lX\f@ctech\divide\v@lY\f@ctech\divide\v@lZ\f@ctech%
    \Figv@ctCreg#1(\v@lX,\v@lY,\v@lZ)\n@rmeucSV{\v@lYa}{#1}%
    \invers@{\v@lXa}{\v@lYa}\edef\v@lv@lXa{\repdecn@mb{\v@lXa}}%
    \v@lX=\v@lv@lXa\v@lX\v@lY=\v@lv@lXa\v@lY\v@lZ=\v@lv@lXa\v@lZ%
    \Figv@ctCreg#1(\v@lX,\v@lY,\v@lZ)\multiply\v@lYa\f@ctech\global\result@t=\v@lYa}}
\ctr@ld@f\def\vecunitC@TD[#1,#2]{\Figg@tXYa{#1}\Figg@tXY{#2}%
    \advance\v@lX-\v@lXa\advance\v@lY-\v@lYa\advance\v@lZ-\v@lZa\c@lvecunitTD}
\ctr@ld@f\def\vecunitCV@TD#1{\Figg@tXY{#1}\c@lvecunitTD}
\ctr@ld@f\def\c@lvecunitTD{\getredf@ctTD\f@ctech(\v@lX,\v@lY,\v@lZ)%
    \divide\v@lX\f@ctech\divide\v@lY\f@ctech\divide\v@lZ\f@ctech%
    \v@lXa=\repdecn@mb{\v@lX}\v@lX%
    \v@lYa=\repdecn@mb{\v@lY}\v@lY\v@lZa=\repdecn@mb{\v@lZ}\v@lZ%
    \advance\v@lXa\v@lYa\advance\v@lXa\v@lZa\sqrt@{\v@lYa}{\v@lXa}%
    \invers@{\v@lXa}{\v@lYa}\edef\v@lv@lXa{\repdecn@mb{\v@lXa}}%
    \v@lX=\v@lv@lXa\v@lX\v@lY=\v@lv@lXa\v@lY\v@lZ=\v@lv@lXa\v@lZ}
\ctr@ln@m\figgetangle
\ctr@ld@f\def\figgetangleDD#1[#2,#3,#4]{\ifGR@cri{\s@uvc@ntr@l\et@tfiggetangleDD\setc@ntr@l{2}%
    \figvectPDD-1[#2,#3]\figvectPDD-2[#2,#4]\vecunit@{-1}{-1}%
    \c@lproscalDD\delt@[-2,-1]\figvectNVDD-1[-1]\c@lproscalDD\v@leur[-2,-1]%
    \arct@n\v@lmax(\delt@,\v@leur)\v@lmax=\rdT@deg\v@lmax%
    \ifdim\v@lmax<\z@\advance\v@lmax\DePI@deg\fi\xdef#1{\repdecn@mb{\v@lmax}}%
    \resetc@ntr@l\et@tfiggetangleDD}\ignorespaces\fi}
\ctr@ld@f\def\figgetangleTD#1[#2,#3,#4,#5]{\ifGR@cri{\s@uvc@ntr@l\et@tfiggetangleTD\setc@ntr@l{2}%
    \figvectPTD-1[#2,#3]\figvectPTD-2[#2,#5]\figvectNVTD-3[-1,-2]%
    \figvectPTD-2[#2,#4]\figvectNVTD-4[-3,-1]%
    \vecunit@{-1}{-1}\c@lproscalTD\delt@[-2,-1]\c@lproscalTD\v@leur[-2,-4]%
    \arct@n\v@lmax(\delt@,\v@leur)\v@lmax=\rdT@deg\v@lmax%
    \ifdim\v@lmax<\z@\advance\v@lmax\DePI@deg\fi\xdef#1{\repdecn@mb{\v@lmax}}%
    \resetc@ntr@l\et@tfiggetangleTD}\ignorespaces\fi}    
\ctr@ld@f\def\figgetdist#1[#2,#3]{\ifGR@cri{\s@uvc@ntr@l\et@tfiggetdist\setc@ntr@l{2}%
    \figvectP-1[#2,#3]\n@rmeuc{\v@lX}{-1}\v@lX=\ptT@unit@\v@lX\xdef#1{\repdecn@mb{\v@lX}}%
    \resetc@ntr@l\et@tfiggetdist}\ignorespaces\fi}
\ctr@ld@f\def\figget#1=#2[#3]{\keln@mun#1|%
    \def\n@mref{a}\ifx\l@debut\n@mref\figgetangle#2[#3]\else
    \def\n@mref{d}\ifx\l@debut\n@mref\figgetdist#2[#3]\else
    \W@rnmeskwd{figget}{#1}\fi\fi\ignorespaces}
\ctr@ld@f\def\Figg@tT#1{\c@ntr@lnum{#1}%
    {\expandafter\expandafter\expandafter\extr@ctT\csname\objc@de\endcsname:%
     \ifnum\B@@ltxt=\z@\ptn@me{#1}\else\csname\objc@de T\endcsname\fi}}
\ctr@ld@f\def\extr@ctT#1,#2,#3/#4:{\def\B@@ltxt{#3}}
\ctr@ld@f\def\Figg@tXY#1{\c@ntr@lnum{#1}%
    \expandafter\expandafter\expandafter\extr@ctC\csname\objc@de\endcsname:}
\ctr@ln@m\extr@ctC
\ctr@ld@f\def\extr@ctCDD#1/#2,#3,#4:{\v@lX=#2\v@lY=#3}
\ctr@ld@f\def\extr@ctCTD#1/#2,#3,#4:{\v@lX=#2\v@lY=#3\v@lZ=#4}
\ctr@ld@f\def\Figg@tXYa#1{\c@ntr@lnum{#1}%
    \expandafter\expandafter\expandafter\extr@ctCa\csname\objc@de\endcsname:}
\ctr@ln@m\extr@ctCa
\ctr@ld@f\def\extr@ctCaDD#1/#2,#3,#4:{\v@lXa=#2\v@lYa=#3}
\ctr@ld@f\def\extr@ctCaTD#1/#2,#3,#4:{\v@lXa=#2\v@lYa=#3\v@lZa=#4}
\ctr@ln@m\t@xt@
\ctr@ld@f\def\figinit#1{\t@stc@tcodech@nge\initpr@lim\Figinit@#1,:\initpss@ttings\ignorespaces}
\ctr@ld@f\def\Figinit@#1,#2:{\setunit@{#1}\def\t@xt@{#2}\ifx\t@xt@\empty\else\Figinit@@#2:\fi}
\ctr@ld@f\def\Figinit@@#1#2:{\if#12 \else\Figs@tproj{#1}\initTD@\fi}
\ctr@ln@w{newif}\ifTr@isDim
\ctr@ld@f\def\UnD@fined{UNDEFINED}
\ctr@ln@m\@utoFN
\ctr@ln@m\@utoFInDone
\ctr@ln@m\disob@unit
\ctr@ld@f\def\initpr@lim{\initb@undb@x\figsetmark{}\figsetptname{$A_{##1}$}\def\Sc@leFact{1}%
    \initDD@\figsetroundcoord{yes}\GR@critrue\expandafter\setupd@te\D@FTupdate:%
    \edef\disob@unit{\UnD@fined}\edef\t@rgetpt{\UnD@fined}\gdef\@utoFInDone{1}\gdef\@utoFN{0}}
\ctr@ld@f\def\initDD@{\Tr@isDimfalse%
    \ifPDFm@ke%
     \let\Ps@rcerc=\Ps@rcercBz%
     \let\Ps@rell=\Ps@rellBz%
    \fi
    \let\c@lDCUn=\c@lDCUnDD%
    \let\c@lDCDeux=\c@lDCDeuxDD%
    \let\c@ldefproj=\relax%
    \let\c@lproscal=\c@lproscalDD%
    \let\c@lprojSP=\relax%
    \let\extr@ctC=\extr@ctCDD%
    \let\extr@ctCa=\extr@ctCaDD%
    \let\extr@ctCF=\extr@ctCFDD%
    \let\Figp@intreg=\Figp@intregDD%
    \let\Figpts@xes=\Figpts@xesDD%
    \let\getredf@ctB=\getredf@ctBDD%
    \let\n@rmeucSV=\n@rmeucSVDD\let\n@rmeuc=\n@rmeucDD\let\n@rmeucC\n@rmeucCDD\let\n@rminf=\n@rminfDD%
    \let\pr@dMatV=\pr@dMatVDD%
    \let\Q@@xes=\Q@@xesDD%
    \let\vecunit@=\vecunit@DD%
    \let\figcoord=\figcoordDD%
    \let\figgetangle=\figgetangleDD%
    \let\figpt=\figptDD%
    \let\figptBezier=\figptBezierDD%
    \let\figptbary=\figptbaryDD%
    \let\figptcirc=\figptcircDD%
    \let\figptcircumcenter=\figptcircumcenterDD%
    \let\figptcopy=\figptcopyDD%
    \let\figptcurvcenter=\figptcurvcenterDD%
    \let\figptell=\figptellDD%
    \let\figptendnormal=\figptendnormalDD%
    \let\figptinterlineplane=\figptinterlineplaneDD%
    \let\figptinterlines=\inters@cDD%
    \let\figptorthocenter=\figptorthocenterDD%
    \let\figptorthoprojline=\figptorthoprojlineDD%
    \let\figptorthoprojplane=\figptorthoprojplaneDD%
    \let\figptrot=\figptrotDD%
    \let\figptscontrol=\figptscontrolDD%
    \let\figptsintercirc=\figptsintercircDD%
    \let\figptsinterlinell=\figptsinterlinellDD%
    \let\figptsorthoprojline=\figptsorthoprojlineDD%
    \let\figptorthoprojplane=\figptorthoprojplaneDD%
    \let\figptsrot=\figptsrotDD%
    \let\figptssym=\figptssymDD%
    \let\figptstra=\figptstraDD%
    \let\figptsym=\figptsymDD%
    \let\figpttraC=\figpttraCDD%
    \let\figpttra=\figpttraDD%
    \let\figptvisilimSL=\figptvisilimSLDD%
    \let\figsetobdist=\figsetobdistDD%
    \let\figsettarget=\figsettargetDD%
    \let\figsetview=\figsetviewDD%
    \let\figvectDBezier=\figvectDBezierDD%
    \let\figvectN=\figvectNDD%
    \let\figvectNV=\figvectNVDD%
    \let\figvectP=\figvectPDD%
    \let\figvectU=\figvectUDD%
    \let\figdrawarccircP=\Q@arccircPDD%
    \let\figdrawarccirc=\Q@arccircDD%
    \let\figdrawarcell=\Q@arcellDD%
    \let\figdrawarcellPA=\Q@arcellPADD%
    \let\figdrawarrowBezier=\Q@arrowBezierDD%
    \let\figdrawarrowcircP=\Q@arrowcircPDD%
    \let\figdrawarrowcirc=\Q@arrowcircDD%
    \let\figdrawarrowhead=\Q@arrowheadDD%
    \let\figdrawarrow=\Q@arrowDD%
    \let\figdrawBezier=\Q@BezierDD%
    \let\figdrawcirc=\Q@circDD%
    \let\figdrawcurve=\Q@curveDD%
    \let\figdrawnormal=\Q@normalDD%
    }
\ctr@ld@f\def\initTD@{\Tr@isDimtrue\initb@undb@xTD\newt@rgetptfalse\newdis@bfalse%
    \let\c@lDCUn=\c@lDCUnTD%
    \let\c@lDCDeux=\c@lDCDeuxTD%
    \let\c@ldefproj=\c@ldefprojTD%
    \let\c@lproscal=\c@lproscalTD%
    \let\extr@ctC=\extr@ctCTD%
    \let\extr@ctCa=\extr@ctCaTD%
    \let\extr@ctCF=\extr@ctCFTD%
    \let\Figp@intreg=\Figp@intregTD%
    \let\Figpts@xes=\Figpts@xesTD%
    \let\getredf@ctB=\getredf@ctBTD%
    \let\n@rmeucSV=\n@rmeucSVTD\let\n@rmeuc=\n@rmeucTD\let\n@rmeucC\n@rmeucCTD\let\n@rminf=\n@rminfTD%
    \let\pr@dMatV=\pr@dMatVTD%
    \let\Q@@xes=\Q@@xesTD%
    \let\vecunit@=\vecunit@TD%
    \let\figcoord=\figcoordTD%
    \let\figgetangle=\figgetangleTD%
    \let\figpt=\figptTD%
    \let\figptBezier=\figptBezierTD%
    \let\figptbary=\figptbaryTD%
    \let\figptcirc=\figptcircTD%
    \let\figptcircumcenter=\figptcircumcenterTD%
    \let\figptcopy=\figptcopyTD%
    \let\figptcurvcenter=\figptcurvcenterTD%
    \let\figptinterlineplane=\figptinterlineplaneTD%
    \let\figptinterlines=\inters@cTD%
    \let\figptorthocenter=\figptorthocenterTD%
    \let\figptorthoprojline=\figptorthoprojlineTD%
    \let\figptorthoprojplane=\figptorthoprojplaneTD%
    \let\figptrot=\figptrotTD%
    \let\figptscontrol=\figptscontrolTD%
    \let\figptsintercirc=\figptsintercircTD%
    \let\figptsorthoprojline=\figptsorthoprojlineTD%
    \let\figptsorthoprojplane=\figptsorthoprojplaneTD%
    \let\figptsrot=\figptsrotTD%
    \let\figptssym=\figptssymTD%
    \let\figptstra=\figptstraTD%
    \let\figptsym=\figptsymTD%
    \let\figpttraC=\figpttraCTD%
    \let\figpttra=\figpttraTD%
    \let\figptvisilimSL=\figptvisilimSLTD%
    \let\figsetobdist=\figsetobdistTD%
    \let\figsettarget=\figsettargetTD%
    \let\figsetview=\figsetviewTD%
    \let\figvectDBezier=\figvectDBezierTD%
    \let\figvectN=\figvectNTD%
    \let\figvectNV=\figvectNVTD%
    \let\figvectP=\figvectPTD%
    \let\figvectU=\figvectUTD%
    \let\figdrawarccircP=\Q@arccircPTD%
    \let\figdrawarccirc=\Q@arccircTD%
    \let\figdrawarcell=\Q@arcellTD%
    \let\figdrawarcellPA=\Q@arcellPATD%
    \let\figdrawarrowBezier=\Q@arrowBezierTD%
    \let\figdrawarrowcircP=\Q@arrowcircPTD%
    \let\figdrawarrowcirc=\Q@arrowcircTD%
    \let\figdrawarrowhead=\Q@arrowheadTD%
    \let\figdrawarrow=\Q@arrowTD%
    \let\figdrawBezier=\Q@BezierTD%
    \let\figdrawcirc=\Q@circTD%
    \let\figdrawcurve=\Q@curveTD%
    }
\ctr@ld@f\def\un@v@ilable#1{\immediate\write16{*** The macro #1 is not available in the current context.}}
\ctr@ld@f\def\figinsert#1{{\def\t@xt@{#1}\relax%
    \ifx\t@xt@\empty\ifnum\@utoFInDone>\z@\Figinsert@\DefGIfilen@me,:\fi%
    \else\expandafter\FiginsertNu@#1 :\fi}\ignorespaces}
\ctr@ld@f\def\FiginsertNu@#1 #2:{\def\t@xt@{#1}\relax\ifx\t@xt@\empty\def\t@xt@{#2}%
    \ifx\t@xt@\empty\ifnum\@utoFInDone>\z@\Figinsert@\DefGIfilen@me,:\fi%
    \else\FiginsertNu@#2:\fi\else\expandafter\FiginsertNd@#1 #2:\fi}
\ctr@ld@f\def\FiginsertNd@#1#2:{\ifcat#1a\Figinsert@#1#2,:\else%
    \ifnum\@utoFInDone>\z@\Figinsert@\DefGIfilen@me,#1#2,:\fi\fi}
\ctr@ln@m\Sc@leFact
\ctr@ld@f\def\Figinsert@#1,#2:{\def\t@xt@{#2}\ifx\t@xt@\empty\xdef\Sc@leFact{1}\else%
    \X@rgdeux@#2\xdef\Sc@leFact{\@rgdeux}\fi%
    \Figdisc@rdLTS{#1}{\t@xt@}\@psfgetbb{\t@xt@}%
    \v@lX=\@psfllx\p@\v@lX=\ptpsT@pt\v@lX\v@lX=\Sc@leFact\v@lX%
    \v@lY=\@psflly\p@\v@lY=\ptpsT@pt\v@lY\v@lY=\Sc@leFact\v@lY%
    \b@undb@x{\v@lX}{\v@lY}%
    \v@lX=\@psfurx\p@\v@lX=\ptpsT@pt\v@lX\v@lX=\Sc@leFact\v@lX%
    \v@lY=\@psfury\p@\v@lY=\ptpsT@pt\v@lY\v@lY=\Sc@leFact\v@lY%
    \b@undb@x{\v@lX}{\v@lY}%
    \ifPDFm@ke\Figinclud@PDF{\t@xt@}{\Sc@leFact}\else%
    \v@lX=\c@nt pt\v@lX=\Sc@leFact\v@lX\edef\F@ct{\repdecn@mb{\v@lX}}%
    \ifx\TeXturesonMacOSltX\special{postscriptfile #1 vscale=\F@ct\space hscale=\F@ct}%
    \else\includegraphics{#1}\fi\fi%
    \message{[\t@xt@]}\ignorespaces}
\ctr@ld@f\def\Figdisc@rdLTS#1#2{\expandafter\Figdisc@rdLTS@#1 :#2}
\ctr@ld@f\def\Figdisc@rdLTS@#1 #2:#3{\def#3{#1}\relax\ifx#3\empty\expandafter\Figdisc@rdLTS@#2:#3\fi}
\ctr@ld@f\def\figinsertE#1{\FiginsertE@#1,:\ignorespaces}
\ctr@ld@f\def\FiginsertE@#1,#2:{{\def\t@xt@{#2}\ifx\t@xt@\empty\xdef\Sc@leFact{1}\else%
    \X@rgdeux@#2\xdef\Sc@leFact{\@rgdeux}\fi%
    \Figdisc@rdLTS{#1}{\t@xt@}\pdfximage{\t@xt@}%
    \setbox\Gb@x=\hbox{\pdfrefximage\pdflastximage}%
    \v@lX=\z@\v@lY=-\Sc@leFact\dp\Gb@x\b@undb@x{\v@lX}{\v@lY}%
    \advance\v@lX\Sc@leFact\wd\Gb@x\advance\v@lY\Sc@leFact\dp\Gb@x%
    \advance\v@lY\Sc@leFact\ht\Gb@x\b@undb@x{\v@lX}{\v@lY}%
    \v@lX=\Sc@leFact\wd\Gb@x\pdfximage width \v@lX {\t@xt@}%
    \rlap{\pdfrefximage\pdflastximage}\message{[\t@xt@]}}\ignorespaces}
\ctr@ld@f\def\X@rgdeux@#1,{\edef\@rgdeux{#1}}
\ctr@ln@m\figpt
\ctr@ld@f\def\figptDD#1:#2(#3,#4){\ifGR@cri\c@ntr@lnum{#1}%
    {\v@lX=#3\unit@\v@lY=#4\unit@\Fig@dmpt{#2}{\z@}}\ignorespaces\fi}
\ctr@ld@f\def\Fig@dmpt#1#2{\def\t@xt@{#1}\ifx\t@xt@\empty\def\B@@ltxt{\z@}%
    \else\expandafter\gdef\csname\objc@de T\endcsname{#1}\def\B@@ltxt{\@ne}\fi%
    \expandafter\xdef\csname\objc@de\endcsname{\ifitis@vect@r\C@dCl@svect%
    \else\C@dCl@spt\fi,\z@,\B@@ltxt/\the\v@lX,\the\v@lY,#2}}
\ctr@ld@f\def\C@dCl@spt{P}
\ctr@ld@f\def\C@dCl@svect{V}
\ctr@ln@m\c@@rdYZ
\ctr@ln@m\c@@rdY
\ctr@ld@f\def\figptTD#1:#2(#3,#4){\ifGR@cri\c@ntr@lnum{#1}%
    \def\c@@rdYZ{#4,0,0}\extrairelepremi@r\c@@rdY\de\c@@rdYZ%
    \extrairelepremi@r\c@@rdZ\de\c@@rdYZ%
    {\v@lX=#3\unit@\v@lY=\c@@rdY\unit@\v@lZ=\c@@rdZ\unit@\Fig@dmpt{#2}{\the\v@lZ}%
    \b@undb@xTD{\v@lX}{\v@lY}{\v@lZ}}\ignorespaces\fi}
\ctr@ln@m\Figp@intreg
\ctr@ld@f\def\Figp@intregDD#1:#2(#3,#4){\c@ntr@lnum{#1}%
    {\result@t=#4\v@lX=#3\v@lY=\result@t\Fig@dmpt{#2}{\z@}}\ignorespaces}
\ctr@ld@f\def\Figp@intregTD#1:#2(#3,#4){\c@ntr@lnum{#1}%
    \def\c@@rdYZ{#4,\z@,\z@}\extrairelepremi@r\c@@rdY\de\c@@rdYZ%
    \extrairelepremi@r\c@@rdZ\de\c@@rdYZ%
    {\v@lX=#3\v@lY=\c@@rdY\v@lZ=\c@@rdZ\Fig@dmpt{#2}{\the\v@lZ}%
    \b@undb@xTD{\v@lX}{\v@lY}{\v@lZ}}\ignorespaces}
\ctr@ln@m\figptBezier
\ctr@ld@f\def\figptBezierDD#1:#2:#3[#4,#5,#6,#7]{\ifGR@cri{\s@uvc@ntr@l\et@tfigptBezierDD%
    \FigptBezier@#3[#4,#5,#6,#7]\Figp@intregDD#1:{#2}(\v@lX,\v@lY)%
    \resetc@ntr@l\et@tfigptBezierDD}\ignorespaces\fi}
\ctr@ld@f\def\figptBezierTD#1:#2:#3[#4,#5,#6,#7]{\ifGR@cri{\s@uvc@ntr@l\et@tfigptBezierTD%
    \FigptBezier@#3[#4,#5,#6,#7]\Figp@intregTD#1:{#2}(\v@lX,\v@lY,\v@lZ)%
    \resetc@ntr@l\et@tfigptBezierTD}\ignorespaces\fi}
\ctr@ld@f\def\FigptBezier@#1[#2,#3,#4,#5]{\setc@ntr@l{2}%
    \edef\T@{#1}\v@leur=\p@\advance\v@leur-#1pt\edef\UNmT@{\repdecn@mb{\v@leur}}%
    \figptcopy-4:/#2/\figptcopy-3:/#3/\figptcopy-2:/#4/\figptcopy-1:/#5/%
    \l@mbd@un=-4 \l@mbd@de=-\thr@@\p@rtent=\m@ne\c@lDecast%
    \l@mbd@un=-4 \l@mbd@de=-\thr@@\p@rtent=-\tw@\c@lDecast%
    \l@mbd@un=-4 \l@mbd@de=-\thr@@\p@rtent=-\thr@@\c@lDecast\Figg@tXY{-4}}
\ctr@ln@m\c@lDCUn
\ctr@ld@f\def\c@lDCUnDD#1#2{\Figg@tXY{#1}\v@lX=\UNmT@\v@lX\v@lY=\UNmT@\v@lY%
    \Figg@tXYa{#2}\advance\v@lX\T@\v@lXa\advance\v@lY\T@\v@lYa%
    \Figp@intregDD#1:(\v@lX,\v@lY)}
\ctr@ld@f\def\c@lDCUnTD#1#2{\Figg@tXY{#1}\v@lX=\UNmT@\v@lX\v@lY=\UNmT@\v@lY\v@lZ=\UNmT@\v@lZ%
    \Figg@tXYa{#2}\advance\v@lX\T@\v@lXa\advance\v@lY\T@\v@lYa\advance\v@lZ\T@\v@lZa%
    \Figp@intregTD#1:(\v@lX,\v@lY,\v@lZ)}
\ctr@ld@f\def\c@lDecast{\relax\ifnum\l@mbd@un<\p@rtent\c@lDCUn{\l@mbd@un}{\l@mbd@de}%
    \advance\l@mbd@un\@ne\advance\l@mbd@de\@ne\c@lDecast\fi}
\ctr@ld@f\def\figptmap#1:#2=#3/#4/#5/{\ifGR@cri{\s@uvc@ntr@l\et@tfigptmap%
    \setc@ntr@l{2}\figvectP-1[#4,#3]\Figg@tXY{-1}%
    \pr@dMatV/#5/\figpttra#1:{#2}=#4/1,-1/%
    \resetc@ntr@l\et@tfigptmap}\ignorespaces\fi}
\ctr@ln@m\pr@dMatV
\ctr@ld@f\def\pr@dMatVDD/#1,#2;#3,#4/{\v@lXa=#1\v@lX\advance\v@lXa#2\v@lY%
    \v@lYa=#3\v@lX\advance\v@lYa#4\v@lY\Figv@ctCreg-1(\v@lXa,\v@lYa)}
\ctr@ld@f\def\pr@dMatVTD/#1,#2,#3;#4,#5,#6;#7,#8,#9/{%
    \v@lXa=#1\v@lX\advance\v@lXa#2\v@lY\advance\v@lXa#3\v@lZ%
    \v@lYa=#4\v@lX\advance\v@lYa#5\v@lY\advance\v@lYa#6\v@lZ%
    \v@lZa=#7\v@lX\advance\v@lZa#8\v@lY\advance\v@lZa#9\v@lZ%
    \Figv@ctCreg-1(\v@lXa,\v@lYa,\v@lZa)}
\ctr@ln@m\figptbary
\ctr@ld@f\def\figptbaryDD#1:#2[#3;#4]{\ifGR@cri{\edef\list@num{#3}\extrairelepremi@r\p@int\de\list@num%
    \s@mme=\z@\@ecfor\c@ef:=#4\do{\advance\s@mme\c@ef}%
    \edef\listec@ef{#4,0}\extrairelepremi@r\c@ef\de\listec@ef%
    \Figg@tXY{\p@int}\divide\v@lX\s@mme\divide\v@lY\s@mme%
    \multiply\v@lX\c@ef\multiply\v@lY\c@ef%
    \@ecfor\p@int:=\list@num\do{\extrairelepremi@r\c@ef\de\listec@ef%
           \Figg@tXYa{\p@int}\divide\v@lXa\s@mme\divide\v@lYa\s@mme%
           \multiply\v@lXa\c@ef\multiply\v@lYa\c@ef%
           \advance\v@lX\v@lXa\advance\v@lY\v@lYa}%
    \Figp@intregDD#1:{#2}(\v@lX,\v@lY)}\ignorespaces\fi}
\ctr@ld@f\def\figptbaryTD#1:#2[#3;#4]{\ifGR@cri{\edef\list@num{#3}\extrairelepremi@r\p@int\de\list@num%
    \s@mme=\z@\@ecfor\c@ef:=#4\do{\advance\s@mme\c@ef}%
    \edef\listec@ef{#4,0}\extrairelepremi@r\c@ef\de\listec@ef%
    \Figg@tXY{\p@int}\divide\v@lX\s@mme\divide\v@lY\s@mme\divide\v@lZ\s@mme%
    \multiply\v@lX\c@ef\multiply\v@lY\c@ef\multiply\v@lZ\c@ef%
    \@ecfor\p@int:=\list@num\do{\extrairelepremi@r\c@ef\de\listec@ef%
           \Figg@tXYa{\p@int}\divide\v@lXa\s@mme\divide\v@lYa\s@mme\divide\v@lZa\s@mme%
           \multiply\v@lXa\c@ef\multiply\v@lYa\c@ef\multiply\v@lZa\c@ef%
           \advance\v@lX\v@lXa\advance\v@lY\v@lYa\advance\v@lZ\v@lZa}%
    \Figp@intregTD#1:{#2}(\v@lX,\v@lY,\v@lZ)}\ignorespaces\fi}
\ctr@ld@f\def\figptbaryR#1:#2[#3;#4]{\ifGR@cri{%
    \v@leur=\z@\@ecfor\c@ef:=#4\do{\maxim@m{\v@lmax}{\c@ef pt}{-\c@ef pt}%
    \ifdim\v@lmax>\v@leur\v@leur=\v@lmax\fi}%
    \ifdim\v@leur<\p@\f@ctech=\@M\else\ifdim\v@leur<\t@n\p@\f@ctech=\@m\else%
    \ifdim\v@leur<\c@nt\p@\f@ctech=\c@nt\else\ifdim\v@leur<\@m\p@\f@ctech=\t@n\else%
    \f@ctech=\@ne\fi\fi\fi\fi%
    \def\listec@ef{0}%
    \@ecfor\c@ef:=#4\do{\sc@lec@nvRI{\c@ef pt}\edef\listec@ef{\listec@ef,\the\s@mme}}%
    \extrairelepremi@r\c@ef\de\listec@ef\figptbary#1:#2[#3;\listec@ef]}\ignorespaces\fi}
\ctr@ld@f\def\sc@lec@nvRI#1{\v@leur=#1\p@rtentiere{\s@mme}{\v@leur}\advance\v@leur-\s@mme\p@%
    \multiply\v@leur\f@ctech\p@rtentiere{\p@rtent}{\v@leur}%
    \multiply\s@mme\f@ctech\advance\s@mme\p@rtent}
\ctr@ln@m\figptcirc
\ctr@ld@f\def\figptcircDD#1:#2:#3;#4(#5){\ifGR@cri{\s@uvc@ntr@l\et@tfigptcircDD%
    \c@lptellDD#1:{#2}:#3;#4,#4(#5)\resetc@ntr@l\et@tfigptcircDD}\ignorespaces\fi}
\ctr@ld@f\def\figptcircTD#1:#2:#3,#4,#5;#6(#7){\ifGR@cri{\s@uvc@ntr@l\et@tfigptcircTD%
    \setc@ntr@l{2}\c@lExtAxes#3,#4,#5(#6)\figptellP#1:{#2}:#3,-4,-5(#7)%
    \resetc@ntr@l\et@tfigptcircTD}\ignorespaces\fi}
\ctr@ln@m\figptcircumcenter
\ctr@ld@f\def\figptcircumcenterDD#1:#2[#3,#4,#5]{\ifGR@cri{\s@uvc@ntr@l\et@tfigptcircumcenterDD%
    \setc@ntr@l{2}\figvectNDD-5[#3,#4]\figptbaryDD-3:[#3,#4;1,1]%
                  \figvectNDD-6[#4,#5]\figptbaryDD-4:[#4,#5;1,1]%
    \resetc@ntr@l{2}\inters@cDD#1:{#2}[-3,-5;-4,-6]%
    \resetc@ntr@l\et@tfigptcircumcenterDD}\ignorespaces\fi}
\ctr@ld@f\def\figptcircumcenterTD#1:#2[#3,#4,#5]{\ifGR@cri{\s@uvc@ntr@l\et@tfigptcircumcenterTD%
    \setc@ntr@l{2}\figvectNTD-1[#3,#4,#5]%
    \figvectPTD-3[#3,#4]\figvectNVTD-5[-1,-3]\figptbaryTD-3:[#3,#4;1,1]%
    \figvectPTD-4[#4,#5]\figvectNVTD-6[-1,-4]\figptbaryTD-4:[#4,#5;1,1]%
    \resetc@ntr@l{2}\inters@cTD#1:{#2}[-3,-5;-4,-6]%
    \resetc@ntr@l\et@tfigptcircumcenterTD}\ignorespaces\fi}
\ctr@ln@m\figptcopy
\ctr@ld@f\def\figptcopyDD#1:#2/#3/{\ifGR@cri{\Figg@tXY{#3}%
    \Figp@intregDD#1:{#2}(\v@lX,\v@lY)}\ignorespaces\fi}
\ctr@ld@f\def\figptcopyTD#1:#2/#3/{\ifGR@cri{\Figg@tXY{#3}%
    \Figp@intregTD#1:{#2}(\v@lX,\v@lY,\v@lZ)}\ignorespaces\fi}
\ctr@ln@m\figptcurvcenter
\ctr@ld@f\def\figptcurvcenterDD#1:#2:#3[#4,#5,#6,#7]{\ifGR@cri{\s@uvc@ntr@l\et@tfigptcurvcenterDD%
    \setc@ntr@l{2}\c@lcurvradDD#3[#4,#5,#6,#7]\edef\Sprim@{\repdecn@mb{\result@t}}%
    \figptBezierDD-1::#3[#4,#5,#6,#7]\figpttraDD#1:{#2}=-1/\Sprim@,-5/%
    \resetc@ntr@l\et@tfigptcurvcenterDD}\ignorespaces\fi}
\ctr@ld@f\def\figptcurvcenterTD#1:#2:#3[#4,#5,#6,#7]{\ifGR@cri{\s@uvc@ntr@l\et@tfigptcurvcenterTD%
    \setc@ntr@l{2}\figvectDBezierTD -5:1,#3[#4,#5,#6,#7]%
    \figvectDBezierTD -6:2,#3[#4,#5,#6,#7]\vecunit@TD{-5}{-5}%
    \edef\Sprim@{\repdecn@mb{\result@t}}\figvectNVTD-1[-6,-5]%
    \figvectNVTD-5[-5,-1]\c@lproscalTD\v@leur[-6,-5]%
    \invers@{\v@leur}{\v@leur}\v@leur=\Sprim@\v@leur\v@leur=\Sprim@\v@leur%
    \figptBezierTD-1::#3[#4,#5,#6,#7]\edef\Sprim@{\repdecn@mb{\v@leur}}%
    \figpttraTD#1:{#2}=-1/\Sprim@,-5/\resetc@ntr@l\et@tfigptcurvcenterTD}\ignorespaces\fi}
\ctr@ld@f\def\c@lcurvradDD#1[#2,#3,#4,#5]{{\figvectDBezierDD -5:1,#1[#2,#3,#4,#5]%
    \figvectDBezierDD -6:2,#1[#2,#3,#4,#5]\vecunit@DD{-5}{-5}%
    \edef\Sprim@{\repdecn@mb{\result@t}}\figvectNVDD-5[-5]\c@lproscalDD\v@leur[-6,-5]%
    \invers@{\v@leur}{\v@leur}\v@leur=\Sprim@\v@leur\v@leur=\Sprim@\v@leur%
    \global\result@t=\v@leur}}
\ctr@ln@m\figptell
\ctr@ld@f\def\figptellDD#1:#2:#3;#4,#5(#6,#7){\ifGR@cri{\s@uvc@ntr@l\et@tfigptell%
    \c@lptellDD#1::#3;#4,#5(#6)\figptrotDD#1:{#2}=#1/#3,#7/%
    \resetc@ntr@l\et@tfigptell}\ignorespaces\fi}
\ctr@ld@f\def\c@lptellDD#1:#2:#3;#4,#5(#6){\c@ssin{\C@}{\S@}{#6}\v@lmin=\C@ pt\v@lmax=\S@ pt%
    \v@lmin=#4\v@lmin\v@lmax=#5\v@lmax%
    \edef\Xc@mp{\repdecn@mb{\v@lmin}}\edef\Yc@mp{\repdecn@mb{\v@lmax}}%
    \setc@ntr@l{2}\figvectC-1(\Xc@mp,\Yc@mp)\figpttraDD#1:{#2}=#3/1,-1/}
\ctr@ld@f\def\figptellP#1:#2:#3,#4,#5(#6){\ifGR@cri{\s@uvc@ntr@l\et@tfigptellP%
    \setc@ntr@l{2}\figvectP-1[#3,#4]\figvectP-2[#3,#5]%
    \v@leur=#6pt\c@lptellP{#3}{-1}{-2}\figptcopy#1:{#2}/-3/%
    \resetc@ntr@l\et@tfigptellP}\ignorespaces\fi}
\ctr@ln@m\@ngle
\ctr@ld@f\def\c@lptellP#1#2#3{\edef\@ngle{\repdecn@mb\v@leur}\c@ssin{\C@}{\S@}{\@ngle}%
    \figpttra-3:=#1/\C@,#2/\figpttra-3:=-3/\S@,#3/}
\ctr@ln@m\figptendnormal
\ctr@ld@f\def\figptendnormalDD#1:#2:#3,#4[#5,#6]{\ifGR@cri{\s@uvc@ntr@l\et@tfigptendnormal%
    \Figg@tXYa{#5}\Figg@tXY{#6}%
    \advance\v@lX-\v@lXa\advance\v@lY-\v@lYa%
    \setc@ntr@l{2}\Figv@ctCreg-1(\v@lX,\v@lY)\vecunit@{-1}{-1}\Figg@tXY{-1}%
    \delt@=#3\unit@\maxim@m{\delt@}{\delt@}{-\delt@}\edef\l@ngueur{\repdecn@mb{\delt@}}%
    \v@lX=\l@ngueur\v@lX\v@lY=\l@ngueur\v@lY%
    \delt@=\p@\advance\delt@-#4pt\edef\l@ngueur{\repdecn@mb{\delt@}}%
    \figptbaryR-1:[#5,#6;#4,\l@ngueur]\Figg@tXYa{-1}%
    \advance\v@lXa\v@lY\advance\v@lYa-\v@lX%
    \setc@ntr@l{1}\Figp@intregDD#1:{#2}(\v@lXa,\v@lYa)\resetc@ntr@l\et@tfigptendnormal}%
    \ignorespaces\fi}
\ctr@ld@f\def\figptexcenter#1:#2[#3,#4,#5]{\ifGR@cri{\let@xte={-}
    \Figptexinsc@nter#1:#2[#3,#4,#5]}\ignorespaces\fi}
\ctr@ld@f\def\figptincenter#1:#2[#3,#4,#5]{\ifGR@cri{\let@xte={}
    \Figptexinsc@nter#1:#2[#3,#4,#5]}\ignorespaces\fi}
\ctr@ld@f
\ctr@ld@f\def\Figptexinsc@nter#1:#2[#3,#4,#5]{%
    \figgetdist\LA@[#4,#5]\figgetdist\LB@[#3,#5]\figgetdist\LC@[#3,#4]%
    \figptbaryR#1:{#2}[#3,#4,#5;\the\let@xte\LA@,\LB@,\LC@]}
\ctr@ln@m\figptinterlineplane
\ctr@ld@f\def\figptinterlineplaneDD{\un@v@ilable{figptinterlineplane}}
\ctr@ld@f\def\figptinterlineplaneTD#1:#2[#3,#4;#5,#6]{\ifGR@cri{\s@uvc@ntr@l\et@tfigptinterlineplane%
    \setc@ntr@l{2}\figvectPTD-1[#3,#5]\vecunit@TD{-2}{#6}%
    \r@pPSTD\v@leur[-2,-1,#4]\edef\v@lcoef{\repdecn@mb{\v@leur}}%
    \figpttraTD#1:{#2}=#3/\v@lcoef,#4/\resetc@ntr@l\et@tfigptinterlineplane}\ignorespaces\fi}
\ctr@ln@m\figptorthocenter
\ctr@ld@f\def\figptorthocenterDD#1:#2[#3,#4,#5]{\ifGR@cri{\s@uvc@ntr@l\et@tfigptorthocenterDD%
    \setc@ntr@l{2}\figvectNDD-3[#3,#4]\figvectNDD-4[#4,#5]%
    \resetc@ntr@l{2}\inters@cDD#1:{#2}[#5,-3;#3,-4]%
    \resetc@ntr@l\et@tfigptorthocenterDD}\ignorespaces\fi}
\ctr@ld@f\def\figptorthocenterTD#1:#2[#3,#4,#5]{\ifGR@cri{\s@uvc@ntr@l\et@tfigptorthocenterTD%
    \setc@ntr@l{2}\figvectNTD-1[#3,#4,#5]%
    \figvectPTD-2[#3,#4]\figvectNVTD-3[-1,-2]%
    \figvectPTD-2[#4,#5]\figvectNVTD-4[-1,-2]%
    \resetc@ntr@l{2}\inters@cTD#1:{#2}[#5,-3;#3,-4]%
    \resetc@ntr@l\et@tfigptorthocenterTD}\ignorespaces\fi}
\ctr@ln@m\figptorthoprojline
\ctr@ld@f\def\figptorthoprojlineDD#1:#2=#3/#4,#5/{\ifGR@cri{\s@uvc@ntr@l\et@tfigptorthoprojlineDD%
    \setc@ntr@l{2}\figvectPDD-3[#4,#5]\figvectNVDD-4[-3]\resetc@ntr@l{2}%
    \inters@cDD#1:{#2}[#3,-4;#4,-3]\resetc@ntr@l\et@tfigptorthoprojlineDD}\ignorespaces\fi}
\ctr@ld@f\def\figptorthoprojlineTD#1:#2=#3/#4,#5/{\ifGR@cri{\s@uvc@ntr@l\et@tfigptorthoprojlineTD%
    \setc@ntr@l{2}\figvectPTD-1[#4,#3]\figvectPTD-2[#4,#5]\vecunit@TD{-2}{-2}%
    \c@lproscalTD\v@leur[-1,-2]\edef\v@lcoef{\repdecn@mb{\v@leur}}%
    \figpttraTD#1:{#2}=#4/\v@lcoef,-2/\resetc@ntr@l\et@tfigptorthoprojlineTD}\ignorespaces\fi}
\ctr@ln@m\figptorthoprojplane
\ctr@ld@f\def\figptorthoprojplaneDD{\un@v@ilable{figptorthoprojplane}}
\ctr@ld@f\def\figptorthoprojplaneTD#1:#2=#3/#4,#5/{\ifGR@cri{\s@uvc@ntr@l\et@tfigptorthoprojplane%
    \setc@ntr@l{2}\figvectPTD-1[#3,#4]\vecunit@TD{-2}{#5}%
    \c@lproscalTD\v@leur[-1,-2]\edef\v@lcoef{\repdecn@mb{\v@leur}}%
    \figpttraTD#1:{#2}=#3/\v@lcoef,-2/\resetc@ntr@l\et@tfigptorthoprojplane}\ignorespaces\fi}
\ctr@ld@f\def\figpthom#1:#2=#3/#4,#5/{\ifGR@cri{\s@uvc@ntr@l\et@tfigpthom%
    \setc@ntr@l{2}\figvectP-1[#4,#3]\figpttra#1:{#2}=#4/#5,-1/%
    \resetc@ntr@l\et@tfigpthom}\ignorespaces\fi}
\ctr@ld@f\def\figptinv#1:#2=#3/#4,#5/{\ifGR@cri{\s@uvc@ntr@l\et@tfigptinv%
    \setc@ntr@l{2}\figvectP-1[#4,#3]\Figg@tXY{-1}%
    \getredf@ctB\f@ctech\n@rmeucC{\delt@}{-1}%
    \delt@=\ptT@unit@\delt@\delt@=\ptT@unit@\delt@%
    \invers@{\delt@}{\delt@}\multiply\f@ctech\f@ctech\divide\delt@\f@ctech%
    \delt@=#5\delt@\edef\v@lcoef{\repdecn@mb{\delt@}}\figpttra#1:{#2}=#4/\v@lcoef,-1/%
    \resetc@ntr@l\et@tfigptinv}\ignorespaces\fi}
\ctr@ln@m\figptrot
\ctr@ld@f\def\figptrotDD#1:#2=#3/#4,#5/{\ifGR@cri{\s@uvc@ntr@l\et@tfigptrotDD%
    \c@ssin{\C@}{\S@}{#5}\setc@ntr@l{2}\figvectPDD-1[#4,#3]\Figg@tXY{-1}%
    \v@lXa=\C@\v@lX\advance\v@lXa-\S@\v@lY%
    \v@lYa=\S@\v@lX\advance\v@lYa\C@\v@lY%
    \Figv@ctCreg-1(\v@lXa,\v@lYa)\figpttraDD#1:{#2}=#4/1,-1/%
    \resetc@ntr@l\et@tfigptrotDD}\ignorespaces\fi}
\ctr@ld@f\def\figptrotTD#1:#2=#3/#4,#5,#6/{\ifGR@cri{\s@uvc@ntr@l\et@tfigptrotTD%
    \c@ssin{\C@}{\S@}{#5}%
    \setc@ntr@l{2}\figptorthoprojplaneTD-3:=#4/#3,#6/\figvectPTD-2[-3,#3]%
    \n@rmeucTD\v@leur{-2}\ifdim\v@leur<\Cepsil@n\Figg@tXYa{#3}\else%
    \edef\v@lcoef{\repdecn@mb{\v@leur}}\figvectNVTD-1[#6,-2]%
    \Figg@tXYa{-1}\v@lXa=\v@lcoef\v@lXa\v@lYa=\v@lcoef\v@lYa\v@lZa=\v@lcoef\v@lZa%
    \v@lXa=\S@\v@lXa\v@lYa=\S@\v@lYa\v@lZa=\S@\v@lZa\Figg@tXY{-2}%
    \advance\v@lXa\C@\v@lX\advance\v@lYa\C@\v@lY\advance\v@lZa\C@\v@lZ%
    \Figg@tXY{-3}\advance\v@lXa\v@lX\advance\v@lYa\v@lY\advance\v@lZa\v@lZ\fi%
    \Figp@intregTD#1:{#2}(\v@lXa,\v@lYa,\v@lZa)\resetc@ntr@l\et@tfigptrotTD}\ignorespaces\fi}
\ctr@ln@m\figptsym
\ctr@ld@f\def\figptsymDD#1:#2=#3/#4,#5/{\ifGR@cri{\s@uvc@ntr@l\et@tfigptsymDD%
    \resetc@ntr@l{2}\figptorthoprojlineDD-5:=#3/#4,#5/\figvectPDD-2[#3,-5]%
    \figpttraDD#1:{#2}=#3/2,-2/\resetc@ntr@l\et@tfigptsymDD}\ignorespaces\fi}
\ctr@ld@f\def\figptsymTD#1:#2=#3/#4,#5/{\ifGR@cri{\s@uvc@ntr@l\et@tfigptsymTD%
    \resetc@ntr@l{2}\figptorthoprojplaneTD-3:=#3/#4,#5/\figvectPTD-2[#3,-3]%
    \figpttraTD#1:{#2}=#3/2,-2/\resetc@ntr@l\et@tfigptsymTD}\ignorespaces\fi}
\ctr@ln@m\figpttra
\ctr@ld@f\def\figpttraDD#1:#2=#3/#4,#5/{\ifGR@cri{\Figg@tXYa{#5}\v@lXa=#4\v@lXa\v@lYa=#4\v@lYa%
    \Figg@tXY{#3}\advance\v@lX\v@lXa\advance\v@lY\v@lYa%
    \Figp@intregDD#1:{#2}(\v@lX,\v@lY)}\ignorespaces\fi}
\ctr@ld@f\def\figpttraTD#1:#2=#3/#4,#5/{\ifGR@cri{\Figg@tXYa{#5}\v@lXa=#4\v@lXa\v@lYa=#4\v@lYa%
    \v@lZa=#4\v@lZa\Figg@tXY{#3}\advance\v@lX\v@lXa\advance\v@lY\v@lYa%
    \advance\v@lZ\v@lZa\Figp@intregTD#1:{#2}(\v@lX,\v@lY,\v@lZ)}\ignorespaces\fi}
\ctr@ln@m\figpttraC
\ctr@ld@f\def\figpttraCDD#1:#2=#3/#4,#5/{\ifGR@cri{\v@lXa=#4\unit@\v@lYa=#5\unit@%
    \Figg@tXY{#3}\advance\v@lX\v@lXa\advance\v@lY\v@lYa%
    \Figp@intregDD#1:{#2}(\v@lX,\v@lY)}\ignorespaces\fi}
\ctr@ld@f\def\figpttraCTD#1:#2=#3/#4,#5,#6/{\ifGR@cri{\v@lXa=#4\unit@\v@lYa=#5\unit@\v@lZa=#6\unit@%
    \Figg@tXY{#3}\advance\v@lX\v@lXa\advance\v@lY\v@lYa\advance\v@lZ\v@lZa%
    \Figp@intregTD#1:{#2}(\v@lX,\v@lY,\v@lZ)}\ignorespaces\fi}
\ctr@ld@f\def\figptsaxes#1:#2(#3){\ifGR@cri{\an@lys@xes#3,:\ifx\t@xt@\empty%
    \ifTr@isDim\Figpts@xes#1:#2(0,#3,0,#3,0,#3)\else\Figpts@xes#1:#2(0,#3,0,#3)\fi%
    \else\Figpts@xes#1:#2(#3)\fi}\ignorespaces\fi}
\ctr@ln@m\Figpts@xes
\ctr@ld@f\def\Figpts@xesDD#1:#2(#3,#4,#5,#6){%
    \s@mme=#1\figpttraC\the\s@mme:$x$=#2/#4,0/%
    \advance\s@mme\@ne\figpttraC\the\s@mme:$y$=#2/0,#6/}
\ctr@ld@f\def\Figpts@xesTD#1:#2(#3,#4,#5,#6,#7,#8){%
    \s@mme=#1\figpttraC\the\s@mme:$x$=#2/#4,0,0/%
    \advance\s@mme\@ne\figpttraC\the\s@mme:$y$=#2/0,#6,0/%
    \advance\s@mme\@ne\figpttraC\the\s@mme:$z$=#2/0,0,#8/}
\ctr@ld@f\def\figptsmap#1=#2/#3/#4/{\ifGR@cri{\s@uvc@ntr@l\et@tfigptsmap%
    \setc@ntr@l{2}\def\list@num{#2}\s@mme=#1%
    \@ecfor\p@int:=\list@num\do{\figvectP-1[#3,\p@int]\Figg@tXY{-1}%
    \pr@dMatV/#4/\figpttra\the\s@mme:=#3/1,-1/\advance\s@mme\@ne}%
    \resetc@ntr@l\et@tfigptsmap}\ignorespaces\fi}
\ctr@ln@m\figptscontrol
\ctr@ld@f\def\figptscontrolDD#1[#2,#3,#4,#5]{\ifGR@cri{\s@uvc@ntr@l\et@tfigptscontrolDD\setc@ntr@l{2}%
    \v@lX=\z@\v@lY=\z@\Figtr@nptDD{-5}{#2}\Figtr@nptDD{2}{#5}%
    \divide\v@lX\@vi\divide\v@lY\@vi%
    \Figtr@nptDD{3}{#3}\Figtr@nptDD{-1.5}{#4}\Figp@intregDD-1:(\v@lX,\v@lY)%
    \v@lX=\z@\v@lY=\z@\Figtr@nptDD{2}{#2}\Figtr@nptDD{-5}{#5}%
    \divide\v@lX\@vi\divide\v@lY\@vi\Figtr@nptDD{-1.5}{#3}\Figtr@nptDD{3}{#4}%
    \s@mme=#1\advance\s@mme\@ne\Figp@intregDD\the\s@mme:(\v@lX,\v@lY)%
    \figptcopyDD#1:/-1/\resetc@ntr@l\et@tfigptscontrolDD}\ignorespaces\fi}
\ctr@ld@f\def\figptscontrolTD#1[#2,#3,#4,#5]{\ifGR@cri{\s@uvc@ntr@l\et@tfigptscontrolTD\setc@ntr@l{2}%
    \v@lX=\z@\v@lY=\z@\v@lZ=\z@\Figtr@nptTD{-5}{#2}\Figtr@nptTD{2}{#5}%
    \divide\v@lX\@vi\divide\v@lY\@vi\divide\v@lZ\@vi%
    \Figtr@nptTD{3}{#3}\Figtr@nptTD{-1.5}{#4}\Figp@intregTD-1:(\v@lX,\v@lY,\v@lZ)%
    \v@lX=\z@\v@lY=\z@\v@lZ=\z@\Figtr@nptTD{2}{#2}\Figtr@nptTD{-5}{#5}%
    \divide\v@lX\@vi\divide\v@lY\@vi\divide\v@lZ\@vi\Figtr@nptTD{-1.5}{#3}\Figtr@nptTD{3}{#4}%
    \s@mme=#1\advance\s@mme\@ne\Figp@intregTD\the\s@mme:(\v@lX,\v@lY,\v@lZ)%
    \figptcopyTD#1:/-1/\resetc@ntr@l\et@tfigptscontrolTD}\ignorespaces\fi}
\ctr@ld@f\def\Figtr@nptDD#1#2{\Figg@tXYa{#2}\v@lXa=#1\v@lXa\v@lYa=#1\v@lYa%
    \advance\v@lX\v@lXa\advance\v@lY\v@lYa}
\ctr@ld@f\def\Figtr@nptTD#1#2{\Figg@tXYa{#2}\v@lXa=#1\v@lXa\v@lYa=#1\v@lYa\v@lZa=#1\v@lZa%
    \advance\v@lX\v@lXa\advance\v@lY\v@lYa\advance\v@lZ\v@lZa}
\ctr@ld@f\def\figptscontrolcurve#1,#2[#3]{\ifGR@cri{\s@uvc@ntr@l\et@tfigptscontrolcurve%
    \def\list@num{#3}\extrairelepremi@r\Ak@\de\list@num%
    \extrairelepremi@r\Ai@\de\list@num\extrairelepremi@r\Aj@\de\list@num%
    \s@mme=#1\figptcopy\the\s@mme:/\Ai@/%
    \setc@ntr@l{2}\figvectP -1[\Ak@,\Aj@]%
    \@ecfor\Ak@:=\list@num\do{\advance\s@mme\@ne\figpttra\the\s@mme:=\Ai@/\curv@roundness,-1/%
       \figvectP -1[\Ai@,\Ak@]\advance\s@mme\@ne\figpttra\the\s@mme:=\Aj@/-\curv@roundness,-1/%
       \advance\s@mme\@ne\figptcopy\the\s@mme:/\Aj@/%
       \edef\Ai@{\Aj@}\edef\Aj@{\Ak@}}\advance\s@mme-#1\divide\s@mme\thr@@%
       \xdef#2{\the\s@mme}%
    \resetc@ntr@l\et@tfigptscontrolcurve}\ignorespaces\fi}
\ctr@ln@m\figptsintercirc
\ctr@ld@f\def\figptsintercircDD#1[#2,#3;#4,#5]{\ifGR@cri{\s@uvc@ntr@l\et@tfigptsintercircDD%
    \setc@ntr@l{2}\let\c@lNVintc=\c@lNVintcDD\Figptsintercirc@#1[#2,#3;#4,#5]%
    \resetc@ntr@l\et@tfigptsintercircDD}\ignorespaces\fi}
\ctr@ld@f\def\figptsintercircTD#1[#2,#3;#4,#5;#6]{\ifGR@cri{\s@uvc@ntr@l\et@tfigptsintercircTD%
    \setc@ntr@l{2}\let\c@lNVintc=\c@lNVintcTD\vecunitC@TD[#2,#6]%
    \Figv@ctCreg-3(\v@lX,\v@lY,\v@lZ)\Figptsintercirc@#1[#2,#3;#4,#5]%
    \resetc@ntr@l\et@tfigptsintercircTD}\ignorespaces\fi}
\ctr@ld@f\def\Figptsintercirc@#1[#2,#3;#4,#5]{\figvectP-1[#2,#4]%
    \vecunit@{-1}{-1}\delt@=\result@t\f@ctech=\result@tent%
    \s@mme=#1\advance\s@mme\@ne\figptcopy#1:/#2/\figptcopy\the\s@mme:/#4/%
    \ifdim\delt@=\z@\else%
    \v@lmin=#3\unit@\v@lmax=#5\unit@\v@leur=\v@lmin\advance\v@leur\v@lmax%
    \ifdim\v@leur>\delt@%
    \v@leur=\v@lmin\advance\v@leur-\v@lmax\maxim@m{\v@leur}{\v@leur}{-\v@leur}%
    \ifdim\v@leur<\delt@%
    \divide\v@lmin\f@ctech\divide\v@lmax\f@ctech\divide\delt@\f@ctech%
    \v@lmin=\repdecn@mb{\v@lmin}\v@lmin\v@lmax=\repdecn@mb{\v@lmax}\v@lmax%
    \invers@{\v@leur}{\delt@}\advance\v@lmax-\v@lmin%
    \v@lmax=-\repdecn@mb{\v@leur}\v@lmax\advance\delt@\v@lmax\delt@=.5\delt@%
    \v@lmax=\delt@\multiply\v@lmax\f@ctech%
    \edef\t@ille{\repdecn@mb{\v@lmax}}\figpttra-2:=#2/\t@ille,-1/%
    \delt@=\repdecn@mb{\delt@}\delt@\advance\v@lmin-\delt@%
    \sqrt@{\v@leur}{\v@lmin}\multiply\v@leur\f@ctech\edef\t@ille{\repdecn@mb{\v@leur}}%
    \c@lNVintc\figpttra#1:=-2/-\t@ille,-1/\figpttra\the\s@mme:=-2/\t@ille,-1/\fi\fi\fi}
\ctr@ld@f\def\c@lNVintcDD{\Figg@tXY{-1}\Figv@ctCreg-1(-\v@lY,\v@lX)} 
\ctr@ld@f\def\c@lNVintcTD{{\Figg@tXY{-3}\v@lmin=\v@lX\v@lmax=\v@lY\v@leur=\v@lZ%
    \Figg@tXY{-1}\c@lprovec{-3}\vecunit@{-3}{-3}
    \Figg@tXY{-1}\v@lmin=\v@lX\v@lmax=\v@lY%
    \v@leur=\v@lZ\Figg@tXY{-3}\c@lprovec{-1}}} 
\ctr@ln@m\figptsinterlinell
\ctr@ld@f\def\figptsinterlinellDD#1[#2,#3,#4,#5;#6,#7]{\ifGR@cri{\s@uvc@ntr@l\et@tfigptsinterlinellDD%
    \figptcopy#1:/#6/\s@mme=#1\advance\s@mme\@ne\figptcopy\the\s@mme:/#7/%
    \v@lmin=#3\unit@\v@lmax=#4\unit@
    \setc@ntr@l{2}\figptbaryDD-4:[#6,#7;1,1]\figptsrotDD-3=-4,#7/#2,-#5/
    \Figg@tXY{-3}\Figg@tXYa{#2}\advance\v@lX-\v@lXa\advance\v@lY-\v@lYa
    \figvectP-1[-3,-2]\Figg@tXYa{-1}\figvectP-3[-4,#7]\Figptsint@rLE{#1}
    \resetc@ntr@l\et@tfigptsinterlinellDD}\ignorespaces\fi}
\ctr@ld@f\def\figptsinterlinellP#1[#2,#3,#4;#5,#6]{\ifGR@cri{\s@uvc@ntr@l\et@tfigptsinterlinellP%
    \figptcopy#1:/#5/\s@mme=#1\advance\s@mme\@ne\figptcopy\the\s@mme:/#6/\setc@ntr@l{2}%
    \figvectP-1[#2,#3]\vecunit@{-1}{-1}\v@lmin=\result@t
    \figvectP-2[#2,#4]\vecunit@{-2}{-2}\v@lmax=\result@t
    \figptbary-4:[#5,#6;1,1]
    \figvectP-3[#2,-4]\c@lproscal\v@lX[-3,-1]\c@lproscal\v@lY[-3,-2]
    \figvectP-3[-4,#6]\c@lproscal\v@lXa[-3,-1]\c@lproscal\v@lYa[-3,-2]
    \Figptsint@rLE{#1}\resetc@ntr@l\et@tfigptsinterlinellP}\ignorespaces\fi}
\ctr@ld@f\def\Figptsint@rLE#1{%
    \getredf@ctDD\f@ctech(\v@lmin,\v@lmax)%
    \getredf@ctDD\p@rtent(\v@lX,\v@lY)\ifnum\p@rtent>\f@ctech\f@ctech=\p@rtent\fi%
    \getredf@ctDD\p@rtent(\v@lXa,\v@lYa)\ifnum\p@rtent>\f@ctech\f@ctech=\p@rtent\fi%
    \divide\v@lmin\f@ctech\divide\v@lmax\f@ctech\divide\v@lX\f@ctech\divide\v@lY\f@ctech%
    \divide\v@lXa\f@ctech\divide\v@lYa\f@ctech%
    \c@rre=\repdecn@mb\v@lXa\v@lmax\mili@u=\repdecn@mb\v@lYa\v@lmin%
    \getredf@ctDD\f@ctech(\c@rre,\mili@u)%
    \c@rre=\repdecn@mb\v@lX\v@lmax\mili@u=\repdecn@mb\v@lY\v@lmin%
    \getredf@ctDD\p@rtent(\c@rre,\mili@u)\ifnum\p@rtent>\f@ctech\f@ctech=\p@rtent\fi%
    \divide\v@lmin\f@ctech\divide\v@lmax\f@ctech\divide\v@lX\f@ctech\divide\v@lY\f@ctech%
    \divide\v@lXa\f@ctech\divide\v@lYa\f@ctech%
    \v@lmin=\repdecn@mb{\v@lmin}\v@lmin\v@lmax=\repdecn@mb{\v@lmax}\v@lmax%
    \edef\G@xde{\repdecn@mb\v@lmin}\edef\P@xde{\repdecn@mb\v@lmax}%
    \c@rre=-\v@lmax\v@leur=\repdecn@mb\v@lY\v@lY\advance\c@rre\v@leur\c@rre=\G@xde\c@rre%
    \v@leur=\repdecn@mb\v@lX\v@lX\v@leur=\P@xde\v@leur\advance\c@rre\v@leur
    \v@lmin=\repdecn@mb\v@lYa\v@lmin\v@lmax=\repdecn@mb\v@lXa\v@lmax%
    \mili@u=\repdecn@mb\v@lX\v@lmax\advance\mili@u\repdecn@mb\v@lY\v@lmin
    \v@lmax=\repdecn@mb\v@lXa\v@lmax\advance\v@lmax\repdecn@mb\v@lYa\v@lmin
    \ifdim\v@lmax>\epsil@n%
    \maxim@m{\v@leur}{\c@rre}{-\c@rre}\maxim@m{\v@lmin}{\mili@u}{-\mili@u}%
    \maxim@m{\v@leur}{\v@leur}{\v@lmin}\maxim@m{\v@lmin}{\v@lmax}{-\v@lmax}%
    \maxim@m{\v@leur}{\v@leur}{\v@lmin}\p@rtentiere{\p@rtent}{\v@leur}\advance\p@rtent\@ne%
    \divide\c@rre\p@rtent\divide\mili@u\p@rtent\divide\v@lmax\p@rtent%
    \delt@=\repdecn@mb{\mili@u}\mili@u\v@leur=\repdecn@mb{\v@lmax}\c@rre%
    \advance\delt@-\v@leur\ifdim\delt@<\z@\else\sqrt@\delt@\delt@%
    \invers@\v@lmax\v@lmax\edef\Uns@rAp{\repdecn@mb\v@lmax}%
    \v@leur=-\mili@u\advance\v@leur-\delt@\v@leur=\Uns@rAp\v@leur%
    \edef\t@ille{\repdecn@mb\v@leur}\figpttra#1:=-4/\t@ille,-3/\s@mme=#1\advance\s@mme\@ne%
    \v@leur=-\mili@u\advance\v@leur\delt@\v@leur=\Uns@rAp\v@leur%
    \edef\t@ille{\repdecn@mb\v@leur}\figpttra\the\s@mme:=-4/\t@ille,-3/\fi\fi}
\ctr@ln@m\figptsorthoprojline
\ctr@ld@f\def\figptsorthoprojlineDD#1=#2/#3,#4/{\ifGR@cri{\s@uvc@ntr@l\et@tfigptsorthoprojlineDD%
    \setc@ntr@l{2}\figvectPDD-3[#3,#4]\figvectNVDD-4[-3]\resetc@ntr@l{2}%
    \def\list@num{#2}\s@mme=#1\@ecfor\p@int:=\list@num\do{%
    \inters@cDD\the\s@mme:[\p@int,-4;#3,-3]\advance\s@mme\@ne}%
    \resetc@ntr@l\et@tfigptsorthoprojlineDD}\ignorespaces\fi}
\ctr@ld@f\def\figptsorthoprojlineTD#1=#2/#3,#4/{\ifGR@cri{\s@uvc@ntr@l\et@tfigptsorthoprojlineTD%
    \setc@ntr@l{2}\figvectPTD-2[#3,#4]\vecunit@TD{-2}{-2}%
    \def\list@num{#2}\s@mme=#1\@ecfor\p@int:=\list@num\do{%
    \figvectPTD-1[#3,\p@int]\c@lproscalTD\v@leur[-1,-2]%
    \edef\v@lcoef{\repdecn@mb{\v@leur}}\figpttraTD\the\s@mme:=#3/\v@lcoef,-2/%
    \advance\s@mme\@ne}\resetc@ntr@l\et@tfigptsorthoprojlineTD}\ignorespaces\fi}
\ctr@ln@m\figptsorthoprojplane
\ctr@ld@f\def\figptsorthoprojplaneDD{\un@v@ilable{figptsorthoprojplane}}
\ctr@ld@f\def\figptsorthoprojplaneTD#1=#2/#3,#4/{\ifGR@cri{\s@uvc@ntr@l\et@tfigptsorthoprojplane%
    \setc@ntr@l{2}\vecunit@TD{-2}{#4}%
    \def\list@num{#2}\s@mme=#1\@ecfor\p@int:=\list@num\do{\figvectPTD-1[\p@int,#3]%
    \c@lproscalTD\v@leur[-1,-2]\edef\v@lcoef{\repdecn@mb{\v@leur}}%
    \figpttraTD\the\s@mme:=\p@int/\v@lcoef,-2/\advance\s@mme\@ne}%
    \resetc@ntr@l\et@tfigptsorthoprojplane}\ignorespaces\fi}
\ctr@ld@f\def\figptshom#1=#2/#3,#4/{\ifGR@cri{\s@uvc@ntr@l\et@tfigptshom%
    \setc@ntr@l{2}\def\list@num{#2}\s@mme=#1%
    \@ecfor\p@int:=\list@num\do{\figvectP-1[#3,\p@int]%
    \figpttra\the\s@mme:=#3/#4,-1/\advance\s@mme\@ne}%
    \resetc@ntr@l\et@tfigptshom}\ignorespaces\fi}
\ctr@ld@f\def\figptsinv#1=#2/#3,#4/{\ifGR@cri{\s@uvc@ntr@l\et@tfigptsinv%
    \setc@ntr@l{2}\def\list@num{#2}\s@mme=#1%
    \@ecfor\p@int:=\list@num\do{\figvectP-1[#3,\p@int]\Figg@tXY{-1}%
    \getredf@ctB\f@ctech\n@rmeucC{\delt@}{-1}%
    \delt@=\ptT@unit@\delt@\delt@=\ptT@unit@\delt@%
    \invers@{\delt@}{\delt@}\multiply\f@ctech\f@ctech\divide\delt@\f@ctech%
    \delt@=#4\delt@\edef\v@lcoef{\repdecn@mb{\delt@}}\figpttra\the\s@mme:=#3/\v@lcoef,-1/%
    \advance\s@mme\@ne}\resetc@ntr@l\et@tfigptsinv}\ignorespaces\fi}
\ctr@ln@m\figptsrot
\ctr@ld@f\def\figptsrotDD#1=#2/#3,#4/{\ifGR@cri{\s@uvc@ntr@l\et@tfigptsrotDD%
    \c@ssin{\C@}{\S@}{#4}\setc@ntr@l{2}\def\list@num{#2}\s@mme=#1%
    \@ecfor\p@int:=\list@num\do{\figvectPDD-1[#3,\p@int]\Figg@tXY{-1}%
    \v@lXa=\C@\v@lX\advance\v@lXa-\S@\v@lY%
    \v@lYa=\S@\v@lX\advance\v@lYa\C@\v@lY%
    \Figv@ctCreg-1(\v@lXa,\v@lYa)\figpttraDD\the\s@mme:=#3/1,-1/\advance\s@mme\@ne}%
    \resetc@ntr@l\et@tfigptsrotDD}\ignorespaces\fi}
\ctr@ld@f\def\figptsrotTD#1=#2/#3,#4,#5/{\ifGR@cri{\s@uvc@ntr@l\et@tfigptsrotTD%
    \c@ssin{\C@}{\S@}{#4}%
    \setc@ntr@l{2}\def\list@num{#2}\s@mme=#1%
    \@ecfor\p@int:=\list@num\do{\figptorthoprojplaneTD-3:=#3/\p@int,#5/%
    \figvectPTD-2[-3,\p@int]%
    \figvectNVTD-1[#5,-2]\n@rmeucTD\v@leur{-2}\edef\v@lcoef{\repdecn@mb{\v@leur}}%
    \Figg@tXYa{-1}\v@lXa=\v@lcoef\v@lXa\v@lYa=\v@lcoef\v@lYa\v@lZa=\v@lcoef\v@lZa%
    \v@lXa=\S@\v@lXa\v@lYa=\S@\v@lYa\v@lZa=\S@\v@lZa\Figg@tXY{-2}%
    \advance\v@lXa\C@\v@lX\advance\v@lYa\C@\v@lY\advance\v@lZa\C@\v@lZ%
    \Figg@tXY{-3}\advance\v@lXa\v@lX\advance\v@lYa\v@lY\advance\v@lZa\v@lZ%
    \Figp@intregTD\the\s@mme:(\v@lXa,\v@lYa,\v@lZa)\advance\s@mme\@ne}%
    \resetc@ntr@l\et@tfigptsrotTD}\ignorespaces\fi}
\ctr@ln@m\figptssym
\ctr@ld@f\def\figptssymDD#1=#2/#3,#4/{\ifGR@cri{\s@uvc@ntr@l\et@tfigptssymDD%
    \setc@ntr@l{2}\figvectPDD-3[#3,#4]\Figg@tXY{-3}\Figv@ctCreg-4(-\v@lY,\v@lX)%
    \resetc@ntr@l{2}\def\list@num{#2}\s@mme=#1%
    \@ecfor\p@int:=\list@num\do{\inters@cDD-5:[#3,-3;\p@int,-4]\figvectPDD-2[\p@int,-5]%
    \figpttraDD\the\s@mme:=\p@int/2,-2/\advance\s@mme\@ne}%
    \resetc@ntr@l\et@tfigptssymDD}\ignorespaces\fi}
\ctr@ld@f\def\figptssymTD#1=#2/#3,#4/{\ifGR@cri{\s@uvc@ntr@l\et@tfigptssymTD%
    \setc@ntr@l{2}\vecunit@TD{-2}{#4}\def\list@num{#2}\s@mme=#1%
    \@ecfor\p@int:=\list@num\do{\figvectPTD-1[\p@int,#3]%
    \c@lproscalTD\v@leur[-1,-2]\v@leur=2\v@leur\edef\v@lcoef{\repdecn@mb{\v@leur}}%
    \figpttraTD\the\s@mme:=\p@int/\v@lcoef,-2/\advance\s@mme\@ne}%
    \resetc@ntr@l\et@tfigptssymTD}\ignorespaces\fi}
\ctr@ln@m\figptstra
\ctr@ld@f\def\figptstraDD#1=#2/#3,#4/{\ifGR@cri{\Figg@tXYa{#4}\v@lXa=#3\v@lXa\v@lYa=#3\v@lYa%
    \def\list@num{#2}\s@mme=#1\@ecfor\p@int:=\list@num\do{\Figg@tXY{\p@int}%
    \advance\v@lX\v@lXa\advance\v@lY\v@lYa%
    \Figp@intregDD\the\s@mme:(\v@lX,\v@lY)\advance\s@mme\@ne}}\ignorespaces\fi}
\ctr@ld@f\def\figptstraTD#1=#2/#3,#4/{\ifGR@cri{\Figg@tXYa{#4}\v@lXa=#3\v@lXa\v@lYa=#3\v@lYa%
    \v@lZa=#3\v@lZa\def\list@num{#2}\s@mme=#1\@ecfor\p@int:=\list@num\do{\Figg@tXY{\p@int}%
    \advance\v@lX\v@lXa\advance\v@lY\v@lYa\advance\v@lZ\v@lZa%
    \Figp@intregTD\the\s@mme:(\v@lX,\v@lY,\v@lZ)\advance\s@mme\@ne}}\ignorespaces\fi}
\ctr@ln@m\figptvisilimSL
\ctr@ld@f\def\figptvisilimSLDD{\un@v@ilable{figptvisilimSL}}
\ctr@ld@f\def\figptvisilimSLTD#1:#2[#3,#4;#5,#6]{\ifGR@cri{\s@uvc@ntr@l\et@tfigptvisilimSLTD%
    \setc@ntr@l{2}\figvectP-1[#3,#4]\n@rminf{\delt@}{-1}%
    \ifcase\CUR@proj\v@lX=\cxa@\p@\v@lY=-\p@\v@lZ=\cxb@\p@
    \Figv@ctCreg-2(\v@lX,\v@lY,\v@lZ)\figvectP-3[#5,#6]\figvectNV-1[-2,-3]%
    \or\figvectP-1[#5,#6]\vecunitCV@TD{-1}\v@lmin=\v@lX\v@lmax=\v@lY
    \v@leur=\v@lZ\v@lX=\cza@\p@\v@lY=\czb@\p@\v@lZ=\czc@\p@\c@lprovec{-1}%
    \or\c@ley@pt{-2}\figvectN-1[#5,#6,-2]\fi
    \edef\Ai@{#3}\edef\Aj@{#4}\figvectP-2[#5,\Ai@]\c@lproscal\v@leur[-1,-2]%
    \ifdim\v@leur>\z@\p@rtent=\@ne\else\p@rtent=\m@ne\fi%
    \figvectP-2[#5,\Aj@]\c@lproscal\v@leur[-1,-2]%
    \ifdim\p@rtent\v@leur>\z@\figptcopy#1:#2/#3/%
    \message{*** \BS@ figptvisilimSL: points are on the same side.}\else%
    \figptcopy-3:/#3/\figptcopy-4:/#4/%
    \loop\figptbary-5:[-3,-4;1,1]\figvectP-2[#5,-5]\c@lproscal\v@leur[-1,-2]%
    \ifdim\p@rtent\v@leur>\z@\figptcopy-3:/-5/\else\figptcopy-4:/-5/\fi%
    \divide\delt@\tw@\ifdim\delt@>\epsil@n\repeat%
    \figptbary#1:#2[-3,-4;1,1]\fi\resetc@ntr@l\et@tfigptvisilimSLTD}\ignorespaces\fi}
\ctr@ld@f\def\c@ley@pt#1{\t@stp@r\ifitis@K\v@lX=\cza@\p@\v@lY=\czb@\p@\v@lZ=\czc@\p@%
    \Figv@ctCreg-1(\v@lX,\v@lY,\v@lZ)\Figp@intreg-2:(\wd\Bt@rget,\ht\Bt@rget,\dp\Bt@rget)%
    \figpttra#1:=-2/-\disob@intern,-1/\else\end\fi}
\ctr@ld@f\def\t@stp@r{\itis@Ktrue\ifnewt@rgetpt\else\itis@Kfalse%
    \message{*** \BS@ figptvisilimXX: target point undefined.}\fi\ifnewdis@b\else%
    \itis@Kfalse\message{*** \BS@ figptvisilimXX: observation distance undefined.}\fi%
    \ifitis@K\else\message{*** This macro must be called after \BS@ figdrawbegin or after
    having set the missing parameter(s) with \BS@ figset proj()}\fi}
\ctr@ld@f\def\figscan#1(#2,#3){{\s@uvc@ntr@l\et@tfigscan\@psfgetbb{#1}\if@psfbbfound\else%
    \def\@psfllx{0}\def\@psflly{20}\def\@psfurx{540}\def\@psfury{640}\fi\figscan@{#2}{#3}%
    \resetc@ntr@l\et@tfigscan}\ignorespaces}
\ctr@ld@f\def\figscan@#1#2{%
    \unit@=\@ne bp\setc@ntr@l{2}\figsetmark{}%
    \def\minst@p{20pt}%
    \v@lX=\@psfllx\p@\v@lX=\Sc@leFact\v@lX\r@undint\v@lX\v@lX%
    \v@lY=\@psflly\p@\v@lY=\Sc@leFact\v@lY\ifdim\v@lY>\z@\r@undint\v@lY\v@lY\fi%
    \delt@=\@psfury\p@\delt@=\Sc@leFact\delt@%
    \advance\delt@-\v@lY\v@lXa=\@psfurx\p@\v@lXa=\Sc@leFact\v@lXa\v@leur=\minst@p%
    \edef\valv@lY{\repdecn@mb{\v@lY}}\edef\LgTr@it{\the\delt@}%
    \loop\ifdim\v@lX<\v@lXa\edef\valv@lX{\repdecn@mb{\v@lX}}%
    \figptDD -1:(\valv@lX,\valv@lY)\figwriten -1:\hbox{\vrule height\LgTr@it}(0)%
    \ifdim\v@leur<\minst@p\else\figsetmark{\raise-8bp\hbox{$\scriptscriptstyle\triangle$}}%
    \figwrites -1:\@ffichnb{0}{\valv@lX}(6)\v@leur=\z@\figsetmark{}\fi%
    \advance\v@leur#1pt\advance\v@lX#1pt\repeat%
    \def\minst@p{10pt}%
    \v@lX=\@psfllx\p@\v@lX=\Sc@leFact\v@lX\ifdim\v@lX>\z@\r@undint\v@lX\v@lX\fi%
    \v@lY=\@psflly\p@\v@lY=\Sc@leFact\v@lY\r@undint\v@lY\v@lY%
    \delt@=\@psfurx\p@\delt@=\Sc@leFact\delt@%
    \advance\delt@-\v@lX\v@lYa=\@psfury\p@\v@lYa=\Sc@leFact\v@lYa\v@leur=\minst@p%
    \edef\valv@lX{\repdecn@mb{\v@lX}}\edef\LgTr@it{\the\delt@}%
    \loop\ifdim\v@lY<\v@lYa\edef\valv@lY{\repdecn@mb{\v@lY}}%
    \figptDD -1:(\valv@lX,\valv@lY)\figwritee -1:\vbox{\hrule width\LgTr@it}(0)%
    \ifdim\v@leur<\minst@p\else\figsetmark{$\triangleright$\kern4bp}%
    \figwritew -1:\@ffichnb{0}{\valv@lY}(6)\v@leur=\z@\figsetmark{}\fi%
    \advance\v@leur#2pt\advance\v@lY#2pt\repeat}
\ctr@ld@f
\ctr@ld@f\def\figscan@E#1(#2,#3){{\s@uvc@ntr@l\et@tfigscan@E%
    \Figdisc@rdLTS{#1}{\t@xt@}\pdfximage{\t@xt@}%
    \setbox\Gb@x=\hbox{\pdfrefximage\pdflastximage}%
    \edef\@psfllx{0}\v@lY=-\dp\Gb@x\edef\@psflly{\repdecn@mb{\v@lY}}%
    \edef\@psfurx{\repdecn@mb{\wd\Gb@x}}%
    \v@lY=\dp\Gb@x\advance\v@lY\ht\Gb@x\edef\@psfury{\repdecn@mb{\v@lY}}%
    \figscan@{#2}{#3}\resetc@ntr@l\et@tfigscan@E}\ignorespaces}
\ctr@ld@f\def\figshowpts[#1,#2]{{\figsetmark{$\bullet$}\figsetptname{\bf ##1}%
    \p@rtent=#2\relax\ifnum\p@rtent<\z@\p@rtent=\z@\fi%
    \s@mme=#1\relax\ifnum\s@mme<\z@\s@mme=\z@\fi%
    \loop\ifnum\s@mme<\p@rtent\pt@rvect{\s@mme}%
    \ifitis@K\figwriten{\the\s@mme}:(4pt)\fi\advance\s@mme\@ne\repeat%
    \pt@rvect{\s@mme}\ifitis@K\figwriten{\the\s@mme}:(4pt)\fi}\ignorespaces}
\ctr@ld@f\def\pt@rvect#1{\set@bjc@de{#1}%
    \expandafter\expandafter\expandafter\inqpt@rvec\csname\objc@de\endcsname:}
\ctr@ld@f\def\inqpt@rvec#1#2:{\if#1\C@dCl@spt\itis@Ktrue\else\itis@Kfalse\fi}
\ctr@ld@f\def\figshowsettings{{%
    \immediate\write16{====================================================================}%
    \immediate\write16{ Current settings are (DDV means "with dynamic default value"):}%
    \immediate\write16{ --- GENERAL ---}%
    \immediate\write16{Scale factor and Unit = \unit@util\space (\the\unit@)
     \space -> \BS@ figinit{ScaleFactorUnit}}%
    \immediate\write16{Update mode = \ifGRupdatem@de yes\else no\fi
     \space-> \BS@ figset(update=yes/no) or \BS@ figsetdefault(update=yes/no)}%
    \immediate\write16{ --- WRITING ---}%
    \immediate\write16{Implicit point name = \ptn@me{i} \space-> \BS@ figset write(ptname={Name})}%
    \immediate\write16{Point marker = \the\c@nsymb \space -> \BS@ figset write(mark=Mark)}%
    \immediate\write16{Print rounded coordinates = \ifr@undcoord yes\else no\fi
     \space-> \BS@ figset write(roundcoord=yes/no)}%
    \immediate\write16{ --- GRAPHICAL (general) ---}%
    \immediate\write16{Color = \CUR@color \space-> \BS@ figset(color=ColorDefinition)}%
    \immediate\write16{Filling mode = \iffillm@de yes\else no\fi
     \space-> \BS@ figset(fillmode=yes/no)}%
    \immediate\write16{Line join = \CUR@join \space-> \BS@ figset(join=miter/round/bevel)}%
    \immediate\write16{Line style = \CUR@dash \space-> \BS@ figset(dash=Index/Pattern)}%
    \immediate\write16{Line width = \CUR@width
     \space-> \BS@ figset(width=real in PostScript units)}%
    \immediate\write16{ --- GRAPHICAL (specific) ---}%
    \immediate\write16{Altitude (all the following attributes are DDV):}%
    \immediate\write16{ Base line color =
     \ifx\DDV@blcolor\D@FTref general color\else\DDV@blcolor\fi
     \space-> \BS@ figset altitude(blcolor=ColorDefinition)}%
    \immediate\write16{ Base line style =
     \ifx\DDV@bldash\D@FTref general style\else\DDV@bldash\fi
     \space-> \BS@ figset altitude(bldash=Index/Pattern)}%
    \immediate\write16{ Base line width =
     \ifx\DDV@blwidth\D@FTref general width\else\DDV@blwidth\fi
     \space-> \BS@ figset altitude(blwidth=real in PostScript units)}%
    \immediate\write16{ Square line color =
     \ifx\DDV@sqcolor\D@FTref general color\else\DDV@sqcolor\fi
     \space-> \BS@ figset altitude(sqcolor=ColorDefinition)}%
    \immediate\write16{ Square line style =
     \ifx\DDV@sqdash\D@FTref general style\else\DDV@sqdash\fi
     \space-> \BS@ figset altitude(sqdash=Index/Pattern)}%
    \immediate\write16{ Square line width =
     \ifx\DDV@sqwidth\D@FTref general width\else\DDV@sqwidth\fi
     \space-> \BS@ figset altitude(sqwidth=real in PostScript units)}%
    \immediate\write16{Arrowhead:}%
    \immediate\write16{ (half-)Angle = \@rrowheadangle
     \space-> \BS@ figset arrowhead(angle=real in degrees)}%
    \immediate\write16{ Filling mode = \if@rrowhfill yes\else no\fi
     \space-> \BS@ figset arrowhead(fillmode=yes/no)}%
    \immediate\write16{ "Outside" = \if@rrowhout yes\else no\fi
     \space-> \BS@ figset arrowhead(out=yes/no)}%
    \immediate\write16{ Length = \@rrowheadlength
     \if@rrowratio\space(not active)\else\space(active)\fi
     \space-> \BS@ figset arrowhead(length=real in user coord.)}%
    \immediate\write16{ Ratio = \@rrowheadratio
     \if@rrowratio\space(active)\else\space(not active)\fi
     \space-> \BS@ figset arrowhead(ratio=real in [0,1])}%
    \immediate\write16{Curve:}%
    \immediate\write16{ Roundness = \curv@roundness
     \space-> \BS@ figset curve(roundness=real in [0,0.5])}%
    \immediate\write16{Flow chart:}%
    \immediate\write16{ Arrow position = \@rrowp@s
     \space-> \BS@ figset flowchart(arrowposition=real in [0,1])}%
    \immediate\write16{ Arrow reference point = \ifcase\@rrowr@fpt start\else end\fi
     \space-> \BS@ figset flowchart(arrowrefpt = start/end)}%
    \immediate\write16{ Background color = \fcbgc@lor
     \space-> \BS@ figset flowchart(bgcolor=ColorDefinition)}%
    \immediate\write16{ Line type = \ifcase\fclin@typ@ curve\else polygon\fi
     \space-> \BS@ figset flowchart(line=polygon/curve)}%
    \immediate\write16{ Padding = (\Xp@dd, \Yp@dd)
     \space-> \BS@ figset flowchart(padding = real in user coord.)}%
    \immediate\write16{\space\space\space\space(or
     \BS@ figset flowchart(xpadding=real, ypadding=real) )}%
    \immediate\write16{ Radius = \fclin@r@d
     \space-> \BS@ figset flowchart(radius=positive real in user coord.)}%
    \immediate\write16{ Shape = \fcsh@pe
     \space-> \BS@ figset flowchart(shape = rectangle, ellipse or lozenge)}%
    \immediate\write16{ Thickness color (DDV) = 
     \ifx\DDV@thickcolor\D@FTref general color\else\DDV@thickcolor\fi
     \space-> \BS@ figset flowchart(thickcolor=ColorDefinition)}%
    \immediate\write16{ Thickness = \thickn@ss
     \space-> \BS@ figset flowchart(thickness = real in user coord.)}%
    \immediate\write16{Mesh:}%
    \immediate\write16{ Diagonal = \c@ntrolmesh
     \space-> \BS@ figset mesh(diag=integer in {-1,0,1})}%
    \immediate\write16{ Lines color (DDV) =
     \ifx\DDV@meshcolor\D@FTref general color\else\DDV@meshcolor\fi
     \space-> \BS@ figset mesh(color=ColorDefinition)}%
    \immediate\write16{ Lines style (DDV) =
     \ifx\DDV@meshdash\D@FTref general style\else\DDV@meshdash\fi
     \space-> \BS@ figset mesh(dash=Index/Pattern)}%
    \immediate\write16{ Lines width (DDV) =
     \ifx\DDV@meshwidth\D@FTref general width\else\DDV@meshwidth\fi
     \space-> \BS@ figset mesh(width=real in PostScript units)}%
    \immediate\write16{Trimesh:}%
    \immediate\write16{ Lines color (DDV) =
     \ifx\DDV@tmeshcolor\D@FTref general color\else\DDV@tmeshcolor\fi
     \space-> \BS@ figset trimesh(color=ColorDefinition)}%
    \immediate\write16{ Lines style (DDV) =
     \ifx\DDV@tmeshdash\D@FTref general style\else\DDV@tmeshdash\fi
     \space-> \BS@ figset trimesh(dash=Index/Pattern)}%
    \immediate\write16{ Lines width (DDV) =
     \ifx\DDV@tmeshwidth\D@FTref general width\else\DDV@tmeshwidth\fi
     \space-> \BS@ figset trimesh(width=real in PostScript units)}%
    \ifTr@isDim%
    \immediate\write16{ --- 3D to 2D PROJECTION ---}%
    \immediate\write16{Projection : \typ@proj \space-> \BS@ figinit{ScaleFactorUnit, ProjType}}%
    \immediate\write16{Longitude (psi) = \v@lPsi \space-> \BS@ figset proj(psi=real in degrees)}%
    \ifcase\CUR@proj\immediate\write16{Depth coeff. (Lambda)
     \space = \v@lTheta \space-> \BS@ figset proj(lambda=real in [0,1])}%
    \else\immediate\write16{Latitude (theta)
     \space = \v@lTheta \space-> \BS@ figset proj(theta=real in degrees)}%
    \fi%
    \ifnum\CUR@proj=\tw@%
    \immediate\write16{Observation distance = \disob@unit
     \space-> \BS@ figset proj(dist=real in user coord.)}%
    \immediate\write16{Target point = \t@rgetpt \space-> \BS@ figset proj(targetpt=pt number)}%
     \v@lX=\ptT@unit@\wd\Bt@rget\v@lY=\ptT@unit@\ht\Bt@rget\v@lZ=\ptT@unit@\dp\Bt@rget%
    \immediate\write16{ Its coordinates are
     (\repdecn@mb{\v@lX}, \repdecn@mb{\v@lY}, \repdecn@mb{\v@lZ})}%
    \fi%
    \fi%
    \immediate\write16{====================================================================}%
    \ignorespaces}}
\ctr@ln@w{newif}\ifitis@vect@r
\ctr@ld@f\def\figvectC#1(#2,#3){{\itis@vect@rtrue\figpt#1:(#2,#3)}\ignorespaces}
\ctr@ld@f\def\Figv@ctCreg#1(#2,#3){{\itis@vect@rtrue\Figp@intreg#1:(#2,#3)}\ignorespaces}
\ctr@ln@m\figvectDBezier
\ctr@ld@f\def\figvectDBezierDD#1:#2,#3[#4,#5,#6,#7]{\ifGR@cri{\s@uvc@ntr@l\et@tfigvectDBezierDD%
    \FigvectDBezier@#2,#3[#4,#5,#6,#7]\v@lX=\c@ef\v@lX\v@lY=\c@ef\v@lY%
    \Figv@ctCreg#1(\v@lX,\v@lY)\resetc@ntr@l\et@tfigvectDBezierDD}\ignorespaces\fi}
\ctr@ld@f\def\figvectDBezierTD#1:#2,#3[#4,#5,#6,#7]{\ifGR@cri{\s@uvc@ntr@l\et@tfigvectDBezierTD%
    \FigvectDBezier@#2,#3[#4,#5,#6,#7]\v@lX=\c@ef\v@lX\v@lY=\c@ef\v@lY\v@lZ=\c@ef\v@lZ%
    \Figv@ctCreg#1(\v@lX,\v@lY,\v@lZ)\resetc@ntr@l\et@tfigvectDBezierTD}\ignorespaces\fi}
\ctr@ld@f\def\FigvectDBezier@#1,#2[#3,#4,#5,#6]{\setc@ntr@l{2}%
    \edef\T@{#2}\v@leur=\p@\advance\v@leur-#2pt\edef\UNmT@{\repdecn@mb{\v@leur}}%
    \ifnum#1=\tw@\def\c@ef{6}\else\def\c@ef{3}\fi%
    \figptcopy-4:/#3/\figptcopy-3:/#4/\figptcopy-2:/#5/\figptcopy-1:/#6/%
    \l@mbd@un=-4 \l@mbd@de=-\thr@@\p@rtent=\m@ne\c@lDecast%
    \ifnum#1=\tw@\c@lDCDeux{-4}{-3}\c@lDCDeux{-3}{-2}\c@lDCDeux{-4}{-3}\else%
    \l@mbd@un=-4 \l@mbd@de=-\thr@@\p@rtent=-\tw@\c@lDecast%
    \c@lDCDeux{-4}{-3}\fi\Figg@tXY{-4}}
\ctr@ln@m\c@lDCDeux
\ctr@ld@f\def\c@lDCDeuxDD#1#2{\Figg@tXY{#2}\Figg@tXYa{#1}%
    \advance\v@lX-\v@lXa\advance\v@lY-\v@lYa\Figp@intregDD#1:(\v@lX,\v@lY)}
\ctr@ld@f\def\c@lDCDeuxTD#1#2{\Figg@tXY{#2}\Figg@tXYa{#1}\advance\v@lX-\v@lXa%
    \advance\v@lY-\v@lYa\advance\v@lZ-\v@lZa\Figp@intregTD#1:(\v@lX,\v@lY,\v@lZ)}
\ctr@ln@m\figvectN
\ctr@ld@f\def\figvectNDD#1[#2,#3]{\ifGR@cri{\Figg@tXYa{#2}\Figg@tXY{#3}%
    \advance\v@lX-\v@lXa\advance\v@lY-\v@lYa%
    \Figv@ctCreg#1(-\v@lY,\v@lX)}\ignorespaces\fi}
\ctr@ld@f\def\figvectNTD#1[#2,#3,#4]{\ifGR@cri{\vecunitC@TD[#2,#4]\v@lmin=\v@lX\v@lmax=\v@lY%
    \v@leur=\v@lZ\vecunitC@TD[#2,#3]\c@lprovec{#1}}\ignorespaces\fi}
\ctr@ln@m\figvectNV
\ctr@ld@f\def\figvectNVDD#1[#2]{\ifGR@cri{\Figg@tXY{#2}\Figv@ctCreg#1(-\v@lY,\v@lX)}\ignorespaces\fi}
\ctr@ld@f\def\figvectNVTD#1[#2,#3]{\ifGR@cri{\vecunitCV@TD{#3}\v@lmin=\v@lX\v@lmax=\v@lY%
    \v@leur=\v@lZ\vecunitCV@TD{#2}\c@lprovec{#1}}\ignorespaces\fi}
\ctr@ln@m\figvectP
\ctr@ld@f\def\figvectPDD#1[#2,#3]{\ifGR@cri{\Figg@tXYa{#2}\Figg@tXY{#3}%
    \advance\v@lX-\v@lXa\advance\v@lY-\v@lYa%
    \Figv@ctCreg#1(\v@lX,\v@lY)}\ignorespaces\fi}
\ctr@ld@f\def\figvectPTD#1[#2,#3]{\ifGR@cri{\Figg@tXYa{#2}\Figg@tXY{#3}%
    \advance\v@lX-\v@lXa\advance\v@lY-\v@lYa\advance\v@lZ-\v@lZa%
    \Figv@ctCreg#1(\v@lX,\v@lY,\v@lZ)}\ignorespaces\fi}
\ctr@ln@m\figvectU
\ctr@ld@f\def\figvectUDD#1[#2]{\ifGR@cri{\n@rmeuc\v@leur{#2}\invers@\v@leur\v@leur%
    \delt@=\repdecn@mb{\v@leur}\unit@\edef\v@ldelt@{\repdecn@mb{\delt@}}%
    \Figg@tXY{#2}\v@lX=\v@ldelt@\v@lX\v@lY=\v@ldelt@\v@lY%
    \Figv@ctCreg#1(\v@lX,\v@lY)}\ignorespaces\fi}
\ctr@ld@f\def\figvectUTD#1[#2]{\ifGR@cri{\n@rmeuc\v@leur{#2}\invers@\v@leur\v@leur%
    \delt@=\repdecn@mb{\v@leur}\unit@\edef\v@ldelt@{\repdecn@mb{\delt@}}%
    \Figg@tXY{#2}\v@lX=\v@ldelt@\v@lX\v@lY=\v@ldelt@\v@lY\v@lZ=\v@ldelt@\v@lZ%
    \Figv@ctCreg#1(\v@lX,\v@lY,\v@lZ)}\ignorespaces\fi}
\ctr@ld@f\def\figvisu#1#2#3{\c@ldefproj\initb@undb@x\xdef\figforTeXFigno{\figforTeXnextFigno}%
    \s@mme=\figforTeXnextFigno\advance\s@mme\@ne\xdef\figforTeXnextFigno{\number\s@mme}%
    \setbox\b@xvisu=\hbox{\ifnum\@utoFN>\z@\figinsert{}\gdef\@utoFInDone{0}\fi\ignorespaces#3}%
    \gdef\@utoFInDone{1}\gdef\@utoFN{0}%
    \v@lXa=-\c@@rdYmin\v@lYa=\c@@rdYmax\advance\v@lYa-\c@@rdYmin%
    \v@lX=\c@@rdXmax\advance\v@lX-\c@@rdXmin%
    \setbox#1=\hbox{#2}\v@lY=-\v@lX\maxim@m{\v@lX}{\v@lX}{\wd#1}%
    \advance\v@lY\v@lX\divide\v@lY\tw@\advance\v@lY-\c@@rdXmin%
    \setbox#1=\vbox{\parindent\z@\hsize=\v@lX\vskip\v@lYa%
    \rlap{\hskip\v@lY\smash{\raise\v@lXa\box\b@xvisu}}%
    \def\t@xt@{#2}\ifx\t@xt@\empty\else\medskip\centerline{#2}\fi}\wd#1=\v@lX}
\ctr@ld@f\def\figDecrementFigno{{\xdef\figforTeXnextFigno{\figforTeXFigno}%
    \s@mme=\figforTeXFigno\advance\s@mme\m@ne\xdef\figforTeXFigno{\number\s@mme}}}
\ctr@ln@w{newbox}\Bt@rget\setbox\Bt@rget=\null
\ctr@ln@w{newbox}\BminTD@\setbox\BminTD@=\null
\ctr@ln@w{newbox}\BmaxTD@\setbox\BmaxTD@=\null
\ctr@ln@w{newif}\ifnewt@rgetpt\ctr@ln@w{newif}\ifnewdis@b
\ctr@ld@f\def\b@undb@xTD#1#2#3{%
    \relax\ifdim#1<\wd\BminTD@\global\wd\BminTD@=#1\fi%
    \relax\ifdim#2<\ht\BminTD@\global\ht\BminTD@=#2\fi%
    \relax\ifdim#3<\dp\BminTD@\global\dp\BminTD@=#3\fi%
    \relax\ifdim#1>\wd\BmaxTD@\global\wd\BmaxTD@=#1\fi%
    \relax\ifdim#2>\ht\BmaxTD@\global\ht\BmaxTD@=#2\fi%
    \relax\ifdim#3>\dp\BmaxTD@\global\dp\BmaxTD@=#3\fi}
\ctr@ld@f\def\c@ldefdisob{{\ifdim\wd\BminTD@<\maxdimen\v@leur=\wd\BmaxTD@\advance\v@leur-\wd\BminTD@%
    \delt@=\ht\BmaxTD@\advance\delt@-\ht\BminTD@\maxim@m{\v@leur}{\v@leur}{\delt@}%
    \delt@=\dp\BmaxTD@\advance\delt@-\dp\BminTD@\maxim@m{\v@leur}{\v@leur}{\delt@}%
    \v@leur=5\v@leur\else\v@leur=800pt\fi\c@ldefdisob@{\v@leur}}}
\ctr@ln@m\disob@intern
\ctr@ln@m\disob@
\ctr@ln@m\divf@ctproj
\ctr@ld@f\def\c@ldefdisob@#1{{\v@leur=#1\ifdim\v@leur<\p@\v@leur=800pt\fi%
    \xdef\disob@intern{\repdecn@mb{\v@leur}}%
    \delt@=\ptT@unit@\v@leur\xdef\disob@unit{\repdecn@mb{\delt@}}%
    \f@ctech=\@ne\loop\ifdim\v@leur>\t@n pt\divide\v@leur\t@n\multiply\f@ctech\t@n\repeat%
    \xdef\disob@{\repdecn@mb{\v@leur}}\xdef\divf@ctproj{\the\f@ctech}}%
    \global\newdis@btrue}
\ctr@ln@m\t@rgetpt
\ctr@ld@f\def\c@ldeft@rgetpt{\newt@rgetpttrue\def\t@rgetpt{CenterBoundBox}{%
    \delt@=\wd\BmaxTD@\advance\delt@-\wd\BminTD@\divide\delt@\tw@%
    \v@leur=\wd\BminTD@\advance\v@leur\delt@\global\wd\Bt@rget=\v@leur%
    \delt@=\ht\BmaxTD@\advance\delt@-\ht\BminTD@\divide\delt@\tw@%
    \v@leur=\ht\BminTD@\advance\v@leur\delt@\global\ht\Bt@rget=\v@leur%
    \delt@=\dp\BmaxTD@\advance\delt@-\dp\BminTD@\divide\delt@\tw@%
    \v@leur=\dp\BminTD@\advance\v@leur\delt@\global\dp\Bt@rget=\v@leur}}
\ctr@ln@m\c@ldefproj
\ctr@ld@f\def\c@ldefprojTD{\ifnewt@rgetpt\else\c@ldeft@rgetpt\fi\ifnewdis@b\else\c@ldefdisob\fi}
\ctr@ld@f\def\c@lprojcav{
    \v@lZa=\cxa@\v@lY\advance\v@lX\v@lZa%
    \v@lZa=\cxb@\v@lY\v@lY=\v@lZ\advance\v@lY\v@lZa\ignorespaces}
\ctr@ln@m\v@lcoef
\ctr@ld@f\def\c@lprojrea{
    \advance\v@lX-\wd\Bt@rget\advance\v@lY-\ht\Bt@rget\advance\v@lZ-\dp\Bt@rget%
    \v@lZa=\cza@\v@lX\advance\v@lZa\czb@\v@lY\advance\v@lZa\czc@\v@lZ%
    \divide\v@lZa\divf@ctproj\advance\v@lZa\disob@ pt\invers@{\v@lZa}{\v@lZa}%
    \v@lZa=\disob@\v@lZa\edef\v@lcoef{\repdecn@mb{\v@lZa}}%
    \v@lXa=\cxa@\v@lX\advance\v@lXa\cxb@\v@lY\v@lXa=\v@lcoef\v@lXa%
    \v@lY=\cyb@\v@lY\advance\v@lY\cya@\v@lX\advance\v@lY\cyc@\v@lZ%
    \v@lY=\v@lcoef\v@lY\v@lX=\v@lXa\ignorespaces}
\ctr@ld@f\def\c@lprojort{
    \v@lXa=\cxa@\v@lX\advance\v@lXa\cxb@\v@lY%
    \v@lY=\cyb@\v@lY\advance\v@lY\cya@\v@lX\advance\v@lY\cyc@\v@lZ%
    \v@lX=\v@lXa\ignorespaces}
\ctr@ld@f\def\Figptpr@j#1:#2/#3/{{\Figg@tXY{#3}\superc@lprojSP%
    \Figp@intregDD#1:{#2}(\v@lX,\v@lY)}\ignorespaces}
\ctr@ln@m\figsetobdist
\ctr@ld@f\def\figsetobdistDD{\un@v@ilable{figsetobdist}}
\ctr@ld@f\def\figsetobdistTD(#1){{\ifCUR@PS\W@rnmesIgn{figset proj(dist=...)}%
    \else\v@leur=#1\unit@\c@ldefdisob@{\v@leur}\fi}\ignorespaces}
\ctr@ln@m\c@lprojSP
\ctr@ln@m\CUR@proj
\ctr@ln@m\typ@proj
\ctr@ln@m\superc@lprojSP
\ctr@ld@f\def\Figs@tproj#1{%
    \if#13 \def@ultproj\else\if#1c\def@ultproj%
    \else\if#1o\xdef\CUR@proj{1}\xdef\typ@proj{orthogonal}%
         \figsetviewTD(\def@ultpsi,\def@ulttheta)%
         \global\let\c@lprojSP=\c@lprojort\global\let\superc@lprojSP=\c@lprojort%
    \else\if#1r\xdef\CUR@proj{2}\xdef\typ@proj{realistic}%
         \figsetviewTD(\def@ultpsi,\def@ulttheta)%
         \global\let\c@lprojSP=\c@lprojrea\global\let\superc@lprojSP=\c@lprojrea%
    \else\def@ultproj\message{*** Unknown projection. Cavalier projection assumed.}%
    \fi\fi\fi\fi}
\ctr@ld@f\def\def@ultproj{\xdef\CUR@proj{0}\xdef\typ@proj{cavalier}\figsetviewTD(\def@ultpsi,0.5)%
         \global\let\c@lprojSP=\c@lprojcav\global\let\superc@lprojSP=\c@lprojcav}
\ctr@ln@m\figsettarget
\ctr@ld@f\def\figsettargetDD{\un@v@ilable{figsettarget}}
\ctr@ld@f\def\figsettargetTD[#1]{{\ifCUR@PS\W@rnmesIgn{figset proj(targetpt=...)}%
    \else\global\newt@rgetpttrue\xdef\t@rgetpt{#1}\Figg@tXY{#1}\global\wd\Bt@rget=\v@lX%
    \global\ht\Bt@rget=\v@lY\global\dp\Bt@rget=\v@lZ\fi}\ignorespaces}
\ctr@ln@m\figsetview
\ctr@ld@f\def\figsetviewDD{\un@v@ilable{figsetview}}
\ctr@ld@f\def\figsetviewTD(#1){\ifCUR@PS\W@rnmesIgn{figset proj(Psi|Theta|Lambda=...)}%
     \else\Figsetview@#1,:\fi\ignorespaces}
\ctr@ld@f\def\Figsetview@#1,#2:{{\xdef\v@lPsi{#1}\def\t@xt@{#2}%
    \ifx\t@xt@\empty\def\@rgdeux{\v@lTheta}\else\X@rgdeux@#2\fi%
    \c@ssin{\costhet@}{\sinthet@}{#1}\v@lmin=\costhet@ pt\v@lmax=\sinthet@ pt%
    \ifcase\CUR@proj%
    \v@leur=\@rgdeux\v@lmin\xdef\cxa@{\repdecn@mb{\v@leur}}%
    \v@leur=\@rgdeux\v@lmax\xdef\cxb@{\repdecn@mb{\v@leur}}\v@leur=\@rgdeux pt%
    \relax\ifdim\v@leur>\p@\message{*** Lambda too large ! See \BS@ figset proj() !}\fi%
    \else%
    \v@lmax=-\v@lmax\xdef\cxa@{\repdecn@mb{\v@lmax}}\xdef\cxb@{\costhet@}%
    \ifx\t@xt@\empty\edef\@rgdeux{\def@ulttheta}\fi\c@ssin{\C@}{\S@}{\@rgdeux}%
    \v@lmax=-\S@ pt%
    \v@leur=\v@lmax\v@leur=\costhet@\v@leur\xdef\cya@{\repdecn@mb{\v@leur}}%
    \v@leur=\v@lmax\v@leur=\sinthet@\v@leur\xdef\cyb@{\repdecn@mb{\v@leur}}%
    \xdef\cyc@{\C@}\v@lmin=-\C@ pt%
    \v@leur=\v@lmin\v@leur=\costhet@\v@leur\xdef\cza@{\repdecn@mb{\v@leur}}%
    \v@leur=\v@lmin\v@leur=\sinthet@\v@leur\xdef\czb@{\repdecn@mb{\v@leur}}%
    \xdef\czc@{\repdecn@mb{\v@lmax}}\fi%
    \xdef\v@lTheta{\@rgdeux}}}
\ctr@ld@f\def\def@ultpsi{40}
\ctr@ld@f\def\def@ulttheta{25}
\ctr@ln@m\l@debut
\ctr@ln@m\n@mref
\ctr@ld@f\def\Figsetpr@j#1=#2|{\keln@mtr#1|%
    \def\n@mref{dep}\ifx\l@debut\n@mref\Figsetd@p{#2}\else
    \def\n@mref{dis}\ifx\l@debut\n@mref%
     \ifnum\CUR@proj=\tw@\figsetobdist(#2)\else\Figset@rr\fi\else
    \def\n@mref{lam}\ifx\l@debut\n@mref\Figsetd@p{#2}\else
    \def\n@mref{lat}\ifx\l@debut\n@mref\Figsetth@{#2}\else
    \def\n@mref{lon}\ifx\l@debut\n@mref\figsetview(#2)\else
    \def\n@mref{psi}\ifx\l@debut\n@mref\figsetview(#2)\else
    \def\n@mref{tar}\ifx\l@debut\n@mref%
     \ifnum\CUR@proj=\tw@\figsettarget[#2]\else\Figset@rr\fi\else
    \def\n@mref{the}\ifx\l@debut\n@mref\Figsetth@{#2}\else
    \W@rnmesAttr{figset proj}{#1}\fi\fi\fi\fi\fi\fi\fi\fi}
\ctr@ld@f\def\Figsetd@p#1{\ifnum\CUR@proj=\z@\figsetview(\v@lPsi,#1)\else\Figset@rr\fi}
\ctr@ld@f\def\Figsetth@#1{\ifnum\CUR@proj=\z@\Figset@rr\else\figsetview(\v@lPsi,#1)\fi}
\ctr@ld@f\def\Figset@rr{\message{*** \BS@ figset proj(): Attribute "\n@mref" ignored, incompatible
    with current projection}}
\ctr@ld@f\def\initb@undb@xTD{\wd\BminTD@=\maxdimen\ht\BminTD@=\maxdimen\dp\BminTD@=\maxdimen%
    \wd\BmaxTD@=-\maxdimen\ht\BmaxTD@=-\maxdimen\dp\BmaxTD@=-\maxdimen}
\ctr@ln@w{newbox}\Gb@x      
\ctr@ln@w{newbox}\Gb@xSC    
\ctr@ln@w{newtoks}\c@nsymb  
\ctr@ln@w{newif}\ifr@undcoord\ctr@ln@w{newif}\ifunitpr@sent
\ctr@ld@f\def\unssqrttw@{0.707106 }
\ctr@ld@f\def\figAst{\raise-1.15ex\hbox{$\ast$}}
\ctr@ld@f\def\figBullet{\raise-1.15ex\hbox{$\bullet$}}
\ctr@ld@f\def\figCirc{\raise-1.15ex\hbox{$\circ$}}
\ctr@ld@f\def\figDiamond{\raise-1.15ex\hbox{$\diamond$}}%
\ctr@ld@f\def\boxit#1#2{\leavevmode\hbox{\vrule\vbox{\hrule\vglue#1%
    \vtop{\hbox{\kern#1{#2}\kern#1}\vglue#1\hrule}}\vrule}}
\ctr@ld@f
\ctr@ld@f
\ctr@ld@f\def\c@nterpt{\ignorespaces%
    \kern-.5\wd\Gb@xSC%
    \raise-.5\ht\Gb@xSC\rlap{\hbox{\raise.5\dp\Gb@xSC\hbox{\copy\Gb@xSC}}}%
    \kern .5\wd\Gb@xSC\ignorespaces}
\ctr@ld@f\def\b@undb@xSC#1#2{{\v@lXa=#1\v@lYa=#2%
    \v@leur=\ht\Gb@xSC\advance\v@leur\dp\Gb@xSC%
    \advance\v@lXa-.5\wd\Gb@xSC\advance\v@lYa-.5\v@leur\b@undb@x{\v@lXa}{\v@lYa}%
    \advance\v@lXa\wd\Gb@xSC\advance\v@lYa\v@leur\b@undb@x{\v@lXa}{\v@lYa}}}
\ctr@ln@m\Dist@n
\ctr@ln@m\l@suite
\ctr@ld@f\def\@keldist#1#2{\edef\Dist@n{#2}\y@tiunit{\Dist@n}%
    \ifunitpr@sent#1=\Dist@n\else#1=\Dist@n\unit@\fi}
\ctr@ld@f\def\y@tiunit#1{\unitpr@sentfalse\expandafter\y@tiunit@#1:}
\ctr@ld@f\def\y@tiunit@#1#2:{\ifcat#1a\unitpr@senttrue\else\def\l@suite{#2}%
    \ifx\l@suite\empty\else\y@tiunit@#2:\fi\fi}
\ctr@ln@m\figcoord
\ctr@ld@f\def\figcoordDD#1{{\v@lX=\ptT@unit@\v@lX\v@lY=\ptT@unit@\v@lY%
    \ifr@undcoord\ifcase#1\v@leur=0.5pt\or\v@leur=0.05pt\or\v@leur=0.005pt%
    \or\v@leur=0.0005pt\else\v@leur=\z@\fi%
    \ifdim\v@lX<\z@\advance\v@lX-\v@leur\else\advance\v@lX\v@leur\fi%
    \ifdim\v@lY<\z@\advance\v@lY-\v@leur\else\advance\v@lY\v@leur\fi\fi%
    (\@ffichnb{#1}{\repdecn@mb{\v@lX}},\ifmmode\else\thinspace\fi%
    \@ffichnb{#1}{\repdecn@mb{\v@lY}})}}
\ctr@ld@f\def\@ffichnb#1#2{{\def\@@ffich{\@ffich#1(}\edef\n@mbre{#2}%
    \expandafter\@@ffich\n@mbre)}}
\ctr@ld@f\def\@ffich#1(#2.#3){{#2\ifnum#1>\z@.\fi\def\dig@ts{#3}\s@mme=\z@%
    \loop\ifnum\s@mme<#1\expandafter\@ffichdec\dig@ts:\advance\s@mme\@ne\repeat}}
\ctr@ld@f\def\@ffichdec#1#2:{\relax#1\def\dig@ts{#20}}
\ctr@ld@f\def\figcoordTD#1{{\v@lX=\ptT@unit@\v@lX\v@lY=\ptT@unit@\v@lY\v@lZ=\ptT@unit@\v@lZ%
    \ifr@undcoord\ifcase#1\v@leur=0.5pt\or\v@leur=0.05pt\or\v@leur=0.005pt%
    \or\v@leur=0.0005pt\else\v@leur=\z@\fi%
    \ifdim\v@lX<\z@\advance\v@lX-\v@leur\else\advance\v@lX\v@leur\fi%
    \ifdim\v@lY<\z@\advance\v@lY-\v@leur\else\advance\v@lY\v@leur\fi%
    \ifdim\v@lZ<\z@\advance\v@lZ-\v@leur\else\advance\v@lZ\v@leur\fi\fi%
    (\@ffichnb{#1}{\repdecn@mb{\v@lX}},\ifmmode\else\thinspace\fi%
     \@ffichnb{#1}{\repdecn@mb{\v@lY}},\ifmmode\else\thinspace\fi%
     \@ffichnb{#1}{\repdecn@mb{\v@lZ}})}}
\ctr@ld@f\def\figsetroundcoord#1{\expandafter\Figsetr@undcoord#1:\ignorespaces}
\ctr@ld@f\def\Figsetr@undcoord#1#2:{\if#1n\r@undcoordfalse\else\r@undcoordtrue\fi}
\ctr@ld@f\def\Figsetwr@te#1=#2|{\keln@mun#1|%
    \def\n@mref{m}\ifx\l@debut\n@mref\figsetmark{#2}\else
    \def\n@mref{p}\ifx\l@debut\n@mref\figsetptname{#2}\else
    \def\n@mref{r}\ifx\l@debut\n@mref\figsetroundcoord{#2}\else
    \W@rnmesAttr{figset write}{#1}\fi\fi\fi}
\ctr@ld@f\def\figsetmark#1{\c@nsymb={#1}\setbox\Gb@xSC=\hbox{\the\c@nsymb}\ignorespaces}
\ctr@ln@m\ptn@me
\ctr@ld@f\def\figsetptname#1{\def\ptn@me##1{#1}\ignorespaces}
\ctr@ld@f\def\FigWrit@L#1:#2(#3,#4){\ignorespaces\@keldist\v@leur{#3}\@keldist\delt@{#4}%
    \C@rp@r@m\def\list@num{#1}\@ecfor\p@int:=\list@num\do{\FigWrit@pt{\p@int}{#2}}}
\ctr@ld@f\def\FigWrit@pt#1#2{\FigWp@r@m{#1}{#2}\Vc@rrect\figWp@si%
    \ifdim\wd\Gb@xSC>\z@\b@undb@xSC{\v@lX}{\v@lY}\fi\figWBB@x}
\ctr@ld@f\def\FigWp@r@m#1#2{\Figg@tXY{#1}%
    \setbox\Gb@x=\hbox{\def\t@xt@{#2}\ifx\t@xt@\empty\Figg@tT{#1}\else#2\fi}\c@lprojSP}
\ctr@ld@f\let\Vc@rrect=\relax
\ctr@ld@f\let\C@rp@r@m=\relax
\ctr@ld@f\def\figwrite[#1]#2{{\ignorespaces\def\list@num{#1}\@ecfor\p@int:=\list@num\do{%
    \setbox\Gb@x=\hbox{\def\t@xt@{#2}\ifx\t@xt@\empty\Figg@tT{\p@int}\else#2\fi}%
    \Figwrit@{\p@int}}}\ignorespaces}
\ctr@ld@f\def\Figwrit@#1{\Figg@tXY{#1}\c@lprojSP%
    \rlap{\kern\v@lX\raise\v@lY\hbox{\unhcopy\Gb@x}}\v@leur=\v@lY%
    \advance\v@lY\ht\Gb@x\b@undb@x{\v@lX}{\v@lY}\advance\v@lX\wd\Gb@x%
    \v@lY=\v@leur\advance\v@lY-\dp\Gb@x\b@undb@x{\v@lX}{\v@lY}}
\ctr@ld@f\def\figwritec[#1]#2{{\ignorespaces\def\list@num{#1}%
    \@ecfor\p@int:=\list@num\do{\Figwrit@c{\p@int}{#2}}}\ignorespaces}
\ctr@ld@f\def\Figwrit@c#1#2{\FigWp@r@m{#1}{#2}%
    \rlap{\kern\v@lX\raise\v@lY\hbox{\rlap{\kern-.5\wd\Gb@x%
    \raise-.5\ht\Gb@x\hbox{\raise.5\dp\Gb@x\hbox{\unhcopy\Gb@x}}}}}%
    \v@leur=\ht\Gb@x\advance\v@leur\dp\Gb@x%
    \advance\v@lX-.5\wd\Gb@x\advance\v@lY-.5\v@leur\b@undb@x{\v@lX}{\v@lY}%
    \advance\v@lX\wd\Gb@x\advance\v@lY\v@leur\b@undb@x{\v@lX}{\v@lY}}
\ctr@ld@f\def\figwritep[#1]{{\ignorespaces\def\list@num{#1}\setbox\Gb@x=\hbox{\c@nterpt}%
    \@ecfor\p@int:=\list@num\do{\Figwrit@{\p@int}}}\ignorespaces}
\ctr@ld@f\def\figwritew#1:#2(#3){\figwritegcw#1:{#2}(#3,0pt)}
\ctr@ld@f\def\figwritee#1:#2(#3){\figwritegce#1:{#2}(#3,0pt)}
\ctr@ld@f\def\figwriten#1:#2(#3){{\def\Vc@rrect{\v@lZ=\v@leur\advance\v@lZ\dp\Gb@x}%
    \Figwrit@NS#1:{#2}(#3)}\ignorespaces}
\ctr@ld@f\def\figwrites#1:#2(#3){{\def\Vc@rrect{\v@lZ=-\v@leur\advance\v@lZ-\ht\Gb@x}%
    \Figwrit@NS#1:{#2}(#3)}\ignorespaces}
\ctr@ld@f\def\Figwrit@NS#1:#2(#3){\let\figWp@si=\FigWp@siNS\let\figWBB@x=\FigWBB@xNS%
    \FigWrit@L#1:{#2}(#3,0pt)}
\ctr@ld@f\def\FigWp@siNS{\rlap{\kern\v@lX\raise\v@lY\hbox{\rlap{\kern-.5\wd\Gb@x%
    \raise\v@lZ\hbox{\unhcopy\Gb@x}}\c@nterpt}}}
\ctr@ld@f\def\FigWBB@xNS{\advance\v@lY\v@lZ%
    \advance\v@lY-\dp\Gb@x\advance\v@lX-.5\wd\Gb@x\b@undb@x{\v@lX}{\v@lY}%
    \advance\v@lY\ht\Gb@x\advance\v@lY\dp\Gb@x%
    \advance\v@lX\wd\Gb@x\b@undb@x{\v@lX}{\v@lY}}
\ctr@ld@f\def\figwritenw#1:#2(#3){{\let\figWp@si=\FigWp@sigW\let\figWBB@x=\FigWBB@xgWE%
    \def\C@rp@r@m{\v@leur=\unssqrttw@\v@leur\delt@=\v@leur%
    \ifdim\delt@=\z@\delt@=\epsil@n\fi}\let@xte={-}\FigWrit@L#1:{#2}(#3,0pt)}\ignorespaces}
\ctr@ld@f\def\figwritesw#1:#2(#3){{\let\figWp@si=\FigWp@sigW\let\figWBB@x=\FigWBB@xgWE%
    \def\C@rp@r@m{\v@leur=\unssqrttw@\v@leur\delt@=-\v@leur%
    \ifdim\delt@=\z@\delt@=-\epsil@n\fi}\let@xte={-}\FigWrit@L#1:{#2}(#3,0pt)}\ignorespaces}
\ctr@ld@f\def\figwritene#1:#2(#3){{\let\figWp@si=\FigWp@sigE\let\figWBB@x=\FigWBB@xgWE%
    \def\C@rp@r@m{\v@leur=\unssqrttw@\v@leur\delt@=\v@leur%
    \ifdim\delt@=\z@\delt@=\epsil@n\fi}\let@xte={}\FigWrit@L#1:{#2}(#3,0pt)}\ignorespaces}
\ctr@ld@f\def\figwritese#1:#2(#3){{\let\figWp@si=\FigWp@sigE\let\figWBB@x=\FigWBB@xgWE%
    \def\C@rp@r@m{\v@leur=\unssqrttw@\v@leur\delt@=-\v@leur%
    \ifdim\delt@=\z@\delt@=-\epsil@n\fi}\let@xte={}\FigWrit@L#1:{#2}(#3,0pt)}\ignorespaces}
\ctr@ld@f\def\figwritegw#1:#2(#3,#4){{\let\figWp@si=\FigWp@sigW\let\figWBB@x=\FigWBB@xgWE%
    \let@xte={-}\FigWrit@L#1:{#2}(#3,#4)}\ignorespaces}
\ctr@ld@f\def\figwritege#1:#2(#3,#4){{\let\figWp@si=\FigWp@sigE\let\figWBB@x=\FigWBB@xgWE%
    \let@xte={}\FigWrit@L#1:{#2}(#3,#4)}\ignorespaces}
\ctr@ld@f\def\FigWp@sigW{\v@lXa=\z@\v@lYa=\ht\Gb@x\advance\v@lYa\dp\Gb@x%
    \ifdim\delt@>\z@\relax%
    \rlap{\kern\v@lX\raise\v@lY\hbox{\rlap{\kern-\wd\Gb@x\kern-\v@leur%
          \raise\delt@\hbox{\raise\dp\Gb@x\hbox{\unhcopy\Gb@x}}}\c@nterpt}}%
    \else\ifdim\delt@<\z@\relax\v@lYa=-\v@lYa%
    \rlap{\kern\v@lX\raise\v@lY\hbox{\rlap{\kern-\wd\Gb@x\kern-\v@leur%
          \raise\delt@\hbox{\raise-\ht\Gb@x\hbox{\unhcopy\Gb@x}}}\c@nterpt}}%
    \else\v@lXa=-.5\v@lYa%
    \rlap{\kern\v@lX\raise\v@lY\hbox{\rlap{\kern-\wd\Gb@x\kern-\v@leur%
          \raise-.5\ht\Gb@x\hbox{\raise.5\dp\Gb@x\hbox{\unhcopy\Gb@x}}}\c@nterpt}}%
    \fi\fi}
\ctr@ld@f\def\FigWp@sigE{\v@lXa=\z@\v@lYa=\ht\Gb@x\advance\v@lYa\dp\Gb@x%
    \ifdim\delt@>\z@\relax%
    \rlap{\kern\v@lX\raise\v@lY\hbox{\c@nterpt\kern\v@leur%
          \raise\delt@\hbox{\raise\dp\Gb@x\hbox{\unhcopy\Gb@x}}}}%
    \else\ifdim\delt@<\z@\relax\v@lYa=-\v@lYa%
    \rlap{\kern\v@lX\raise\v@lY\hbox{\c@nterpt\kern\v@leur%
          \raise\delt@\hbox{\raise-\ht\Gb@x\hbox{\unhcopy\Gb@x}}}}%
    \else\v@lXa=-.5\v@lYa%
    \rlap{\kern\v@lX\raise\v@lY\hbox{\c@nterpt\kern\v@leur%
          \raise-.5\ht\Gb@x\hbox{\raise.5\dp\Gb@x\hbox{\unhcopy\Gb@x}}}}%
    \fi\fi}
\ctr@ld@f\def\FigWBB@xgWE{\advance\v@lY\delt@%
    \advance\v@lX\the\let@xte\v@leur\advance\v@lY\v@lXa\b@undb@x{\v@lX}{\v@lY}%
    \advance\v@lX\the\let@xte\wd\Gb@x\advance\v@lY\v@lYa\b@undb@x{\v@lX}{\v@lY}}
\ctr@ld@f\def\figwritegcw#1:#2(#3,#4){{\let\figWp@si=\FigWp@sigcW\let\figWBB@x=\FigWBB@xgcWE%
    \let@xte={-}\FigWrit@L#1:{#2}(#3,#4)}\ignorespaces}
\ctr@ld@f\def\figwritegce#1:#2(#3,#4){{\let\figWp@si=\FigWp@sigcE\let\figWBB@x=\FigWBB@xgcWE%
    \let@xte={}\FigWrit@L#1:{#2}(#3,#4)}\ignorespaces}
\ctr@ld@f\def\FigWp@sigcW{\rlap{\kern\v@lX\raise\v@lY\hbox{\rlap{\kern-\wd\Gb@x\kern-\v@leur%
     \raise-.5\ht\Gb@x\hbox{\raise\delt@\hbox{\raise.5\dp\Gb@x\hbox{\unhcopy\Gb@x}}}}%
     \c@nterpt}}}
\ctr@ld@f\def\FigWp@sigcE{\rlap{\kern\v@lX\raise\v@lY\hbox{\c@nterpt\kern\v@leur%
    \raise-.5\ht\Gb@x\hbox{\raise\delt@\hbox{\raise.5\dp\Gb@x\hbox{\unhcopy\Gb@x}}}}}}
\ctr@ld@f\def\FigWBB@xgcWE{\v@lZ=\ht\Gb@x\advance\v@lZ\dp\Gb@x%
    \advance\v@lX\the\let@xte\v@leur\advance\v@lY\delt@\advance\v@lY.5\v@lZ%
    \b@undb@x{\v@lX}{\v@lY}%
    \advance\v@lX\the\let@xte\wd\Gb@x\advance\v@lY-\v@lZ\b@undb@x{\v@lX}{\v@lY}}
\ctr@ld@f\def\figwritebn#1:#2(#3){{\def\Vc@rrect{\v@lZ=\v@leur}\Figwrit@NS#1:{#2}(#3)}\ignorespaces}
\ctr@ld@f\def\figwritebs#1:#2(#3){{\def\Vc@rrect{\v@lZ=-\v@leur}\Figwrit@NS#1:{#2}(#3)}\ignorespaces}
\ctr@ld@f\def\figwritebw#1:#2(#3){{\let\figWp@si=\FigWp@sibW\let\figWBB@x=\FigWBB@xbWE%
    \let@xte={-}\FigWrit@L#1:{#2}(#3,0pt)}\ignorespaces}
\ctr@ld@f\def\figwritebe#1:#2(#3){{\let\figWp@si=\FigWp@sibE\let\figWBB@x=\FigWBB@xbWE%
    \let@xte={}\FigWrit@L#1:{#2}(#3,0pt)}\ignorespaces}
\ctr@ld@f\def\FigWp@sibW{\rlap{\kern\v@lX\raise\v@lY\hbox{\rlap{\kern-\wd\Gb@x\kern-\v@leur%
          \hbox{\unhcopy\Gb@x}}\c@nterpt}}}
\ctr@ld@f\def\FigWp@sibE{\rlap{\kern\v@lX\raise\v@lY\hbox{\c@nterpt\kern\v@leur%
          \hbox{\unhcopy\Gb@x}}}}
\ctr@ld@f\def\FigWBB@xbWE{\v@lZ=\ht\Gb@x\advance\v@lZ\dp\Gb@x%
    \advance\v@lX\the\let@xte\v@leur\advance\v@lY\ht\Gb@x\b@undb@x{\v@lX}{\v@lY}%
    \advance\v@lX\the\let@xte\wd\Gb@x\advance\v@lY-\v@lZ\b@undb@x{\v@lX}{\v@lY}}
\ctr@ln@w{newread}\frf@g  \ctr@ln@w{newwrite}\fwf@g
\ctr@ln@w{newif}\ifCUR@PS
\ctr@ln@w{newif}\ifGR@cri
\ctr@ln@w{newif}\ifUse@llipse
\ctr@ln@w{newif}\ifGRdebugm@de \GRdebugm@defalse 
\ctr@ln@w{newif}\ifPDFm@ke
\ifx\pdfliteral\undefined\else\ifnum\pdfoutput>\z@\PDFm@ketrue\fi\fi
\ctr@ld@f\def\initPDF@rDVI{%
\ifPDFm@ke
 \let\figscan=\figscan@E
 \let\newGr@FN=\newGr@FNPDF
 \ctr@ld@f\def\c@mcurveto{c}
 \ctr@ld@f\def\c@mfill{f}
 \ctr@ld@f\def\c@mgsave{q}
 \ctr@ld@f\def\c@mgrestore{Q}
 \ctr@ld@f\def\c@mlineto{l}
 \ctr@ld@f\def\c@mmoveto{m}
 \ctr@ld@f\def\c@msetgray{g}     \ctr@ld@f\def\c@msetgrayStroke{G}
 \ctr@ld@f\def\c@msetcmykcolor{k}\ctr@ld@f\def\c@msetcmykcolorStroke{K}
 \ctr@ld@f\def\c@msetrgbcolor{rg}\ctr@ld@f\def\c@msetrgbcolorStroke{RG}
 \ctr@ld@f\def\d@fprimarC@lor{\CUR@color\space\CUR@colorc@md%
               \space\CUR@color\space\CUR@colorc@mdStroke}
 \ctr@ld@f\def\c@msetdash{d}
 \ctr@ld@f\def\c@msetlinejoin{j}
 \ctr@ld@f\def\c@msetlinewidth{w}
 \ctr@ld@f\def\f@gclosestroke{\immediate\write\fwf@g{s}}
 \ctr@ld@f\def\f@gfill{\immediate\write\fwf@g{\fillc@md}}
 \ctr@ld@f\def\f@gnewpath{}
 \ctr@ld@f\def\f@gstroke{\immediate\write\fwf@g{S}}
\else
 \let\figinsertE=\figinsert
 \let\newGr@FN=\newGr@FNDVI
 \ctr@ld@f\def\c@mcurveto{curveto}
 \ctr@ld@f\def\c@mfill{fill}
 \ctr@ld@f\def\c@mgsave{gsave}
 \ctr@ld@f\def\c@mgrestore{grestore}
 \ctr@ld@f\def\c@mlineto{lineto}
 \ctr@ld@f\def\c@mmoveto{moveto}
 \ctr@ld@f\def\c@msetgray{setgray}          \ctr@ld@f\def\c@msetgrayStroke{}
 \ctr@ld@f\def\c@msetcmykcolor{setcmykcolor}\ctr@ld@f\def\c@msetcmykcolorStroke{}
 \ctr@ld@f\def\c@msetrgbcolor{setrgbcolor}  \ctr@ld@f\def\c@msetrgbcolorStroke{}
 \ctr@ld@f\def\d@fprimarC@lor{\CUR@color\space\CUR@colorc@md}
 \ctr@ld@f\def\c@msetdash{setdash}
 \ctr@ld@f\def\c@msetlinejoin{setlinejoin}
 \ctr@ld@f\def\c@msetlinewidth{setlinewidth}
 \ctr@ld@f\def\f@gclosestroke{\immediate\write\fwf@g{closepath\space stroke}}
 \ctr@ld@f\def\f@gfill{\immediate\write\fwf@g{\fillc@md}}
 \ctr@ld@f\def\f@gnewpath{\immediate\write\fwf@g{newpath}}
 \ctr@ld@f\def\f@gstroke{\immediate\write\fwf@g{stroke}}
\fi}
\ctr@ld@f\def\c@pypsfile#1#2{\c@pyfil@{\immediate\write#1}{#2}}
\ctr@ld@f\def\Figinclud@PDF#1#2{\openin\frf@g=#1\pdfliteral{q #2 0 0 #2 0 0 cm}%
    \c@pyfil@{\pdfliteral}{\frf@g}\pdfliteral{Q}\closein\frf@g}
\ctr@ln@w{newif}\ifmored@ta
\ctr@ln@m\bl@nkline
\ctr@ld@f\def\c@pyfil@#1#2{\def\bl@nkline{\par}{\catcode`\%=12
    \loop\ifeof#2\mored@tafalse\else\mored@tatrue\immediate\read#2 to\tr@c
    \ifx\tr@c\bl@nkline\else#1{\tr@c}\fi\fi\ifmored@ta\repeat}}
\ctr@ld@f\def\keln@mun#1#2|{\def\l@debut{#1}\def\l@suite{#2}}
\ctr@ld@f\def\keln@mde#1#2#3|{\def\l@debut{#1#2}\def\l@suite{#3}}
\ctr@ld@f\def\keln@mtr#1#2#3#4|{\def\l@debut{#1#2#3}\def\l@suite{#4}}
\ctr@ld@f\def\keln@mqu#1#2#3#4#5|{\def\l@debut{#1#2#3#4}\def\l@suite{#5}}
\ctr@ld@f\let\@psffilein=\frf@g 
\ctr@ln@w{newif}\if@psffileok    
\ctr@ln@w{newif}\if@psfbbfound   
\ctr@ln@w{newif}\if@psfverbose   
\@psfverbosetrue
\ctr@ln@m\@psfllx \ctr@ln@m\@psflly
\ctr@ln@m\@psfurx \ctr@ln@m\@psfury
\ctr@ln@m\resetcolonc@tcode
\ctr@ld@f\def\@psfgetbb#1{\global\@psfbbfoundfalse%
\global\def\@psfllx{0}\global\def\@psflly{0}%
\global\def\@psfurx{30}\global\def\@psfury{30}%
\openin\@psffilein=#1\relax
\ifeof\@psffilein\errmessage{I couldn't open #1, will ignore it}\else
   \edef\resetcolonc@tcode{\catcode`\noexpand\:\the\catcode`\:\relax}%
   {\@psffileoktrue \chardef\other=12
    \def\do##1{\catcode`##1=\other}\dospecials \catcode`\ =10 \resetcolonc@tcode
    \loop
       \read\@psffilein to \@psffileline
       \ifeof\@psffilein\@psffileokfalse\else
          \expandafter\@psfaux\@psffileline:. \\%
       \fi
   \if@psffileok\repeat
   \if@psfbbfound\else
    \if@psfverbose\message{No bounding box comment in #1; using defaults}\fi\fi
   }\closein\@psffilein\fi}%
\ctr@ln@m\@psfbblit
\ctr@ln@m\@psfpercent
{\catcode`\%=12 \global\let\@psfpercent=
\ctr@ln@m\@psfaux
\long\def\@psfaux#1#2:#3\\{\ifx#1\@psfpercent
   \def\testit{#2}\ifx\testit\@psfbblit
      \@psfgrab #3 . . . \\%
      \@psffileokfalse
      \global\@psfbbfoundtrue
   \fi\else\ifx#1\par\else\@psffileokfalse\fi\fi}%
\ctr@ld@f\def\@psfempty{}%
\ctr@ld@f\def\@psfgrab #1 #2 #3 #4 #5\\{%
\global\def\@psfllx{#1}\ifx\@psfllx\@psfempty
      \@psfgrab #2 #3 #4 #5 .\\\else
   \global\def\@psflly{#2}%
   \global\def\@psfurx{#3}\global\def\@psfury{#4}\fi}%
\ctr@ld@f\def\PSwrit@cmd#1#2#3{{\Figg@tXY{#1}\c@lprojSP\b@undb@x{\v@lX}{\v@lY}%
    \v@lX=\ptT@ptps\v@lX\v@lY=\ptT@ptps\v@lY%
    \immediate\write#3{\repdecn@mb{\v@lX}\space\repdecn@mb{\v@lY}\space#2}}}
\ctr@ld@f\def\PSwrit@cmdS#1#2#3#4#5{{\Figg@tXY{#1}\c@lprojSP\b@undb@x{\v@lX}{\v@lY}%
    \global\result@t=\v@lX\global\result@@t=\v@lY%
    \v@lX=\ptT@ptps\v@lX\v@lY=\ptT@ptps\v@lY%
    \immediate\write#3{\repdecn@mb{\v@lX}\space\repdecn@mb{\v@lY}\space#2}}%
    \edef#4{\the\result@t}\edef#5{\the\result@@t}}
\ctr@ld@f\def\update@ttr#1#2#3{\Figdisc@rdLTS{#3}{\n@mref}%
    \ifx\n@mref\D@FTref#2{#1}\else#2{#3}\fi}
\ctr@ld@f\def\D@FTref{default}
\ctr@ld@f\def\W@rnmesAttr#1#2{%
    \immediate\write16{*** Unknown attribute: \BS@ #1(..., #2=...)}}
\ctr@ld@f\def\W@rnmeskwd#1#2{%
    \immediate\write16{*** Unknown keyword #2 in \BS@ #1}}
\ctr@ld@f\def\W@rnmesIgn#1{\immediate\write16{*** \BS@ #1 is ignored inside a
     \BS@ figdrawbegin-\BS@ figdrawend block.}}
\ctr@ld@f\def\Psset@lti#1=#2|{\keln@mtr#1|%
    \def\n@mref{blc}\ifx\l@debut\n@mref\update@ttr\D@FTref\P@setblcolor{#2}\else
    \def\n@mref{bld}\ifx\l@debut\n@mref\update@ttr\D@FTref\P@setbldash{#2}\else
    \def\n@mref{blw}\ifx\l@debut\n@mref\update@ttr\D@FTref\P@setblwidth{#2}\else
    \def\n@mref{sqc}\ifx\l@debut\n@mref\update@ttr\D@FTref\P@setsqcolor{#2}\else
    \def\n@mref{sqd}\ifx\l@debut\n@mref\update@ttr\D@FTref\P@setsqdash{#2}\else
    \def\n@mref{sqw}\ifx\l@debut\n@mref\update@ttr\D@FTref\P@setsqwidth{#2}\else
    \W@rnmesAttr{figset altitude}{#1}\fi\fi\fi\fi\fi\fi}
\ctr@ln@m\DDV@blcolor
\ctr@ld@f\def\P@setblcolor#1{\edef\DDV@blcolor{#1}}
\ctr@ln@m\DDV@bldash
\ctr@ld@f\def\P@setbldash#1{\edef\DDV@bldash{#1}}
\ctr@ln@m\DDV@blwidth
\ctr@ld@f\def\P@setblwidth#1{\edef\DDV@blwidth{#1}}
\ctr@ln@m\DDV@sqcolor
\ctr@ld@f\def\P@setsqcolor#1{\edef\DDV@sqcolor{#1}}
\ctr@ln@m\DDV@sqdash
\ctr@ld@f\def\P@setsqdash#1{\edef\DDV@sqdash{#1}}
\ctr@ln@m\DDV@sqwidth
\ctr@ld@f\def\P@setsqwidth#1{\edef\DDV@sqwidth{#1}}
\ctr@ld@f\def\figdrawaltitude#1[#2,#3,#4]{{\ifCUR@PS\ifGR@cri%
    \PSc@mment{altitude Square Dim=#1, Triangle=[#2 / #3,#4]}%
    \s@uvc@ntr@l\et@tpsaltitude\resetc@ntr@l{2}\figptorthoprojline-5:=#2/#3,#4/%
    \figvectP -1[#3,#4]\n@rminf{\v@leur}{-1}\vecunit@{-3}{-1}%
    \figvectP -1[-5,#3]\n@rminf{\v@lmin}{-1}\figvectP -2[-5,#4]\n@rminf{\v@lmax}{-2}%
    \ifdim\v@lmin<\v@lmax\s@mme=#3\else\v@lmax=\v@lmin\s@mme=#4\fi%
    \figvectP -4[-5,#2]\vecunit@{-4}{-4}\delt@=#1\unit@%
    \edef\t@ille{\repdecn@mb{\delt@}}\figpttra-1:=-5/\t@ille,-3/%
    \figptstra-3=-5,-1/\t@ille,-4/\figdrawline[#2,-5]%
    \Pss@tspecifSt{color=\DDV@sqcolor,dash=\DDV@sqdash,width=\DDV@sqwidth}%
    \figdrawline[-1,-2,-3]%
    \Psrest@reSt{color=\DDV@sqcolor,dash=\DDV@sqdash,width=\DDV@sqwidth}%
    \ifdim\v@leur<\v@lmax%
    \Pss@tspecifSt{color=\DDV@blcolor,dash=\DDV@bldash,width=\DDV@blwidth}%
    \figdrawline[-5,\the\s@mme]%
    \Psrest@reSt{color=\DDV@blcolor,dash=\DDV@bldash,width=\DDV@blwidth}%
    \fi\PSc@mment{End altitude}\resetc@ntr@l\et@tpsaltitude\fi\fi}}
\ctr@ld@f\def\Ps@rcerc#1;#2(#3,#4){\ellBB@x#1;#2,#2(#3,#4,0)%
    \f@gnewpath{\delt@=#2\unit@\delt@=\ptT@ptps\delt@%
    \BdingB@xfalse%
    \PSwrit@cmd{#1}{\repdecn@mb{\delt@}\space #3\space #4\space arc}{\fwf@g}}}
\ctr@ln@m\figdrawarccirc
\ctr@ld@f\def\Q@arccircDD#1;#2(#3,#4){\ifCUR@PS\ifGR@cri%
    \PSc@mment{arccircDD Center=#1 ; Radius=#2 (Ang1=#3, Ang2=#4)}%
    \iffillm@de\Ps@rcerc#1;#2(#3,#4)%
    \f@gfill%
    \else\Ps@rcerc#1;#2(#3,#4)\f@gstroke\fi%
    \PSc@mment{End arccircDD}\fi\fi}
\ctr@ld@f\def\Q@arccircTD#1,#2,#3;#4(#5,#6){{\ifCUR@PS\ifGR@cri\s@uvc@ntr@l\et@tpsarccircTD%
    \PSc@mment{arccircTD Center=#1,P1=#2,P2=#3 ; Radius=#4 (Ang1=#5, Ang2=#6)}%
    \setc@ntr@l{2}\c@lExtAxes#1,#2,#3(#4)\Q@arcellPATD#1,-4,-5(#5,#6)%
    \PSc@mment{End arccircTD}\resetc@ntr@l\et@tpsarccircTD\fi\fi}}
\ctr@ld@f\def\c@lExtAxes#1,#2,#3(#4){%
    \figvectPTD-5[#1,#2]\vecunit@{-5}{-5}\figvectNTD-4[#1,#2,#3]\vecunit@{-4}{-4}%
    \figvectNVTD-3[-4,-5]\delt@=#4\unit@\edef\r@yon{\repdecn@mb{\delt@}}%
    \figpttra-4:=#1/\r@yon,-5/\figpttra-5:=#1/\r@yon,-3/}
\ctr@ln@m\figdrawarccircP
\ctr@ld@f\def\Q@arccircPDD#1;#2[#3,#4]{{\ifCUR@PS\ifGR@cri\s@uvc@ntr@l\et@tpsarccircPDD%
    \PSc@mment{arccircPDD Center=#1; Radius=#2, [P1=#3, P2=#4]}%
    \Ps@ngleparam#1;#2[#3,#4]\ifdim\v@lmin>\v@lmax\advance\v@lmax\DePI@deg\fi%
    \edef\@ngdeb{\repdecn@mb{\v@lmin}}\edef\@ngfin{\repdecn@mb{\v@lmax}}%
    \figdrawarccirc#1;\r@dius(\@ngdeb,\@ngfin)%
    \PSc@mment{End arccircPDD}\resetc@ntr@l\et@tpsarccircPDD\fi\fi}}
\ctr@ld@f\def\Q@arccircPTD#1;#2[#3,#4,#5]{{\ifCUR@PS\ifGR@cri\s@uvc@ntr@l\et@tpsarccircPTD%
    \PSc@mment{arccircPTD Center=#1; Radius=#2, [P1=#3, P2=#4, P3=#5]}%
    \setc@ntr@l{2}\c@lExtAxes#1,#3,#5(#2)\figdrawarcellPP#1,-4,-5[#3,#4]%
    \PSc@mment{End arccircPTD}\resetc@ntr@l\et@tpsarccircPTD\fi\fi}}
\ctr@ld@f\def\Ps@ngleparam#1;#2[#3,#4]{\setc@ntr@l{2}%
    \figvectPDD-1[#1,#3]\vecunit@{-1}{-1}\Figg@tXY{-1}\arct@n\v@lmin(\v@lX,\v@lY)%
    \figvectPDD-2[#1,#4]\vecunit@{-2}{-2}\Figg@tXY{-2}\arct@n\v@lmax(\v@lX,\v@lY)%
    \v@lmin=\rdT@deg\v@lmin\v@lmax=\rdT@deg\v@lmax%
    \v@leur=#2pt\maxim@m{\mili@u}{-\v@leur}{\v@leur}%
    \edef\r@dius{\repdecn@mb{\mili@u}}}
\ctr@ld@f\def\Ps@rcercBz#1;#2(#3,#4){\Ps@rellBz#1;#2,#2(#3,#4,0)}
\ctr@ld@f\def\Ps@rellBz#1;#2,#3(#4,#5,#6){%
    \ellBB@x#1;#2,#3(#4,#5,#6)\BdingB@xfalse%
    \c@lNbarcs{#4}{#5}\v@leur=#4pt\setc@ntr@l{2}\figptell-13::#1;#2,#3(#4,#6)%
    \f@gnewpath\PSwrit@cmd{-13}{\c@mmoveto}{\fwf@g}%
    \s@mme=\z@\bcl@rellBz#1;#2,#3(#6)\BdingB@xtrue}
\ctr@ld@f\def\bcl@rellBz#1;#2,#3(#4){\relax%
    \ifnum\s@mme<\p@rtent\advance\s@mme\@ne%
    \advance\v@leur\delt@\edef\@ngle{\repdecn@mb\v@leur}\figptell-14::#1;#2,#3(\@ngle,#4)%
    \advance\v@leur\delt@\edef\@ngle{\repdecn@mb\v@leur}\figptell-15::#1;#2,#3(\@ngle,#4)%
    \advance\v@leur\delt@\edef\@ngle{\repdecn@mb\v@leur}\figptell-16::#1;#2,#3(\@ngle,#4)%
    \figptscontrolDD-18[-13,-14,-15,-16]%
    \PSwrit@cmd{-18}{}{\fwf@g}\PSwrit@cmd{-17}{}{\fwf@g}%
    \PSwrit@cmd{-16}{\c@mcurveto}{\fwf@g}%
    \figptcopyDD-13:/-16/\bcl@rellBz#1;#2,#3(#4)\fi}
\ctr@ld@f\def\Ps@rell#1;#2,#3(#4,#5,#6){\ellBB@x#1;#2,#3(#4,#5,#6)%
    \f@gnewpath{\v@lmin=#2\unit@\v@lmin=\ptT@ptps\v@lmin%
    \v@lmax=#3\unit@\v@lmax=\ptT@ptps\v@lmax\BdingB@xfalse%
    \PSwrit@cmd{#1}%
    {#6\space\repdecn@mb{\v@lmin}\space\repdecn@mb{\v@lmax}\space #4\space #5\space ellipse}{\fwf@g}}%
    \global\Use@llipsetrue}
\ctr@ln@m\figdrawarcell
\ctr@ld@f\def\Q@arcellDD#1;#2,#3(#4,#5,#6){{\ifCUR@PS\ifGR@cri%
    \PSc@mment{arcellDD Center=#1 ; XRad=#2, YRad=#3 (Ang1=#4, Ang2=#5, Inclination=#6)}%
    \iffillm@de\Ps@rell#1;#2,#3(#4,#5,#6)%
    \f@gfill%
    \else\Ps@rell#1;#2,#3(#4,#5,#6)\f@gstroke\fi%
    \PSc@mment{End arcellDD}\fi\fi}}
\ctr@ld@f\def\Q@arcellTD#1;#2,#3(#4,#5,#6){{\ifCUR@PS\ifGR@cri\s@uvc@ntr@l\et@tpsarcellTD%
    \PSc@mment{arcellTD Center=#1 ; XRad=#2, YRad=#3 (Ang1=#4, Ang2=#5, Inclination=#6)}%
    \setc@ntr@l{2}\figpttraC -8:=#1/#2,0,0/\figpttraC -7:=#1/0,#3,0/%
    \figvectC -4(0,0,1)\figptsrot -8=-8,-7/#1,#6,-4/\Q@arcellPATD#1,-8,-7(#4,#5)%
    \PSc@mment{End arcellTD}\resetc@ntr@l\et@tpsarcellTD\fi\fi}}
\ctr@ln@m\figdrawarcellPA
\ctr@ld@f\def\Q@arcellPADD#1,#2,#3(#4,#5){{\ifCUR@PS\ifGR@cri\s@uvc@ntr@l\et@tpsarcellPADD%
    \PSc@mment{arcellPADD Center=#1,PtAxis1=#2,PtAxis2=#3 (Ang1=#4, Ang2=#5)}%
    \setc@ntr@l{2}\figvectPDD-1[#1,#2]\vecunit@DD{-1}{-1}\v@lX=\ptT@unit@\result@t%
    \edef\XR@d{\repdecn@mb{\v@lX}}\Figg@tXY{-1}\arct@n\v@lmin(\v@lX,\v@lY)%
    \v@lmin=\rdT@deg\v@lmin\edef\Inclin@{\repdecn@mb{\v@lmin}}%
    \figgetdist\YR@d[#1,#3]\Q@arcellDD#1;\XR@d,\YR@d(#4,#5,\Inclin@)%
    \PSc@mment{End arcellPADD}\resetc@ntr@l\et@tpsarcellPADD\fi\fi}}
\ctr@ld@f\def\Q@arcellPATD#1,#2,#3(#4,#5){{\ifCUR@PS\ifGR@cri\s@uvc@ntr@l\et@tpsarcellPATD%
    \PSc@mment{arcellPATD Center=#1,PtAxis1=#2,PtAxis2=#3 (Ang1=#4, Ang2=#5)}%
    \iffillm@de\Ps@rellPATD#1,#2,#3(#4,#5)%
    \f@gfill%
    \else\Ps@rellPATD#1,#2,#3(#4,#5)\f@gstroke\fi%
    \PSc@mment{End arcellPATD}\resetc@ntr@l\et@tpsarcellPATD\fi\fi}}
\ctr@ld@f\def\Ps@rellPATD#1,#2,#3(#4,#5){\let\c@lprojSP=\relax%
    \setc@ntr@l{2}\figvectPTD-1[#1,#2]\figvectPTD-2[#1,#3]\c@lNbarcs{#4}{#5}%
    \v@leur=#4pt\c@lptellP{#1}{-1}{-2}\Figptpr@j-5:/-3/%
    \f@gnewpath\PSwrit@cmdS{-5}{\c@mmoveto}{\fwf@g}{\X@un}{\Y@un}%
    \edef\C@nt@r{#1}\s@mme=\z@\bcl@rellPATD}
\ctr@ld@f\def\bcl@rellPATD{\relax%
    \ifnum\s@mme<\p@rtent\advance\s@mme\@ne%
    \advance\v@leur\delt@\c@lptellP{\C@nt@r}{-1}{-2}\Figptpr@j-4:/-3/%
    \advance\v@leur\delt@\c@lptellP{\C@nt@r}{-1}{-2}\Figptpr@j-6:/-3/%
    \advance\v@leur\delt@\c@lptellP{\C@nt@r}{-1}{-2}\Figptpr@j-3:/-3/%
    \v@lX=\z@\v@lY=\z@\Figtr@nptDD{-5}{-5}\Figtr@nptDD{2}{-3}%
    \divide\v@lX\@vi\divide\v@lY\@vi%
    \Figtr@nptDD{3}{-4}\Figtr@nptDD{-1.5}{-6}\v@lmin=\v@lX\v@lmax=\v@lY%
    \v@lX=\z@\v@lY=\z@\Figtr@nptDD{2}{-5}\Figtr@nptDD{-5}{-3}%
    \divide\v@lX\@vi\divide\v@lY\@vi\Figtr@nptDD{-1.5}{-4}\Figtr@nptDD{3}{-6}%
    \BdingB@xfalse%
    \Figp@intregDD-4:(\v@lmin,\v@lmax)\PSwrit@cmdS{-4}{}{\fwf@g}{\X@de}{\Y@de}%
    \Figp@intregDD-4:(\v@lX,\v@lY)\PSwrit@cmdS{-4}{}{\fwf@g}{\X@tr}{\Y@tr}%
    \BdingB@xtrue\PSwrit@cmdS{-3}{\c@mcurveto}{\fwf@g}{\X@qu}{\Y@qu}%
    \B@zierBB@x{1}{\Y@un}(\X@un,\X@de,\X@tr,\X@qu)%
    \B@zierBB@x{2}{\X@un}(\Y@un,\Y@de,\Y@tr,\Y@qu)%
    \edef\X@un{\X@qu}\edef\Y@un{\Y@qu}\figptcopyDD-5:/-3/\bcl@rellPATD\fi}
\ctr@ld@f\def\c@lNbarcs#1#2{%
    \delt@=#2pt\advance\delt@-#1pt\maxim@m{\v@lmax}{\delt@}{-\delt@}%
    \v@leur=\v@lmax\divide\v@leur45 \p@rtentiere{\p@rtent}{\v@leur}\advance\p@rtent\@ne%
    \s@mme=\p@rtent\multiply\s@mme\thr@@\divide\delt@\s@mme}
\ctr@ld@f\def\figdrawarcellPP#1,#2,#3[#4,#5]{{\ifCUR@PS\ifGR@cri\s@uvc@ntr@l\et@tpsarcellPP%
    \PSc@mment{arcellPP Center=#1,PtAxis1=#2,PtAxis2=#3 [Point1=#4, Point2=#5]}%
    \setc@ntr@l{2}\figvectP-2[#1,#3]\vecunit@{-2}{-2}\v@lmin=\result@t%
    \invers@{\v@lmax}{\v@lmin}%
    \figvectP-1[#1,#2]\vecunit@{-1}{-1}\v@leur=\result@t%
    \v@leur=\repdecn@mb{\v@lmax}\v@leur\edef\AsB@{\repdecn@mb{\v@leur}}
    \c@lAngle{#1}{#4}{\v@lmin}\edef\@ngdeb{\repdecn@mb{\v@lmin}}%
    \c@lAngle{#1}{#5}{\v@lmax}\ifdim\v@lmin>\v@lmax\advance\v@lmax\DePI@deg\fi%
    \edef\@ngfin{\repdecn@mb{\v@lmax}}\figdrawarcellPA#1,#2,#3(\@ngdeb,\@ngfin)%
    \PSc@mment{End arcellPP}\resetc@ntr@l\et@tpsarcellPP\fi\fi}}
\ctr@ld@f\def\c@lAngle#1#2#3{\figvectP-3[#1,#2]%
    \c@lproscal\delt@[-3,-1]\c@lproscal\v@leur[-3,-2]%
    \v@leur=\AsB@\v@leur\arct@n#3(\delt@,\v@leur)#3=\rdT@deg#3}
\ctr@ln@w{newif}\if@rrowratio\@rrowratiotrue
\ctr@ln@w{newif}\if@rrowhfill
\ctr@ln@w{newif}\if@rrowhout
\ctr@ld@f\def\Psset@rrowhe@d#1=#2|{\keln@mun#1|%
    \def\n@mref{a}\ifx\l@debut\n@mref\update@ttr\D@FTarrowheadangle\Q@s@tarrowheadangle{#2}\else
    \def\n@mref{f}\ifx\l@debut\n@mref\update@ttr\D@FTarrowheadfill\Q@s@tarrowheadfill{#2}\else
    \def\n@mref{l}\ifx\l@debut\n@mref\update@ttr\D@FTarrowheadlength\Q@s@tarrowheadlength{#2}\else
    \def\n@mref{o}\ifx\l@debut\n@mref\update@ttr\D@FTarrowheadout\Q@s@tarrowheadout{#2}\else
    \def\n@mref{r}\ifx\l@debut\n@mref\update@ttr\D@FTarrowheadratio\Q@s@tarrowheadratio{#2}\else
    \W@rnmesAttr{figset arrowhead}{#1}\fi\fi\fi\fi\fi}
\ctr@ln@m\@rrowheadangle
\ctr@ln@m\C@AHANG \ctr@ln@m\S@AHANG \ctr@ln@m\UNSS@N
\ctr@ld@f\def\Q@s@tarrowheadangle#1{\edef\@rrowheadangle{#1}{\c@ssin{\C@}{\S@}{#1}%
    \xdef\C@AHANG{\C@}\xdef\S@AHANG{\S@}\v@lmax=\S@ pt%
    \invers@{\v@leur}{\v@lmax}\maxim@m{\v@leur}{\v@leur}{-\v@leur}%
    \xdef\UNSS@N{\the\v@leur}}}
\ctr@ld@f\def\Q@s@tarrowheadfill#1{\expandafter\set@rrowhfill#1:}
\ctr@ld@f\def\set@rrowhfill#1#2:{\if#1n\@rrowhfillfalse\else\@rrowhfilltrue\fi}
\ctr@ld@f\def\Q@s@tarrowheadout#1{\expandafter\set@rrowhout#1:}
\ctr@ld@f\def\set@rrowhout#1#2:{\if#1n\@rrowhoutfalse\else\@rrowhouttrue\fi}
\ctr@ln@m\@rrowheadlength
\ctr@ld@f\def\Q@s@tarrowheadlength#1{\edef\@rrowheadlength{#1}\@rrowratiofalse}
\ctr@ln@m\@rrowheadratio
\ctr@ld@f\def\Q@s@tarrowheadratio#1{\edef\@rrowheadratio{#1}\@rrowratiotrue}
\ctr@ln@m\D@FTarrowheadlength
\ctr@ld@f\def\figresetarrowhead{%
    \Q@s@tarrowheadangle{\D@FTarrowheadangle}%
    \Q@s@tarrowheadfill{\D@FTarrowheadfill}%
    \Q@s@tarrowheadout{\D@FTarrowheadout}%
    \Q@s@tarrowheadratio{\D@FTarrowheadratio}%
    \d@fm@cdim\D@FTarrowheadlength{\D@FTh@rdahlength}
    \Q@s@tarrowheadlength{\D@FTarrowheadlength}}
\ctr@ld@f\def\D@FTarrowheadratio{0.1}
\ctr@ld@f\def\D@FTarrowheadangle{20}
\ctr@ld@f\def\D@FTarrowheadfill{no}
\ctr@ld@f\def\D@FTarrowheadout{no}
\ctr@ld@f\def\D@FTh@rdahlength{8pt}
\ctr@ln@m\figdrawarrow
\ctr@ld@f\def\Q@arrowDD[#1,#2]{{\ifCUR@PS\ifGR@cri\s@uvc@ntr@l\et@tpsarrow%
    \PSc@mment{arrowDD [Pt1,Pt2]=[#1,#2]}\Q@s@tfillmode{no}%
    \Q@arrowheadDD[#1,#2]\setc@ntr@l{2}\figdrawline[#1,-3]%
    \PSc@mment{End arrowDD}\resetc@ntr@l\et@tpsarrow\fi\fi}}
\ctr@ld@f\def\Q@arrowTD[#1,#2]{{\ifCUR@PS\ifGR@cri\s@uvc@ntr@l\et@tpsarrowTD%
    \PSc@mment{arrowTD [Pt1,Pt2]=[#1,#2]}\resetc@ntr@l{2}%
    \Figptpr@j-5:/#1/\Figptpr@j-6:/#2/\let\c@lprojSP=\relax\Q@arrowDD[-5,-6]%
    \PSc@mment{End arrowTD}\resetc@ntr@l\et@tpsarrowTD\fi\fi}}
\ctr@ln@m\figdrawarrowhead
\ctr@ld@f\def\Q@arrowheadDD[#1,#2]{{\ifCUR@PS\ifGR@cri\s@uvc@ntr@l\et@tpsarrowheadDD%
    \if@rrowhfill\def\@hangle{-\@rrowheadangle}\else\def\@hangle{\@rrowheadangle}\fi%
    \if@rrowratio%
    \if@rrowhout\def\@hratio{-\@rrowheadratio}\else\def\@hratio{\@rrowheadratio}\fi%
    \PSc@mment{arrowheadDD Ratio=\@hratio, Angle=\@hangle, [Pt1,Pt2]=[#1,#2]}%
    \Ps@rrowhead\@hratio,\@hangle[#1,#2]%
    \else%
    \if@rrowhout\def\@hlength{-\@rrowheadlength}\else\def\@hlength{\@rrowheadlength}\fi%
    \PSc@mment{arrowheadDD Length=\@hlength, Angle=\@hangle, [Pt1,Pt2]=[#1,#2]}%
    \Ps@rrowheadfd\@hlength,\@hangle[#1,#2]%
    \fi%
    \PSc@mment{End arrowheadDD}\resetc@ntr@l\et@tpsarrowheadDD\fi\fi}}
\ctr@ld@f\def\Q@arrowheadTD[#1,#2]{{\ifCUR@PS\ifGR@cri\s@uvc@ntr@l\et@tpsarrowheadTD%
    \PSc@mment{arrowheadTD [Pt1,Pt2]=[#1,#2]}\resetc@ntr@l{2}%
    \Figptpr@j-5:/#1/\Figptpr@j-6:/#2/\let\c@lprojSP=\relax\Q@arrowheadDD[-5,-6]%
    \PSc@mment{End arrowheadTD}\resetc@ntr@l\et@tpsarrowheadTD\fi\fi}}
\ctr@ld@f\def\Ps@rrowhead#1,#2[#3,#4]{\v@leur=#1\p@\maxim@m{\v@leur}{\v@leur}{-\v@leur}%
    \ifdim\v@leur>\Cepsil@n{
    \PSc@mment{@rrowhead Ratio=#1, Angle=#2, [Pt1,Pt2]=[#3,#4]}\v@leur=\UNSS@N%
    \v@leur=\CUR@width\v@leur\v@leur=\ptpsT@pt\v@leur\delt@=.5\v@leur
    \setc@ntr@l{2}\figvectPDD-3[#4,#3]%
    \Figg@tXY{-3}\v@lX=#1\v@lX\v@lY=#1\v@lY\Figv@ctCreg-3(\v@lX,\v@lY)%
    \vecunit@{-4}{-3}\mili@u=\result@t%
    \ifdim#2pt>\z@\v@lXa=-\C@AHANG\delt@%
     \edef\c@ef{\repdecn@mb{\v@lXa}}\figpttraDD-3:=-3/\c@ef,-4/\fi%
    \edef\c@ef{\repdecn@mb{\delt@}}%
    \v@lXa=\mili@u\v@lXa=\C@AHANG\v@lXa%
    \v@lYa=\ptpsT@pt\p@\v@lYa=\CUR@width\v@lYa\v@lYa=\sDcc@ngle\v@lYa%
    \advance\v@lXa-\v@lYa\gdef\sDcc@ngle{0}%
    \ifdim\v@lXa>\v@leur\edef\c@efendpt{\repdecn@mb{\v@leur}}%
    \else\edef\c@efendpt{\repdecn@mb{\v@lXa}}\fi%
    \Figg@tXY{-3}\v@lmin=\v@lX\v@lmax=\v@lY%
    \v@lXa=\C@AHANG\v@lmin\v@lYa=\S@AHANG\v@lmax\advance\v@lXa\v@lYa%
    \v@lYa=-\S@AHANG\v@lmin\v@lX=\C@AHANG\v@lmax\advance\v@lYa\v@lX%
    \setc@ntr@l{1}\Figg@tXY{#4}\advance\v@lX\v@lXa\advance\v@lY\v@lYa%
    \setc@ntr@l{2}\Figp@intregDD-2:(\v@lX,\v@lY)%
    \v@lXa=\C@AHANG\v@lmin\v@lYa=-\S@AHANG\v@lmax\advance\v@lXa\v@lYa%
    \v@lYa=\S@AHANG\v@lmin\v@lX=\C@AHANG\v@lmax\advance\v@lYa\v@lX%
    \setc@ntr@l{1}\Figg@tXY{#4}\advance\v@lX\v@lXa\advance\v@lY\v@lYa%
    \setc@ntr@l{2}\Figp@intregDD-1:(\v@lX,\v@lY)%
    \ifdim#2pt<\z@\fillm@detrue\figdrawline[-2,#4,-1]
    \else\figptstraDD-3=#4,-2,-1/\c@ef,-4/\s@uvdash{\typ@dash}\Q@s@tdash{\D@FTdash}%
    \figdrawline[-2,-3,-1]\Q@s@tdash{\typ@dash}\fi
    \ifdim#1pt>\z@\figpttraDD-3:=#4/\c@efendpt,-4/\else\figptcopyDD-3:/#4/\fi%
    \PSc@mment{End @rrowhead}}\fi}
\ctr@ld@f\def\sDcc@ngle{0}
\ctr@ld@f\def\Ps@rrowheadfd#1,#2[#3,#4]{{%
    \PSc@mment{@rrowheadfd Length=#1, Angle=#2, [Pt1,Pt2]=[#3,#4]}%
    \setc@ntr@l{2}\figvectPDD-1[#3,#4]\n@rmeucDD{\v@leur}{-1}\v@leur=\ptT@unit@\v@leur%
    \invers@{\v@leur}{\v@leur}\v@leur=#1\v@leur\edef\R@tio{\repdecn@mb{\v@leur}}%
    \Ps@rrowhead\R@tio,#2[#3,#4]\PSc@mment{End @rrowheadfd}}}
\ctr@ln@m\figdrawarrowBezier
\ctr@ld@f\def\Q@arrowBezierDD[#1,#2,#3,#4]{{\ifCUR@PS\ifGR@cri\s@uvc@ntr@l\et@tpsarrowBezierDD%
    \PSc@mment{arrowBezierDD Control points=#1,#2,#3,#4}\setc@ntr@l{2}%
    \if@rrowratio\c@larclengthDD\v@leur,10[#1,#2,#3,#4]\else\v@leur=\z@\fi%
    \Ps@rrowB@zDD\v@leur[#1,#2,#3,#4]%
    \PSc@mment{End arrowBezierDD}\resetc@ntr@l\et@tpsarrowBezierDD\fi\fi}}
\ctr@ld@f\def\Q@arrowBezierTD[#1,#2,#3,#4]{{\ifCUR@PS\ifGR@cri\s@uvc@ntr@l\et@tpsarrowBezierTD%
    \PSc@mment{arrowBezierTD Control points=#1,#2,#3,#4}\resetc@ntr@l{2}%
    \Figptpr@j-7:/#1/\Figptpr@j-8:/#2/\Figptpr@j-9:/#3/\Figptpr@j-10:/#4/%
    \let\c@lprojSP=\relax\ifnum\CUR@proj<\tw@\Q@arrowBezierDD[-7,-8,-9,-10]%
    \else\f@gnewpath\PSwrit@cmd{-7}{\c@mmoveto}{\fwf@g}%
    \if@rrowratio\c@larclengthDD\mili@u,10[-7,-8,-9,-10]\else\mili@u=\z@\fi%
    \p@rtent=\NBz@rcs\advance\p@rtent\m@ne\subB@zierTD\p@rtent[#1,#2,#3,#4]%
    \f@gstroke%
    \advance\v@lmin\p@rtent\delt@
    \v@leur=\v@lmin\advance\v@leur0.33333 \delt@\edef\unti@rs{\repdecn@mb{\v@leur}}%
    \v@leur=\v@lmin\advance\v@leur0.66666 \delt@\edef\deti@rs{\repdecn@mb{\v@leur}}%
    \figptcopyDD-8:/-10/\c@lsubBzarc\unti@rs,\deti@rs[#1,#2,#3,#4]%
    \figptcopyDD-8:/-4/\figptcopyDD-9:/-3/\Ps@rrowB@zDD\mili@u[-7,-8,-9,-10]\fi%
    \PSc@mment{End arrowBezierTD}\resetc@ntr@l\et@tpsarrowBezierTD\fi\fi}}
\ctr@ld@f\def\c@larclengthDD#1,#2[#3,#4,#5,#6]{{\p@rtent=#2\figptcopyDD-5:/#3/%
    \delt@=\p@\divide\delt@\p@rtent\c@rre=\z@\v@leur=\z@\s@mme=\z@%
    \loop\ifnum\s@mme<\p@rtent\advance\s@mme\@ne\advance\v@leur\delt@%
    \edef\T@{\repdecn@mb{\v@leur}}\figptBezierDD-6::\T@[#3,#4,#5,#6]%
    \figvectPDD-1[-5,-6]\n@rmeucDD{\mili@u}{-1}\advance\c@rre\mili@u%
    \figptcopyDD-5:/-6/\repeat\global\result@t=\ptT@unit@\c@rre}#1=\result@t}
\ctr@ld@f\def\Ps@rrowB@zDD#1[#2,#3,#4,#5]{{\Q@s@tfillmode{no}%
    \if@rrowratio\delt@=\@rrowheadratio#1\else\delt@=\@rrowheadlength pt\fi%
    \v@leur=\C@AHANG\delt@\edef\R@dius{\repdecn@mb{\v@leur}}%
    \FigptintercircB@zDD-5::0,\R@dius[#5,#4,#3,#2]%
    \Q@s@tarrowheadlength{\repdecn@mb{\delt@}}\Q@arrowheadDD[-5,#5]%
    \let\n@rmeuc=\n@rmeucDD\figgetdist\R@dius[#5,-3]%
    \FigptintercircB@zDD-6::0,\R@dius[#5,#4,#3,#2]%
    \figptBezierDD-5::0.33333[#5,#4,#3,#2]\figptBezierDD-3::0.66666[#5,#4,#3,#2]%
    \figptscontrolDD-5[-6,-5,-3,#2]\Q@BezierDD1[-6,-5,-4,#2]}}
\ctr@ln@m\figdrawarrowcirc
\ctr@ld@f\def\Q@arrowcircDD#1;#2(#3,#4){{\ifCUR@PS\ifGR@cri\s@uvc@ntr@l\et@tpsarrowcircDD%
    \PSc@mment{arrowcircDD Center=#1 ; Radius=#2 (Ang1=#3,Ang2=#4)}%
    \Q@s@tfillmode{no}\Pscirc@rrowhead#1;#2(#3,#4)%
    \setc@ntr@l{2}\figvectPDD -4[#1,-3]\vecunit@{-4}{-4}%
    \Figg@tXY{-4}\arct@n\v@lmin(\v@lX,\v@lY)%
    \v@lmin=\rdT@deg\v@lmin\v@leur=#4pt\advance\v@leur-\v@lmin%
    \maxim@m{\v@leur}{\v@leur}{-\v@leur}%
    \ifdim\v@leur>\DemiPI@deg\relax\ifdim\v@lmin<#4pt\advance\v@lmin\DePI@deg%
    \else\advance\v@lmin-\DePI@deg\fi\fi\edef\ar@ngle{\repdecn@mb{\v@lmin}}%
    \ifdim#3pt<#4pt\figdrawarccirc#1;#2(#3,\ar@ngle)\else\figdrawarccirc#1;#2(\ar@ngle,#3)\fi%
    \PSc@mment{End arrowcircDD}\resetc@ntr@l\et@tpsarrowcircDD\fi\fi}}
\ctr@ld@f\def\Q@arrowcircTD#1,#2,#3;#4(#5,#6){{\ifCUR@PS\ifGR@cri\s@uvc@ntr@l\et@tpsarrowcircTD%
    \PSc@mment{arrowcircTD Center=#1,P1=#2,P2=#3 ; Radius=#4 (Ang1=#5, Ang2=#6)}%
    \resetc@ntr@l{2}\c@lExtAxes#1,#2,#3(#4)\let\c@lprojSP=\relax%
    \figvectPTD-11[#1,-4]\figvectPTD-12[#1,-5]\c@lNbarcs{#5}{#6}%
    \if@rrowratio\v@lmax=\degT@rd\v@lmax\edef\D@lpha{\repdecn@mb{\v@lmax}}\fi%
    \advance\p@rtent\m@ne\mili@u=\z@%
    \v@leur=#5pt\c@lptellP{#1}{-11}{-12}\Figptpr@j-9:/-3/%
    \f@gnewpath\PSwrit@cmdS{-9}{\c@mmoveto}{\fwf@g}{\X@un}{\Y@un}%
    \edef\C@nt@r{#1}\s@mme=\z@\bcl@rcircTD\f@gstroke%
    \advance\v@leur\delt@\c@lptellP{#1}{-11}{-12}\Figptpr@j-5:/-3/%
    \advance\v@leur\delt@\c@lptellP{#1}{-11}{-12}\Figptpr@j-6:/-3/%
    \advance\v@leur\delt@\c@lptellP{#1}{-11}{-12}\Figptpr@j-10:/-3/%
    \figptscontrolDD-8[-9,-5,-6,-10]%
    \if@rrowratio\c@lcurvradDD0.5[-9,-8,-7,-10]\advance\mili@u\result@t%
    \maxim@m{\mili@u}{\mili@u}{-\mili@u}\mili@u=\ptT@unit@\mili@u%
    \mili@u=\D@lpha\mili@u\advance\p@rtent\@ne\divide\mili@u\p@rtent\fi%
    \Ps@rrowB@zDD\mili@u[-9,-8,-7,-10]%
    \PSc@mment{End arrowcircTD}\resetc@ntr@l\et@tpsarrowcircTD\fi\fi}}
\ctr@ld@f\def\bcl@rcircTD{\relax%
    \ifnum\s@mme<\p@rtent\advance\s@mme\@ne%
    \advance\v@leur\delt@\c@lptellP{\C@nt@r}{-11}{-12}\Figptpr@j-5:/-3/%
    \advance\v@leur\delt@\c@lptellP{\C@nt@r}{-11}{-12}\Figptpr@j-6:/-3/%
    \advance\v@leur\delt@\c@lptellP{\C@nt@r}{-11}{-12}\Figptpr@j-10:/-3/%
    \figptscontrolDD-8[-9,-5,-6,-10]\BdingB@xfalse%
    \PSwrit@cmdS{-8}{}{\fwf@g}{\X@de}{\Y@de}\PSwrit@cmdS{-7}{}{\fwf@g}{\X@tr}{\Y@tr}%
    \BdingB@xtrue\PSwrit@cmdS{-10}{\c@mcurveto}{\fwf@g}{\X@qu}{\Y@qu}%
    \if@rrowratio\c@lcurvradDD0.5[-9,-8,-7,-10]\advance\mili@u\result@t\fi%
    \B@zierBB@x{1}{\Y@un}(\X@un,\X@de,\X@tr,\X@qu)%
    \B@zierBB@x{2}{\X@un}(\Y@un,\Y@de,\Y@tr,\Y@qu)%
    \edef\X@un{\X@qu}\edef\Y@un{\Y@qu}\figptcopyDD-9:/-10/\bcl@rcircTD\fi}
\ctr@ld@f\def\Pscirc@rrowhead#1;#2(#3,#4){{%
    \PSc@mment{circ@rrowhead Center=#1 ; Radius=#2 (Ang1=#3,Ang2=#4)}%
    \v@leur=#2\unit@\edef\s@glen{\repdecn@mb{\v@leur}}\v@lY=\z@\v@lX=\v@leur%
    \resetc@ntr@l{2}\Figv@ctCreg-3(\v@lX,\v@lY)\figpttraDD-5:=#1/1,-3/%
    \figptrotDD-5:=-5/#1,#4/%
    \figvectPDD-3[#1,-5]\Figg@tXY{-3}\v@leur=\v@lX%
    \ifdim#3pt<#4pt\v@lX=\v@lY\v@lY=-\v@leur\else\v@lX=-\v@lY\v@lY=\v@leur\fi%
    \Figv@ctCreg-3(\v@lX,\v@lY)\vecunit@{-3}{-3}%
    \if@rrowratio\v@leur=#4pt\advance\v@leur-#3pt\maxim@m{\mili@u}{-\v@leur}{\v@leur}%
    \mili@u=\degT@rd\mili@u\v@leur=\s@glen\mili@u\edef\s@glen{\repdecn@mb{\v@leur}}%
    \mili@u=#2\mili@u\mili@u=\@rrowheadratio\mili@u\else\mili@u=\@rrowheadlength pt\fi%
    \figpttraDD-6:=-5/\s@glen,-3/\v@leur=#2pt\v@leur=2\v@leur%
    \invers@{\v@leur}{\v@leur}\c@rre=\repdecn@mb{\v@leur}\mili@u
    \mili@u=\c@rre\mili@u=\repdecn@mb{\c@rre}\mili@u%
    \v@leur=\p@\advance\v@leur-\mili@u
    \invers@{\mili@u}{2\v@leur}\delt@=\c@rre\delt@=\repdecn@mb{\mili@u}\delt@%
    \xdef\sDcc@ngle{\repdecn@mb{\delt@}}
    \sqrt@{\mili@u}{\v@leur}\arct@n\v@leur(\mili@u,\c@rre)%
    \v@leur=\rdT@deg\v@leur
    \ifdim#3pt<#4pt\v@leur=-\v@leur\fi%
    \if@rrowhout\v@leur=-\v@leur\fi\edef\cor@ngle{\repdecn@mb{\v@leur}}%
    \figptrotDD-6:=-6/-5,\cor@ngle/\Q@arrowheadDD[-6,-5]%
    \PSc@mment{End circ@rrowhead}}}
\ctr@ln@m\figdrawarrowcircP
\ctr@ld@f\def\Q@arrowcircPDD#1;#2[#3,#4]{{\ifCUR@PS\ifGR@cri%
    \PSc@mment{arrowcircPDD Center=#1; Radius=#2, [P1=#3,P2=#4]}%
    \s@uvc@ntr@l\et@tpsarrowcircPDD\Ps@ngleparam#1;#2[#3,#4]%
    \ifdim\v@leur>\z@\ifdim\v@lmin>\v@lmax\advance\v@lmax\DePI@deg\fi%
    \else\ifdim\v@lmin<\v@lmax\advance\v@lmin\DePI@deg\fi\fi%
    \edef\@ngdeb{\repdecn@mb{\v@lmin}}\edef\@ngfin{\repdecn@mb{\v@lmax}}%
    \figdrawarrowcirc#1;\r@dius(\@ngdeb,\@ngfin)%
    \PSc@mment{End arrowcircPDD}\resetc@ntr@l\et@tpsarrowcircPDD\fi\fi}}
\ctr@ld@f\def\Q@arrowcircPTD#1;#2[#3,#4,#5]{{\ifCUR@PS\ifGR@cri\s@uvc@ntr@l\et@tpsarrowcircPTD%
    \PSc@mment{arrowcircPTD Center=#1; Radius=#2, [P1=#3,P2=#4,P3=#5]}%
    \figgetangleTD\@ngfin[#1,#3,#4,#5]\v@leur=#2pt%
    \maxim@m{\mili@u}{-\v@leur}{\v@leur}\edef\r@dius{\repdecn@mb{\mili@u}}%
    \ifdim\v@leur<\z@\v@lmax=\@ngfin pt\advance\v@lmax-\DePI@deg%
    \edef\@ngfin{\repdecn@mb{\v@lmax}}\fi\Q@arrowcircTD#1,#3,#5;\r@dius(0,\@ngfin)%
    \PSc@mment{End arrowcircPTD}\resetc@ntr@l\et@tpsarrowcircPTD\fi\fi}}
\ctr@ld@f\def\figdrawaxes#1(#2){{\ifCUR@PS\ifGR@cri\s@uvc@ntr@l\et@tpsaxes%
    \PSc@mment{axes Origin=#1 Range=(#2)}\an@lys@xes#2,:\resetc@ntr@l{2}%
    \ifx\t@xt@\empty\ifTr@isDim\Q@@xes#1(0,#2,0,#2,0,#2)\else\Q@@xes#1(0,#2,0,#2)\fi%
    \else\Q@@xes#1(#2)\fi\PSc@mment{End axes}\resetc@ntr@l\et@tpsaxes\fi\fi}}
\ctr@ld@f\def\an@lys@xes#1,#2:{\def\t@xt@{#2}}
\ctr@ln@m\Q@@xes
\ctr@ld@f\def\Q@@xesDD#1(#2,#3,#4,#5){%
    \figpttraC-5:=#1/#2,0/\figpttraC-6:=#1/#3,0/\Q@arrowDD[-5,-6]%
    \figpttraC-5:=#1/0,#4/\figpttraC-6:=#1/0,#5/\Q@arrowDD[-5,-6]}
\ctr@ld@f\def\Q@@xesTD#1(#2,#3,#4,#5,#6,#7){%
    \figpttraC-7:=#1/#2,0,0/\figpttraC-8:=#1/#3,0,0/\Q@arrowTD[-7,-8]%
    \figpttraC-7:=#1/0,#4,0/\figpttraC-8:=#1/0,#5,0/\Q@arrowTD[-7,-8]%
    \figpttraC-7:=#1/0,0,#6/\figpttraC-8:=#1/0,0,#7/\Q@arrowTD[-7,-8]}
\ctr@ln@m\newGr@FN
\ctr@ld@f\def\newGr@FNPDF#1{\s@mme=\Gr@FNb\advance\s@mme\@ne\xdef\Gr@FNb{\number\s@mme}}
\ctr@ld@f\def\newGr@FNDVI#1{\newGr@FNPDF{}\xdef#1{\jobname GI\Gr@FNb.anx}}
\ctr@ld@f\def\figdrawbegin#1{\newGr@FN\DefGIfilen@me\gdef\@utoFN{0}%
    \def\t@xt@{#1}\relax\ifx\t@xt@\empty\GRupdatem@detrue%
    \gdef\@utoFN{1}\Psb@ginfig\DefGIfilen@me\else\expandafter\Psb@ginfigNu@#1 :\fi}
\ctr@ld@f\def\Psb@ginfigNu@#1 #2:{\def\t@xt@{#1}\relax\ifx\t@xt@\empty\def\t@xt@{#2}%
    \ifx\t@xt@\empty\GRupdatem@detrue\gdef\@utoFN{1}\Psb@ginfig\DefGIfilen@me%
    \else\Psb@ginfigNu@#2:\fi\else\Psb@ginfig{#1}\fi}
\ctr@ln@m\PSfilen@me \ctr@ln@m\auxfilen@me
\ctr@ld@f\def\Psb@ginfig#1{\ifCUR@PS\else%
    \edef\PSfilen@me{#1}\edef\auxfilen@me{\jobname.anx}%
    \ifGRupdatem@de\GR@critrue\else\openin\frf@g=\PSfilen@me\relax%
    \ifeof\frf@g\GR@critrue\else\GR@crifalse\fi\closein\frf@g\fi%
    \CUR@PStrue\c@ldefproj\expandafter\setupd@te\D@FTupdate:%
    \ifGR@cri\initb@undb@x%
    \immediate\openout\fwf@g=\auxfilen@me\initpss@ttings\fi%
    \fi}
\ctr@ld@f\def\Gr@FNb{0}
\ctr@ld@f\def\figforTeXFileno{\Gr@FNb}
\ctr@ld@f\def\figforTeXFigno{0 }
\ctr@ld@f\def\figforTeXnextFigno{1 }
\ctr@ld@f\edef\DefGIfilen@me{\jobname GI.anx}
\ctr@ld@f\def\initpss@ttings{\figreset{altitude,arrowhead,curve,general,flowchart,mesh,trimesh}%
    \Use@llipsefalse}
\ctr@ld@f\def\B@zierBB@x#1#2(#3,#4,#5,#6){{\c@rre=\t@n\epsil@n
    \v@lmax=#4\advance\v@lmax-#5\v@lmax=\thr@@\v@lmax\advance\v@lmax#6\advance\v@lmax-#3%
    \mili@u=#4\mili@u=-\tw@\mili@u\advance\mili@u#3\advance\mili@u#5%
    \v@lmin=#4\advance\v@lmin-#3\maxim@m{\v@leur}{-\v@lmax}{\v@lmax}%
    \maxim@m{\delt@}{-\mili@u}{\mili@u}\maxim@m{\v@leur}{\v@leur}{\delt@}%
    \maxim@m{\delt@}{-\v@lmin}{\v@lmin}\maxim@m{\v@leur}{\v@leur}{\delt@}%
    \ifdim\v@leur>\c@rre\invers@{\v@leur}{\v@leur}\edef\Uns@rM@x{\repdecn@mb{\v@leur}}%
    \v@lmax=\Uns@rM@x\v@lmax\mili@u=\Uns@rM@x\mili@u\v@lmin=\Uns@rM@x\v@lmin%
    \maxim@m{\v@leur}{-\v@lmax}{\v@lmax}\ifdim\v@leur<\c@rre%
    \maxim@m{\v@leur}{-\mili@u}{\mili@u}\ifdim\v@leur<\c@rre\else%
    \invers@{\mili@u}{\mili@u}\v@leur=-0.5\v@lmin%
    \v@leur=\repdecn@mb{\mili@u}\v@leur\m@jBBB@x{\v@leur}{#1}{#2}(#3,#4,#5,#6)\fi%
    \else\delt@=\repdecn@mb{\mili@u}\mili@u\v@leur=\repdecn@mb{\v@lmax}\v@lmin%
    \advance\delt@-\v@leur\ifdim\delt@<\z@\else\invers@{\v@lmax}{\v@lmax}%
    \edef\Uns@rAp{\repdecn@mb{\v@lmax}}\sqrt@{\delt@}{\delt@}%
    \v@leur=-\mili@u\advance\v@leur\delt@\v@leur=\Uns@rAp\v@leur%
    \m@jBBB@x{\v@leur}{#1}{#2}(#3,#4,#5,#6)%
    \v@leur=-\mili@u\advance\v@leur-\delt@\v@leur=\Uns@rAp\v@leur%
    \m@jBBB@x{\v@leur}{#1}{#2}(#3,#4,#5,#6)\fi\fi\fi}}
\ctr@ld@f\def\m@jBBB@x#1#2#3(#4,#5,#6,#7){{\relax\ifdim#1>\z@\ifdim#1<\p@%
    \edef\T@{\repdecn@mb{#1}}\v@lX=\p@\advance\v@lX-#1\edef\UNmT@{\repdecn@mb{\v@lX}}%
    \v@lX=#4\v@lY=#5\v@lZ=#6\v@lXa=#7\v@lX=\UNmT@\v@lX\advance\v@lX\T@\v@lY%
    \v@lY=\UNmT@\v@lY\advance\v@lY\T@\v@lZ\v@lZ=\UNmT@\v@lZ\advance\v@lZ\T@\v@lXa%
    \v@lX=\UNmT@\v@lX\advance\v@lX\T@\v@lY\v@lY=\UNmT@\v@lY\advance\v@lY\T@\v@lZ%
    \v@lX=\UNmT@\v@lX\advance\v@lX\T@\v@lY%
    \ifcase#2\or\v@lY=#3\or\v@lY=\v@lX\v@lX=#3\fi\b@undb@x{\v@lX}{\v@lY}\fi\fi}}
\ctr@ld@f\def\PsB@zier#1[#2]{{\f@gnewpath%
    \s@mme=\z@\def\list@num{#2,0}\extrairelepremi@r\p@int\de\list@num%
    \PSwrit@cmdS{\p@int}{\c@mmoveto}{\fwf@g}{\X@un}{\Y@un}\p@rtent=#1\bclB@zier}}
\ctr@ld@f\def\bclB@zier{\relax%
    \ifnum\s@mme<\p@rtent\advance\s@mme\@ne\BdingB@xfalse%
    \extrairelepremi@r\p@int\de\list@num\PSwrit@cmdS{\p@int}{}{\fwf@g}{\X@de}{\Y@de}%
    \extrairelepremi@r\p@int\de\list@num\PSwrit@cmdS{\p@int}{}{\fwf@g}{\X@tr}{\Y@tr}%
    \BdingB@xtrue%
    \extrairelepremi@r\p@int\de\list@num\PSwrit@cmdS{\p@int}{\c@mcurveto}{\fwf@g}{\X@qu}{\Y@qu}%
    \B@zierBB@x{1}{\Y@un}(\X@un,\X@de,\X@tr,\X@qu)%
    \B@zierBB@x{2}{\X@un}(\Y@un,\Y@de,\Y@tr,\Y@qu)%
    \edef\X@un{\X@qu}\edef\Y@un{\Y@qu}\bclB@zier\fi}
\ctr@ln@m\figdrawBezier
\ctr@ld@f\def\Q@BezierDD#1[#2]{\ifCUR@PS\ifGR@cri%
    \PSc@mment{BezierDD N arcs=#1, Control points=#2}%
    \iffillm@de\PsB@zier#1[#2]%
    \f@gfill%
    \else\PsB@zier#1[#2]\f@gstroke\fi%
    \PSc@mment{End BezierDD}\fi\fi}
\ctr@ln@m\et@tpsBezierTD
\ctr@ld@f\def\Q@BezierTD#1[#2]{\ifCUR@PS\ifGR@cri\s@uvc@ntr@l\et@tpsBezierTD%
    \PSc@mment{BezierTD N arcs=#1, Control points=#2}%
    \iffillm@de\PsB@zierTD#1[#2]%
    \f@gfill%
    \else\PsB@zierTD#1[#2]\f@gstroke\fi%
    \PSc@mment{End BezierTD}\resetc@ntr@l\et@tpsBezierTD\fi\fi}
\ctr@ld@f\def\PsB@zierTD#1[#2]{\ifnum\CUR@proj<\tw@\PsB@zier#1[#2]\else\PsB@zier@TD#1[#2]\fi}
\ctr@ld@f\def\PsB@zier@TD#1[#2]{{\f@gnewpath%
    \s@mme=\z@\def\list@num{#2,0}\extrairelepremi@r\p@int\de\list@num%
    \let\c@lprojSP=\relax\setc@ntr@l{2}\Figptpr@j-7:/\p@int/%
    \PSwrit@cmd{-7}{\c@mmoveto}{\fwf@g}%
    \loop\ifnum\s@mme<#1\advance\s@mme\@ne\extrairelepremi@r\p@intun\de\list@num%
    \extrairelepremi@r\p@intde\de\list@num\extrairelepremi@r\p@inttr\de\list@num%
    \subB@zierTD\NBz@rcs[\p@int,\p@intun,\p@intde,\p@inttr]\edef\p@int{\p@inttr}\repeat}}
\ctr@ld@f\def\subB@zierTD#1[#2,#3,#4,#5]{\delt@=\p@\divide\delt@\NBz@rcs\v@lmin=\z@%
    {\Figg@tXY{-7}\edef\X@un{\the\v@lX}\edef\Y@un{\the\v@lY}%
    \s@mme=\z@\loop\ifnum\s@mme<#1\advance\s@mme\@ne%
    \v@leur=\v@lmin\advance\v@leur0.33333 \delt@\edef\unti@rs{\repdecn@mb{\v@leur}}%
    \v@leur=\v@lmin\advance\v@leur0.66666 \delt@\edef\deti@rs{\repdecn@mb{\v@leur}}%
    \advance\v@lmin\delt@\edef\trti@rs{\repdecn@mb{\v@lmin}}%
    \figptBezierTD-8::\trti@rs[#2,#3,#4,#5]\Figptpr@j-8:/-8/%
    \c@lsubBzarc\unti@rs,\deti@rs[#2,#3,#4,#5]\BdingB@xfalse%
    \PSwrit@cmdS{-4}{}{\fwf@g}{\X@de}{\Y@de}\PSwrit@cmdS{-3}{}{\fwf@g}{\X@tr}{\Y@tr}%
    \BdingB@xtrue\PSwrit@cmdS{-8}{\c@mcurveto}{\fwf@g}{\X@qu}{\Y@qu}%
    \B@zierBB@x{1}{\Y@un}(\X@un,\X@de,\X@tr,\X@qu)%
    \B@zierBB@x{2}{\X@un}(\Y@un,\Y@de,\Y@tr,\Y@qu)%
    \edef\X@un{\X@qu}\edef\Y@un{\Y@qu}\figptcopyDD-7:/-8/\repeat}}
\ctr@ld@f\def\NBz@rcs{2}
\ctr@ld@f\def\c@lsubBzarc#1,#2[#3,#4,#5,#6]{\figptBezierTD-5::#1[#3,#4,#5,#6]%
    \figptBezierTD-6::#2[#3,#4,#5,#6]\Figptpr@j-4:/-5/\Figptpr@j-5:/-6/%
    \figptscontrolDD-4[-7,-4,-5,-8]}
\ctr@ln@m\figdrawcirc
\ctr@ld@f\def\Q@circDD#1(#2){\ifCUR@PS\ifGR@cri\PSc@mment{circDD Center=#1 (Radius=#2)}%
    \Q@arccircDD#1;#2(0,360)\PSc@mment{End circDD}\fi\fi}
\ctr@ld@f\def\Q@circTD#1,#2,#3(#4){\ifCUR@PS\ifGR@cri%
    \PSc@mment{circTD Center=#1,P1=#2,P2=#3 (Radius=#4)}%
    \Q@arccircTD#1,#2,#3;#4(0,360)\PSc@mment{End circTD}\fi\fi}
\ctr@ln@m\p@urcent
{\catcode`\%=12\gdef\p@urcent{
\ctr@ld@f\def\PSc@mment#1{\ifGRdebugm@de\immediate\write\fwf@g{\p@urcent\space#1}\fi}
\ctr@ln@m\acc@louv \ctr@ln@m\acc@lfer
{\catcode`\[=1\catcode`\{=12\gdef\acc@louv[{}}
{\catcode`\]=2\catcode`\}=12\gdef\acc@lfer{}]]
\ctr@ld@f\def\PSdict@{\ifUse@llipse%
    \immediate\write\fwf@g{/ellipsedict 9 dict def ellipsedict /mtrx matrix put}%
    \immediate\write\fwf@g{/ellipse \acc@louv ellipsedict begin}%
    \immediate\write\fwf@g{ /endangle exch def /startangle exch def}%
    \immediate\write\fwf@g{ /yrad exch def /xrad exch def}%
    \immediate\write\fwf@g{ /rotangle exch def /y exch def /x exch def}%
    \immediate\write\fwf@g{ /savematrix mtrx currentmatrix def}%
    \immediate\write\fwf@g{ x y translate rotangle rotate xrad yrad scale}%
    \immediate\write\fwf@g{ 0 0 1 startangle endangle arc}%
    \immediate\write\fwf@g{ savematrix setmatrix end\acc@lfer def}%
    \fi\PShe@der{EndProlog}}
\ctr@ld@f\def\Pssetc@rve#1=#2|{\keln@mun#1|%
    \def\n@mref{r}\ifx\l@debut\n@mref\update@ttr\D@FTroundness\Q@s@troundness{#2}\else
    \W@rnmesAttr{figset curve}{#1}\fi}
\ctr@ln@m\curv@roundness
\ctr@ld@f\def\Q@s@troundness#1{\edef\curv@roundness{#1}}
\ctr@ld@f\def\D@FTroundness{0.2} 
\ctr@ln@m\figdrawcurve
\ctr@ld@f\def\Q@curveDD[#1]{{\ifCUR@PS\ifGR@cri\PSc@mment{curveDD Points=#1}%
    \s@uvc@ntr@l\et@tpscurveDD%
    \iffillm@de\Psc@rveDD\curv@roundness[#1]%
    \f@gfill%
    \else\Psc@rveDD\curv@roundness[#1]\f@gstroke\fi%
    \PSc@mment{End curveDD}\resetc@ntr@l\et@tpscurveDD\fi\fi}}
\ctr@ld@f\def\Q@curveTD[#1]{{\ifCUR@PS\ifGR@cri%
    \PSc@mment{curveTD Points=#1}\s@uvc@ntr@l\et@tpscurveTD\let\c@lprojSP=\relax%
    \iffillm@de\Psc@rveTD\curv@roundness[#1]%
    \f@gfill%
    \else\Psc@rveTD\curv@roundness[#1]\f@gstroke\fi%
    \PSc@mment{End curveTD}\resetc@ntr@l\et@tpscurveTD\fi\fi}}
\ctr@ld@f\def\Psc@rveDD#1[#2]{%
    \def\list@num{#2}\extrairelepremi@r\Ak@\de\list@num%
    \extrairelepremi@r\Ai@\de\list@num\extrairelepremi@r\Aj@\de\list@num%
    \f@gnewpath\PSwrit@cmdS{\Ai@}{\c@mmoveto}{\fwf@g}{\X@un}{\Y@un}%
    \setc@ntr@l{2}\figvectPDD -1[\Ak@,\Aj@]%
    \@ecfor\Ak@:=\list@num\do{\figpttraDD-2:=\Ai@/#1,-1/\BdingB@xfalse%
       \PSwrit@cmdS{-2}{}{\fwf@g}{\X@de}{\Y@de}%
       \figvectPDD -1[\Ai@,\Ak@]\figpttraDD-2:=\Aj@/-#1,-1/%
       \PSwrit@cmdS{-2}{}{\fwf@g}{\X@tr}{\Y@tr}\BdingB@xtrue%
       \PSwrit@cmdS{\Aj@}{\c@mcurveto}{\fwf@g}{\X@qu}{\Y@qu}%
       \B@zierBB@x{1}{\Y@un}(\X@un,\X@de,\X@tr,\X@qu)%
       \B@zierBB@x{2}{\X@un}(\Y@un,\Y@de,\Y@tr,\Y@qu)%
       \edef\X@un{\X@qu}\edef\Y@un{\Y@qu}\edef\Ai@{\Aj@}\edef\Aj@{\Ak@}}}
\ctr@ld@f\def\Psc@rveTD#1[#2]{\ifnum\CUR@proj<\tw@\Psc@rvePPTD#1[#2]\else\Psc@rveCPTD#1[#2]\fi}
\ctr@ld@f\def\Psc@rvePPTD#1[#2]{\setc@ntr@l{2}%
    \def\list@num{#2}\extrairelepremi@r\Ak@\de\list@num\Figptpr@j-5:/\Ak@/%
    \extrairelepremi@r\Ai@\de\list@num\Figptpr@j-3:/\Ai@/%
    \extrairelepremi@r\Aj@\de\list@num\Figptpr@j-4:/\Aj@/%
    \f@gnewpath\PSwrit@cmdS{-3}{\c@mmoveto}{\fwf@g}{\X@un}{\Y@un}%
    \figvectPDD -1[-5,-4]%
    \@ecfor\Ak@:=\list@num\do{\Figptpr@j-5:/\Ak@/\figpttraDD-2:=-3/#1,-1/%
       \BdingB@xfalse\PSwrit@cmdS{-2}{}{\fwf@g}{\X@de}{\Y@de}%
       \figvectPDD -1[-3,-5]\figpttraDD-2:=-4/-#1,-1/%
       \PSwrit@cmdS{-2}{}{\fwf@g}{\X@tr}{\Y@tr}\BdingB@xtrue%
       \PSwrit@cmdS{-4}{\c@mcurveto}{\fwf@g}{\X@qu}{\Y@qu}%
       \B@zierBB@x{1}{\Y@un}(\X@un,\X@de,\X@tr,\X@qu)%
       \B@zierBB@x{2}{\X@un}(\Y@un,\Y@de,\Y@tr,\Y@qu)%
       \edef\X@un{\X@qu}\edef\Y@un{\Y@qu}\figptcopyDD-3:/-4/\figptcopyDD-4:/-5/}}
\ctr@ld@f\def\Psc@rveCPTD#1[#2]{\setc@ntr@l{2}%
    \def\list@num{#2}\extrairelepremi@r\Ak@\de\list@num%
    \extrairelepremi@r\Ai@\de\list@num\extrairelepremi@r\Aj@\de\list@num%
    \Figptpr@j-7:/\Ai@/%
    \f@gnewpath\PSwrit@cmd{-7}{\c@mmoveto}{\fwf@g}%
    \figvectPTD -9[\Ak@,\Aj@]%
    \@ecfor\Ak@:=\list@num\do{\figpttraTD-10:=\Ai@/#1,-9/%
       \figvectPTD -9[\Ai@,\Ak@]\figpttraTD-11:=\Aj@/-#1,-9/%
       \subB@zierTD\NBz@rcs[\Ai@,-10,-11,\Aj@]\edef\Ai@{\Aj@}\edef\Aj@{\Ak@}}}
\ctr@ld@f\def\figdrawend{\ifCUR@PS\ifGR@cri\immediate\closeout\fwf@g%
    \immediate\openout\fwf@g=\PSfilen@me\relax%
    \ifPDFm@ke\PSBdingB@x\else%
    \immediate\write\fwf@g{\p@urcent\string!PS-Adobe-2.0 EPSF-2.0}%
    \PShe@der{Creator\string: TeX (fig4tex.tex)}%
    \PShe@der{Title\string: \PSfilen@me}%
    \PShe@der{CreationDate\string: \the\day/\the\month/\the\year}%
    \PSBdingB@x%
    \PShe@der{EndComments}\PSdict@\fi%
    \immediate\write\fwf@g{\c@mgsave}%
    \openin\frf@g=\auxfilen@me\c@pypsfile\fwf@g\frf@g\closein\frf@g%
    \immediate\write\fwf@g{\c@mgrestore}%
    \PSc@mment{End of file.}\immediate\closeout\fwf@g%
    \immediate\openout\fwf@g=\auxfilen@me\immediate\closeout\fwf@g%
    \immediate\write16{File \PSfilen@me\space created.}\fi\fi\CUR@PSfalse\GR@critrue}
\ctr@ld@f\def\PShe@der#1{\immediate\write\fwf@g{\p@urcent\p@urcent#1}}
\ctr@ld@f\def\PSBdingB@x{{\v@lX=\ptT@ptps\c@@rdXmin\v@lY=\ptT@ptps\c@@rdYmin%
     \v@lXa=\ptT@ptps\c@@rdXmax\v@lYa=\ptT@ptps\c@@rdYmax%
     \PShe@der{BoundingBox\string: \repdecn@mb{\v@lX}\space\repdecn@mb{\v@lY}%
     \space\repdecn@mb{\v@lXa}\space\repdecn@mb{\v@lYa}}}}
\ctr@ld@f\def\figdrawfcconnect[#1]{{\ifCUR@PS\ifGR@cri\PSc@mment{fcconnect Points=#1}%
    \Q@s@tfillmode{no}\s@uvc@ntr@l\et@tpsfcconnect\resetc@ntr@l{2}%
    \fcc@nnect@[#1]\resetc@ntr@l\et@tpsfcconnect\PSc@mment{End fcconnect}\fi\fi}}
\ctr@ld@f\def\fcc@nnect@[#1]{\let\N@rm=\n@rmeucDD\def\list@num{#1}%
    \extrairelepremi@r\Ai@\de\list@num\edef\pr@m{\Ai@}\v@leur=\z@\p@rtent=\@ne\c@llgtot%
    \ifcase\fclin@typ@\edef\list@num{[\pr@m,#1,\Ai@}\expandafter\figdrawcurve\list@num]%
    \else\ifdim\fclin@r@d\p@>\z@\Pslin@conge[#1]\else\figdrawline[#1]\fi\fi%
    \v@leur=\@rrowp@s\v@leur\edef\list@num{#1,\Ai@,0}%
    \extrairelepremi@r\Ai@\de\list@num\mili@u=\epsil@n\c@llgpart%
    \advance\mili@u-\epsil@n\advance\mili@u-\delt@\advance\v@leur-\mili@u%
    \ifcase\fclin@typ@\invers@\mili@u\delt@%
    \ifnum\@rrowr@fpt>\z@\advance\delt@-\v@leur\v@leur=\delt@\fi%
    \v@leur=\repdecn@mb\v@leur\mili@u\edef\v@lt{\repdecn@mb\v@leur}%
    \extrairelepremi@r\Ak@\de\list@num%
    \figvectPDD-1[\pr@m,\Aj@]\figpttraDD-6:=\Ai@/\curv@roundness,-1/%
    \figvectPDD-1[\Ak@,\Ai@]\figpttraDD-7:=\Aj@/\curv@roundness,-1/%
    \delt@=\@rrowheadlength\p@\delt@=\C@AHANG\delt@\edef\R@dius{\repdecn@mb{\delt@}}%
    \ifcase\@rrowr@fpt%
    \FigptintercircB@zDD-8::\v@lt,\R@dius[\Ai@,-6,-7,\Aj@]\Q@arrowheadDD[-5,-8]\else%
    \FigptintercircB@zDD-8::\v@lt,\R@dius[\Aj@,-7,-6,\Ai@]\Q@arrowheadDD[-8,-5]\fi%
    \else\advance\delt@-\v@leur%
    \p@rtentiere{\p@rtent}{\delt@}\edef\C@efun{\the\p@rtent}%
    \p@rtentiere{\p@rtent}{\v@leur}\edef\C@efde{\the\p@rtent}%
    \figptbaryDD-5:[\Ai@,\Aj@;\C@efun,\C@efde]\ifcase\@rrowr@fpt%
    \delt@=\@rrowheadlength\unit@\delt@=\C@AHANG\delt@\edef\t@ille{\repdecn@mb{\delt@}}%
    \figvectPDD-2[\Ai@,\Aj@]\vecunit@{-2}{-2}\figpttraDD-5:=-5/\t@ille,-2/\fi%
    \Q@arrowheadDD[\Ai@,-5]\fi}
\ctr@ld@f\def\c@llgtot{\@ecfor\Aj@:=\list@num\do{\figvectP-1[\Ai@,\Aj@]\N@rm\delt@{-1}%
    \advance\v@leur\delt@\advance\p@rtent\@ne\edef\Ai@{\Aj@}}}
\ctr@ld@f\def\c@llgpart{\extrairelepremi@r\Aj@\de\list@num\figvectP-1[\Ai@,\Aj@]\N@rm\delt@{-1}%
    \advance\mili@u\delt@\ifdim\mili@u<\v@leur\edef\pr@m{\Ai@}\edef\Ai@{\Aj@}\c@llgpart\fi}
\ctr@ld@f\def\Pslin@conge[#1]{\ifnum\p@rtent>\tw@{\def\list@num{#1}%
    \extrairelepremi@r\Ai@\de\list@num\extrairelepremi@r\Aj@\de\list@num%
    \figptcopy-6:/\Ai@/\figvectP-3[\Ai@,\Aj@]\vecunit@{-3}{-3}\v@lmax=\result@t%
    \@ecfor\Ak@:=\list@num\do{\figvectP-4[\Aj@,\Ak@]\vecunit@{-4}{-4}%
    \minim@m\v@lmin\v@lmax\result@t\v@lmax=\result@t%
    \det@rm\delt@[-3,-4]\maxim@m\mili@u{\delt@}{-\delt@}\ifdim\mili@u>\Cepsil@n%
    \ifdim\delt@>\z@\figgetangleDD\Angl@[\Aj@,\Ak@,\Ai@]\else%
    \figgetangleDD\Angl@[\Aj@,\Ai@,\Ak@]\fi%
    \v@leur=\PI@deg\advance\v@leur-\Angl@\p@\divide\v@leur\tw@%
    \edef\Angl@{\repdecn@mb\v@leur}\c@ssin{\C@}{\S@}{\Angl@}\v@leur=\fclin@r@d\unit@%
    \v@leur=\S@\v@leur\mili@u=\C@\p@\invers@\mili@u\mili@u%
    \v@leur=\repdecn@mb{\mili@u}\v@leur%
    \minim@m\v@leur\v@leur\v@lmin\edef\t@ille{\repdecn@mb{\v@leur}}%
    \figpttra-5:=\Aj@/-\t@ille,-3/\figdrawline[-6,-5]\figpttra-6:=\Aj@/\t@ille,-4/%
    \figvectNVDD-3[-3]\figvectNVDD-8[-4]\inters@cDD-7:[-5,-3;-6,-8]%
    \ifdim\delt@>\z@\figdrawarccircP-7;\fclin@r@d[-5,-6]\else\figdrawarccircP-7;\fclin@r@d[-6,-5]\fi%
    \else\figdrawline[-6,\Aj@]\figptcopy-6:/\Aj@/\fi
    \edef\Ai@{\Aj@}\edef\Aj@{\Ak@}\figptcopy-3:/-4/}\figdrawline[-6,\Aj@]}\else\figdrawline[#1]\fi}
\ctr@ld@f\def\figdrawfcnode[#1]#2{{\ifCUR@PS\ifGR@cri\PSc@mment{fcnode Points=#1}%
    \s@uvc@ntr@l\et@tpsfcnode\resetc@ntr@l{2}%
    \def\t@xt@{#2}\ifx\t@xt@\empty\def\g@tt@xt{\setbox\Gb@x=\hbox{\Figg@tT{\p@int}}}%
    \else\def\g@tt@xt{\setbox\Gb@x=\hbox{#2}}\fi%
    \v@lmin=\h@rdfcXp@dd\advance\v@lmin\Xp@dd\unit@\multiply\v@lmin\tw@%
    \v@lmax=\h@rdfcYp@dd\advance\v@lmax\Yp@dd\unit@\multiply\v@lmax\tw@%
    \Figv@ctCreg-8(\unit@,-\unit@)\def\list@num{#1}%
    \delt@=\CUR@width bp\divide\delt@\tw@%
    \fcn@de\PSc@mment{End fcnode}\resetc@ntr@l\et@tpsfcnode\fi\fi}}
\ctr@ld@f\def\d@butn@de{\g@tt@xt\v@lX=\wd\Gb@x%
    \v@lY=\ht\Gb@x\advance\v@lY\dp\Gb@x\advance\v@lX\v@lmin\advance\v@lY\v@lmax}
\ctr@ld@f\def\fcn@deE{%
    \@ecfor\p@int:=\list@num\do{\d@butn@de\v@lX=\unssqrttw@\v@lX\v@lY=\unssqrttw@\v@lY%
    \ifdim\thickn@ss\p@>\z@
    \v@lXa=\v@lX\advance\v@lXa\delt@\v@lXa=\ptT@unit@\v@lXa\edef\XR@d{\repdecn@mb\v@lXa}%
    \v@lYa=\v@lY\advance\v@lYa\delt@\v@lYa=\ptT@unit@\v@lYa\edef\YR@d{\repdecn@mb\v@lYa}%
    \arct@n\v@leur(\v@lXa,\v@lYa)\v@leur=\rdT@deg\v@leur\edef\@nglde{\repdecn@mb\v@leur}%
    {\c@lptellDD-2::\p@int;\XR@d,\YR@d(\@nglde)}
    \advance\v@leur-\PI@deg\edef\@nglun{\repdecn@mb\v@leur}%
    {\c@lptellDD-3::\p@int;\XR@d,\YR@d(\@nglun)}%
    \figptstra-6=-3,-2,\p@int/\thickn@ss,-8/\Q@s@tfillmode{yes}%
    \Pss@tspecifSt{color=\DDV@thickcolor}%
    \figdrawline[-2,-3,-6,-5]\figdrawarcell-4;\XR@d,\YR@d(\@nglun,\@nglde,0)%
    \Psrest@reSt{color=\DDV@thickcolor}\fi
    \v@lX=\ptT@unit@\v@lX\v@lY=\ptT@unit@\v@lY%
    \edef\XR@d{\repdecn@mb\v@lX}\edef\YR@d{\repdecn@mb\v@lY}%
    \Q@s@tfillmode{yes}\Pss@tspecifSt{color=\fcbgc@lor}%
    \figdrawarcell\p@int;\XR@d,\YR@d(0,360,0)%
    \Q@s@tfillmode{no}\Psrest@reSt{color=\fcbgc@lor}\figdrawarcell\p@int;\XR@d,\YR@d(0,360,0)}}
\ctr@ld@f\def\fcn@deL{\delt@=\ptT@unit@\delt@\edef\t@ille{\repdecn@mb\delt@}%
    \@ecfor\p@int:=\list@num\do{\Figg@tXYa{\p@int}\d@butn@de%
    \ifdim\v@lX>\v@lY\itis@Ktrue\else\itis@Kfalse\fi%
    \advance\v@lXa-\v@lX\Figp@intreg-1:(\v@lXa,\v@lYa)%
    \advance\v@lXa\v@lX\advance\v@lYa-\v@lY\Figp@intreg-2:(\v@lXa,\v@lYa)%
    \advance\v@lXa\v@lX\advance\v@lYa\v@lY\Figp@intreg-3:(\v@lXa,\v@lYa)%
    \advance\v@lXa-\v@lX\advance\v@lYa\v@lY\Figp@intreg-4:(\v@lXa,\v@lYa)%
    \ifdim\thickn@ss\p@>\z@
    \Figg@tXYa{\p@int}\Q@s@tfillmode{yes}\Pss@tspecifSt{color=\DDV@thickcolor}%
    \c@lpt@xt{-1}{-4}\c@lpt@xt@\v@lXa\v@lYa\v@lX\v@lY\c@rre\delt@%
    \Figp@intregDD-9:(\v@lZ,\v@lYa)\Figp@intregDD-11:(\v@lZa,\v@lYa)%
    \c@lpt@xt{-4}{-3}\c@lpt@xt@\v@lYa\v@lXa\v@lY\v@lX\delt@\c@rre%
    \Figp@intregDD-12:(\v@lXa,\v@lZ)\Figp@intregDD-10:(\v@lXa,\v@lZa)%
    \ifitis@K\figptstra-7=-9,-10,-11/\thickn@ss,-8/\figdrawline[-9,-11,-5,-6,-7]\else%
    \figptstra-7=-10,-11,-12/\thickn@ss,-8/\figdrawline[-10,-12,-5,-6,-7]\fi%
    \Psrest@reSt{color=\DDV@thickcolor}\fi
    \Q@s@tfillmode{yes}\Pss@tspecifSt{color=\fcbgc@lor}\figdrawline[-1,-2,-3,-4]%
    \Q@s@tfillmode{no}\Psrest@reSt{color=\fcbgc@lor}\figdrawline[-1,-2,-3,-4,-1]}}
\ctr@ld@f\def\c@lpt@xt#1#2{\figvectN-7[#1,#2]\vecunit@{-7}{-7}\figpttra-5:=#1/\t@ille,-7/%
    \figvectP-7[#1,#2]\Figg@tXY{-7}\c@rre=\v@lX\delt@=\v@lY\Figg@tXY{-5}}
\ctr@ld@f\def\c@lpt@xt@#1#2#3#4#5#6{\v@lZ=#6\invers@{\v@lZ}{\v@lZ}\v@leur=\repdecn@mb{#5}\v@lZ%
    \v@lZ=#2\advance\v@lZ-#4\mili@u=\repdecn@mb{\v@leur}\v@lZ%
    \v@lZ=#3\advance\v@lZ\mili@u\v@lZa=-\v@lZ\advance\v@lZa\tw@#1}
\ctr@ld@f\def\fcn@deR{\@ecfor\p@int:=\list@num\do{\Figg@tXYa{\p@int}\d@butn@de%
    \advance\v@lXa-0.5\v@lX\advance\v@lYa-0.5\v@lY\Figp@intreg-1:(\v@lXa,\v@lYa)%
    \advance\v@lXa\v@lX\Figp@intreg-2:(\v@lXa,\v@lYa)%
    \advance\v@lYa\v@lY\Figp@intreg-3:(\v@lXa,\v@lYa)%
    \advance\v@lXa-\v@lX\Figp@intreg-4:(\v@lXa,\v@lYa)%
    \ifdim\thickn@ss\p@>\z@
    \Q@s@tfillmode{yes}\Pss@tspecifSt{color=\DDV@thickcolor}%
    \Figv@ctCreg-5(-\delt@,-\delt@)\figpttra-9:=-1/1,-5/%
    \Figv@ctCreg-5(\delt@,-\delt@)\figpttra-10:=-2/1,-5/%
    \Figv@ctCreg-5(\delt@,\delt@)\figpttra-11:=-3/1,-5/%
    \figptstra-7=-9,-10,-11/\thickn@ss,-8/\figdrawline[-9,-11,-5,-6,-7]%
    \Psrest@reSt{color=\DDV@thickcolor}\fi
    \Q@s@tfillmode{yes}\Pss@tspecifSt{color=\fcbgc@lor}\figdrawline[-1,-2,-3,-4]%
    \Q@s@tfillmode{no}\Psrest@reSt{color=\fcbgc@lor}\figdrawline[-1,-2,-3,-4,-1]}}
\ctr@ld@f\def\Pssetfl@wchart#1=#2|{\keln@mtr#1|%
    \def\n@mref{arr}\ifx\l@debut\n@mref\expandafter\keln@mtr\l@suite|%
     \def\n@mref{owp}\ifx\l@debut\n@mref\update@ttr\D@FTfcarrowposition\P@setfcarrowposition{#2}\else
     \def\n@mref{owr}\ifx\l@debut\n@mref\update@ttr\D@FTfcarrowrefpt\P@setfcarrowrefpt{#2}\else
     \W@rnmesAttr{figset flowchart}{#1}\fi\fi\else%
    \def\n@mref{bgc}\ifx\l@debut\n@mref\update@ttr\D@FTfcbgcolor\P@setfcbgcolor{#2}\else
    \def\n@mref{lin}\ifx\l@debut\n@mref\update@ttr\D@FTfcline\P@setfcline{#2}\else
    \def\n@mref{pad}\ifx\l@debut\n@mref\update@ttr\D@FTfcxpadding\P@setfcxpadding{#2}%
                                       \update@ttr\D@FTfcypadding\P@setfcypadding{#2}\else
    \def\n@mref{rad}\ifx\l@debut\n@mref\update@ttr\D@FTfcradius\P@setfcradius{#2}\else
    \def\n@mref{sha}\ifx\l@debut\n@mref\update@ttr\D@FTfcshape\P@setfcshape{#2}\else
    \def\n@mref{thi}\ifx\l@debut\n@mref\expandafter\keln@mtr\l@suite|%
     \def\n@mref{ckc}\ifx\l@debut\n@mref\update@ttr\D@FTref\P@setfcthickcolor{#2}\else
     \def\n@mref{ckn}\ifx\l@debut\n@mref\update@ttr\D@FTfcthickness\P@setfcthickness{#2}\else
     \W@rnmesAttr{figset flowchart}{#1}\fi\fi\else%
    \def\n@mref{xpa}\ifx\l@debut\n@mref\update@ttr\D@FTfcxpadding\P@setfcxpadding{#2}\else
    \def\n@mref{ypa}\ifx\l@debut\n@mref\update@ttr\D@FTfcypadding\P@setfcypadding{#2}\else
    \W@rnmesAttr{figset flowchart}{#1}\fi\fi\fi\fi\fi\fi\fi\fi\fi}
\ctr@ln@m\@rrowp@s
\ctr@ld@f\def\P@setfcarrowposition#1{\edef\@rrowp@s{#1}}
\ctr@ln@m\@rrowr@fpt
\ctr@ld@f\def\P@setfcarrowrefpt#1{\setfcr@fpt#1|}
\ctr@ld@f\def\setfcr@fpt#1#2|{\if#1e\def\@rrowr@fpt{1}\else\def\@rrowr@fpt{0}\fi}
\ctr@ln@m\fcbgc@lor
\ctr@ld@f\def\P@setfcbgcolor#1{\edef\fcbgc@lor{#1}}
\ctr@ln@m\fclin@typ@
\ctr@ld@f\def\P@setfcline#1{\setfccurv@#1|}
\ctr@ld@f\def\setfccurv@#1#2|{\if#1c\def\fclin@typ@{0}\else\def\fclin@typ@{1}\fi}
\ctr@ln@m\fclin@r@d
\ctr@ld@f\def\P@setfcradius#1{\edef\fclin@r@d{#1}}
\ctr@ln@m\fcn@de \ctr@ln@m\fcsh@pe
\ctr@ln@m\h@rdfcXp@dd \ctr@ln@m\h@rdfcYp@dd
\ctr@ld@f\def\P@setfcshape#1{\setfcshap@#1|}
\ctr@ld@f\def\setfcshap@#1#2|{%
    \if#1e\let\fcn@de=\fcn@deE\def\h@rdfcXp@dd{4pt}\def\h@rdfcYp@dd{4pt}%
     \edef\fcsh@pe{ellipse}\else%
    \if#1l\let\fcn@de=\fcn@deL\def\h@rdfcXp@dd{4pt}\def\h@rdfcYp@dd{4pt}%
     \edef\fcsh@pe{lozenge}\else%
          \let\fcn@de=\fcn@deR\def\h@rdfcXp@dd{6pt}\def\h@rdfcYp@dd{6pt}%
     \edef\fcsh@pe{rectangle}\fi\fi}
\ctr@ln@m\DDV@thickcolor
\ctr@ld@f\def\P@setfcthickcolor#1{\edef\DDV@thickcolor{#1}}
\ctr@ln@m\thickn@ss
\ctr@ld@f\def\P@setfcthickness#1{\edef\thickn@ss{#1}}
\ctr@ln@m\Xp@dd
\ctr@ld@f\def\P@setfcxpadding#1{\edef\Xp@dd{#1}}
\ctr@ln@m\Yp@dd
\ctr@ld@f\def\P@setfcypadding#1{\edef\Yp@dd{#1}}
\ctr@ld@f\def\figdrawline[#1]{{\ifCUR@PS\ifGR@cri\PSc@mment{line Points=#1}%
    \let\figdrawlign@=\Pslign@P\Pslin@{#1}\PSc@mment{End line}\fi\fi}}
\ctr@ld@f\def\figdrawlineF#1{{\ifCUR@PS\ifGR@cri\PSc@mment{lineF Filename=#1}%
    \let\figdrawlign@=\Pslign@F\Pslin@{#1}\PSc@mment{End lineF}\fi\fi}}
\ctr@ld@f\def\figdrawlineC(#1){{\ifCUR@PS\ifGR@cri\PSc@mment{lineC}%
    \let\figdrawlign@=\Pslign@C\Pslin@{#1}\PSc@mment{End lineC}\fi\fi}}
\ctr@ld@f\def\Pslin@#1{\iffillm@de\figdrawlign@{#1}%
    \f@gfill%
    \else\figdrawlign@{#1}\ifx\derp@int\premp@int%
    \f@gclosestroke%
    \else\f@gstroke\fi\fi}
\ctr@ld@f\def\Pslign@P#1{\def\list@num{#1}\extrairelepremi@r\p@int\de\list@num%
    \edef\premp@int{\p@int}\f@gnewpath%
    \PSwrit@cmd{\p@int}{\c@mmoveto}{\fwf@g}%
    \@ecfor\p@int:=\list@num\do{\PSwrit@cmd{\p@int}{\c@mlineto}{\fwf@g}%
    \edef\derp@int{\p@int}}}
\ctr@ld@f\def\Pslign@F#1{\s@uvc@ntr@l\et@tPslign@F\setc@ntr@l{2}\openin\frf@g=#1\relax%
    \ifeof\frf@g\message{*** File #1 not found !}\end\else%
    \read\frf@g to\tr@c\edef\premp@int{\tr@c}\expandafter\extr@ctCF\tr@c:%
    \f@gnewpath\PSwrit@cmd{-1}{\c@mmoveto}{\fwf@g}%
    \loop\read\frf@g to\tr@c\ifeof\frf@g\mored@tafalse\else\mored@tatrue\fi%
    \ifmored@ta\expandafter\extr@ctCF\tr@c:\PSwrit@cmd{-1}{\c@mlineto}{\fwf@g}%
    \edef\derp@int{\tr@c}\repeat\fi\closein\frf@g\resetc@ntr@l\et@tPslign@F}
\ctr@ln@m\extr@ctCF
\ctr@ld@f\def\extr@ctCFDD#1 #2:{\v@lX=#1\unit@\v@lY=#2\unit@\Figp@intregDD-1:(\v@lX,\v@lY)}
\ctr@ld@f\def\extr@ctCFTD#1 #2 #3:{\v@lX=#1\unit@\v@lY=#2\unit@\v@lZ=#3\unit@%
    \Figp@intregTD-1:(\v@lX,\v@lY,\v@lZ)}
\ctr@ld@f\def\Pslign@C#1{\s@uvc@ntr@l\et@tPslign@C\setc@ntr@l{2}%
    \def\list@num{#1}\extrairelepremi@r\p@int\de\list@num%
    \edef\premp@int{\p@int}\f@gnewpath%
    \expandafter\Pslign@C@\p@int:\PSwrit@cmd{-1}{\c@mmoveto}{\fwf@g}%
    \@ecfor\p@int:=\list@num\do{\expandafter\Pslign@C@\p@int:%
    \PSwrit@cmd{-1}{\c@mlineto}{\fwf@g}\edef\derp@int{\p@int}}%
    \resetc@ntr@l\et@tPslign@C}
\ctr@ld@f\def\Pslign@C@#1 #2:{{\def\t@xt@{#1}\ifx\t@xt@\empty\Pslign@C@#2:
    \else\extr@ctCF#1 #2:\fi}}
\ctr@ld@f\def\Pssetm@sh#1=#2|{\keln@mde#1|%
    \def\n@mref{co}\ifx\l@debut\n@mref\update@ttr\D@FTref\P@setmeshcolor{#2}\else
    \def\n@mref{da}\ifx\l@debut\n@mref\update@ttr\D@FTref\P@setmeshdash{#2}\else
    \def\n@mref{di}\ifx\l@debut\n@mref\update@ttr\D@FTmeshdiag\Q@s@tmeshdiag{#2}\else
    \def\n@mref{wi}\ifx\l@debut\n@mref\update@ttr\D@FTref\P@setmeshwidth{#2}\else
    \W@rnmesAttr{figset mesh}{#1}\fi\fi\fi\fi}
\ctr@ln@m\c@ntrolmesh
\ctr@ld@f\def\Q@s@tmeshdiag#1{\edef\c@ntrolmesh{#1}}
\ctr@ld@f\def\D@FTmeshdiag{0}    
\ctr@ln@m\DDV@meshcolor
\ctr@ld@f\def\P@setmeshcolor#1{\edef\DDV@meshcolor{#1}}
\ctr@ln@m\DDV@meshdash
\ctr@ld@f\def\P@setmeshdash#1{\edef\DDV@meshdash{#1}}
\ctr@ln@m\DDV@meshwidth
\ctr@ld@f\def\P@setmeshwidth#1{\edef\DDV@meshwidth{#1}}
\ctr@ld@f\def\figdrawmesh#1,#2[#3,#4,#5,#6]{{\ifCUR@PS\ifGR@cri%
    \PSc@mment{mesh N1=#1, N2=#2, Quadrangle=[#3,#4,#5,#6]}\s@uvc@ntr@l\et@tpsmesh%
    \Pss@tspecifSt{color=\DDV@meshcolor,dash=\DDV@meshdash,width=\DDV@meshwidth}%
    \setc@ntr@l{2}%
    \ifnum#1>\@ne\Psmeshp@rt#1[#3,#4,#5,#6]\fi%
    \ifnum#2>\@ne\Psmeshp@rt#2[#4,#5,#6,#3]\fi%
    \ifnum\c@ntrolmesh>\z@\Psmeshdi@g#1,#2[#3,#4,#5,#6]\fi%
    \ifnum\c@ntrolmesh<\z@\Psmeshdi@g#2,#1[#4,#5,#6,#3]\fi%
    \Psrest@reSt{color=\DDV@meshcolor,dash=\DDV@meshdash,width=\DDV@meshwidth}%
    \figdrawline[#3,#4,#5,#6,#3]\PSc@mment{End mesh}\resetc@ntr@l\et@tpsmesh\fi\fi}}
\ctr@ld@f\def\Psmeshp@rt#1[#2,#3,#4,#5]{{\l@mbd@un=\@ne\l@mbd@de=#1\loop%
    \ifnum\l@mbd@un<#1\advance\l@mbd@de\m@ne\figptbary-1:[#2,#3;\l@mbd@de,\l@mbd@un]%
    \figptbary-2:[#5,#4;\l@mbd@de,\l@mbd@un]\figdrawline[-1,-2]\advance\l@mbd@un\@ne\repeat}}
\ctr@ld@f\def\Psmeshdi@g#1,#2[#3,#4,#5,#6]{\figptcopy-2:/#3/\figptcopy-3:/#6/%
    \l@mbd@un=\z@\l@mbd@de=#1\loop\ifnum\l@mbd@un<#1%
    \advance\l@mbd@un\@ne\advance\l@mbd@de\m@ne\figptcopy-1:/-2/\figptcopy-4:/-3/%
    \figptbary-2:[#3,#4;\l@mbd@de,\l@mbd@un]%
    \figptbary-3:[#6,#5;\l@mbd@de,\l@mbd@un]\Psmeshdi@gp@rt#2[-1,-2,-3,-4]\repeat}
\ctr@ld@f\def\Psmeshdi@gp@rt#1[#2,#3,#4,#5]{{\l@mbd@un=\z@\l@mbd@de=#1\loop%
    \ifnum\l@mbd@un<#1\figptbary-5:[#2,#5;\l@mbd@de,\l@mbd@un]%
    \advance\l@mbd@de\m@ne\advance\l@mbd@un\@ne%
    \figptbary-6:[#3,#4;\l@mbd@de,\l@mbd@un]\figdrawline[-5,-6]\repeat}}
\ctr@ln@m\figdrawnormal
\ctr@ld@f\def\Q@normalDD#1,#2[#3,#4]{{\ifCUR@PS\ifGR@cri%
    \PSc@mment{normal Length=#1, Lambda=#2 [Pt1,Pt2]=[#3,#4]}%
    \s@uvc@ntr@l\et@tpsnormal\resetc@ntr@l{2}\figptendnormal-6::#1,#2[#3,#4]%
    \figptcopyDD-5:/-1/\figdrawarrow[-5,-6]%
    \PSc@mment{End normal}\resetc@ntr@l\et@tpsnormal\fi\fi}}
\ctr@ld@f\def\figreset#1{\trtlis@rg{#1}{\Psreset@}}
\ctr@ld@f\def\Psreset@#1|{\def\t@xt@{#1}\ifx\t@xt@\empty\P@resetg@n
    \else\keln@mde#1|%
    \def\n@mref{al}\ifx\l@debut\n@mref%
        \figset altitude(blcolor=default,bldash=default,blwidth=default,%
        sqcolor=default,sqdash=default,sqwidth=default)\else
    \def\n@mref{ar}\ifx\l@debut\n@mref\figresetarrowhead\else
    \def\n@mref{cu}\ifx\l@debut\n@mref\figset curve(roundness=\D@FTroundness)\else
    \def\n@mref{ge}\ifx\l@debut\n@mref\P@resetg@n\else
    \def\n@mref{fl}\ifx\l@debut\n@mref%
        \figset flowchart(arrowp=\D@FTfcarrowposition,arrowr=\D@FTfcarrowrefpt,%
	bgcolor=\D@FTfcbgcolor,line=\D@FTfcline,radius=\D@FTfcradius,%
	shape=\D@FTfcshape,thickcolor=default,thickness=\D@FTfcthickness,%
	xpadd=\D@FTfcxpadding,ypadd=\D@FTfcypadding)\else
    \def\n@mref{me}\ifx\l@debut\n@mref\figset mesh(diag=\D@FTmeshdiag,%
        color=default,dash=default,width=default)\else
    \def\n@mref{tr}\ifx\l@debut\n@mref%
        \figset trimesh(color=default,dash=default,width=default)\else
    \W@rnmeskwd{figreset}{#1}\fi\fi\fi\fi\fi\fi\fi\fi}
\ctr@ld@f\def\P@resetg@n{\figset (color=\D@FTcolor,dash=\D@FTdash,fill=\D@FTfill,%
    join=\D@FTjoin,width=\D@FTwidth)}
\ctr@ld@f\def\figset#1(#2){\def\t@xt@{#1}\ifx\t@xt@\empty\trtlis@rg{#2}{\Pssetg@n}
    \else\keln@mde#1|%
    \def\n@mref{al}\ifx\l@debut\n@mref\trtlis@rg{#2}{\Psset@lti}\else
    \def\n@mref{ar}\ifx\l@debut\n@mref\trtlis@rg{#2}{\Psset@rrowhe@d}\else
    \def\n@mref{cu}\ifx\l@debut\n@mref\trtlis@rg{#2}{\Pssetc@rve}\else
    \def\n@mref{fl}\ifx\l@debut\n@mref\trtlis@rg{#2}{\Pssetfl@wchart}\else
    \def\n@mref{ge}\ifx\l@debut\n@mref\trtlis@rg{#2}{\Pssetg@n}\else
    \def\n@mref{me}\ifx\l@debut\n@mref\trtlis@rg{#2}{\Pssetm@sh}\else
    \def\n@mref{pr}\ifx\l@debut\n@mref\ifCUR@PS\W@rnmesIgn{figset proj(...)}%
     \else\trtlis@rg{#2}{\Figsetpr@j}\fi\else
    \def\n@mref{tr}\ifx\l@debut\n@mref\trtlis@rg{#2}{\Pssettrim@sh}\else
    \def\n@mref{wr}\ifx\l@debut\n@mref\let\M@cro=\Figsetwr@te\trtlis@rgtok{#2,|}\else
    \W@rnmeskwd{figset}{#1}\fi\fi\fi\fi\fi\fi\fi\fi\fi\fi\ignorespaces}
\ctr@ld@f\def\figsetdefault#1(#2){\ifCUR@PS\W@rnmesIgn{figsetdefault}\else%
    \def\t@xt@{#1}\ifx\t@xt@\empty\trtlis@rg{#2}{\Pssd@g@n}\else\keln@mun#1|
    \def\n@mref{a}\ifx\l@debut\n@mref\trtlis@rg{#2}{\Pssd@@rrowhe@d}\else
    \def\n@mref{c}\ifx\l@debut\n@mref\trtlis@rg{#2}{\Pssd@c@rve}\else
    \def\n@mref{g}\ifx\l@debut\n@mref\trtlis@rg{#2}{\Pssd@g@n}\else
    \def\n@mref{f}\ifx\l@debut\n@mref\trtlis@rg{#2}{\Pssd@fl@wchart}\else
    \def\n@mref{m}\ifx\l@debut\n@mref\trtlis@rg{#2}{\Pssd@m@sh}\else
    \W@rnmeskwd{figsetdefault}{#1}\fi\fi\fi\fi\fi\fi\initpss@ttings\fi}
\ctr@ld@f\def\Pssd@g@n#1=#2|{\keln@mun#1|%
    \def\n@mref{c}\ifx\l@debut\n@mref\edef\D@FTcolor{#2}\else
    \def\n@mref{d}\ifx\l@debut\n@mref\edef\D@FTdash{#2}\else
    \def\n@mref{f}\ifx\l@debut\n@mref\edef\D@FTfill{#2}\else
    \def\n@mref{j}\ifx\l@debut\n@mref\edef\D@FTjoin{#2}\else
    \def\n@mref{u}\ifx\l@debut\n@mref\edef\D@FTupdate{#2}\Q@s@tupdate{#2}\else
    \def\n@mref{w}\ifx\l@debut\n@mref\edef\D@FTwidth{#2}\else
    \W@rnmesAttr{figsetdefault}{#1}\fi\fi\fi\fi\fi\fi}
\ctr@ld@f\def\Pssd@@rrowhe@d#1=#2|{\keln@mun#1|%
    \def\n@mref{a}\ifx\l@debut\n@mref\edef\D@FTarrowheadangle{#2}\else
    \def\n@mref{f}\ifx\l@debut\n@mref\edef\D@FTarrowheadfill{#2}\else
    \def\n@mref{l}\ifx\l@debut\n@mref\y@tiunit{#2}\ifunitpr@sent%
     \edef\D@FTh@rdahlength{#2}\else\edef\D@FTh@rdahlength{#2pt}%
     \message{*** \BS@ figsetdefault (..., #1=#2, ...) : unit is missing, pt is assumed.}%
     \fi\else
    \def\n@mref{o}\ifx\l@debut\n@mref\edef\D@FTarrowheadout{#2}\else
    \def\n@mref{r}\ifx\l@debut\n@mref\edef\D@FTarrowheadratio{#2}\else
    \W@rnmesAttr{figsetdefault arrowhead}{#1}\fi\fi\fi\fi\fi}
\ctr@ld@f\def\Pssd@c@rve#1=#2|{\keln@mun#1|%
    \def\n@mref{r}\ifx\l@debut\n@mref\edef\D@FTroundness{#2}\else%
    \W@rnmesAttr{figsetdefault curve}{#1}\fi}
\ctr@ld@f\def\Pssd@fl@wchart#1=#2|{\keln@mtr#1|%
    \def\n@mref{arr}\ifx\l@debut\n@mref\expandafter\keln@mtr\l@suite|%
     \def\n@mref{owp}\ifx\l@debut\n@mref\edef\D@FTfcarrowposition{#2}\else
     \def\n@mref{owr}\ifx\l@debut\n@mref\edef\D@FTfcarrowrefpt{#2}\else
                     \W@rnmesAttr{figsetdefault flowchart}{#1}\fi\fi\else%
    \def\n@mref{bgc}\ifx\l@debut\n@mref\edef\D@FTfcbgcolor{#2}\else
    \def\n@mref{lin}\ifx\l@debut\n@mref\edef\D@FTfcline{#2}\else
    \def\n@mref{pad}\ifx\l@debut\n@mref\edef\D@FTfcxpadding{#2}%
                    \edef\D@FTfcypadding{#2}\else
    \def\n@mref{rad}\ifx\l@debut\n@mref\edef\D@FTfcradius{#2}\else
    \def\n@mref{sha}\ifx\l@debut\n@mref\edef\D@FTfcshape{#2}\else
    \def\n@mref{thi}\ifx\l@debut\n@mref\expandafter\keln@mtr\l@suite|%
     \def\n@mref{ckn}\ifx\l@debut\n@mref\edef\D@FTfcthickness{#2}\else
                     \W@rnmesAttr{figsetdefault flowchart}{#1}\fi\else%
    \def\n@mref{xpa}\ifx\l@debut\n@mref\edef\D@FTfcxpadding{#2}\else
    \def\n@mref{ypa}\ifx\l@debut\n@mref\edef\D@FTfcypadding{#2}\else
    \W@rnmesAttr{figsetdefault flowchart}{#1}\fi\fi\fi\fi\fi\fi\fi\fi\fi}
\ctr@ld@f\def\D@FTfcarrowposition{0.5}
\ctr@ld@f\def\D@FTfcarrowrefpt{start}
\ctr@ld@f\def\D@FTfcbgcolor{1}
\ctr@ld@f\def\D@FTfcline{polygon}
\ctr@ld@f\def\D@FTfcradius{0}
\ctr@ld@f\def\D@FTfcshape{rectangle}
\ctr@ld@f\def\D@FTfcthickness{0}
\ctr@ld@f\def\D@FTfcxpadding{0}
\ctr@ld@f\def\D@FTfcypadding{0}
\ctr@ld@f\def\Pssd@m@sh#1=#2|{\keln@mun#1|%
    \def\n@mref{d}\ifx\l@debut\n@mref\edef\D@FTmeshdiag{#2}\else%
    \W@rnmesAttr{figsetdefault mesh}{#1}\fi}
\ctr@ln@w{newif}\iffillm@de
\ctr@ld@f\def\Q@s@tfillmode#1{\expandafter\setfillm@de#1:}
\ctr@ld@f\def\setfillm@de#1#2:{\if#1n\fillm@defalse\else\fillm@detrue\fi}
\ctr@ld@f\def\D@FTfill{no}     
\ctr@ln@w{newif}\ifGRupdatem@de
\ctr@ld@f\def\Q@s@tupdate#1{\ifCUR@PS\W@rnmesIgn{figset (update=...)}%
    \else\expandafter\setupd@te#1:\fi}
\ctr@ld@f\def\setupd@te#1#2:{\if#1n\GRupdatem@defalse\else\GRupdatem@detrue\fi}
\ctr@ld@f\def\D@FTupdate{no}     
\ctr@ln@m\CUR@color \ctr@ln@m\CUR@colorc@md
\ctr@ld@f\def\s@uvcolor#1{\edef#1{\CUR@color}}
\ctr@ld@f\def\D@FTcolor{0}       
\ctr@ld@f\def\Pssetc@lor#1{\ifGR@cri\result@tent=\@ne\expandafter\c@lnbV@l#1 :%
    \def\CUR@color{}\def\CUR@colorc@md{}%
    \ifcase\result@tent\or\Q@s@tgray{#1}\or\or\Q@s@trgb{#1}\or\Q@s@tcmyk{#1}\fi\fi}
\ctr@ln@m\CUR@colorc@mdStroke
\ctr@ld@f\def\Q@s@tcmyk#1{\ifGR@cri\def\CUR@color{#1}\def\CUR@colorc@md{\c@msetcmykcolor}%
    \def\CUR@colorc@mdStroke{\c@msetcmykcolorStroke}%
    \ifCUR@PS\PSc@mment{setcmyk Color=#1}\us@primarC@lor\fi\fi}
\ctr@ld@f\def\Q@s@trgb#1{\ifGR@cri\def\CUR@color{#1}\def\CUR@colorc@md{\c@msetrgbcolor}%
    \def\CUR@colorc@mdStroke{\c@msetrgbcolorStroke}%
    \ifCUR@PS\PSc@mment{setrgb Color=#1}\us@primarC@lor\fi\fi}
\ctr@ld@f\def\Q@s@tgray#1{\ifGR@cri\def\CUR@color{#1}\def\CUR@colorc@md{\c@msetgray}%
    \def\CUR@colorc@mdStroke{\c@msetgrayStroke}%
    \ifCUR@PS\PSc@mment{setgray Gray level=#1}\us@primarC@lor\fi\fi}
\ctr@ln@m\fillc@md
\ctr@ld@f\def\us@primarC@lor{\immediate\write\fwf@g{\d@fprimarC@lor}%
    \let\fillc@md=\prfillc@md}
\ctr@ld@f\def\prfillc@md{\d@fprimarC@lor\space\c@mfill}
\ctr@ld@f\def\c@lnbV@l#1 #2:{\def\t@xt@{#1}\relax\ifx\t@xt@\empty\c@lnbV@l#2:
    \else\c@lnbV@l@#1 #2:\fi}
\ctr@ld@f\def\c@lnbV@l@#1 #2:{\def\t@xt@{#2}\ifx\t@xt@\empty%
    \def\t@xt@{#1}\ifx\t@xt@\empty\advance\result@tent\m@ne\fi
    \else\advance\result@tent\@ne\c@lnbV@l@#2:\fi}
\ctr@ld@f\def\Blackcmyk{0 0 0 1}
\ctr@ld@f\def\Whitecmyk{0 0 0 0}
\ctr@ld@f\def\Cyancmyk{1 0 0 0}
\ctr@ld@f\def\Magentacmyk{0 1 0 0}
\ctr@ld@f\def\Yellowcmyk{0 0 1 0}
\ctr@ld@f\def\Redcmyk{0 1 1 0}
\ctr@ld@f\def\Greencmyk{1 0 1 0}
\ctr@ld@f\def\Bluecmyk{1 1 0 0}
\ctr@ld@f\def\Graycmyk{0 0 0 0.50}
\ctr@ld@f\def\BrickRedcmyk{0 0.89 0.94 0.28} 
\ctr@ld@f\def\Browncmyk{0 0.81 1 0.60} 
\ctr@ld@f\def\ForestGreencmyk{0.91 0 0.88 0.12} 
\ctr@ld@f\def\Goldenrodcmyk{ 0 0.10 0.84 0} 
\ctr@ld@f\def\Marooncmyk{0 0.87 0.68 0.32} 
\ctr@ld@f\def\Orangecmyk{0 0.61 0.87 0} 
\ctr@ld@f\def\Purplecmyk{0.45 0.86 0 0} 
\ctr@ld@f\def\RoyalBluecmyk{1. 0.50 0 0} 
\ctr@ld@f\def\Violetcmyk{0.79 0.88 0 0} 
\ctr@ld@f\def\Blackrgb{0 0 0}
\ctr@ld@f\def\Whitergb{1 1 1}
\ctr@ld@f\def\Redrgb{1 0 0}
\ctr@ld@f\def\Greenrgb{0 1 0}
\ctr@ld@f\def\Bluergb{0 0 1}
\ctr@ld@f\def\Cyanrgb{0 1 1}
\ctr@ld@f\def\Magentargb{1 0 1}
\ctr@ld@f\def\Yellowrgb{1 1 0}
\ctr@ld@f\def\Grayrgb{0.5 0.5 0.5}
\ctr@ld@f\def\Chocolatergb{0.824 0.412 0.118}
\ctr@ld@f\def\DarkGoldenrodrgb{0.722 0.525 0.043}
\ctr@ld@f\def\DarkOrangergb{1 0.549 0}
\ctr@ld@f\def\Firebrickrgb{0.698 0.133 0.133}
\ctr@ld@f\def\ForestGreenrgb{0.133 0.545 0.133}
\ctr@ld@f\def\Goldrgb{1 0.843 0}
\ctr@ld@f\def\HotPinkrgb{1 0.412 0.706}
\ctr@ld@f\def\Maroonrgb{0.690 0.188 0.376}
\ctr@ld@f\def\Pinkrgb{1 0.753 0.796}
\ctr@ld@f\def\RoyalBluergb{0.255 0.412 0.882}
\ctr@ld@f\def\Pssetg@n#1=#2|{\keln@mun#1|%
    \def\n@mref{c}\ifx\l@debut\n@mref\update@ttr\D@FTcolor\Pssetc@lor{#2}\else
    \def\n@mref{d}\ifx\l@debut\n@mref\update@ttr\D@FTdash\Q@s@tdash{#2}\else
    \def\n@mref{f}\ifx\l@debut\n@mref\update@ttr\D@FTfill\Q@s@tfillmode{#2}\else
    \def\n@mref{j}\ifx\l@debut\n@mref\update@ttr\D@FTjoin\Q@s@tjoin{#2}\else
    \def\n@mref{u}\ifx\l@debut\n@mref\update@ttr\D@FTupdate\Q@s@tupdate{#2}\else
    \def\n@mref{w}\ifx\l@debut\n@mref\update@ttr\D@FTwidth\Q@s@twidth{#2}\else
    \W@rnmesAttr{figset}{#1}\fi\fi\fi\fi\fi\fi}
\ctr@ln@m\CUR@dash
\ctr@ld@f\def\s@uvdash#1{\edef#1{\CUR@dash}}
\ctr@ld@f\def\D@FTdash{1}        
\ctr@ld@f\def\Q@s@tdash#1{\ifGR@cri\edef\CUR@dash{#1}\ifCUR@PS\expandafter\Pssetd@sh#1 :\fi\fi}
\ctr@ld@f\def\Pssetd@shI#1{\PSc@mment{setdash Index=#1}\ifcase#1%
    \or\immediate\write\fwf@g{[] 0 \c@msetdash}
    \or\immediate\write\fwf@g{[6 2] 0 \c@msetdash}
    \or\immediate\write\fwf@g{[4 2] 0 \c@msetdash}
    \or\immediate\write\fwf@g{[2 2] 0 \c@msetdash}
    \or\immediate\write\fwf@g{[1 2] 0 \c@msetdash}
    \or\immediate\write\fwf@g{[2 4] 0 \c@msetdash}
    \or\immediate\write\fwf@g{[3 5] 0 \c@msetdash}
    \or\immediate\write\fwf@g{[3 3] 0 \c@msetdash}
    \or\immediate\write\fwf@g{[3 5 1 5] 0 \c@msetdash}
    \or\immediate\write\fwf@g{[6 4 2 4] 0 \c@msetdash}
    \fi}
\ctr@ld@f\def\Pssetd@sh#1 #2:{{\def\t@xt@{#1}\ifx\t@xt@\empty\Pssetd@sh#2:
    \else\def\t@xt@{#2}\ifx\t@xt@\empty\Pssetd@shI{#1}\else\s@mme=\@ne\def\debutp@t{#1}%
    \an@lysd@sh#2:\ifodd\s@mme\edef\debutp@t{\debutp@t\space\finp@t}\def\finp@t{0}\fi%
    \PSc@mment{setdash Pattern=#1 #2}%
    \immediate\write\fwf@g{[\debutp@t] \finp@t\space\c@msetdash}\fi\fi}}
\ctr@ld@f\def\an@lysd@sh#1 #2:{\def\t@xt@{#2}\ifx\t@xt@\empty\def\finp@t{#1}\else%
    \edef\debutp@t{\debutp@t\space#1}\advance\s@mme\@ne\an@lysd@sh#2:\fi}
\ctr@ln@m\CUR@width
\ctr@ld@f\def\s@uvwidth#1{\edef#1{\CUR@width}}
\ctr@ld@f\def\D@FTwidth{0.4}     
\ctr@ld@f\def\Q@s@twidth#1{\ifGR@cri\edef\CUR@width{#1}\ifCUR@PS%
    \PSc@mment{setwidth Width=#1}\immediate\write\fwf@g{#1 \c@msetlinewidth}\fi\fi}
\ctr@ln@m\CUR@join
\ctr@ld@f\def\s@uvjoin#1{\edef#1{\CUR@join}}
\ctr@ld@f\def\D@FTjoin{miter}   
\ctr@ld@f\def\Q@s@tjoin#1{\ifGR@cri\edef\CUR@join{#1}\ifCUR@PS\expandafter\Pssetj@in#1:\fi\fi}
\ctr@ld@f\def\Pssetj@in#1#2:{\PSc@mment{setjoin join=#1}%
    \if#1r\def\t@xt@{1}\else\if#1b\def\t@xt@{2}\else\def\t@xt@{0}\fi\fi%
    \immediate\write\fwf@g{\t@xt@\space\c@msetlinejoin}}
\ctr@ld@f\def\Pss@tspecifSt#1{\trtlis@rg{#1}{\Pss@tspecifSt@}}
\ctr@ld@f\def\Pss@tspecifSt@#1=#2|{\keln@mun#1|%
    \def\n@mref{c}\ifx\l@debut\n@mref\def\n@mref{#2}\ifx\n@mref\D@FTref\else%
     \s@uvcolor{\typ@color}\Pssetc@lor{#2}\fi\else
    \def\n@mref{d}\ifx\l@debut\n@mref\def\n@mref{#2}\ifx\n@mref\D@FTref\else%
     \s@uvdash{\typ@dash}\Q@s@tdash{#2}\fi\else
    \def\n@mref{j}\ifx\l@debut\n@mref\def\n@mref{#2}\ifx\n@mref\D@FTref\else%
     \s@uvjoin{\typ@join}\Q@s@tjoin{#2}\fi\else
    \def\n@mref{w}\ifx\l@debut\n@mref\def\n@mref{#2}\ifx\n@mref\D@FTref\else%
     \s@uvwidth{\typ@width}\Q@s@twidth{#2}\fi\else
    \W@rnmeskwd{Pss@tspecifSt}{#1}\fi\fi\fi\fi}
\ctr@ld@f\def\Psrest@reSt#1{\trtlis@rg{#1}{\Psrest@reSt@}}
\ctr@ld@f\def\Psrest@reSt@#1=#2|{\keln@mun#1|%
    \def\n@mref{c}\ifx\l@debut\n@mref\def\n@mref{#2}\ifx\n@mref\D@FTref\else%
     \Pssetc@lor{\typ@color}\fi\else
    \def\n@mref{d}\ifx\l@debut\n@mref\def\n@mref{#2}\ifx\n@mref\D@FTref\else%
     \Q@s@tdash{\typ@dash}\fi\else
    \def\n@mref{j}\ifx\l@debut\n@mref\def\n@mref{#2}\ifx\n@mref\D@FTref\else%
     \Q@s@tjoin{\typ@join}\fi\else
    \def\n@mref{w}\ifx\l@debut\n@mref\def\n@mref{#2}\ifx\n@mref\D@FTref\else%
     \Q@s@twidth{\typ@width}\fi\else
    \W@rnmeskwd{Psrest@reSt}{#1}\fi\fi\fi\fi}
\ctr@ld@f\def\Pssettrim@sh#1=#2|{\keln@mde#1|%
    \def\n@mref{co}\ifx\l@debut\n@mref\update@ttr\D@FTref\P@settmeshcolor{#2}\else
    \def\n@mref{da}\ifx\l@debut\n@mref\update@ttr\D@FTref\P@settmeshdash{#2}\else
    \def\n@mref{wi}\ifx\l@debut\n@mref\update@ttr\D@FTref\P@settmeshwidth{#2}\else
    \W@rnmesAttr{figset trimesh}{#1}\fi\fi\fi}
\ctr@ln@m\DDV@tmeshcolor
\ctr@ld@f\def\P@settmeshcolor#1{\edef\DDV@tmeshcolor{#1}}
\ctr@ln@m\DDV@tmeshdash
\ctr@ld@f\def\P@settmeshdash#1{\edef\DDV@tmeshdash{#1}}
\ctr@ln@m\DDV@tmeshwidth
\ctr@ld@f\def\P@settmeshwidth#1{\edef\DDV@tmeshwidth{#1}}
\ctr@ld@f\def\figdrawtrimesh#1[#2,#3,#4]{{\ifCUR@PS\ifGR@cri%
    \PSc@mment{trimesh Type=#1, Triangle=[#2,#3,#4]}%
    \s@uvc@ntr@l\et@tpstrimesh\ifnum#1>\@ne%
    \Pss@tspecifSt{color=\DDV@tmeshcolor,dash=\DDV@tmeshdash,width=\DDV@tmeshwidth}%
    \setc@ntr@l{2}%
    \Pstrimeshp@rt#1[#2,#3,#4]\Pstrimeshp@rt#1[#3,#4,#2]\Pstrimeshp@rt#1[#4,#2,#3]%
    \Psrest@reSt{color=\DDV@tmeshcolor,dash=\DDV@tmeshdash,width=\DDV@tmeshwidth}%
    \fi\figdrawline[#2,#3,#4,#2]%
    \PSc@mment{End trimesh}\resetc@ntr@l\et@tpstrimesh\fi\fi}}
\ctr@ld@f\def\Pstrimeshp@rt#1[#2,#3,#4]{{\l@mbd@un=\@ne\l@mbd@de=#1\loop\ifnum\l@mbd@de>\@ne%
    \advance\l@mbd@de\m@ne\figptbary-1:[#2,#3;\l@mbd@de,\l@mbd@un]%
    \figptbary-2:[#2,#4;\l@mbd@de,\l@mbd@un]\figdrawline[-1,-2]%
    \advance\l@mbd@un\@ne\repeat}}
\initpr@lim\initpss@ttings\initPDF@rDVI
\ctr@ln@w{newbox}\figBoxA
\ctr@ln@w{newbox}\figBoxB
\ctr@ln@w{newbox}\figBoxC
\catcode`\@=12

\newcommand{\beq}{\begin{equation}}
\newcommand{\eeq}{\end{equation}}

\newcommand{\sfI}{\mathsf{I}}
\newcommand{\bfx}{{\bf x}}
\newcommand{\bfy}{{\bf y}}

\newcommand{\comm}[1]{}
\newcommand\lbra{\left}
\newcommand\rbra{\right}

\newcommand\fqf{\mathfrak{h}_{\omega,\tt}}

\newcommand\argmin{{\rm argmin}\,}

\makeatletter
\renewcommand\theequation{\thesection.\arabic{equation}}
\@addtoreset{equation}{section}

\def\@listI{
    \leftmargin\leftmargini
    \parsep 1.5pt plus 1pt minus 1pt
    \topsep -1.5pt plus 1pt minus 1pt
    \itemsep \parsep}
\let\@listi\@listI
\makeatother

\usepackage{color}
\newcommand{\Bk}{\color{black}}
\newcommand{\Rd}{\color{red}}
\newcommand{\Gn}{\color{green}}
\newcommand{\Bl}{\color{blue}}
\def\to#1{\Bl {#1} \Bk}

\newcommand{\R}{\mathbb{R}}

\newcommand{\spn}{\mathrm{span}\,}

\newcommand{\fq}{\mathfrak{q}}
\newcommand{\frQ}{\mathfrak{Q}}

\newcommand{\Gui}{\mathsf{Gui}}
\newcommand{\Lay}{\mathsf{Lay}}
\newcommand{\Str}{\mathsf{Str}}
\newcommand{\Dom}{\mathsf{Dom}}
\newcommand{\Tri}{\mathsf{Tri}}
\newcommand{\Hst}{\mathsf{Hst}}

\newcommand{\far}{\mathsf{far}}
\newcommand{\near}{\mathsf{near}}
\newcommand{\trans}{\mathsf{trans}}
\newcommand{\cyl}{\mathsf{cyl}}

\newcommand{\Fqft}{\mathfrak{Q}_{\omega,\Omega_\tt}}

\newcommand{\ps}[2]{\left\langle#1,#2\right\rangle}
\newcommand{\vabs}[1]{\left|#1\right|}

\def\myotimes{\otimes}

\newcommand\sfT{\mathsf{T}}
\newcommand\sfR{\mathsf{R}}
\newcommand\frH{\mathfrak{H}}
\newcommand\frc{\mathfrak{c}}
\newcommand\frd{\mathfrak{d}}
\newcommand\frT{\mathfrak{T}}
\newcommand\frl{\mathfrak{l}}

\def\FrmV{\frH_{\alpha,\tt, V}}
\def\FrmV{\frH_{\alpha,\tt, V}}

\def\OpV{\sfH_{\alpha, \tt, V}}

\def\Frmsc{Q_{\omega}}
\def\Op{\sfH_{\omega,\tt}}
\def\Opsc{\sfH_{\omega}}

\def\OpL{\mathsf{L}_{\alpha,\cC}}
\def\OpLf{\mathsf{L}_{\alpha,\Gamma_\tt}}
\def\frmLf{Q_{\alpha,\Gamma_\tt}}
\def\frmLft{\wt{Q}_{\alpha,\Gamma_\tt}}
\def\frmLfr{{Q}_{\alpha,\Gamma}}
\def\frmLfrt{\wt{Q}_{\alpha,\Gamma}}

\def\frmLWz{Q_0}
\def\frmLWo{Q_1}
\def\frmLWj{Q_j}

\def\frmLz{Q_0'}

\def\frmLzz{Q_{00}'}
\def\frmLzo{Q_{01}'}
\def\frmLzj{Q_{0j}'}

\def\frmLzpr{Q_0'}

\def\frmLWot{Q_{\alpha, \Omega^{K}_\tt}}
\def\frmLWone{\wt{Q}_{\alpha,\Omega^{K}_\tt}}

\def\frmVg{\frT_{\alpha, \tt,\gamma, V}}
\def\opVg{\sfT_{\alpha,\tt, \gamma, V}}

\def\frmg{\frT_{\alpha, \tt, \gamma}}
\newcommand\sfG{\mathsf{G}}

\def\opg{\sfT_{\alpha, \tt,\gamma}}
\newcommand\bfe{{\bf e}}
\newcommand\bfA{{\bf A}}
\newcommand\bfB{{\bf B}}
\newcommand\Frm{Q_{\omega,\tt}}
\def\op{\sfT_{\alpha,\tt}}

\def\opID{\sfh_{\alpha,L}^{\rm D}}
\def\opIN{\sfh_{\alpha,L}^{\rm N}}
\def\frmID{\frq_{\alpha,L}^{\rm D}}
\def\frmIN{\frq_{\alpha,L}^{\rm N}}

\def\frmKSN{\frq_c^{\rm N}}
\def\frmKSD{\frq_c^{\rm D}}

\newcommand\N{\cN}

\def\frmR{\frT_{\tt,R}}
\def\opR{\sfT_{\tt,R}}

\def\frmRN{\frS_{\tt, R}}

\def\frmRI{\fra_{\tt, R}}
\def\frmRII{\frb_{\tt, R}}
\def\frmRIII{\frc_{\tt, R}}
\def\frmRIV{\frd_{\tt, R}}

\def\frmRk{\frS_{\aa,\tt,R}^{(k)}}

\def\frmtrans{\frt_{\aa,\tt, R}}
\def\frmlong{\frl_{\aa,\tt,R}}

\newcommand\medsum{{\textstyle\sum}}

\newcommand\Cd{{\cC_{d,\theta}}}

\newcommand\m{\mathfrak{m}}
\def\tt{\theta}
\def\aa{\alpha}
\def\lm{\lambda}
\def\sfh{\mathsf{h}}
\def\Gt{{\Gamma_\tt}}
\def\St{{\Sigma_\tt}}
\newcommand\G{\Gamma}
\def\arr{\rightarrow}
\def\frg{\mathfrak{g}}
\def\fri{\mathfrak{i}}

\def\restric#1#2{{#1} |_{#2}}

\newcommand\myemph[1]{\textbf{\emph{#1}}} 

\definecolor{darkred}{rgb}{0.5,0.1,0.1}
\newcommand{\vl}[1]{{\color{red} #1}}
\newcommand{\vlc}[1]{{\color{red} VL: #1}}
\newcommand{\tob}[1]{{\color{blue} #1}}

\newcommand\s{\sigma}
\newcommand\sd{\sigma_{\rm disc}}
\newcommand\sess{\sigma_{\rm ess}}

\newcommand\ii{{\mathsf{i}}}
\newcommand\p{\partial}

\newcommand\pp{\prime}
\newcommand\dpp{{\prime\prime}}
\newcommand\tpp{{\prime\prime\prime}}

\newcommand\sfH{\mathsf{H}}
\newcommand\sfF{\mathsf{F}}
\renewcommand\sfh{\mathsf{h}}
\newcommand\one{\mathbbm{1}}
\newcommand\Z{\mathsf{z}}
\newcommand\sfP{\mathsf{P}}
\newcommand\dd{{\mathsf{d}}}
\newcommand\uhr{\upharpoonright}

\newcommand\frf{{\mathfrak f}}
\newcommand\frj{{\mathfrak j}}
\newcommand\omg{\omega}

\newcommand\sfw{\mathsf{w}}
\newcommand\sfz{\mathsf{z}}

\newcommand\sfS{\mathsf{S}}
\newcommand\sfC{\mathsf{C}}
\newcommand\sfW{\mathsf{W}}
\newcommand\sfZ{\mathsf{Z}}
\newcommand\sfU{\mathsf{U}}
\newcommand\sfV{\mathsf{V}}

\newcommand*{\medcap}{\mathbin{\scalebox{1.5}{\ensuremath{\cap}}}}%
\newcommand*{\medcup}{\mathbin{\scalebox{1.5}{\ensuremath{\cup}}}}%
\newcommand*{\medoplus}{\mathbin{\scalebox{1.5}{\ensuremath{\oplus}}}}%

\def\radius{3cm}
\def\softness{0.4}

\definecolor{softred}{rgb}{1,\softness,\softness}
\definecolor{softgreen}{rgb}{\softness,1,\softness}
\definecolor{softblue}{rgb}{\softness,\softness,1}
\definecolor{softrg}{rgb}{1,1,\softness}
\definecolor{softrb}{rgb}{1,\softness,1}
\definecolor{softgb}{rgb}{\softness,1,1}

\newcounter{counter_a}
\newenvironment{myenum}{\begin{list}{{\rm(\roman{counter_a})}}%
{\usecounter{counter_a}
\setlength{\itemsep}{1.ex}\setlength{\topsep}{1.5ex}
\setlength{\leftmargin}{5ex}\setlength{\labelwidth}{5ex}}}{\end{list}}
\newcommand\ds{\displaystyle}
\usepackage[latin1]{inputenc}
\usepackage[T1]{fontenc}
\newcommand{\red}{\color{red}}

\newcommand{\eg}{{\it e.g.}\,}
\newcommand{\ie}{{\it i.e.}\,}
\newcommand{\cf}{{\it cf.}\,}

\numberwithin{figure}{section}
\numberwithin{equation}{section}
\theoremstyle{plain}
\newtheorem*{thm*}{Theorem}
\newtheorem{thm}{Theorem}[section]

\newtheorem{lem}[thm]{Lemma}
\newtheorem{prop}[thm]{Proposition}

\newtheorem{cor}[thm]{Corollary}

\newtheorem{dfn}[thm]{Definition}
\theoremstyle{remark}

\theoremstyle{plain}

\newcommand{\dsp}{\displaystyle}
\newcommand{\spec}{{\mathrm{spec}}}
\newcommand{\rmd}{\mathrm{d}}
\newcommand{\rmi}{\mathrm{i}}
\newcommand{\supp}{\mathrm{supp}\,}
\newcommand{\beu}{\begin{equation*}}
\newcommand{\eeu}{\end{equation*}}
\newcommand{\besu}{\begin{equation*}
\begin{aligned}}
\newcommand{\eesu}{\end{aligned}
\end{equation*}}
\newcommand{\bes}{\begin{equation}
\begin{aligned}}
\newcommand{\ees}{\end{aligned}
\end{equation}}

\newcommand\cA{\mathcal A}
\newcommand\cB{\mathcal B}
\newcommand\cD{\mathcal D}
\newcommand\cF{\mathcal F}
\newcommand\cG{\mathcal G}
\newcommand\cH{\mathcal H}
\newcommand\cK{\mathcal K}
\newcommand\cJ{\mathcal J}
\newcommand\cL{\mathcal L}
\newcommand\cM{\mathcal M}
\newcommand\cN{\mathcal N}
\newcommand\cP{\mathcal P}
\newcommand\cR{\mathcal R}
\newcommand\cS{\mathcal S}
\newcommand\CC{\mathbb C}
\newcommand\NN{\mathbb N}
\newcommand\RR{\mathbb R}
\newcommand\ZZ{\mathbb Z}
\newcommand\frA{\mathfrak A}
\newcommand\frB{\mathfrak B}
\newcommand\frS{\mathfrak S}
\newcommand\fra{\mathfrak a}
\newcommand\frq{\mathfrak q}
\newcommand\frs{\mathfrak s}
\newcommand\frp{\mathfrak p}
\newcommand\frh{\mathfrak h}
\newcommand\fraD{\mathfrak{a}_{\rm D}}
\newcommand\fraN{\mathfrak{a}}
\newcommand\dis{\displaystyle}
\newcommand\ov{\overline}
\newcommand\wt{\widetilde}
\newcommand\wh{\widehat}
\newcommand{\defeq}{\mathrel{\mathop:}=}
\newcommand{\eqdef}{=\mathrel{\mathop:}}
\newcommand{\defequ}{\mathrel{\mathop:}\hspace*{-0.72ex}&=}
\newcommand\vectn[2]{\begin{pmatrix} #1 \\ #2 \end{pmatrix}}
\newcommand\vect[2]{\begin{pmatrix} #1 \\[1ex] #2 \end{pmatrix}}
\newcommand\restr[1]{\!\bigm|_{#1}}
\newcommand\RE{\text{\rm Re}}

\newcommand\sign{{\rm sign\,}}

\newcommand\void[1]{}

\newcommand\eps{\varepsilon}
\newcommand\ran{{\rm ran\,}}
\newcommand\sigp{\sigma_{\rm p}}
\DeclareMathOperator\tr{tr}

\def\sA{{\mathfrak A}}   \def\sB{{\mathfrak B}}   \def\sC{{\mathfrak C}}
\def\sD{{\mathfrak D}}   \def\sE{{\mathfrak E}}   \def\sF{{\mathfrak F}}
\def\sG{{\mathfrak G}}   \def\sH{{\mathfrak H}}   \def\sI{{\mathfrak I}}
\def\sJ{{\mathfrak J}}   \def\sK{{\mathfrak K}}   \def\sL{{\mathfrak L}}
\def\sM{{\mathfrak M}}   \def\sN{{\mathfrak N}}   \def\sO{{\mathfrak O}}
\def\sP{{\mathfrak P}}   \def\sQ{{\mathfrak Q}}   \def\sR{{\mathfrak R}}
\def\sS{{\mathfrak S}}   \def\sT{{\mathfrak T}}   \def\sU{{\mathfrak U}}
\def\sV{{\mathfrak V}}   \def\sW{{\mathfrak W}}   \def\sX{{\mathfrak X}}
\def\sY{{\mathfrak Y}}   \def\sZ{{\mathfrak Z}}

\def\frb{{\mathfrak b}}
\def\frl{{\mathfrak l}}
\def\frs{{\mathfrak s}}
\def\frt{{\mathfrak t}}

\def\dA{{\mathbb A}}   \def\dB{{\mathbb B}}   \def\dC{{\mathbb C}}
\def\dD{{\mathbb D}}   \def\dE{{\mathbb E}}   \def\dF{{\mathbb F}}
\def\dG{{\mathbb G}}   \def\dH{{\mathbb H}}   \def\dI{{\mathbb I}}
\def\dJ{{\mathbb J}}   \def\dK{{\mathbb K}}   \def\dL{{\mathbb L}}
\def\dM{{\mathbb M}}   \def\dN{{\mathbb N}}   \def\dO{{\mathbb O}}
\def\dP{{\mathbb P}}   \def\dQ{{\mathbb Q}}   \def\dR{{\mathbb R}}
\def\dS{{\mathbb S}}   \def\dT{{\mathbb T}}   \def\dU{{\mathbb U}}
\def\dV{{\mathbb V}}   \def\dW{{\mathbb W}}   \def\dX{{\mathbb X}}
\def\dY{{\mathbb Y}}   \def\dZ{{\mathbb Z}}

\def\cA{{\mathcal A}}   \def\cB{{\mathcal B}}   \def\cC{{\mathcal C}}
\def\cD{{\mathcal D}}   \def\cE{{\mathcal E}}   \def\cF{{\mathcal F}}
\def\cG{{\mathcal G}}   \def\cH{{\mathcal H}}   \def\cI{{\mathcal I}}
\def\cJ{{\mathcal J}}   \def\cK{{\mathcal K}}   \def\cL{{\mathcal L}}
\def\cM{{\mathcal M}}   \def\cN{{\mathcal N}}   \def\cO{{\mathcal O}}
\def\cP{{\mathcal P}}   \def\cQ{{\mathcal Q}}   \def\cR{{\mathcal R}}
\def\cS{{\mathcal S}}   \def\cT{{\mathcal T}}   \def\cU{{\mathcal U}}
\def\cV{{\mathcal V}}   \def\cW{{\mathcal W}}   \def\cX{{\mathcal X}}
\def\cY{{\mathcal Y}}   \def\cZ{{\mathcal Z}}

\renewcommand{\div}{\mathrm{div}\,}
\newcommand{\grad}{\mathrm{grad}\,}
\newcommand{\Tr}{\mathrm{Tr}\,}

\newcommand{\dom}{\mathrm{dom}\,}
\newcommand{\mes}{\mathrm{mes}\,}

\newcommand\Lcyl{{L^2_{\mathsf{cyl}}(\Lay(\tt))}}

\def\fOp{\sfF_{\omega, \tt}}
\def\fFrm{\frf_{\omega, \tt}}
\newcommand{\Fqf}{\mathfrak{q}_{\omega, \tt}}
\newcommand{\Fhf}{\mathfrak{h}_{\omega, \tt}}

\newcommand{\frmq}{\mathfrak{q}_{\omega, \tt}}
\newcommand{\frmh}{\mathfrak{h}_{\omega, \tt}}
\newcommand{\frmf}{\mathfrak{f}_{\omega, \tt}}
\newcommand{\wtfrmf}{\wt{\mathfrak{f}}_{\omega, \tt}}
\newcommand{\whfrmf}{\wh{\mathfrak{f}}_{\omega, \tt}}

\usepackage[normalem]{ulem}
\definecolor{DarkGreen}{rgb}{0,0.5,0.1}
\newcommand{\txtD}{\textcolor{DarkGreen}}
\newcommand{\discuss}{\textcolor{red}}
\newcommand\soutD{\bgroup\markoverwith
{\textcolor{DarkGreen}{\rule[.5ex]{2pt}{1pt}}}\ULon}
\newcommand{\Hm}[1]{\leavevmode{\marginpar{\tiny%
$\hbox to 0mm{\hspace*{-0.5mm}$\leftarrow$\hss}%
\vcenter{\vrule depth 0.1mm height 0.1mm width \the\marginparwidth}%
\hbox to 0mm{\hss$\rightarrow$\hspace*{-0.5mm}}$\\\relax\raggedright
#1}}}
\newcommand{\noteD}[1]{\Hm{\textcolor{DarkGreen}{#1}}}
\title[Spectral transitions for Aharonov-Bohm Laplacians on conical layers]
{Spectral transitions for Aharonov-Bohm Laplacians\\ on conical layers}

\author{D.~Krej\v{c}i\v{r}\'ik} 
\address{Department of Theoretical Physics,
Nuclear Physics Institute, Czech Academy of Sciences, 250 68, \v{R}e\v{z} near Prague, Czech Republic}
\email{krejcirik@ujf.cas.cz}

\author{V.~Lotoreichik} 
\address{Department of Theoretical Physics,
Nuclear Physics Institute, Czech Academy of Sciences, 250 68, \v{R}e\v{z} near Prague, Czech Republic}
\email{lotoreichik@ujf.cas.cz}

\author{T.~Ourmi\`eres-Bonafos}
\address{BCAM - Basque Center for Applied Mathematics, Alameda de Mazarredo, 14 E48009 Bilbao, Basque Country -  Spain}
\email{tourmieres@bcamath.org}

\begin{document}



\begin{abstract}
We consider the Laplace operator in a tubular neighbourhood of a conical surface of revolution,
subject to an Aharonov-Bohm magnetic field supported on the axis of symmetry
and Dirichlet boundary conditions on the boundary of the domain. 
We show that there exists a critical total magnetic flux 
depending on the aperture of the conical surface
for which the system undergoes an abrupt spectral transition 
from infinitely many eigenvalues below the essential spectrum
to an empty discrete spectrum.
For the critical flux we establish a Hardy-type inequality.  
In the regime with infinite discrete spectrum 
we obtain sharp spectral asymptotics with refined estimate of 
the remainder and investigate the dependence of the eigenvalues	
on the aperture of the surface and the flux of the magnetic field.
\end{abstract}                                                            

\maketitle

\section{Introduction}

\subsection{Motivation and state of the art}
Various physical properties of quantum systems 
can be explained through a careful spectral analysis of 
the underlying Hamiltonian. 
In this paper we consider the Hamiltonian of a quantum particle
constrained to a tubular neighbourhood of a conical surface
by hard-wall boundary conditions
and subjected to an external Aharonov-Bohm magnetic field 
supported on the axis of symmetry.
It turns out that the system exhibits a \emph{spectral transition}:
depending on the geometric aperture of the conical surface,
there exists a critical total magnetic flux which suddenly switches
from infinitely many bound states to an empty discrete spectrum.

The choice of such a system requires some comments.
First, the existence of infinitely many bound states below the threshold of the essential spectrum is a common property shared by Laplacians on various conical structures. 
This was first found in~\cite{DEK01, CEK04}, 
revisited in~\cite{ET10}, and further analysed in~\cite{DOR15} 
for the Dirichlet Laplacian 
in the tubular neighbourhood of the conical surface.
In agreement with these pioneering works, 
in this paper we use the term \emph{layer}
to denote the tubular neighbourhood.
Later, the same effect was observed for other realisations of Laplacians on conical structures~\cite{BEL14, BDPR15, BR15, BPP16, LO16, P15}. 
Second, the motivation for combining Dirichlet Laplacians on conical layers
with magnetic fields has a clear physical importance in quantum mechanics~\cite{SST69}.
Informally speaking, magnetic fields act as ``repulsive'' interactions
whereas the specific geometry of the layer acts as an ``attractive'' interaction.
Therefore, one expects that if a magnetic field is not too strong
to change the essential spectrum but
strong enough to compensate the binding effect of the geometry, 
the number of eigenvalues can become finite or the discrete can even fully
disappear.  

Our main goal is to demonstrate this effect 
for an idealised situation of an infinitely thin and long solenoid
put along the axis of symmetry of the conical layer, 
which is conventionally realised by a singular Aharonov-Bohm-type magnetic potential. 
First of all, 
we prove that the essential spectrum is stable under the geometric 
and magnetic perturbations considered in this paper.
As the main result, we establish the occurrence of an abrupt 
spectral transition regarding the existence and number
of discrete eigenvalues. 
In the \emph{sub-critical} regime, 
when the magnetic field is weak,
we prove the existence of infinitely many bound states
below the essential spectrum and obtain a precise accumulation rate
of the eigenvalues with refined estimate of the remainder. 
The method of this proof is inspired by~\cite{DOR15}, see also \cite{LO16}.
In the case of the critical magnetic flux 
we obtain a global Hardy inequality which,
in particular, implies that there are no bound states in the sup-critical regime.

A similar phenomenon is observed in~\cite{NR16} where it is shown that a sufficiently strong Aharonov-Bohm point interaction can remove finitely many bound states in the model of a quantum waveguide laterally coupled through a window~\cite{ESTV96, P99}. There are also many other models where a sort of competition between binding
and repulsion caused by different mechanism occurs. For example, 
bending of a quantum waveguide acts as an attractive interaction 
\cite{DE95,CDFK05}
whereas twisting of it acts as a repulsive interaction 
\cite{EKK08,K08}.
Thus, bound states in such a  waveguide exist 
only if the bending is in a certain sense stronger than twisting.
It is also conjectured in~\cite[Sec. IX]{S00} (but not proven so far) 
that a similar effect can arise for atomic many-body Hamiltonians
at specific critical values of the nucleus charge. Here, both binding
and repulsive forces are played by Coulombic interactions.

\subsection{Aharonov-Bohm magnetic Dirichlet Laplacian on a conical layer}
\label{subsec:AB_field}
Given an angle $\theta\in(0,\pi/2)$, 
our configuration space is a $\pi/2$-tubular neighbourhood  
of a conical surface of opening angle~$2\theta$.
Such a domain will be denoted here by $\Lay(\tt)$
and called a \emph{conical layer}. 
Because of the rotational symmetry,
it is best described in cylindrical coordinates.

To this purpose, let $(x_1, x_2, x_3)$ be the Cartesian coordinates 
on the Euclidean space $\dR^3$ 
and $\dR^2_+$ be the positive half-plane $(0,+\infty)\times\dR$. 
We consider cylindrical coordinates 
$(r, z, \phi) \in \R_+^2 \times \dS^1$ defined {\it via} the following standard relations
\begin{equation}
	x_1 = r\cos\phi,\qquad x_2=r\sin\phi,\qquad x_3 =z.
	\label{eqn:co_cyl}
\end{equation}
For further use, we also introduce the axis of symmetry
$\Gamma :=\{(r,z,\phi)\in \R_+^2 \times \dS^1 \colon r = 0\}$.
We abbreviate by $(\bfe_r, \bfe_\phi, \bfe_z)$ the \emph{moving frame}
\[
	\bfe_r := (\cos\phi,\sin\phi,0),
	\qquad 
	\bfe_\phi := (-\sin\phi,\cos\phi,0),
	\qquad
	\bfe_z := (0,0,1),
\]
associated with the cylindrical coordinates $(r,z,\phi)$. 

To introduce the conical layer $\Lay(\theta)$ 
with half-opening angle $\theta\in(0,\pi/2)$, 
we first define its \emph{meridian domain} $\Gui(\tt)\subset\dR^2_+$
(see Figure~\ref{F:1}) by
\begin{equation}
	\Gui(\tt) = 
	\Big\{(r,z)\in\R_+^2\colon 
	- \frac\pi{\sin\tt} < z,\quad 
	\max(0,z\tan\tt) < r < z\tan\tt + \frac\pi{\cos\tt}\Big\}.
\end{equation}
Then the conical layer $\Lay(\tt)$
associated with $\Gui(\tt)$ is defined in 
cylindrical coordinates~\eqref{eqn:co_cyl} by
\begin{equation}
	\Lay(\tt) := \Gui(\tt)\times \dS^1.
\end{equation}
The layer $\Lay(\tt)$ can be seen as a sub-domain
of $\dR^3$ constructed {\it via} rotation
of the meridian domain $\Gui(\tt)$ around the axis $\Gamma$.

\begin{figure}[h!]
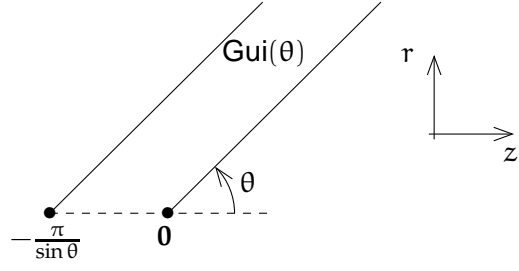

\figinit{0.35cm}
\figpt 0:(0,0)
\figpt 1:(-4.44288294,0)

\figpt 2:(-5,0)
\figpt 3:(4,0) 

\figpt 4:(8,8)
\figpt 5:(3.55711706,8)

\figpt 6:(2,0)
\figpt 7:(2,2)

\figpt 8:(2.2,1.2)
\figpt 9:(3.75,5)

\figpt 10:(10,2.8)
\figpt 11:(10,6)

\figpt 12:(9.8,3)
\figpt 13:(13,3)
\figdrawbegin{}
\figdrawline [0,4]
\figdrawline [1,5]
\figdrawarrow[10,11]
\figdrawarrow[12,13]
\figdrawarrowcircP 0;2.5[6,7]
\figset (dash=8)
\figdrawline [1,3]
\figdrawend

\figvisu{\figBoxA}{}{
\figwrites 13:{$z$} (0.5)
\figwritew 11:{$r$} (0.5)
\figwritee 8:{$\theta$} (0.5)
\figwriten 9:{$\Gui(\theta)$} (0.5)
\figset write(mark=$\bullet$)
\figwrites 1:{$-\frac{\pi}{\sin\theta}$} (0.5)
\figwrites 0:{$\mathbf{0}$} (0.5)
}

\centerline{\box\figBoxA}
\caption{The meridian domain $\Gui(\theta)$.}
\label{F:1}
\end{figure}

For later purposes we split the boundary $\p\Gui(\tt)$ of $\Gui(\tt)$ into
two parts defined as
\[
	\p_0 \Gui(\tt) := \lbra\{(0,z)\colon -\pi < z\sin\tt < 0\rbra\},
	\qquad
	\p_1 \Gui(\tt) := \p \Gui(\tt)\setminus\ov{\p_0 \Gui(\tt)}.
\]
The distance between the two connected components of $\p_1\Gui(\tt)$
is said to be the \emph{width} of the layer $\Lay(\tt)$.
We point out that the meridian domain is normalised so that the
width of $\Lay(\tt)$ equals $\pi$ for any value of $\tt$. This normalization
simplifies notations significantly 
and it also preserves all possible spectral features 
without loss of generality, because the problem with an arbitrary width 
is related to the present setting by a simple scaling.

In order to define the \emph{Aharonov-Bohm magnetic field (AB-field)} we are interested in,
we introduce a real-valued function $\omega\in L^2(\dS^1)$ and
the vector potential $\bfA_\omega\colon\R_+^2\times \dS^1 \arr  \R^3$ by
\begin{equation}\label{eq:A}
	{\bf A}_{\omega}(r,z,\phi) :=
	\frac{\omega(\phi)}{r}{\bfe_{\phi}}.
\end{equation}
This vector potential is naturally associated with the singular AB-field 
\begin{equation}\label{eq:field}
	\bfB_\omega 
	= 
	\nabla \times \bfA_\omega 
	=  
	2\pi\Phi_{\omega}\delta_\Gamma{\bf e}_z,
\end{equation} 
where $\delta_\Gamma$ is the $\delta$-distribution supported on $\Gamma$
and $\Phi_\omega$ is the \emph{magnetic flux}
\[
	\Phi_\omega 
	:=
	\frac{1}{2\pi}\int_{0}^{2\pi}\omega(\phi)\dd \phi.
\]
Note that to check identity~\eqref{eq:field} it suffices 
to compute $\nabla \times \bfA_\omega$ in the distributional 
sense~\cite[Chap. 3]{McL}.

We introduce the usual cylindrical $L^2$-spaces on
$\dR^3$ and on $\Lay(\tt)$
\[
	L_\cyl^2(\dR^3):= 
	L^2(\dR^2_+\times\dS^1;r\dd r \dd z \dd \phi),\qquad
	L_\cyl^2(\Lay(\tt)):= L^2(\Gui(\tt)\times\dS^1;r\dd r \dd z \dd \phi).
\]
For further use, we also introduce the cylindrical Sobolev space $H^1_\cyl(\Lay(\tt))$ defined as
\[
	H^1_\cyl\big(\Lay(\tt)\big) := 
	\bigg\{ u \in L_\cyl^2(\Lay(\tt)) : 
	\int_{\Lay(\tt)}
	\bigg(|\partial_ru|^2 + |\partial_zu|^2 + 
	\frac{|\partial_\phi u|^2}{r^2}\bigg) r\dd r \dd z \dd\phi 
	< +\infty\bigg\}.
\]
The space $H^1_\cyl(\Lay(\tt))$ is endowed with the norm $\|\cdot\|_{H^1_\cyl(\Lay(\tt))}$ defined, for all $u \in H^1_\cyl(\Lay(\tt))$, by
\[
	\|u\|_{H^1_\cyl(\Lay(\tt))}^2 = \|u\|_{L^2_\cyl(\Lay(\tt))}^2 + \int_{\Lay(\tt)}\bigg(|\partial_ru|^2 + |\partial_zu|^2 + \frac{|\partial_\phi u|^2}{r^2}\bigg) r\dd r \dd z \dd\phi.
\]
Now, we define the non-negative symmetric densely defined quadratic form
on the Hilbert space $L_\cyl^2(\Lay(\tt))$ by
\begin{equation}\label{eq:Q_omega_0}
	Q_{\omega,\tt,0}[u] := 
	\|(\ii\nabla - \bfA_\omega)u\|_\Lcyl^2,\qquad 
	\dom Q_{\omega,\tt,0} := 
	\cC_0^\infty(\Lay(\tt)).
\end{equation}
The quadratic form $Q_{\omega,\tt,0}$ is closable by~\cite[Thm. VI.1.27]{Kato}, because it can be written {\it via} integration by parts as 
\[
	Q_{\omega,\tt,0}[u] = \ps{\sfH_{\omega,\tt,0} u}{u}_\Lcyl
\]
where the operator $\sfH_{\omega,\tt,0} u := (\ii \nabla  - \bfA_\omega)^2u$ with $\dom\sfH_{\omega,\tt,0}:= \cC_0^\infty(\Lay(\tt))$ is non-negative, symmetric, and densely defined in $L^2(\Lay(\tt))$.
In the sequel, it is convenient to have a special notation for the closure of 
$Q_{\omega,\tt,0}$
\begin{equation}\label{eq:Q_omega}
	Q_{\omega,\tt} := \ov{Q_{\omega,\tt,0}}.
\end{equation}

Now we are in a position to introduce the main object of this paper.
\begin{dfn}
	The self-adjoint operator $\Op$ in $L_\cyl^2(\Lay(\tt))$ 
	associated with	the form $Q_{\omega,\tt}$ 
	\emph{via} the first representation theorem~\cite[Thm. VI.2.1]{Kato}
	is regarded as the \emph{Aharonov-Bohm magnetic Dirichlet Laplacian} on 
	the conical layer $\Lay(\tt)$.
\end{dfn}
The Hamiltonian $\Op$ can be seen as an idealization for a more
physically realistic self-adjoint Hamiltonian $\sfH_{\omega,\tt, W}$
associated with the closure of the quadratic form
\[
	u\in \cC^\infty_0\big(\dR^2_+\times\dS^1\big) 
	\mapsto 
	\|(\ii\nabla - {\bf A}_\omega)u\|^2_{L^2_{\mathsf{cyl}}(\dR^3)} 
	+ (W u,u)_{L^2_{\mathsf{cyl}}(\dR^3)}
\]
where the potential $W\colon\dR^2_+\times\dS^1\arr\dR$ is a piecewise constant function given by
\[
	W(r,z,\phi) = 
	\begin{cases} 
		0,& (r,z,\phi)\in \Lay(\tt),\\
		W_0,& (r,z,\phi)\notin \Lay(\tt).
	\end{cases}
\]
The strong resolvent convergence of $\sfH_{\omega,\tt,W}$ to $\Op$ in the limit
$W_0\arr +\infty$ follows from the monotone convergence for quadratic forms~\cite[\S VIII.7]{RS-I}.

Before going any further, we remark that 
$\Phi_{\omega} + k\in\dR$ with $k\in\dZ$ can
alternatively be seen as a constant real-function in $L^2(\dS^1)$ and that
\begin{equation}
	\bfA_{\Phi_\omega + k} - \bfA_{\omega} = \nabla V
	\quad
	\text{with}
	\quad
	V(\phi) := (\Phi_\omega + k)\phi - \int_0^{\phi}\omega(\xi)\dd\xi.
	\label{eqn:def_V_gauge}
\end{equation}
\emph{The gauge transform} is defined as
\beq
	\sfG_V\colon \Lcyl\arr \Lcyl,\qquad \sfG_V u := e^{\ii V} u.
\eeq
Clearly, the operator $\sfG_V$ is unitary. 
By Proposition~\ref{prop:unit_fq} proven in Appendix~\ref{app:A}
the operators $\Op$ and $\sfH_{\Phi_\omega+k,\tt}$ are unitarily equivalent 
{\it via} the transform $\sfG_V$.
Therefore, taking $k= -\argmin_{k\in\dZ}\{|k-\Phi_\omega|\}$ we can reduce the case of general $\omega\in L^2(\dS^1 ; \dR)$ to constant $\omega\in[-1/2,1/2]$. 
For symmetry reasons $\sfH_{\omega,\tt}$ is unitary equivalent to 
$\sfH_{-\omega,\tt}$ for any $\omega\in\dR$. Thus, the case of constant 
$\omega\in[-1/2,1/2]$ is further reduced to $\omega\in[0,1/2]$.

When $\omega=0$, we remark that the quadratic form $Q_{0,\tt,0}$ coincides with the quadratic form of a Dirichlet Laplacian in cylindrical coordinates. Moreover, we have
\[
	\overline{\cC_0^\infty(\Lay(\tt))}^{\|\cdot\|_{H^1_\cyl(\Lay(\tt))}} = \overline{\cC_0^\infty(\Lay_0(\tt))}^{\|\cdot\|_{H^1_\cyl(\Lay(\tt))}},
\]
where $\Lay_0(\tt) = \big(\Gui(\tt)\cup\partial_0\Gui(\tt)\big)\times \dS^1$. Consequently, the case $\omega = 0$ reduces to the one analysed in~\cite{DEK01, DOR15, ET10} and we exclude it from our considerations.
From now on, we assume that $\omega\in (0,1/2]$ is a constant,
without loss of generality.

For $\omega\in (0,1/2]$ 
the quadratic form $\Frm$ associated with $\Op$ simply reads
\[
	\Frm[u] = 
	\int_{\Lay(\tt)} 
	\bigg(|\p_r u|^2 + |\p_z u|^2 + \frac{|\ii\p_\phi u - \omega u|^2}{r^2}\bigg)
	r \dd r\dd z\dd \phi.
\]
Following the strategy of~\cite[\S 3.4.1]{K13}, we consider on the Hilbert space $L^2(\dS^1)$ the ordinary differential self-adjoint operator $\sfh_\omega$
\begin{equation}
	\sfh_\omega v := \ii v' - \omega v,\qquad 
	\dom\sfh_\omega := \big\{v\in H^1(\dS^1) \colon v(0) = v(2\pi)\big\}.
\end{equation}
The eigenvalues $\{m-\omega\}_{m\in\dZ}$ of $\sfh_\omega$ are associated with the orthonormal basis of $L^2(\dS^1)$ given by
\begin{equation}\label{eq:vm}
	v_m(\phi) = (2\pi)^{-1/2} e^{\ii m\phi},\qquad m\in\dZ.
\end{equation}
For any $m\in\dZ$ and $u\in \Lcyl$, we introduce the
projector 
%
\begin{equation}\label{eq:projector}
	(\pi^{[m]}u)(r,z) = \ps{u(r,z,\phi)}{v_m(\phi)}_{L^2(\dS^1)}.
\end{equation}
According to the approach of~\cite[\S XIII.16]{RS78}, 
see also \cite{DOR15, LO16}
for related considerations, we can decompose $\Op$, with respect to this basis, as
\begin{equation}\label{eq:orthogonal_decomposition}
	\Op \cong \bigoplus_{m\in\dZ}\fOp^{[m]},
\end{equation}
where the symbol $\cong$ stands for the unitary equivalence relation
and, for all $m\in\dZ$, the operators $\fOp^{[m]}$ 
acting on $L^2(\Gui(\tt);r\dd r \dd z)$ are the \emph{fibers} of~$\Op$.
They are associated through the first
representation theorem with the closed, densely defined, symmetric non-negative
quadratic forms
\begin{equation}\label{eq:forms}
	\frmf^{[m]}[u] 
	:= 
	\int_{\Gui(\tt)} 
	\bigg(
	|\p_r u|^2 + |\p_z u|^2 + \frac{(m-\omega)^2}{r^2}|u|^2\bigg)r\dd r\dd z,
	\quad
	\dom \fFrm^{[m]} :=
	\pi^{[m]}\big(\dom \Frm\big).
\end{equation}
The domain of the operator $\fOp^{[m]}$ can
be deduced from the form $\fFrm^{[m]}$
in the standard way {\it via} the first representation theorem.

Finally, we introduce the unitary operator 
$\sfU \colon  L^2(\Gui(\tt);r\dd r\dd z) \arr L^2(\Gui(\tt))$,
$\sfU u := \sqrt{r}u$. This unitary operator allows to 
transform the quadratic forms $\fFrm^{[m]}$
into other ones expressed in a flat metric. Indeed, the quadratic form 
$\fFrm^{[m]}$ is
unitarily equivalent {\it via} $\sfU$ to the form on the Hilbert space
$L^2(\Gui(\tt))$ defined as
\begin{equation}
	\frmq^{[m]}[u]  := 
	\int_{\Gui(\tt)} 
	\Big(|\p_r u|^2 + |\p_z u|^2 + 
	\frac{(m-\omega)^2-1/4}{r^2}|u|^2\Big) \dd r\dd z,
	\quad
	\dom\frmq^{[m]} := \sfU(\dom \fFrm^{[m]}).
	\label{eqn:fq_flat}
\end{equation}
In fact, one can prove that $\cC_0^\infty(\Gui(\tt))$ is a form core for $\frmq^{[m]}$ and that its form domain satisfies
\begin{equation}\label{eq:domain_incl}
	\dom\frmq^{[m]} =  H_0^1(\Gui(\tt)).
\end{equation}
We refer to Appendix~\ref{app:B} for a justification of~\eqref{eq:domain_incl} and we would like to emphasise that \eqref{eq:domain_incl} does not hold for $\omega=0$ 
but we excluded this case from our considerations.

It will be handy in what follows to drop the superscript $[0]$ for $m = 0$
and to set
\begin{equation}\label{eq:no-supscript}
	\fOp := \fOp^{[0]},\qquad \frmf := \frmf^{[0]},\qquad 
	\frmq := \frmq^{[0]}. 
\end{equation}


\subsection{Main results}
We introduce a few notation before stating the main results of this paper.
The set of positive integers is denoted by $\mathbb{N} := \{1,2,\dots\}$ and the set of natural integers is denoted by 
$\mathbb{N}_0 := \mathbb{N}\cup\{0\}$. Let $\sfT$ be a semi-bounded self-adjoint operator 
associated with the quadratic form $\frt$. 
We denote by $\sess(\sfT)$ and $\sd(\sfT)$
the essential and the discrete spectrum of $\sfT$, respectively.
By $\s(\sfT)$, we denote the spectrum of $\sfT$ 
(\ie $\s(\sfT) = \sess(\sfT)\cup\sd(\sfT)$).

 Let $\frt_1$ and $\frt_1$ be two quadratic forms of domains $\dom(\frt_1)$ and $\dom(\frt_2)$, respectively. We say that we have the form ordering $\frt_1 \prec \frt_2$ if
\[
	\dom(\frt_2)\subset\dom(\frt_1)\quad\text{and}\quad \frt_1[u]\leq\frt_2[u],\ \text{for all } u\in \dom(\frt_2).
\]

We set $E_{\rm ess}(\sfT) := \inf\sess(\sfT)$ and, for $k\in\mathbb{N}$, $E_k(\sfT)$ denotes the $k$-th Rayleigh quotient of $\sfT$, defined as
\[
	E_k(\sfT) = \sup_{u_1,\dots,u_{k-1}\in\dom\frt}\inf_{
	\begin{smallmatrix}
	u\in\mathsf{span}(u_1,\dots,u_{k-1})^\perp\\
	u\in\dom\frt\setminus\{0\}
	\end{smallmatrix}}\frac{\frt[u]}{\|u\|^2}.
\]
From the min-max principle (see \eg \cite[Chap. XIII]{RS78}), we know that if $E_k(\sfT)\in(-\infty, E_{\rm ess}(\sfT))$, the $k$-th Rayleigh quotient is a discrete eigenvalue of finite multiplicity. Especially, we have the following description of the discrete spectrum
below $E_{\rm ess}(\sfT)$
\[
	\sd(\sfT)\cap(-\infty,E_{\rm ess}(\sfT))
	 = 
	 \big\{ E_k(\sfT) : k\in\dN, E_k(\sfT) < E_{\rm ess}(\sfT)\big\}.
\]
Consequently, if $E_k(\sfT)\in\sd(\sfT)$, it is the $k$-th eigenvalue with multiplicity taken into account. We define the counting function of $\sfT$ as
\[
	\N_E(\sfT) := \#\big\{k\in\dN\colon E_k(\sfT) < E\big\}, 
	\qquad E \le E_{\rm ess}(\sfT).
\]	
When working with the quadratic form $\frt$, 
we use the notations $\sess(\frt)$, $\sd(\frt)$, $\s(\frt)$,
$E_{\rm ess}(\frt)$, $E_k(\frt)$ and $\N_E(\frt)$  instead.

Our first result gives the description of the essential spectrum of $\Op$.
%
\begin{thm} 
	Let $\tt\in(0,\pi/2)$ and $\omega\in (0,1/2]$. There holds,
	\[
		\sess(\Op) = [1,+\infty).
	\]
\label{th:esspec}
\end{thm}
%
The minimum at $1$ of the essential spectrum is a consequence of the normalisation of the width of $\Lay(\tt)$ to $\pi$. 
The method of the proof of Theorem~\ref{th:esspec} relies on a construction of singular sequences as well as on form decomposition techniques. A similar approach is used \eg in~\cite{CEK04, DEK01, ET10} for Dirichlet conical layers without magnetic fields 
and in~\cite{BEL14} for Schr\"odinger operators with $\delta$-interactions supported on conical surfaces. In this paper we simplify the argument by constructing singular sequences in the generalized sense~\cite{KL14} on the level of quadratic forms.

Now we state a proposition that gives
a lower bound on the spectra of the fibers $\fOp^{[m]}$ with $m\ne 0$.
%
\begin{prop} 
	Let $\tt\in(0,\pi/2)$ and $\omega\in (0,1/2]$. There holds
	\[
		\inf \sigma(\fOp^{[m]})\geq 1,\quad \forall m \neq 0.
	\]
\label{prop:redaxy}
\end{prop}
%
Relying on this proposition and on Theorem~\ref{th:esspec}, we see that the investigation of the discrete spectrum of $\Op$ reduces to
the axisymmetric fiber $\fOp$ 
of decomposition~\eqref{eq:orthogonal_decomposition}.
When there is no magnetic field ($\omega = 0$) 
this result can be found in~\cite[Prop. 3.1]{ET10}. An analogous 
statement holds also for $\delta$-interactions supported on conical
surfaces~\cite[Prop. 2.5]{LO16}.

Now, we formulate a result on the ordering between Rayleigh quotients.
\begin{prop}
	Let $0 < \tt_1 \le \tt_2 < \pi/2$, $\omega_1 \in (0,1/2]$, 
	and $\omega_2\in[\cos\tt_2(\cos\tt_1)^{-1} \omega_1,1/2]$. 
	Then
	\[
		E_k(\sfF_{\omega_1,\tt_1}) \leq E_k(\sfF_{\omega_2,\tt_2}).
	\]
	holds for all $k\in\dN$.
\label{prop:ord_rayl}
\end{prop}
If the Rayleigh quotients in Proposition~\ref{prop:ord_rayl} are indeed eigenvalues, we get immediately an ordering of the eigenvalues for different apertures $\theta$ and values of $\omega$. In particular, if $\omega_1 = \omega_2$, we obtain that the Rayleigh quotients are non-decreasing functions of the aperture $\theta$. 
The latter property is reminiscent of analogous results for broken
waveguides~\cite[Prop. 3.1]{DLR12} and for Dirichlet
conical layers without magnetic fields~\cite[Prop.~1.2]{DOR15}. 
A similar claim also holds for $\delta$-interactions supported on broken lines~\cite[Prop.~5.12]{EN03} and on conical surfaces~\cite[Prop.~1.3]{LO16}. 
The new aspect of Proposition \ref{prop:ord_rayl}
is that we obtain a monotonicity result with respect to two parameters. 
Proposition~\ref{prop:ord_rayl}
implies that the eigenvalues are non-decreasing if we
weaken the magnetic field and compensate by making the aperture of the conical layer
smaller and {\it vice versa}.

The next theorem is the first main result of this paper.
\begin{thm} Let $\tt\in(0,\pi/2)$ and $\omega\in (0,1/2]$. 
	The following statements hold.
	\begin{myenum}
	\item
	For $\cos\tt \leq 2\omega$, $\# \sd (\fOp) = 0$.
	\item 
	For $\cos\tt > 2\omega$, $\# \sd (\fOp) 
	= \infty$ and
	\[
		\N_{1-E}(\fOp) 
		= 
		\frac{\sqrt{\cos^2\tt - 4\omega^2}}{4\pi\sin\tt}|\ln E| + \cO(1),
		\qquad E\arr 0+.
	\]
	\end{myenum}
	\label{thm:struc_sdisc}
\end{thm}
For a fixed $\tt \in (0,\pi/2)$, Theorem~\ref{thm:struc_sdisc} yields the
existence of a critical flux
\begin{equation}\label{eq:omega_cr}
	\omega_{\rm cr} = \omega_{\rm cr}(\tt) := \frac{\cos\tt}{2}
\end{equation}
at which the number of eigenvalues undergoes an abrupt transition from infinity to zero. This is, to our knowledge, the first example of a geometrically non-trivial model that exhibits such a behaviour. In comparison, in the special case $\omega = 0$, this phenomenon arises at $\tt = \pi/2$ which is geometrically simple because the domain $\Lay(\pi/2)$ can be seen in the Cartesian coordinates as the layer between two parallel planes at distance 
$\pi$.

The spectral asymptotics proven in Theorem~\ref{thm:struc_sdisc}\,(ii) 
is reminiscent of~\cite[Thm.~1.4]{DOR15}. However, it can be seen
that the magnetic field enters the coefficient in front of the main term. 
As a slight improvement upon \cite[Thm.~1.4]{DOR15},
in Theorem~\ref{thm:struc_sdisc} we explicitly state
that the remainder in this asymptotics is just $\cO(1)$.
The main new feature in Theorem~\ref{thm:struc_sdisc},
compared to the previous publications on the subject, is
the absence of discrete spectrum $\fOp$ for strong magnetic fields
stated in Theorem~\ref{thm:struc_sdisc}\,(i).
This result is achieved by proving 
a Hardy-type inequality for the quadratic form 
$\frq_\tt := \frq_{\omega_{\rm cr},\tt}$.
This inequality is the second main result of this paper. It is also of independent interest in view of potential applications 
in the context of the associated heat semigroup, 
\emph{cf.}~\cite{K13,CK14}.
\begin{thm}[Hardy-type inequality]\label{thm:Hardy}
	Let $\tt\in(0,\pi/2)$.
	There exists $c > 0$ such that
	\begin{equation}\label{eq:Hardy}
		\frq_{\tt}[u] - \|u\|^2_{L^2(\Gui(\tt))}
		\ge 
		c
		\int_{\Gui(\tt)} 
		\frac{(r\cos\tt - z\sin\tt)^3}
		{1 + \frac{r^2}{\sin^2\tt} 
		\ln^2\big(\frac{r}{\cos\tt}\frac{2}{r\cos\tt-z\sin\tt}\big)}
		|u|^2 \dd r \dd z
	\end{equation}
	holds for any $u\in \cC^\infty_0(\Gui(\tt))$.
\end{thm}
Finally, we point out that Theorem~\ref{thm:Hardy}
implies that for any $V \in \cC^\infty_0(\Lay(\tt))$ 
\begin{equation}\label{eq:non-crit}
	\#\sd({\sfH_{\omega_{\rm cr},\tt}} - \mu V ) = 0
\end{equation}
holds for all sufficiently small $\mu > 0$.
This observation can be extended to some potentials 
$V\in C^\infty_0(\ov{\Lay(\tt)})$,
but we can not derive~\eqref{eq:non-crit} for any 
$V\in C^\infty_0(\ov{\Lay(\tt)})$ from Theorem~\ref{thm:Hardy},
because the weight on the right-hand side of~\eqref{eq:Hardy}
vanishes on the part of $\p\Gui(\tt)$ satisfying $r = z\tan\tt$.  
It is an open question whether a global Hardy inequality with weight
non-vanishing on the whole $\p\Gui(\tt)$ can be proven.

%
%
%
%
\subsection{Structure of the paper} 
In Section~\ref{sec:esspec} we prove Theorem~\ref{th:esspec} about the structure of the essential spectrum.
In Section~\ref{sec:discspec} we
reduce the analysis of the discrete spectrum of $\Op$ to
the discrete spectrum of its axisymmetric fiber, 
prove Proposition~\ref{prop:ord_rayl} 
about inequalities between the Rayleigh quotients,
and Theorem~\ref{thm:struc_sdisc}\,(ii) 
on infiniteness of the discrete spectrum and its spectral asymptotics.
Theorem~\ref{thm:struc_sdisc}\,(i) on absence of discrete spectrum and 
Theorem~\ref{thm:Hardy} on a Hardy-type inequality are proven in 
Section~\ref{sec:Hardy}.  Some technical arguments are gathered into 
Appendices~\ref{app:A} and~\ref{app:B}.



%
\section{Essential spectrum}\label{sec:esspec}
In this section we prove Theorem~\ref{th:esspec} on the structure of the
essential spectrum of $\Op$. Observe that for any $m \ne 0$ the form ordering 
$\frmf \prec \frmf^{[m]}$ follows directly from~\eqref{eq:forms}. 
Hence, according to decomposition~\eqref{eq:orthogonal_decomposition}, to prove
Theorem~\ref{th:esspec} it suffices only to verify $\sess(\frmf) = [1,+\infty)$
which is equivalent to checking that $\sess(\frmq) = [1,+\infty)$. 

To simplify the argument we reformulate the problem in another set of coordinates performing the rotation
\begin{equation}\label{eq:rotation}
	s = z\cos\tt + r\sin\tt,\qquad t=-z\sin\tt+r\cos\tt,
\end{equation}
that transforms the meridian domain $\Gui(\tt)$ 
into the half-strip with corner $\Omega_\tt$ 
(see Figure~\ref{fig:figureomega})  defined by
\begin{equation}
	\Omega_\tt = 
	\big\{(s,t)\in\R \times (0,\pi) \colon 
	s > -t\cot\tt\big\}.
	\label{eqn:defome}
\end{equation}
In the sequel of this subsection, $\ps{\cdot}{\cdot}$ and $\|\cdot\|$ denote the inner product and the norm on $L^2(\Omega_\tt)$, respectively.

\begin{figure}[b]
\figinit{0.5cm}
\figpt 0:(0,0)
\figpt 1:(-5,0)
\figpt 2:(10,0)

\figpt 3:(-5,3.14)
\figpt 4:(10,3.14)

\figpt 5:(4,1.5)

\figpt 10:(15,0)
\figpt 11:(15,3.2)

\figpt 12:(14.8,0)
\figpt 13:(18,0)
\figdrawbegin{}
\figdrawline [0,2]
\figdrawline [3,4]
\figdrawline [0,3]
\figdrawarrow[10,11]
\figdrawarrow[12,13]
\figset (dash=8)
\figdrawline [0,1]
\figdrawline [1,3]
\figdrawend

\figvisu{\figBoxA}{}{
\figwrites 13:{$s$} (0.5)
\figwritew 11:{$t$} (0.5)
\figwritee 5:{$\Omega_\theta$}(0.5)
\figset write(mark=$\bullet$)
\figwrites 1:{$-\pi\cot\theta$} (0.5)
\figwrites 0:{$\mathbf{0}$} (0.5)
\figwriten 3:{$(-\pi\cot\theta,\pi)$} (0.5)
}

\centerline{\box\figBoxA}
\caption{The domain $\Omega_\tt$.}
\label{fig:figureomega}
\end{figure}\noindent

Rotation~\eqref{eq:rotation} naturally defines
a unitary operator 
\beq\label{eq:Utt}
	\sfU_\tt \colon  L^2(\Omega_\tt)\arr L^2(\Gui(\tt)),
	\qquad
	(\sfU_\tt u)(r,z) := u(z\cos\tt + r\sin\tt,-z\sin\tt+r\cos\tt),
\eeq
and induces a new quadratic form
\begin{subequations}
\begin{align}
	& \frmh[u] 
	\!:= \!
	\frmq[\sfU_\tt u] 
	=
	\int_{\Omega_\tt}
	\Big(|\p_s u|^2 + |\p_t u|^2 -
	\frac{\gamma |u|^2}{(s + t\cot\tt)^2}\Big)
	\dd s\dd t,\quad \dom \frmh  := H_0^1(\Omega_\tt),\label{eqn:frQ_flat}\\
	&\quad\text{where}\quad
	\gamma  = \gamma(\omega,\tt) := \frac{1/4 - \omega^2}{\sin^2\tt}.
	\label{eq:gamma}
\end{align}
\end{subequations}
Since the form $\frmh$  is unitarily equivalent to $\frmq$, 
proving Theorem~\ref{th:esspec} is equivalent to showing that
$\sess(\frmh) = [1,+\infty)$. We split this verification
into checking the two inclusions.

\subsection{The inclusion \texorpdfstring{\boldmath{$\sess(\Fhf) \supset [1,+\infty)$}}{inclusion0}}
We verify this inclusion by constructing singular sequences for 
$\frmh$ in the generalized sense~\cite[App. A]{KL14} for every point of the interval $[1,+\infty)$. 
Let us start by fixing a function 
$\chi\in C^\infty_0(1,2)$ such that $\|\chi\|_{L^2(1,2)} = 1$.
For all $p\in\dR_+$, we define the functions $u_{n,p}\colon \Omega_\tt\arr\dC$, 
$n\in\dN$, as
\begin{equation}\label{eq:unp}
	u_{n,p}(s,t) 
	:= 
	\lbra(\frac{1}{\sqrt{n}}\chi\lbra(\frac{s}{n}\rbra)\exp(\ii p s)\rbra)\lbra(\sqrt{\frac{2}{\pi}}\sin(t)\rbra).
\end{equation}
According to~\eqref{eq:domain_incl} it is not difficult to check that 
$u_{n,p} \in  \dom \Fhf$.
It is also convenient to introduce the associated functions $v_{n,p}, w_{n,p}\colon \Omega_\tt\arr\dC$, $n\in\dN$, as
\[
\begin{split}
	v_{n,p}(s,t) 
	& := 
	\lbra( \frac{1}{n^{3/2}}\chi^\pp\lbra(\frac{s}{n}\rbra)\exp(\ii p s)\rbra)
	\lbra(\sqrt{\frac{2}{\pi}}\sin(t)\rbra),\\
	w_{n,p}(s,t) 
	& := 
	\lbra(\frac{1}{\sqrt{n}}\chi\lbra(\frac{s}{n}\rbra)\exp(\ii p s)\rbra)\lbra(\sqrt{\frac{2}{\pi}}\cos(t)\rbra).
\end{split}
\]
First, we get
\begin{align}
	\|u_{n,p}\|^2
	& =
	\frac{2}{\pi}\int_0^\pi
	\int_n^{2n}\frac{1}{n}\lbra|\chi\lbra(\frac{s}{n}\rbra)\rbra|^2\sin^2(t) \dd s \dd t \label{eq:unp_norm}
	= 1.\\
	\|v_{n,p}\|^2 & = 
	\frac{2}{\pi}\frac{1}{n^2}
	\int_0^\pi
	\int_n^{2n}\frac{1}{n}\lbra|\chi^\pp\lbra(\frac{s}{n}\rbra)\rbra|^2\sin^2(t) \dd s \dd t
	= \frac{1}{n^2}
	\|\chi^\pp\|_{L^2(1,2)}^2 \arr 0,\qquad n\arr \infty. 
	\label{eq:vnp_norm}
\end{align}
Further, we compute the partial derivatives $\p_s u_{n,p}$ and $\p_t u_{n,p}$
\begin{equation}\label{eq:deriv}
	(\p_s u_{n,p})(s,t) 
	= \ii p u_{n,p}(s,t) + v_{n,p}(s,t),\qquad
	(\p_t u_{n,p})(s,t) = w_{n,p}(s,t),
\end{equation}
and we define an auxiliary potential by
\begin{equation}\label{eq:V_aux}
	V_{\omega,\tt}(s,t) := \frac{\gamma(\omega,\tt)}{(s + t\cot\tt)^2}.
\end{equation}
For any $\phi \in \dom\Fhf$ we have
\[
\begin{split}
	I_{n,p}(\phi) 
	:= \ & 
	\frh_{\omega,\tt}[\phi,u_{n,p}] - (1+p^2) \langle \phi,  u_{n,p}\rangle\\
	= \ &
	\langle\nabla \phi,  \nabla u_{n,p}\rangle - 
	\langle V_{\omega,\tt} \phi, u_{n,p}\rangle - (1+p^2) \langle \phi,  u_{n,p}\rangle\\
	= \ &
	\underbrace{\bigg(\lbra\langle\nabla \phi, 
	\begin{pmatrix}
		\ii p u_{n,p}\\
		w_{n,p}
	\end{pmatrix}\rbra\rangle - (1+p^2) \langle \phi,  u_{n,p}\rangle\bigg)}_{=: J_{n,p}(\phi)}
	+
	\underbrace{\bigg(\lbra\langle\nabla \phi,  			
	\begin{pmatrix}
		v_{n,p}\\
		0
	\end{pmatrix} \rbra\rangle - \langle V_{\omega,\tt} \phi, u_{n,p}\rangle
	\bigg)}_{=:K_{n,p}(\phi)}.
\end{split}	
\]
Integrating by parts and applying the Cauchy-Schwarz inequality we obtain
\[
\begin{split}
	|J_{n,p}(\phi)| & 
	= 
	\lbra|-\langle\phi, \ii p \p_s u_{n,p} + \p_t w_{n,p}\rangle - (1+p^2) \langle \phi,  u_{n,p}\rangle\rbra|\\
	& = \lbra| 
	\langle\phi,  p^2 u_{n,p} +  u_{n,p}\rangle- (1+p^2) \langle \phi,  u_{n,p}\rangle -  \langle\phi, \ii p v_{n,p}\rangle\rbra|
	= |\langle \phi, \ii p v_{n,p}\rangle|
  	 \le  p\|\phi\|\|v_{n,p}\|.\\
\end{split}
\]
Applying the Cauchy-Schwarz inequality once again 
and using~\eqref{eq:unp_norm}
and~\eqref{eq:deriv} we get
\[
	|K_{n,p}(\phi)| 
	\le  \|\phi\|\sup_{(s,t)\in (n,2n)\times (0,\pi)} |V_{\omega,\tt}(s,t)|
	+ \|\nabla \phi\|\lbra\| v_{n,p}\rbra\| 
	= \frac{\gamma}{n^2}\|\phi\| + \|\nabla \phi\|\|v_{n,p}\|.
\]
Let us define the norm $\|\cdot\|_{+1}$ as
\[	
	\|\phi\|_{+1}^2 := \frmh[\phi] + \|\phi\|^2,
	\qquad \phi \in \dom\frmh.
\]
Clearly, $\|\phi\|_{+1} \ge \|\phi\|$ and, moreover, for sufficiently small $\eps >0$, it holds 
\[
	\omega(\eps) := \sqrt{1/4 + (1-\eps)^{-1}(\omega^2-1/4)}\in (0,1/2]
\]
and
\[
	\|\phi\|_{+1}^2 
	\ge 
	\frmh[\phi] 
	=
	\eps\|\nabla\phi\|^2 + 
	(1-\eps)\frh_{\omega(\eps),\tt}[\phi] \ge \eps\|\nabla\phi\|^2,
\]
where we used $\frh_{\omega(\eps),\tt}[\phi] \ge 0$ in the last step.
Therefore, for any $\phi\in \dom\frh_{\omega,\tt}$, $\phi \ne 0$, we have
by~\eqref{eq:vnp_norm}
\begin{equation}\label{eq:Inp_frac}
	\frac{|I_{n,p}(\phi)|}{\|\phi\|_{+1}}
	\le
	\frac{|J_{n,p}(\phi)|}{\|\phi\|_{+1}} + \frac{|K_{n,p}(\phi)|}{\|\phi\|_{+1}}
	\le 
	p \|v_{n,p}\| + \frac{\gamma}{n^2} + \eps^{-1/2}\|v_{n,p}\|\arr 0,\qquad n\arr \infty.
\end{equation}
Here, the upper bound on $\frac{|I_{n,p}(\phi)|}{\|\phi\|_{+1}}$
is given by a vanishing sequence which is independent of $\phi$.

Since the supports of $u_{2^k,p}$ and $u_{2^l,p}$ with $k\ne l$
are disjoint, the sequence $\{u_{2^k,p}\}$ converges weakly to zero. 
Hence,~\eqref{eq:unp_norm} and~\eqref{eq:Inp_frac} imply that
$\{u_{2^k,p}\}$ is a singular sequence in the generalized sense~\cite[App. A]{KL14} 
for $\Fhf$ corresponding to the point $1+p^2$. 
Therefore, by~\cite[Thm. 5]{KL14}, $1+p^2\in \sess(\Fhf)$ for all $p\in\dR_+$ 
and it follows that $[1,+\infty)\subset \sess(\Fhf)$.

\subsection{The inclusion \texorpdfstring{\boldmath{$\sess(\frmh) \subset [1,+\infty)$}}{inclusion}}
We check this inclusion using the form decomposition method. For $n\in\dN$ we define two subsets of $\Omega_\tt$
\begin{equation}\label{eq:subdomains}
	\Omega_{n}^+  := \{(s,t)\in\Omega_\tt\colon s < n\}, 
	\qquad
	\Omega_{n}^- := \{(s,t)\in\Omega_\tt\colon s > n\},
\end{equation}
as shown in Figure~\ref{fig:figureformdec}.
For the sake of simplicity we do not indicate dependence of 
$\Omega_n^+$ on $\tt$. We also introduce
\[
	\Lambda_n := \{(s,t)\in\Omega_\tt\colon s = n\}.
\]
For $u \in L^2(\Omega_\tt)$
we set $u^\pm := u|_{\Omega_n^\pm}$. Further, we introduce the
Sobolev-type spaces 
\begin{equation}\label{eq:Sobolev}
	H_{0,\rm N}^1(\Omega_{n}^{\pm}) :=
	\big\{u\in H^1(\Omega_{n}^{\pm})\colon 
	u|_{\p\Omega_{n}^\pm \setminus \Lambda_n} = 0\big\}
\end{equation}
and consider the following quadratic forms
\begin{equation}\label{eq:forms_pm}
	\frh_{\omega,\tt,n}^\pm[u]  := 
	\int_{\Omega_{n}^\pm}\Big( |\p_s u^\pm|^2 + |\p_t u^\pm|^2 - 
	V_{\omega,\tt}|u^\pm|^2\Big)
	\dd s \dd t,\qquad
	\dom \frh_{\omega,\tt, n}^\pm 
	 := H^1_{0,\rm N}(\Omega_{n}^\pm),
\end{equation}
where $V_{\omega,\tt}$ is as in~\eqref{eq:V_aux}.
\begin{figure}[b]
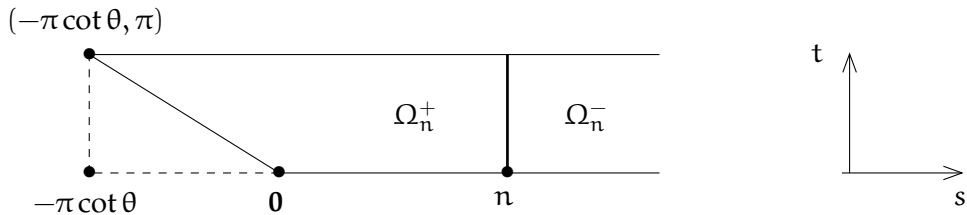

\figinit{0.5cm}
\figpt 0:(0,0)
\figpt 1:(-5,0)
\figpt 2:(10,0)

\figpt 3:(-5,3.14)
\figpt 4:(10,3.14)

\figpt 5:(2.5,1.5)

\figpt 10:(15,0)
\figpt 11:(15,3.2)

\figpt 12:(14.8,0)
\figpt 13:(18,0)

\figpt 14:(6,0)
\figpt 15:(6,3.14)
\figpt 16:(7,1.5)
\figdrawbegin{}
\figdrawline [0,2]
\figdrawline [3,4]
\figdrawline [0,3]
\figdrawarrow[10,11]
\figdrawarrow[12,13]
\figset (dash=8)
\figdrawline [0,1]
\figdrawline [1,3]
\figset(width=1)
\figset (dash=1)
\figdrawline [14,15]
\figdrawend

\figvisu{\figBoxA}{}{
\figwrites 13:{$s$} (0.5)
\figwritew 11:{$t$} (0.5)
\figwritee 5:{$\Omega_n^{+}$}(0.5)
\figwritee 16:{$\Omega_n^{-}$}(0.5)
\figset write(mark=$\bullet$)
\figwrites 1:{$-\pi\cot\theta$} (0.5)
\figwrites 0:{$\mathbf{0}$} (0.5)
\figwriten 3:{$(-\pi\cot\theta,\pi)$} (0.5)
\figwrites 14:{$n$} (0.5)
}

\centerline{\box\figBoxA}
\caption{The domain $\Omega_\tt$ and the subdomains $\Omega_n^\pm$.}
\label{fig:figureformdec}
\end{figure}\noindent
One can verify that  the form $\frh_{\omega,\tt, n}^\pm$
is closed, densely defined, symmetric and semibounded from below in $L^2(\Omega_n^\pm)$.


Due to the compact embedding of $H_{0,\rm N}^1(\Omega_{n}^+)$
into $L^2(\Omega_{n}^+)$ the spectrum
of $\frh^+_{\omega,\tt, n}$ is purely discrete. 
The spectrum of $\frh^-_{\omega,\tt, n}$ can be estimated from below as follows
\begin{equation}\label{eq:bnd}
	\inf\s(\frh^-_{\omega,\tt,n}) 
	\ge 1 - \sup_{(s,t)\in \Omega_{n}^-} V_{\omega,\tt}(s,t) 
	= 1 -  \frac{\gamma}{n^2}.
\end{equation}
The discreteness of the spectrum for $\frh_{\omega,\tt,n}^+$ and the estimate~\eqref{eq:bnd} imply that
\[
	\inf\sess(\frh_{\omega,\tt,n}^+\oplus \frh_{\omega,\tt,n}^-) 
	\ge 1 - \frac{\gamma}{n^2}.
\]
Notice that the ordering 
$\frh_{\omega,\tt,n}^+ \oplus \frh_{\omega,\tt,n}^- \prec \frmh$ holds. 
Hence, by the min-max principle we have
\[
	\inf\sess(\frmh) 
	\ge \inf\sess(\frh_{\omega,\tt,n}^+\oplus \frh_{\omega,\tt,n}^-) \ge 
	1 - \frac{\gamma}{n^2},
\]
and passing to the limit $n\arr\infty$ we get $\inf\sess(\frmh) \ge 1$.

\section{Discrete spectrum}
\label{sec:discspec}
The aim of this section is to discuss properties of the discrete spectrum of~$\Op$,
which has the physical meaning of quantum bound states.
In subsection \ref{subsec:red_axi} we reduce the study of the discrete spectrum of $\Op$ to its axisymmetric fiber $\fOp$ introduced in \eqref{eq:no-supscript}. Then, in subsection \ref{subsec:rayl_ine}, we prove Proposition \ref{prop:ord_rayl} about the ordering of the Rayleigh quotients. Finally, in subsection \ref{sec:countfunc}, we are interested in the asymptotics of the counting function in the regime $\omega \in (0,\omega_{\rm cr}(\tt))$ and we give a proof of Theorem~\ref{thm:struc_sdisc}\,(ii).

\subsection{Reduction to the axisymmetric operator}
\label{subsec:red_axi}

The goal of this subsection is to prove Proposition~\ref{prop:redaxy}.
In the proof we use the strategy developed in~\cite{DOR15, ET10} for Dirichlet conical layers without magnetic fields.

Consider the quadratic forms in the flat metric $\frmq^{[m]}$ 
given in~\eqref{eqn:fq_flat}. For all $m \neq 0$ and $\omega\in (0,1/2]$,
we have $(m-\omega)^2 \geq 1/4$. Consequently, for any $u\in H_0^1(\Gui(\tt))$, 
we get
\begin{equation}
	\frmq^{[m]}[u] \geq \|\nabla u\|_{L^2(\Gui(\tt))}^2.
\label{eqn:redHst}
\end{equation}
Any function $u\in H_0^1(\Gui(\tt))$ can be extended by zero to the strip 
\[
	\Str(\tt) := 
	\Big\{(r,z)\in\R^2 \colon z \tan \tt < r < z\tan\tt + \frac{\pi}{\cos \tt}\Big\},
\]
defining a function $u_0 \in H^1_0(\Str(\tt))$.
Hence, inequality~\eqref{eqn:redHst} can be re-written as
\[
	\frmq^{[m]}[u] \geq \|\nabla u_0\|_{L^2(\Str(\tt))}^2.
\]
The right-hand side of the last inequality is the quadratic form of
the two-dimensional Dirichlet Laplacian 
in a strip of width $\pi$. The spectrum of this operator is only essential and equals $[1,+\infty)$. 
Hence, by the min-max principle we get
\[
	\frmq^{[m]}[u] \geq \|u_0\|_{L^2(\Str(\tt))}^2 = \|u\|_{L^2(\Gui(\tt))}^2.
\]
Finally, applying the min-max principle to the quadratic form $\frmq^{[m]}$ we 
obtain
\[
	\inf \s(\frmq^{[m]}) \geq 1.
\]
This achieves the proof of Proposition~\ref{prop:redaxy}.

\subsection{Rayleigh quotients inequalities} 
\label{subsec:rayl_ine}

The aim of this subsection is to prove Proposition~\ref{prop:ord_rayl}.
This proof follows the same strategy as the proof 
of a related statement about broken waveguides developed in~\cite[\S 3]{DLR12}.

It will be more convenient to work with the quadratic form $\frmf$
in the non-flat metric.
Let the domain $\Omega_\tt$ be defined as in~\eqref{eqn:defome} through rotation~\eqref{eq:rotation}. 
This rotation induces a unitary operator $\sfR_\tt \colon L^2(\Gui(\tt); r\dd r\dd z) \arr L^2(\Omega_\tt; (s\sin\tt + t\cos\tt)\dd s\dd t)$. For $u\in\dom \frmf$, we set $\wt u(s,t) = u (r,z)$ and 
obtain the identity $\frmf[u] = \wtfrmf[\wt{u}]$ with the new quadratic form
\[
\begin{split}
	\wtfrmf[\wt{u}] 
	& := 
	\int_{\Omega_\tt}
	\Big(|\p_s \wt{u}|^2 + |\p_t \wt{u}|^2 + 
	\frac{\omega^2|\wt{u}|^2}{(s\sin\tt + t\cos\tt)^2}\Big)
	(s\sin\tt + t\cos\tt)\dd s\dd t,\\
	\dom  \wtfrmf  & := 
	\sfR_\tt(\dom\frmf),
\end{split}	
\]
which is unitarily equivalent to $\frmf$.
Now, in order to get rid of the dependence on $\tt$ of the integration domain $\Omega_\tt$, 
we perform the change of variables $(s,t) \mapsto (\hat{s},\hat{t}) = (s\tan\tt,t)$ 
that transforms the domain $\Omega_\tt$ into $\Omega := \Omega_{\pi/4}$. 
Setting $\wh{u}(\wh{s},\wh{t}) = \wt{u}(s,t)$ we get for the Rayleigh quotients
\[
\begin{split}
	\frac{\frmf[u]}{\|u\|^2_{L^2(\Gui(\tt);r\dd r \dd z)}} 
	= \ & 
	\frac{\int_{\Omega}
	\big(\tan^2\tt |\p_{\hat{s}} \wh{u}|^2 + |\p_{\hat{t}}\wh{u}|^2 
	+ \omega^2\cos^{-2}\tt(\hat{s} + \hat{t})^{-2}|\wh{u}|^2\big)(\hat{s}+\hat{t})
	\cos\tt\cot\tt\dd\hat{s}\dd\hat{t}}{
	\int_{\Omega}
	|\wh{u}|^2(\hat{s} + \hat{t})\cos\tt\cot\tt\dd\hat{s}\dd\hat{t}}\\
	= \ &
	\frac{\int_{\Omega}
	\big(\tan^2\tt|\p_{\hat{s}} \wh{u}|^2 + |\p_{\hat{t}}\wh{u}|^2 
	+ \omega^2\cos^{-2}\tt(\hat{s} + \hat{t})^{-2}|\wh{u}|^2\big)
	(\hat{s}+\hat{t})
	\dd\hat{s}\dd\hat{t}}{\int_{\Omega}|\wh{u}|^2(\wh{s} + \wh{t})
	\dd\hat{s}\dd\hat{t}} \\
	:= \ &
	\frac{\whfrmf[\wh{u}]}{\int_{\Omega}|\wh{u}|^2(\hat{s} + \hat{t})
	\dd\hat{s}\dd\hat{t}}
\end{split}
\]
The domain of the quadratic form $\whfrmf$ does not depend on $\tt$. 
However, we transferred the dependence on $\tt$ into the expression of 
$\whfrmf[\wh u]$. Now, let $0< \tt_1\le \tt_2<\pi/2$, $\omega_1\in (0,1/2]$ and 
$\omega_2\in[\cos\tt_2(\cos\tt_1)^{-1}\omega_1,1/2]$. Then we get
\begin{equation}
	\wh{\frf}_{\omega_2,\tt_2}[\wh{u}] - \wh{\frf}_{\omega_1,\tt_1}[\wh{u}] 
	= 
	\int_{\Omega} 
	\bigg[
	(\tan^2\tt_2 - \tan^2\tt_1)|\p_{\hat{s}}\wh{u}|^2
	+ 	
	\Big(\frac{\omega_2^2}{\cos^2\tt_2} - \frac{\omega_1^2}{\cos^2\tt_1}\Big)
	\frac{|\wh{u}|^2}{(\hat{s} + \hat{t})^2}\bigg](\wh{s} + \wh{t})
	\dd\hat{s}\dd\hat{t}.	\label{eqn:raylquo}
\end{equation}
Since the tangent is an increasing function, the first term on the right hand side is 
non-negative. As $\omega_2$ is chosen, the second term is also non-negative. 
Therefore, for any $k\in\dN$, the min-max principle
and~\eqref{eqn:raylquo} yield
$E_k(\wh{\frf}_{\omega_1,\tt_1}) \leq E_k(\wh{\frf}_{\omega_2,\tt_2})$
which is equivalent to 
\[
	E_k(\sfF_{\omega_1,\tt_1}) \leq E_k(\sfF_{\omega_2,\tt_2}).
\]
This achieves the proof of Proposition~\ref{prop:ord_rayl}.

%
\subsection{Asymptotics of the counting function}
\label{sec:countfunc}
This subsection is devoted to the proof of Theorem~\ref{thm:struc_sdisc}\,(ii). 
All along this subsection, $\tt\in(0,\pi/2)$ and 
$\omega\in (0,\omega_{\rm cr}(\tt))$ with $\omega_{\rm cr}(\tt)
= (1/2)\cos\tt$ as in~\eqref{eq:omega_cr}.
The proof follows the same steps as in~\cite[\S3]{DOR15}. 
However, in presence of a magnetic field the proof simplifies because
instead of working with the form $\frmf$ introduced in~\eqref{eq:forms}
we can work with the unitarily equivalent
quadratic form $\frmh$ defined in~\eqref{eqn:frQ_flat}. 
In particular, we avoid using IMS localization formula. 

The main idea is to reduce the problem to the known spectral asymptotics of 
one-dimensional operators. To this aim, first, we recall the result of~\cite{KS88}, later extended in~\cite{HM08}. 
Further, let $\gamma > 0$ be fixed. We are interested in the spectral properties of the self-adjoint 
operators acting on $L^2(1,+\infty)$ associated
with the closed, densely defined symmetric and semi-bounded quadratic form,
\[
	\frq_\gamma^{\rm N}[f] 
	:= 
	\int_1^{\infty}|f^\pp(x)|^2 - \frac{\gamma|f(x)|^2}{x^2}\dd x,\qquad 
	\dom \frq_\gamma^{\rm N} := H^1(1,+\infty),
\]
and with its restriction
\[
	\frq_\gamma^{\rm D}[f] 
	:= 
	\frq_\gamma^{\rm N}[f],
	\qquad \dom \frq_\gamma^{\rm D} := H_0^1(1,+\infty).
\]
It is well known that 
$\sess(\frq_\gamma^{\rm D}) = \sess(\frq_\gamma^{\rm N}) = [0,+\infty)$
and it can be shown by a proper choice of test functions that 
$\#\sd(\frq_\gamma^{\rm D}) = \#\sd(\frq_\gamma^{\rm N}) = \infty$
for all $\gamma > 1/4$.

\begin{thm}[{\cite[Thm. 1]{KS88}, \cite[Thm. 1]{HM08}}]
	As $E\arr 0+$ the counting functions 
	of $\frq_\gamma^{\rm D}$ and $\frq_\gamma^{\rm N}$
	with $\gamma > 1/4$	satisfy	
	\[
		\cN_{-E}(\frq_\gamma^{\rm D})  = 
		\frac{1}{2\pi}\sqrt{\gamma-\frac14}|\ln E| + \cO(1),
		\qquad
		\cN_{-E}(\frq_\gamma^{\rm N}) = 
		\frac{1}{2\pi}\sqrt{\gamma-\frac14}|\ln E| + \cO(1).
	\]
\label{thm:KS}
\end{thm}
In Proposition~\ref{prop:lbcount} we establish a lower bound for 
$\N_{1-E}(\frmh)$ while an upper bound is obtained in 
Proposition~\ref{prop:uppbound}. Together with Theorem~\ref{thm:KS} 
these bounds yield Theorem~\ref{thm:struc_sdisc}~(ii).

Let the sub-domains $\Omega^\pm := \Omega_1^\pm$ (for $n= 1$) of $\Omega_\tt$ 
be as in~\eqref{eq:subdomains} and the Sobolev-type spaces
$H_{0,{\rm N}}^1(\Omega^\pm)$ be as in~\eqref{eq:Sobolev}.
Let also the quadratic forms 
$\frh_{\omega,\tt}^\pm := \frh_{\omega,\tt,1}^\pm$ be as 
in~\eqref{eq:forms_pm}. Define the restriction $\frh_{\omega,\tt,\rm D}^-$
of $\frh_{\omega,\tt}^-$ by
\[
	\frh_{\omega,\tt,\rm D}^-[u] := \frh_{\omega,\tt}^-[u],  \qquad
	\dom \frh_{\omega,\tt,\rm D}^- := H^1_0(\Omega^-).  
\]
%
%

To obtain a lower bound, we use a Dirichlet bracketing technique. 
\begin{prop} 
	Let $\tt\in(0,\pi/2)$, $\omega \in (0,\omega_{\rm cr}(\tt))$ be fixed
	and let $\gamma = \gamma(\omega,\tt)$ be as in~\eqref{eq:gamma}.
	For any $E > 0$ set $\wh{E}=(1+\pi\cot\tt)^2E$. Then
	the bound
	\[
		\N_{-\wh{E}}(\frq_\gamma^{\rm D})\leq \N_{1-E}(\frmh),
	\]
	holds for all $E > 0$.
\label{prop:lbcount}
\end{prop}

\begin{proof}
	Any $u\in H_0^1(\Omega^-)$ can be extended 
	by zero in $\Omega_\tt$, defining $u_0 \in H_0^1(\Omega_\tt)$
	such that $\frh^-_{\omega,\tt,{\rm D}}[u] = \frmh[u_0]$. 
	Then, the min-max principle yields
	\begin{equation}\label{eqn:lbcount}
		\N_{1-E}(\frh^-_{\omega,\tt, \rm D})\leq\N_{1-E}(\frmh).
	\end{equation}
	Now, we bound $(s+t\cot\tt)^2$ 
	from above by $(s + \pi\cot\tt)^2$ and for any $u\in H_0^1(\Omega^-)$, 
	we get
	\begin{equation}
		\frh^-_{\omega,\tt,\rm D}[u] 
		\leq 
		\int_{\Omega^-}|\p_s u|^2 + |\p_t u|^2 - 
		\frac{\gamma|u|^2}{(s+\pi\cot\tt)^2}\dd s\dd t.
		\label{eqn:lbbound}
	\end{equation}
	Further, we introduce the quadratic forms 
	for one-dimensional operators
	\begin{align*}
			\wh\frq_\gamma^{\rm D}[f] & := \displaystyle
			\int_1^{+\infty}|f^\pp(x)|^2 - 
			\frac{\gamma |f(x)|^2}{(x+\pi\cot\tt)^2}\dd x,
			&\dom\wh\frq_\gamma^{\rm D}  & :=  H^1_0(1,+\infty),\\
			\frq^{\rm D}_{(0,\pi)}[f] & :=  \displaystyle\int_0^\pi |f'(x)|^2 \dd x, 
			&\dom\frq^{\rm D}_{(0,\pi)} & := H^1_0(0,\pi).
	\end{align*}
	

	%
	The right hand side of~\eqref{eqn:lbbound} 
	can be represented as 
	$\wh\frq^{\rm D}_\gamma \otimes \fri_2 + \fri_1 \otimes 
	\frq^{\rm D}_{(0,\pi)}$ with respect to the tensor product
	decomposition $L^2(\Omega^-) = L^2(1,+\infty) \otimes L^2(0,\pi)$
	where $\fri_1$, $\fri_2$ are the quadratic forms of the identity operators
	on $ L^2(1,+\infty)$ and on $L^2(0,\pi)$, respectively.
	The eigenvalues of $\frq^{\rm D}_{(0,\pi)}$ are given
	by $\{k^2\}_{k\in\dN}$ and hence
	\begin{equation}
		\N_{-E}(\wh\frq_\gamma^{\rm D})
		\leq\N_{1-E}(\frh^-_{\omega,\tt, \rm D}).
		\label{eqn:lb2}
	\end{equation}
	Finally, we perform the change of variables 
	$y = (1+\pi\cot\tt)^{-1}(x + \pi\cot\tt)$. For all functions 
	$f\in\dom \wh\frq_\gamma^{\rm D}$, 
	we denote $g(y) = f(x)$. We get
	\[
	\frac{\wh\frq_\gamma^{\rm D}[f]}{\int_1^{+\infty}|f(x)|^2\dd x} 
	= 
	(1+\pi\cot\tt)^{-2}
	\frac{\frq_\gamma^{\rm D}[g]}
	{\int_1^{+\infty}|g(y)|^2\dd y}.
	\]
	Finally, using~\eqref{eqn:lbcount},~\eqref{eqn:lb2} 
	and the min-max principle, 
	we get the desired bound on $\N_{1-E}(\frmh)$.
\end{proof}

To obtain an upper bound, we use a Neumann bracketing technique. 
\begin{prop} 
	Let $\tt\in(0,\pi/2)$ and $\omega\in(0,\omega_{\rm cr}(\tt))$ be fixed
	and let $\gamma = \gamma(\omega,\tt)$ be as in~\eqref{eq:gamma}. 
	Then there exists a constant $C = C(\omega,\tt) > 0$
	such that 
	\[
		\N_{1-E}(\frmh) 
		\leq C + \N_{-E}(\frq_\gamma^{\rm N}) 
	\]
	holds for all $E > 0$.
	\label{prop:uppbound}
\end{prop}
To prove Proposition~\ref{prop:uppbound} we will need the following two lemmas whose proofs are postponed until the end of the subsection.
\begin{lem} 
	Let $\tt\in(0,\pi/2)$ and $\omega\in (0,\omega_{\rm cr}(\tt))$
	be fixed. Then there exists a constant $C = C(\omega,\tt)>0$ such that
	\[
		\N_{1-E}(\frmh^+) \leq C
	\]
	holds for all $E>0$.
	\label{lem:upbound}
\end{lem}
\begin{lem} Let $\tt\in(0,\pi/2)$ and $\omega\in (0,\omega_{\rm cr}(\tt))$
	be fixed and let $\gamma =\gamma(\omega,\tt)$ be as in~\eqref{eq:gamma}. 
	Then
	\[
		\N_{1-E}(\frmh^-) 
		\leq \N_{-E}(\frq_\gamma^{\rm N})
	\]
	holds for all $E > 0$.
	\label{lem:upboundstrip}
\end{lem}
\begin{proof}[Proof of Proposition \ref{prop:uppbound}] 
Note that we have the following form ordering
\[
	\frh_{\omega,\tt}^+\oplus\frh_{\omega,\tt}^-\prec \frmh
\]
and the min-max principle gives
\begin{equation}
	\N_{1-E}(\frmh) \leq \N_{1-E}(\frmh^+) + \N_{1-E}(\frmh^-).
	\label{eqn:neumbrack}
\end{equation}
The statement follows directly combining~\eqref{eqn:neumbrack}, 
Lemma~\ref{lem:upbound} and Lemma~\ref{lem:upboundstrip}.
\end{proof}

We conclude this part by the proofs of Lemmas~\ref{lem:upbound} 
and~\ref{lem:upboundstrip}.

\begin{proof}[Proof of Lemma~\ref{lem:upbound}] 
Recall that the space $H_{0,{\rm N}}^1(\Omega^+)$ is compactly embedded into $L^2(\Omega^+)$. Consequently, $\s(\frmh^+)$ is purely discrete and consists 
of a non-decreasing sequence of eigenvalues of finite multiplicity 
that goes to $+\infty$. In particular, 
there exists a constant $C = C(\omega,\tt)>0$ such that
\[
	\N_{1-E}(\frmh^+) 
	\leq \N_1(\frmh^+) \leq C.\qedhere
\]
\end{proof}

\begin{proof}[Proof of Lemma~\ref{lem:upboundstrip}]
	In $\Omega^-$, we can bound $(s+t\cot\tt)^2$ from below by $s^2$. 
	For any $u \in \dom\frmh^-$, 
	we get
	\[
		\int_{\Omega^-} |\p_s u|^2 + |\p_t u|^2 -
		\frac{\gamma|u|^2}{s^2}\dd s\dd t 
		\leq \frmh^-[u].
	\]
	The left-hand side can be seen as the tensor product
	$\frq_\gamma^{\rm N} \otimes \fri_2 + \fri_1 \otimes \frq^{\rm D}_{(0,\pi)}$
	with respect to the decomposition $L^2(\Omega^-) = L^2(1,+\infty)\otimes L^2(0,\pi)$
	where the form $\frq^{\rm D}_{(0,\pi)}$ is defined in the proof of 
	Proposition~\ref{prop:lbcount}.
	Since the eigenvalues of $\frq^{\rm D}_{(0,\pi)}$ 
	are given by $\{k^2\}_{k\in\dN}$, we deduce that	
	\[
		\N_{1-E}(\frmh^-) \leq \N_{-E}(\frq_\gamma^{\rm N}).\qedhere
	\]
	%
\end{proof}

\begin{proof}[Proof of Theorem \ref{thm:struc_sdisc}~(ii)]
	Combining Proposition \ref{prop:lbcount} and Proposition \ref{prop:uppbound}, 
	for any $E>0$ we get 
	\begin{equation}
		\N_{-(1+\pi\cot\tt)^2 E}(\frq_{\gamma}^{\rm D}) 
		\leq \N_{1-E}(\frmh)
		\leq C + \N_{-E}(\frq_\gamma^{\rm N}).
		\label{eqn:enccount}
	\end{equation}
	For the lower and upper bounds
	on $\N_{1-E}(\frmh)$ given in~\eqref{eqn:enccount}, 
	Theorem~\ref{thm:KS} implies that as $E\arr 0+$ holds
	\[
	\begin{split}
		C + \N_{-E}(\frq_\gamma^{\rm N}) 
		& = \frac{1}{2\pi}\sqrt{\gamma - \frac{1}{4}}|\ln E|
		+ \cO(1),\\
		\N_{-(1+\pi\cot\tt)^2 E}(\frq_\gamma^{\rm D}) 
		& = \frac{1}{2\pi}\sqrt{\gamma - \frac{1}{4}}|\ln((1+\pi\cot\tt)^2 E)|
		+ \cO(1)  
		= \frac{1}{2\pi}\sqrt{\gamma - \frac{1}{4}}|\ln E|
		+ \cO(1).  
	\end{split}	
	\]
	Hence, Theorem~\ref{thm:struc_sdisc}\,(ii) follows 
	from the identity
	\[
		\sqrt{\gamma - \frac14} = 
		\frac{\sqrt{\cos^2\tt - 4\omega^2}}{2\sin\tt}.
		\qedhere
	\]
\end{proof}

\section{A Hardy-type inequality}
\label{sec:Hardy}
The aim of this section is to prove Theorem~\ref{thm:Hardy}. 
Instead of working with the quadratic form $\frmq$
which is used in the formulation of Theorem~\ref{thm:Hardy}
it is more convenient to work with $\frmh$ defined in~\eqref{eqn:frQ_flat}. 
We go back to the form $\frmq$ only in the end of this section.
Recall that we denote by $\langle\cdot,\cdot\rangle$ and $\|\cdot\|$,
respectively, the inner product and the norm in $L^2(\Omega_\tt)$.

In this section we are only interested in the critical case 
$\omega = \omega_{\rm cr}(\tt) = (1/2)\cos\tt$
for which $\gamma(\omega_{\rm cr}(\tt),\tt) = 1/4$ holds
where $\gamma(\omega,\tt)$ is defined in~\eqref{eq:gamma}.
To make the notations more handy
we define $\frh_\tt := \frh_{\omega_{\rm cr},\tt}$. 
For further use, for any $(s,t)\in\Omega_\tt$, we introduce
\[
	\rho := \rho(s,t) = 
	s+t\cot\tt,\qquad \rho_0:=\rho_0(t) = \frac{1}2t\cot\tt.
\]
With this notation the domain $\Omega_\tt$ can be represented as
\[
	\Omega_\tt = \big\{(s,t)\in\dR\times (0,\pi)\colon s > -2\rho_0(t)\big\}
\]
and the quadratic form $\frh_\tt$ can be written as
\[
	\frh_{\tt}[u] = 
	\int_{\Omega_\tt}|\p_s u|^2 + |\p_t u|^2 - \frac{|u|^2}{4\rho^2} 
	\dd s\dd t,\qquad \dom \frh_\tt = H^1_0(\Omega_\tt).
\]
%
%
The emptiness of the discrete spectrum stated in 
Theorem~\ref{thm:struc_sdisc}\,(i)
is an immediate consequence of Theorem~\ref{thm:Hardy} 
and of the min-max principle because for any $\omega\geq \omega_{\rm cr}$ 
the form ordering $\frh_{\tt} \prec \frmh$ holds.
%
Another consequence of Theorem~\ref{thm:Hardy} 
is the non-criticality of $\Op$ as stated in~\eqref{eq:non-crit}.
%

To prove Theorem~\ref{thm:Hardy}, we adapt the strategy developed 
in~\cite[\S 3]{CK14}. First, in subsection~\ref{subsec:local_Hardy} 
we prove a local Hardy-type inequality for the quadratic form 
$\frh_\tt$ taking advantage of the usual one-dimensional Hardy inequality. Second, in subsection~\ref{subsec:Hardy_lb}, we obtain a refined lower bound that allows us, in subsection~\ref{subsec:Hardy}, to prove 
Theorem~\ref{thm:Hardy}.
%
\subsection{A local Hardy inequality}\label{subsec:local_Hardy} 
Let us introduce the triangle $\cT_\tt$ (see Figure~\ref{fig:sub_Ttheta}),
which is a sub-domain of $\Omega_\tt$ defined as
\[
\begin{split}
	\cT_\tt := \ & \big\{(s,t)\in\Omega_\tt \colon s < -\rho_0(t)/2\big\}\\
= \ &\big\{(s,t)\in\dR \times (0,\pi) \colon -2\rho_0(t) < s <-\rho_0(t)/2\big\}.
\end{split}		
\]

\begin{figure}[h!!!]
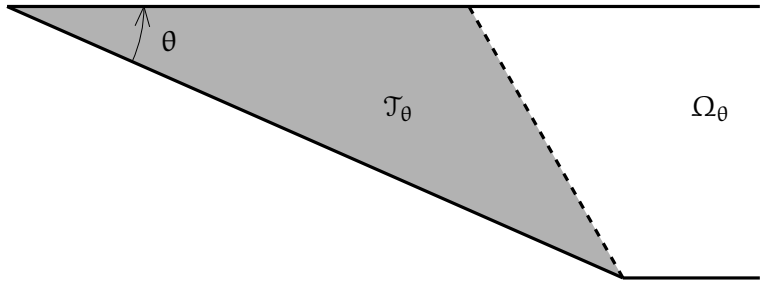

\figinit{0.9cm}
\figpt 0:(0,0)
\figpt 1:(2,0)
\figpt 2:(-9,4)
\figpt 4:(2,4)
\figpt 5:(-2.25,4)
\figpt 6:(-9,0)
\figpt 7:(-2,0)
\figpt 8:(-3.5,2.5)
\figpt 9:(1,2.5)
\figpt 10:(-6.75,3.5)

\figdrawbegin{}
\figset (fillmode=yes, color=0.7)
\figdrawline [0,5,2,0]
\figset (fillmode=no, color=0)
\figset (width=1.2)
\figdrawline [0,1]
\figdrawline [2,4]
\figdrawline [0,2]
\figset (width=1.2,dash=8)
\figdrawline [0,5]
\figset (width=0,dash=0)
\figdrawarrowcircP 2;2[0,4]
\figdrawend

\figvisu{\figBoxA}{}{
\figwritee 8:{$\mathcal{T}_\tt$} (0)
\figwritee 9:{$\Omega_\tt$} (0)
\figwritee 10:{$\tt$} (0)
}
\centerline{\box\figBoxA}
\caption{The domain $\Omega_\tt$ and the subdomain $\mathcal{T}_\tt$ (in grey).}
\label{fig:sub_Ttheta}
\end{figure}%

\noindent 
We also need to define the auxiliary function
\begin{equation}\label{eq:f}
	f(t) := \frac{\pi^2}{(\pi-t/4)^{2}}-1.
\end{equation}
Note that $f(t) \ge 0$ in $\cT_\tt$.
\begin{prop}\label{prop:local_Hardy}
	For any $u\in\cC_0^\infty(\Omega_\tt)$ the inequality
	\[
		\int_{\Omega_\tt}|\p_t u|^2\dd s\dd t  - \|u\|^2 
		\geq 	
		\int_{\cT_\tt} f(t)|u|^2 \dd s \dd t,
	\]
	holds with $f(\cdot)$ as in~\eqref{eq:f}.
\end{prop}
Before going through the proof of Proposition~\ref{prop:local_Hardy}, 
we notice that
\[
	\frh_\tt[u] - \|u\|^2 
	= 
	\int_{\Omega_\tt}|\p_t u|^2\dd s \dd t -\|u\|^2 + 
	\int_{t=0}^{\pi}\int_{s>-t\cot\tt} |\p_s u|^2 -
	\frac{|u|^2}{4\rho^2}\dd s\dd t.
\]
In fact, the last term on the right-hand side is positive. 
It can be seen by performing, in the $s$-integral, the change of variable $\sigma = \rho(s,t)$ for any fixed $t\in(0,\pi)$ and using the classical one-dimensional Hardy inequality (see \eg \cite[\S VI.4., eq. (4.6)]{Kato}). Together with 
Proposition~\ref{prop:local_Hardy}, it gives the following corollary.
\begin{cor}[Local Hardy inequality]\label{cor:local_Hardy}
	For any $u\in\cC_0^\infty(\Omega_\tt)$
	the inequality
	\[
		\frh_\tt[u]  - \|u\|^2 
		\geq \int_{\cT_\tt} f(t)|u|^2 \dd s \dd t,
	\]
	holds with $f(\cdot)$ as in~\eqref{eq:f}. 
\end{cor}
\begin{proof}[Proof of Proposition~\ref{prop:local_Hardy}] 
	Let $u\in\cC_0^\infty(\Omega_\tt)$.
	For fixed $s \in (-\pi\cot\tt,0)$ 
	the function 
	\[
		(-s\tan\tt,\pi)\ni t \mapsto u(s,t)
	\] 
	satisfies Dirichlet boundary conditions
	at $t = -s\tan\tt$ and $t = \pi$. 
	Let
	\[
		\lm_1(s) := \frac{\pi^2}{(\pi - |s|\tan\tt)^{2}}
	\]	 
	be the first eigenvalue of the Dirichlet Laplacian 
	on the interval $(-s\tan\tt,\pi)$.
	Hence, we get
	\[
		\int_{\Omega_\tt}|\p_t u|^2 \dd t\dd s - \|u\|^2 
		\geq \int_{\Omega_\tt} \big( h(s) - 1\big)|u|^2 \dd s\dd t,
	\]
	with
	\[
		h(s) 
		:= 
		\left
		\{
		\begin{array}{lll}
			\lm_1(s),&&  s\in (-\pi\cot\tt,0),\\
			1, &&  s\in [0,+\infty).
		\end{array}
		\right.
	\]
	Particularly, we remark that for any $s>-\pi\cot\tt$ 
	we have $h(s)-1\geq0$. It yields
	\[
		\int_{\Omega_\tt}|\p_t u|^2 \dd s\dd t- \|u\|^2
		\geq 
		\int_{\cT_\tt}(h(s)-1)|u|^2\dd s\dd t.
	\]
	Finally, as $h(\cdot)$ is non-increasing we obtain
	\[
	\begin{split}
		\int_{\Omega_\tt}|\p_t u|^2 \dd s\dd t - \|u\|^2
		&\geq
		\int_{\cT_\tt}\big( h(s) - 1 \big) |u|^2\dd s\dd t\\
		&=
		\int_{t=0}^\pi
		\int_{s=-2\rho_0}^{-\rho_0/2}
		\big( h(s) - 1 \big) |u|^2\dd s\dd t\\
		&\geq
		\int_{t=0}^\pi\int_{s=-2\rho_0}^{-\rho_0/2} 
		\Big(\lm_1\big(-\rho_0/2\big)-1 \Big)|u|^2\dd s\dd t\\
		&=
		\int_{\cT_\tt} 
		\Big(\lm_1\big(-\rho_0/2\big)-1 \Big)|u|^2\dd s\dd t
		 =
		\int_{\cT_\tt} f(t)|u|^2\dd s\dd t.\qedhere
	\end{split}
	\]
	%
\end{proof}

\subsection{A refined lower-bound}\label{subsec:Hardy_lb}
In this subsection we prove the following statement.
\begin{prop}
	For any $\eps \in (0,\pi^{-3})$
	\[
		\int_{\Omega_\tt}|\p_s u|^2
		 - \frac{1}{4\rho^2}|u|^2
		\dd s\dd t 
		\geq 
		\frac{\eps}{16}
		\int_{\Omega_\tt}\frac{t^3}{1+\rho^2\ln^2(\rho/\rho_0)}|u|^2
		\dd s\dd t
		 - 
		\eps\int_{\cT_\tt}
		t^3\left(\frac{4}{\rho_0^2}+\frac1{8}\right)|u|^2\dd s\dd t
	\]	
	holds for all $u\in\cC_0^\infty(\Omega_\tt)$. 
\label{prop:Hardy_1D}
\end{prop}
To prove Proposition~\ref{prop:Hardy_1D} we need the following lemma
whose proof follows the same lines as the one of~\cite[Lem.~3.1]{CK14}. 
However, we provide it here for the sake of completeness.
%
%
In the proofs of this lemma and of Proposition~\ref{prop:Hardy_1D},
we use that for $t\in (0,\pi)$ and $g \in H^1_0(-2\rho_0(t),+\infty) $
\begin{equation}\label{eq:transform}
	\begin{split}
		\int_{s>-2\rho_0}|(\rho^{-1/2}g)^\pp|^2 \rho\dd s 
		& =
		\int_{s>-2\rho_0}\big|\rho^{-1/2}g^\pp - 1/2 \rho^{-3/2} g\big|^2
		 \rho\dd s \\
		& =
		\int_{s>-2\rho_0}|g^\pp|^2 - \frac{1}{2\rho}(|g|^2)^\pp
		+ \frac{1}{4\rho^2}|g|^2\dd s \\
		& =
		\int_{s>-2\rho_0} | g^\pp|^2 - \frac{1}{4\rho^2}|g|^2\dd s. 
	\end{split}
\end{equation}

\begin{lem}
	For any fixed $t\in (0,\pi)$ the inequality
	\[
		\int_{s>-\rho_0(t)}
		|g^\pp(s)|^2 - \frac{1}{4\rho^2}|g(s)|^2 \dd s
		\geq 
		\frac14 \int_{s>-\rho_0(t)}\frac{|g(s)|^2}{\rho^2\ln^2(\rho/\rho_0)}\dd s
	\]
	holds for all $g\in H_0^1(-\rho_0(t),+\infty)$.
\label{lem:Hardy_CK14} 
\end{lem}
%
%
\begin{proof}
	Let $t\in (0,\pi)$ and $g\in\cC_0^\infty(-\rho_0(t),+\infty)$
	be fixed.
	We notice that for any $\aa > 0$
	\begin{equation}\label{eq:gamma_est}
	\begin{split}
		& \int_{s>-\rho_0}\bigg|(\rho^{-1/2}g)^\pp 
		- 
		\frac{\aa\rho^{-1/2}g}{\rho\ln(\rho/\rho_0)}\bigg|^2\rho \dd s \\
		& \qquad\quad
		=
		\int_{s>-\rho_0}|(\rho^{-1/2}g)^\pp|^2 \rho\dd s 
		+ 
		\aa^2\int_{s>-\rho_0}\frac{|g|^2}{\rho^2\ln^2(\rho/\rho_0)}\dd s
		 - 
		\aa\int_{s>-\rho_0}\frac{(\big|\rho^{-1/2}g ) \big|^2)^\pp}
		{\ln(\rho/\rho_0)}\dd s.
	\end{split}
	\end{equation}
	For the first term on the right hand side in~\eqref{eq:gamma_est} we get
	by~\eqref{eq:transform} that
	\begin{equation}\label{eq:gamma_est_term1}
		\int_{s>-\rho_0}|(\rho^{-1/2}g)^\pp|^2 \rho\dd s 	
		=
		\int_{s>-\rho_0}| g^\pp|^2 - \frac{1}{4\rho^2}|g|^2 \dd s. 
	\end{equation}	
	Performing an integration by parts in the last term of the right-hand side
	in~\eqref{eq:gamma_est} we obtain
	\begin{equation}\label{eq:gamma_est_term3}
		\int_{s>-\rho_0}\frac{\big(|\rho^{-1/2}g ) \big|^2)^\pp}
		{\ln(\rho/\rho_0)}\dd s = 
		\int_{s>-\rho_0}
		\frac{|g|^2}{\rho^2\ln^2(\rho/\rho_0)}\dd s.
	\end{equation}
	Combining~\eqref{eq:gamma_est},~\eqref{eq:gamma_est_term1}, 
	and~\eqref{eq:gamma_est_term3} we get
	\[
		\int_{s>-\rho_0}|g^\pp|^2 - \frac{1}{4\rho^2}|g|^2 \dd s 
		\geq 
		(\aa -\aa^2)
		\int_{s>-\rho_0}\frac{|g|^2}{\rho^2\ln^2(\rho/\rho_0)}\dd s.
	\]
	It remains to set $\aa = 1/2$.
	
	The extension of this result to $g\in H^1_0(-\rho_0(t),+\infty)$ relies
	on the density of $\cC_0^\infty(-\rho_0(t),+\infty)$ in 
	$H^1_0(-\rho_0(t),+\infty)$ with respect to the $H^1$-norm
	and a standard continuity argument.
\end{proof}
Now we have all the tools to prove Proposition~\ref{prop:Hardy_1D}.
\begin{proof}[Proof of Proposition~\ref{prop:Hardy_1D}]  
	First, we define the cut-off function $\xi\colon\Omega_\tt\arr\dR$ by
	\[
		\xi(s,t) :=\left\{
			\begin{array}{lll}
				0,&& s \in (-2\rho_0(t),-\rho_0(t)),\\
				2\rho_0(t)^{-1}(s+\rho_0(t)),&&
				s\in(-\rho_0(t), -\rho_0(t)/2),\\
				1,&& s\in(-\rho_0(t)/2,+\infty).
			\end{array}
		\right.
	\]
	The partial derivative of $\xi$ with respect to the $s$-variable
	is given by
	\begin{equation}\label{eq:ps_xi}
		(\p_s\xi)(s,t) = 
		\left\{
		\begin{array}{lll}
			2\rho_0(t)^{-1}, && s \in (-\rho_0(t),-\rho_0(t)/2),\\ 
			0, && s \in (-2\rho_0(t), -\rho_0(t))\cup(-\rho_0(t)/2,+\infty),
		\end{array}
		\right.
	\end{equation}
	Further, for any 
	$u\in\cC_0^\infty(\Omega_\tt)$ and fixed $t\in(0,\pi)$
	using $(a+b)^2 \le 2a^2 + 2b^2$, $a,b\in\dR$, we get
	\[
		\int_{s>-2\rho_0}
		\frac{|u|^2}{1+\rho^2\ln^2(\rho/\rho_0)}\dd s
		\leq 
		2\int_{s>-\rho_0}\frac{|\xi u|^2}{\rho^2\ln^2(\rho/\rho_0)}\dd s 
		+ 2\int_{s>-2\rho_0}|(1-\xi)u|^2\dd s,
	\]
	where in both integrals we increased the integrands
	by making the denominators smaller.
	Note that for fixed $t\in (0,\pi)$ we have 
	$s\mapsto \xi(s,t)u(s,t) \in H^1_0(-\rho_0(t),+\infty)$.
	Applying Lemma~\ref{lem:Hardy_CK14} and using~\eqref{eq:transform} we get
	\[
	\begin{split}
		\int_{s>-2\rho_0}\frac{|u|^2}{1+\rho^2\ln^2(\rho/\rho_0)}\dd s
		& \leq 
		8
		\int_{s>-\rho_0}|\p_s (\xi u )|^2 	-\frac{|\xi u|^2}{4\rho^2} \dd s
		+ 
		2 \int_{s=-2\rho_0}^{-\rho_0/2}|u|^2\dd s\\
		& =
		8\int_{s>-\rho_0}|\p_s (\rho^{-1/2}\xi u )|^2 \rho \dd s
		+ 
		2 \int_{s=-2\rho_0}^{-\rho_0/2}|u|^2\dd s\\
		& \le 
		16\int_{s>-\rho_0}
		\Big(|\xi\p_s (\rho^{-1/2} u )|^2 \rho + |u \p_s \xi  |^2\Big) \dd s
		+ 
		2 \int_{s=-2\rho_0}^{-\rho_0/2}|u|^2\dd s\\
		& \le
		16\int_{s>-2\rho_0}
		|\p_s (\rho^{-1/2} u )|^2 \rho \dd s
		+ 
		\int_{s=-\rho_0}^{-\rho_0/2}\frac{64}{\rho_0^2} |u|^2\dd s
		+
		2\int_{s=-2\rho_0}^{-\rho_0/2} |u|^2\dd s\\
		& \le 
		16\int_{s>-2\rho_0}
		\bigg(|\p_s u|^2 - \frac{|u|^2}{4\rho^2}\bigg) \dd s
		+
		\int_{s=-2\rho_0}^{-\rho_0/2}\bigg(\frac{64}{\rho_0^2} + 2\bigg) |u|^2\dd s,
	\end{split}
	\]
	%
	which is equivalent to
	\[
		\int_{s>-2\rho_0}
		\bigg(|\p_s u|^2 - \frac{|u|^2}{4\rho^2}\bigg) \dd s
		\ge
		\frac{1}{16}
		\int_{s>-2\rho_0}\frac{|u|^2}{1+\rho^2\ln^2(\rho/\rho_0)}\dd s
		-
		\int_{s=-2\rho_0}^{-\rho_0/2}
		\bigg(\frac{4}{\rho_0^2} +\frac{1}{8}\bigg) |u|^2 \dd s 
	\]
	Finally, we multiply each side by $\eps t^3$ and integrate for $t\in(0,\pi)$
	\begin{equation*}
		\int_{\Omega_\tt}\eps t^3
		\bigg(|\p_s u|^2 - \frac{|u|^2}{4\rho^2}\bigg) \dd s\dd t 
		\geq
		\frac{\eps}{16}
		\int_{\Omega_\tt}
		\frac{t^3}{1+\rho^2\ln^2(\rho/\rho_0)}|u|^2\dd s\dd t 
		- \eps\int_{\cT_\tt} t^3
		\Big(\frac{4}{\rho_0^{2}} + \frac{1}{8}\Big) |u|^2 \dd s\dd t.
	\label{eqn:calc1}
	\end{equation*}
	Since for any $\eps \in (0, \pi^{-3})$ holds 
	$0< \eps t^3 < 1$,
	the inequality in Proposition~\ref{prop:Hardy_1D} follows.
\end{proof}

\subsection{Proof of Theorem \ref{thm:Hardy}}\label{subsec:Hardy}
By Propositions~\ref{prop:local_Hardy} and~\ref{prop:Hardy_1D} we have
\begin{equation}
\begin{split}
	\label{eq:non_optimized_Hardy}
	\frh_\tt[u] -\|u\|^2
	& = \int_{\Omega_\tt}
	\bigg(|\p_s u|^2 - \frac{|u|^2}{4\rho^2}\bigg) \dd s\dd t 
	+
	\int_{\Omega_\tt} |\p_t u|^2\dd s \dd t\\
	& \geq
	\frac{\eps}{16}
	\int_{\Omega_\tt}
	\frac{t^3}{1+\rho^2\ln^2(\rho/\rho_0)}|u|^2\dd s\dd t 
	+
	\int_{\cT_\tt}
	\bigg[f(t) - \eps t^3 
	\bigg(\frac{4}{\rho^2_0} + \frac{1}{8}\bigg)\bigg]|u|^2\dd s\dd t,
\end{split}	
\end{equation}
for all $u\in\cC_0^\infty(\Omega_\tt)$. 
For the second term on the right-hand side of~\eqref{eq:non_optimized_Hardy} 
to be positive it suffices to verify that for all $t\in(0,\pi)$
\begin{equation}
	h_{\eps}(t) := 
	f(t) -\frac{16}{\cot^2\tt} \eps t - 
	\frac{1}{8}\eps t^3\ge 0.
\label{eqn:positivity}
\end{equation}
By definition, $f$ in~\eqref{eq:f} is a $C^\infty$-smooth bounded
function on $(0,\pi)$ and for any $a\in(0,\pi)$ 
and all $t\in(a,\pi)$ we have $f(t)\geq f(a)>0$. 
Moreover, $f(t) = (2\pi)^{-1}t +\cO(t^2)$ when $t\arr 0+$. Consequently, 
we can find $\eps_0 > 0$ small enough such that for all 
$\eps \in (0,\eps_0)$ inequality~\eqref{eqn:positivity} holds. 
Going back to the form $\frq_\tt$ we get that 
there exists $c > 0$ such that for any $u\in \cC_0^\infty(\Gui(\tt))$ 
holds
\[
\begin{split}
	\frq_\tt[u] -\|u\|^2_{L^2(\Gui(\tt)} 
	& 
	= 	
	\frh_\tt[\sfU_\tt^{-1} u] -\|\sfU_\tt^{-1}u\|^2\\
	& \ge 
	c
	\int_{\Omega_\tt}
	\frac{t^3}{1+\rho^2\ln^2(\rho/\rho_0)}|(\sfU_\tt^{-1} u)(s,t)|^2
	\dd s\dd t \\
	& = 
	c
	\int_{\Gui(\tt)}
	\frac{(r\cos\tt - z\sin\tt)^3 }
	{1 + \frac{r^2}{\sin^2\tt} 
	\ln^2\big(\frac{r}{\cos\tt}\frac{2}{r\cos\tt-z\sin\tt}\big)}
	|u|^2\dd r\dd z,
\end{split}
\]
where we used the unitary transform $\sfU_\tt$ defined in~\eqref{eq:Utt}.
This finishes the proof of Theorem~\ref{thm:Hardy}.

\appendix
\section{Gauge invariance}
\label{app:A}

In this appendix we justify the unitary equivalence between 
the self-adjoint operators $\sfH_\omega$ and $\sfH_{\Phi_\omega + k}$
for all real-valued function $\omega\in L^2(\dS^1)$ and $k\in\dZ$.
The justification relies on the explicit construction of a unitary transform.

Throughout this appendix, $\omega$ always denotes a real-valued function. Before formulating the main result of this appendix we recall that
for $\omega\in L^2(\dS^1)$, 
we define the norm induced by the quadratic form $Q_{\omega,\tt}$ defined
in~\eqref{eq:Q_omega} as
\[
	\|u\|_{+1,\omega}^2 := Q_{\omega,\tt}[u] + \|u\|_{\Lcyl}^2,
	\qquad u \in \dom Q_{\omega,\tt}.
\]
Recall that the flux $\Phi_\omega\in\dR$, the function $V \in \cC([0,2\pi])$ 
and the unitary gauge transform  $\sfG_V\colon \Lcyl\arr\Lcyl$ are
associated with $\omega$ and $k$ as
\begin{equation}\label{eq:assoc}
	\Phi_\omega := \frac{1}{2\pi}\int_0^{2\pi}\omega(\phi)\dd \phi,
	\qquad
	V(\phi) := (\Phi_\omega + k)\phi - \int_0^\phi\omega(\xi) \dd \xi,
	\qquad
	\sfG_V\psi := e^{\ii V}\psi.
\end{equation}
The following proposition is the main result of this appendix.
\begin{prop}\label{prop:unit_fq}
	Let $\omega\in L^2(\dS^1)$ and $k\in\dZ$.
	Let $\Phi_\omega$, $V$ and $\sfG_V$ be as in~\eqref{eq:assoc}.
	Then, the following hold:
	\begin{myenum}
		\item $\dom Q_{\omega,\tt}
		 = \sfG_V\big(\dom Q_{\Phi_\omega + k,\tt}\big)$;
		\item $Q_{\omega,\tt}[\sfG_V u] = Q_{\Phi_\omega + k,\tt}[u]$ 
		for all $u\in \dom Q_{\Phi_\omega + k,\tt}$.
	\end{myenum}	
	In particular, the operators 
	$\sfH_{\omega,\tt}$ and $\sfH_{\Phi_\omega + k,\tt}$ are unitarily 	
	equivalent.
\end{prop}
Therefore, taking $k= -\argmin_{k\in\dZ}\{|k-\omega|\}$ in~\eqref{eq:assoc}
we can reduce the case of a general
$\omega\in L^2(\dS^1)$ {\it via} the transform $\sfG_V$ 
to a constant $\omega\in [-1/2,1/2]$. 


Before proving Proposition~\ref{prop:unit_fq} we need to state several
lemmas whose proofs are postponed until the end of this appendix.
\begin{lem}\label{lem:omega_smooth}
	Let $\omega\in \cC^\infty(\dS^1)$ and $k \in\dZ$. 
	Let $\Phi_\omega$, $V$ and $\sfG_V$ be associated with $\omega$
	and $k$ as in~\eqref{eq:assoc}.
	Then, the following statements hold:
	\begin{myenum}
		\item 
		$\cC^\infty_0(\Lay(\tt)) = \sfG_V\big(\cC^\infty_0(\Lay(\tt))\big)$;
		\item $Q_{\omega,\tt}[\sfG_V u] = Q_{\Phi_\omega + k,\tt}[u]$
		for all $u\in \cC^\infty_0(\Lay(\tt))$.
	\end{myenum}
\end{lem}
\begin{lem}\label{lem:omega_approx} 
	Let $\omega\in L^2(\dS^1)$ and $(\omega_n)_{n\in\dN}$ be a sequence of real-valued
	functions $\cC^\infty(\dS^1)$ 
	such that $\|\omega_n - \omega\|_{L^2(\dS^1)}\arr 0$
	as $n\arr\infty$.
	Let $\Phi_\omega$, $V$, $\sfG_V$ be associated with $\omega$,$k$
	and $\Phi_{\omega_n}$, $V_n$, $\sfG_{V_n}$ 
	be associated with $\omega_n$, $k$ as in~\eqref{eq:assoc}.
	Then, as $n\arr \infty$, the following hold:
	\begin{myenum}
		\item $\|\omega_n - \omega\|_{L^1(\dS^1)}\arr 0$;
		\item $|\Phi_{\omega_n}-\Phi_\omega| \arr 0$;
		\item
		$V_n(\phi) \arr V(\phi)$ for any $\phi \in \dS^1$;
		\item $\sfG_{V_n} \arr \sfG_{V}$ in the strong sense;
		\item 
		$Q_{\omega_n,\tt}[\sfG_{V_n}u] - Q_{\omega,\tt}[\sfG_{V_n} u] \arr 0$
		and 
		$Q_{\Phi_{\omega_n} + k,\tt}[u] \arr Q_{\Phi_\omega + k,\tt}[u]$
		for any 
		$u \in \cC^\infty_0(\Lay(\tt))$.
	\end{myenum}
	\label{lem:appendixA}
\end{lem}

\begin{lem}\label{lem:inclusions} 
	Let $\omega\in L^2(\dS^1)$ and $k\in\dZ$. 
	Let $\Phi_\omega$, $V$, and $\sfG_V$ be associated with $\omega$ 
	and $k$ as in~\eqref{eq:assoc}.
	Then, the following 	statements hold:
	\begin{myenum}
		\item 
		$\sfG_V\big(\cC_0^\infty(\Lay(\tt))\big)\subset \dom Q_{\omega,\tt}$;
		\item 
		$Q_{\omega,\tt}[\sfG_V u] = Q_{\Phi_\omega + k,\tt}[u]$ 
		for all $u\in  \cC^\infty_0(\Lay(\tt))$.
	\end{myenum}	
\end{lem}

In the proof of Proposition~\ref{prop:unit_fq} we use 
Lemmas~\ref{lem:omega_smooth} and~\ref{lem:inclusions}.
The statement of Lemma~\ref{lem:omega_approx} is only needed later
in the proof of Lemma~\ref{lem:inclusions}.

\begin{proof}[Proof of Proposition \ref{prop:unit_fq}] 
	Let $u\in\dom Q_{\Phi_\omega+k,\tt}$ and let $(u_n)_{n\in\dN}$ 
	be a sequence of functions in $\cC_0^\infty(\Lay(\tt))$ such 
	that $\|u_n - u\|_{+1,\Phi_\omega + k}\arr 0$ as $n\arr \infty$. 
	The sequence $(u_n)_{n\in\dN}$ exists because $\cC_0^\infty(\Lay(\tt))$
	is a core for the form $Q_{\Phi_\omega+k,\tt}$.
	
	\noindent {\rm (i)} 
	Since the norm $\|\cdot\|_{+1,\Phi_\omega+k}$ is stronger
	than the norm $\|\cdot\|_\Lcyl$ we get 
	\begin{equation}\label{eq:un_convergence1}
		\|u_n - u\|_\Lcyl \arr 0,\qquad n\arr \infty.
	\end{equation}
	Let us consider the sequence $(\sfG_Vu_n)_{n\in\dN}$. 
	Due to~\eqref{eq:un_convergence1} we have
	\begin{equation}\label{eq:un_convergence2}
		\|\sfG_V u_n - \sfG_V u\|_\Lcyl \arr 0,\qquad n\arr \infty.
	\end{equation}
	By Lemma~\ref{lem:inclusions}\,(i),
	we know that $\sfG_Vu_n\in \dom Q_{\omega,\tt}$
	for all $n\in\dN$.
	Now, we prove that $(\sfG_Vu_n)_{n\in\dN}$ 
	is a Cauchy sequence in the norm $\|\cdot\|_{+1,\omega}$.
	Indeed, by Lemma~\ref{lem:inclusions}\,(ii) we have
	\begin{equation*}\label{eq:2norms}
	\begin{split}
			\|\sfG_{V}(u_{n+p} - u_n)\|_{+1,\omega}^2
			& =  
			Q_{\omega,\tt}[\sfG_V(u_{n+p} - u_n)] + 
			\|\sfG_V(u_{n+p} - u_n)\|_{L_\cyl^2(\Lay(\tt))}^2\\
			& =  
			Q_{\Phi_\omega+k,\tt}[u_{n+p} - u_n] + \|u_{n+p} - u_n\|_\Lcyl^2
			=\|u_{n+p} - u_n\|_{+1,\Phi_\omega+k}^2.	
	\end{split}
	\end{equation*}
	Thus, $(\sfG_V u_n)_{n\in\dN}$ is a Cauchy sequence 
	in the norm $\|\cdot\|_{+1,\omega}$
	and therefore it converges to a function $v\in\dom Q_{\omega,\tt}$ 
	in this norm. Since the norm $\|\cdot\|_{+1,\omega}$
	is stronger than $\|\cdot\|_\Lcyl$ we get 
	$\|\sfG_Vu_n - v\|_\Lcyl\arr 0$ as $n\arr \infty$.
	Taking~\eqref{eq:un_convergence2} into account
	we conclude $\sfG_V u = v \in \dom Q_{\omega,\tt}$, \ie
	we have proven that 
	$\sfG_V\big(\dom Q_{\Phi_\omega+k,\tt}\big) \subset \dom Q_{\omega,\tt}$.
	As a by-product we have strengthened~\eqref{eq:un_convergence2}
	up to
	\begin{equation}\label{eq:un_convergence3}
		\|\sfG_V u_n - \sfG_V u\|_{+1,\omega} \arr 0,\qquad n\arr\infty.
	\end{equation}
	Because the reverse inclusion
	$\sfG_V\big(\dom Q_{\Phi_\omega+k,\tt}\big) \supset \dom Q_{\omega,\tt}$
	can be proven in a similar way we omit this argument here.
	
	\noindent {\rm (ii)} 
	First, observe that
	\[
		\|u_n\|_{+1,\Phi_\omega+k} 
		\underset{n\rightarrow\infty}{\longrightarrow}\|u\|_{+1,\Phi_\omega+k}
		\quad\text{and}\quad 
		\|\sfG_V u_n\|_{+1,\omega} \underset{n\rightarrow\infty}
		{\longrightarrow}\|\sfG_V u\|_{+1,\omega},
	\]
	where the second limit is a particular 
	consequence of~\eqref{eq:un_convergence3} in the proof of~{\rm (i)}.
	Further, in view of the definition of the norms $\|\cdot\|_{+1,\omega}$ and
	$\|\cdot\|_{+1,\Phi_\omega+k}$, we obtain
	\begin{equation}
		Q_{\Phi_\omega+k,\tt}[u_n] 
		\underset{n\rightarrow\infty}{\longrightarrow}Q_{\Phi_\omega+k,\tt}[u]
		\quad\text{and}\quad 
		Q_{\omega,\tt}[\sfG_V u_n] 
		\underset{n\rightarrow\infty}{\longrightarrow}
		Q_{\omega,\tt}[\sfG_V u].
		\label{eqn:cvg_fq_n}
	\end{equation}
	Note that by Lemma~\ref{lem:inclusions}\,(ii) 
	we have $Q_{\omega,\tt}[\sfG_V u_n] = Q_{\Phi_\omega + k,\tt}[u_n]$ 
	for any $n\in\dN$. Thus, passing to the limit $n\arr\infty$ and 
	taking into account~\eqref{eqn:cvg_fq_n} we end up with
	\[
		Q_{\omega,\tt}[\sfG_V u] 
		= 
		\lim_{n\arr \infty} Q_{\omega,\tt}[\sfG_V u_n]
		=
		\lim_{n\arr \infty} Q_{\Phi_\omega + k,\tt}[ u_n]
		=
		Q_{\Phi_\omega + k,\tt}[u].
	\]
	Finally, the unitary equivalence of the operators $\sfH_{\omega,\tt}$
	and $\sfH_{\Phi_\omega + k,\tt}$ follows from the first
	representation theorem. The operator $\sfG_V$ plays the role
	of the corresponding transform which establishes
	unitary equivalence.
\end{proof}

Now, we deal with the proofs of Lemmas~\ref{lem:omega_smooth},
~\ref{lem:omega_approx}, and~\ref{lem:inclusions}. 

\begin{proof}[Proof of Lemma~\ref{lem:omega_smooth}]
	{\rm (i)} 
	The identity 
	$\cC^\infty_0(\Lay(\tt)) = \sfG_V\big(\cC^\infty_0(\Lay(\tt))\big)$ 
	is a straightforward consequence of $e^{\ii V(\cdot)} \in \mathcal{C}^\infty(\dS^1)$. 
	The details are omitted.
		
	\noindent {\rm (ii)} 
	For any $u\in \cC^\infty_0(\Lay(\tt))$ we get
	by direct computation
	\[
	\begin{split}
		Q_{\omega,\tt}[\sfG_V u] 
		& = 
		\|(\ii \nabla  - \bfA_\omega)e^{\ii V} u\|^2_\Lcyl\\
		& =
		\|e^{\ii V} (\ii \nabla  - \bfA_\omega - \nabla V)  u\|^2_\Lcyl\\
		& =
		\lbra\|\lbra(\ii \nabla  - \bfA_\omega - 
		r^{-1}\bfe_\phi V^\pp(\phi)\rbra)
		 u \rbra\|^2_\Lcyl\\
		& =
		\lbra\|\lbra(\ii \nabla  - \bfA_\omega - r^{-1}\bfe_\phi 
		(\Phi_\omega + k) + 
		\bfA_\omega\rbra)
		 u \rbra\|^2_\Lcyl\\
		& =
		\lbra\|\lbra(\ii \nabla  - \bfA_{\Phi_\omega + k}\rbra)
		 u \rbra\|^2_\Lcyl = Q_{\Phi_\omega + k,\tt}[u].\qedhere	
	\end{split}	
	\]
\end{proof}

\begin{proof}[Proof of Lemma \ref{lem:appendixA}]
	The claims of {\rm (i)} and {\rm (ii)} 
	are a direct consequence of the inclusion 
	$L^2(\dS^1)\subset L^1(\dS^1)$. 
	Indeed, thanks to the Cauchy-Schwarz inequality, we have
	\[
		|\Phi_{\omega_n} - \Phi_\omega| 
		\le 
		\|\omega_n - \omega\|_{L^1(\dS^1)} 	
		= 
		\int_0^{2\pi}|\omega_n(\xi) - \omega(\xi)|\dd\xi
		\le
		\sqrt{2\pi} \|\omega_n - \omega\|_{L^2(\dS^1)}
		\underset{n\rightarrow \infty}{\longrightarrow} 0.
	\]
	The claim of {\rm (iii)} follows from {\rm (i)} 
	and {\rm (ii)} as
	\[
		|V_n(\phi) - V(\phi)| 
		=
		\lbra |(\Phi_{\omega_n}-\Phi_\omega)\phi + 
		\int_0^\phi\big(\omega_n(\xi) - \omega(\xi)\big)\dd\xi\rbra|\\
		\leq 
		|\Phi_{\omega_n}-\Phi_\omega| \phi + 
		\|\omega_n - \omega\|_{L^1(\dS^1)} 
		\underset{n\rightarrow \infty}{\longrightarrow} 0.
	\]
	Using the identity $2i\sin(x) = e^{\ii x} - e^{-\ii x}$ we obtain
	for any $u\in \Lcyl$
	\begin{equation}
		\|\sfG_{V_n}u - \sfG_Vu\|_\Lcyl 
		= 
		\|(e^{\ii V_n} - e^{\ii V}) u\|_\Lcyl
		=
		2\|\sin\big((V-V_n)/2\big)u\|_\Lcyl.
	\label{eqn:do_cvg}
	\end{equation}
	Elementary properties 
	of the sine function give $|\sin\big((V-V_n)/2\big)u|^2\leq |u|^2$. 
	Thanks to {\rm (iii)} 
	we know that $\sin\big((V-V_n)/2\big) \arr 0$ as $n\arr \infty$
	(pointwise). 
	Consequently, passing to the limit in~\eqref{eqn:do_cvg}, we get
	the claim of~{\rm (iv)} by the Lebesgue dominated convergence theorem.
	Finally, for any $u \in \cC^\infty_0(\Lay(\tt))$ we get
	\[
		\big|(Q_{\omega_n,\tt}[u])^{1/2} -(Q_{\omega,\tt}[u])^{1/2}\big|
		\le 
		\|(\bfA_{\omega_n} - \bfA_\omega)u\|_\Lcyl 
		\le C\|\omega_n - \omega\|_{L^2(\dS^1)}\arr 0, 
	\]	
	where the constant $C > 0$ depends on $\|u\|_{L^\infty(\Lay(\tt))}$
	and $\supp u$ only. Hence, the second limit in~{\rm (v)} immediately follows. 
	The first limit in~{\rm (v)} is a consequence of the above bound 
	and of the fact that
	$\|\sfG_{V_n}u\|_{L^\infty(\Lay(\tt))}$ and $\supp (\sfG_{V_n}u)$
	are independent of $n$.
\end{proof}

\begin{proof}[Proof of Lemma~\ref{lem:inclusions}]
	{\rm (i)}
	By definition, $\dom Q_{\omega,\tt}$ 
	is the closure of $\cC_0^\infty(\Lay(\tt))$ 
	with respect to the norm $\|\cdot\|_{+1,\omega}$.
	Let $u\in\cC_0^\infty(\Lay(\tt))$ and 
	$(\omega_n)_{n\in\dN}$ be a sequence of real-valued functions 
	$\cC^\infty(\dS^1)$ such
	that $\|\omega_n - \omega\|_{L^2(\dS^1)}\arr 0$ as $n\arr \infty$.

	First, we prove that $\sfG_{V_n}u \in \cC_0^\infty(\Lay(\tt))$ is a Cauchy 
	sequence in the norm $\|\cdot\|_{+1,\omega}$.
	Due to Lemma~\ref{lem:omega_approx}\,(iv) we already know that 
	\begin{equation}\label{eq:conv0}
		\|\sfG_{V_n} u - \sfG_V u\|_\Lcyl \arr 0,\qquad n\arr\infty.
	\end{equation}
	Further, $Q_{\omega,\tt}[(\sfG_{V_{n+p}} - \sfG_{V_n})u]$ 
	can be bounded from above by
	\begin{equation}\label{??}
			Q_{\omega,\tt}[(\sfG_{V_{n+p}} - \sfG_{V_n})u]
			 = 
			\|(\ii\nabla - \bfA_\omega)(e^{\ii V_{n+p}} - e^{\ii V_n})u\|^2_\Lcyl	
			\le 2(J_{n,p} + K_{n,p}),
	\end{equation}
	where $J_{n,p}$ and $K_{n,p}$ are defined by
	\begin{equation}\label{eq:JK}
		J_{n,p}  := 
		\|(e^{\ii V_{n+p}} - e^{\ii V_n})(\ii\nabla - \bfA_\omega)u\|_\Lcyl^2,
		\qquad
		K_{n,p}  :=
		\|\big(\nabla(e^{\ii V_{n+p}} - e^{\ii V_n})\big)u\|_\Lcyl^2.
	\end{equation}		
	Because $(\ii\nabla - \bfA_\omega)u \in \cC^\infty_0(\Lay(\tt))$,
	Lemma~\ref{lem:omega_approx}\,(iv) implies that 
	$J_{n,p}\arr 0$ as $n,p\rightarrow\infty$. 
	Let us deal with the term $K_{n,p}$. 
	Computing the gradient taking into account the expression of $V_{n}$, 
	we get
	\begin{equation}\label{eq:grad}
	\begin{split}
		\nabla (e^{\ii V_{n+p}} - e^{\ii V_n}) 
		& = 
		\big[e^{\ii V_{n+p}}\Phi_{n+p} 
		- e^{\ii V_n}\Phi_n\big]
		\frac{\bfe_\phi}{r} 
		- 
		\big[e^{\ii V_{n+p}}\omega_{n+p}(\phi) - e^{\ii V_n}\omega_n(\phi)\big]	
		\frac{\bfe_\phi}{r}\\
		& = 
		\bfx_{n,p} + \bfy_{n,p},
	\end{split}	
	\end{equation}
	where, for all $q\in\mathbb{N}$, $\Phi_q := \Phi_{\omega_q} +k$ and the terms $\bfx_{n,p}$, $\bfy_{n,p}$
	on the right-hand side are defined by
	\[
	\begin{split}
		\bfx_{n,p}
		&:= 
		\big((e^{\ii V_{n+p}} - e^{\ii V_n})\Phi_{n+p} + 
		e^{\ii V_n} (\Phi_{n+p} - \Phi_n)\big)\frac{\bfe_\phi}{r},\\
		\bfy_{n,p} &:=
		\big((e^{\ii V_{n+p}} - e^{\ii V_n})\omega_{n+p}(\phi) + 
		e^{\ii V_n}(\omega_{n+p}(\phi) - \omega_n(\phi)\big)
		\frac{\bfe_\phi}{r}.
	\end{split}
	\]
	Note that $u\in\cC_0^\infty(\Lay(\tt))$ yields 
	$v := r^{-1}u\in\cC_0^\infty(\Lay(\tt))$.
	The norm of $\bfx_{n,p} u$ can be estimated as
	\begin{equation}
		\begin{split}
			\|\bfx_{n,p} u\|_\Lcyl 
			& \le 
			|\Phi_{n+p}| \cdot \|(\sfG_{V_{n+p}} - \sfG_{V_n})v\|_\Lcyl 
			+ 
			|\Phi_{n+p} - \Phi_n|\cdot\|v\|_\Lcyl.
		\end{split}
		\label{eqn:up_bond_K_1}
	\end{equation}
	Lemma~\ref{lem:omega_approx}\,(iv) implies
	\[
		\|(\sfG_{V_{n+p}} - \sfG_{V_n})v\|_\Lcyl
		\underset{n,p\rightarrow \infty}{\longrightarrow}0.
	\]
	By Lemma~\ref{lem:omega_approx}\,(ii) the sequence $|\Phi_{n+p}|$ 
	is bounded so that the first term on the right-hand side 
	of~\eqref{eqn:up_bond_K_1} tends to $0$ as $n,p\rightarrow\infty$.
	Again by Lemma~\ref{lem:omega_approx}\,(ii) 
	the sequence $\Phi_n$, being convergent,
	is a Cauchy sequence. 
	Consequently, the second term on the right-hand side 
	of~\eqref{eqn:up_bond_K_1} also tends to $0$ as $n,p\rightarrow\infty$.
	Hence, we have proved that
	\begin{equation}\label{eq:xnp_conv}
		\|\bfx_{n,p} u\|_\Lcyl \arr 0,\qquad n,p\arr \infty.
	\end{equation}
	For the norm of $\bfy_{n,p}u$ we get
	\begin{equation}\label{eqn:eq_upb_K_2}
		\begin{split}
			\|\bfy_{n,p} u\|_\Lcyl
			& \le 
			\|(\sfG_{V_{n+p}} - \sfG_{V_n})\omega_{n+p}v\|_\Lcyl 
			+ \|(\omega_{n+p}-\omega_n)v\|_\Lcyl\\
			& \le 
			\|(\sfG_{V_{n+p}} - \sfG_{V_n})(\omega_{n+p} -\omega)v\|_\Lcyl \\
			& \qquad\qquad + 
			\|(\sfG_{V_{n+p}} - \sfG_{V_n})\omega v\|_\Lcyl
			+ \|(\omega_{n+p}-\omega_n)v\|_\Lcyl.
		\end{split}
	\end{equation}
	Using that $\|\sfG_{V_{n+p}} - \sfG_{V_n}\|$ is bounded 
	and that $v \in\cC_0^\infty(\Lay(\tt))$ we get 
	that the first term
	on the right-hand side of~\eqref{eqn:eq_upb_K_2} 	
	satisfies
	\[
		\|(\sfG_{V_{n+p}} - \sfG_{V_n})(\omega_{n+p} -\omega)v\|_\Lcyl^2
		\leq 
		C \|\omega_{n+p} - \omega\|_{L^2(\dS^1)},\quad 
		\text{ for some } C>0.
	\]
	Consequently it goes to $0$ as $n,p\rightarrow\infty$. 
	The second term 
	$\|(\sfG_{V_{n+p}} - \sfG_{V_n})\omega v\|_\Lcyl$ 
	on the right-hand side of~\eqref{eqn:eq_upb_K_2} 
	tends to $0$ by Lemma~\ref{lem:omega_approx}\,(iv).
	Again employing that $v\in\cC_0^\infty(\Lay(\tt))$ 
	and that $\omega_n$ is convergent in the norm $\|\cdot\|_{L^2(\dS^1)}$
	we get that the last term $\|(\omega_{n+p}-\omega_n)v\|_\Lcyl$
	on the right-hand side of~\eqref{eqn:eq_upb_K_2} also tends 
	to zero as $n,p\rightarrow\infty$. 
	Thus, we have shown
	\begin{equation}\label{eq:ynp_conv}
		\|\bfy_{n,p}u\|_\Lcyl \arr 0, \qquad n,p\arr\infty.
	\end{equation}

	Finally, combining~\eqref{eq:JK},~\eqref{eq:grad},~\eqref{eq:xnp_conv}, 
	and~\eqref{eq:ynp_conv}, we get that 
	$K_{n,p} \arr 0$ as $n,p\rightarrow\infty$. 
	Thus, $\sfG_{V_n} u$ 
	is a Cauchy sequence in the norm $\|\cdot\|_{+1,\omega}$. 
	Hence, it converges to a function 
	$w\in\dom Q_{\omega,\tt}$ in this norm.
	In particular, $\|\sfG_{V_n} u - w\|_\Lcyl\arr 0$ as $n\arr\infty$.
	In view of \eqref{eq:conv0} we get $w = \sfG_V u\in\dom Q_{\omega,\tt}$.
	Thus, we obtain
	\[
		\|\sfG_{V_n} u - \sfG_V u\|_{+1,\omega}\arr 0,\qquad n\arr\infty.
	\]
	Finally, applying Lemma~\ref{lem:omega_smooth}\,(ii) and
	Lemma~\ref{lem:appendixA}\,(v) we end up with
	\[
		Q_{\omega,\tt}[\sfG_V u] 
		= 
		\lim_{n\arr\infty} Q_{\omega,\tt}[\sfG_{V_n} u] 
		= 
		\lim_{n\arr\infty} Q_{\omega_n,\tt}[\sfG_{V_n} u]		
		=
		\lim_{n\arr\infty} Q_{\Phi_{\omega_n} + k,\tt}[u]
		= 
		Q_{\Phi_\omega + k,\tt}[u].\qedhere
	\]
\end{proof}


\section{Description of the domain of \texorpdfstring{$\Fqf^{[m]}$}{Fqf}}
\label{app:B}

The aim of this appendix is to give a simple description of the domain of the quadratic forms $\frmq^{[m]}$ with $\omega \in (0,1/2]$ 
defined in~\eqref{eqn:fq_flat}. The main result
of this appendix reads as follows.
\begin{prop}\label{prop:desc_domfq} 
	Let $\omega \in (0,1/2]$. The domain of the form
	$\frmq^{[m]}$ defined in~\eqref{eqn:fq_flat}
	is given by
	\[
		\dom \frmq^{[m]} = H_0^1(\Gui(\tt)).
	\]
\end{prop}
Before proving Proposition~\ref{prop:desc_domfq} 
we introduce the norm $\|\cdot\|_{+1,m}$ associated 
to the quadratic form $\frmq^{[m]}$ as
\begin{equation}\label{eqn:fq_flat_norm}
	\|u\|_{+1,m}^2 := \frmq^{[m]}[u] + \|u\|_{L^2(\Gui(\tt))}^2,
	\qquad u\in \dom \frmq^{[m]}.
\end{equation}
The proof of Proposition~\ref{prop:desc_domfq} goes along the following lines. First, we remark that $\cC_0^\infty(\Gui(\tt))$ is a form core for 
$\frmq^{[m]}$ and, second, we prove that the norms 
$\|\cdot\|_{H^1(\Gui(\tt))}$ and $\|\cdot\|_{+1,m}$ are topologically equivalent on $\cC_0^\infty(\Gui(\tt))$. These properties are stated in the following two lemmas whose proofs are postponed to the end of this appendix.
\begin{lem}\label{lem:B1}
	Let $\omega\in (0,1/2]$. $\cC_0^\infty(\Gui(\tt))$ is a core for the form $\frmq^{[m]}$ defined in~\eqref{eqn:fq_flat}.
\end{lem}

\begin{lem}\label{lem:B2}
	Let $\tt\in (0,\pi/2)$, $\omega\in (0,1/2]$, and $m\in \dZ$. 
	Then there exist $C_j = C_j(\omega,\tt,m) >0$, 
	$j=1,2$, such that 
	\[
		C_1\|u\|_{H^1(\Gui(\tt))} \leq 
		\|u\|_{+1,m} \leq 
		C_2 \|u\|_{H^1(\Gui(\tt))},\qquad \forall u\in \cC_0^\infty(\Gui(\tt)).
	\]
\end{lem}
We now have all the tools to prove Proposition~\ref{prop:desc_domfq}.

\begin{proof}[Proof of Proposition~\ref{prop:desc_domfq}] 
	Combining Lemmas~\ref{lem:B1},~\ref{lem:B2} and~\cite[Thm. VI 1.21]{Kato}
	we obtain
	\[
		\dom\frmq^{[m]} 
		= 
		\ov{\cC^\infty_0(\Gui(\tt))}^{\|\cdot\|_{+1,m}}
		= 
		\ov{\cC^\infty_0(\Gui(\tt))}^{\|\cdot\|_{H^1(\Gui(\tt))}}
		= H^1_0(\Gui(\tt)).\qedhere
	\]
\end{proof}
Finally, we conclude this appendix 
by the proofs of Lemmas~\ref{lem:B1} and~\ref{lem:B2}.
\begin{proof}[Proof of Lemma~\ref{lem:B1}] 
	Let the projection $\pi^{[m]}$ 
	be defined as in~\eqref{eq:projector}. Let us introduce the associated
	orthogonal projector $\Pi^{[m]}$ in $\Lcyl$ by
	\[
		\Pi^{[m]} u := v_m(\phi)(\pi^{[m]}u)(r,z)
	\]
	with $v_m$ as in~\eqref{eq:vm}.
	For any $v\in \dom Q_{\omega,\tt}$ we have
	\begin{subequations}
	\begin{align}
			\|v\|_{L_\cyl^2(\Lay(\tt))}^2 &= \|\Pi^{[m]} v\|_\Lcyl^2 
			+ \|(\sfI - \Pi^{[m]})v\|_\Lcyl^2,\\
			\Frm[v] & =  Q_{\omega,\tt}[\Pi^{[m]} v] 
			+
			\Frm[(\sfI-\Pi^{[m]}) v].
		\label{eqn:maj_proj}
	\end{align}
	\end{subequations}
	Let $u\in\dom \frmq^{[m]}$ be fixed.
	Thanks to~\eqref{eq:orthogonal_decomposition} 
	and~\eqref{eqn:fq_flat}, we know that 
	$v = (2\pi)^{-1/2}r^{-1/2} u e^{\ii m\phi} \in \dom Q_{\omega,\tt}$.
	Consequently, there exists $v_n\in\cC_0^\infty(\Lay(\tt))$ such that
	\[
		\Frm[v_n - v] + \|v_n - v\|_\Lcyl^2 \arr0,
		\qquad n\arr\infty.
	\]
	By~\eqref{eqn:maj_proj} and using the non-negativity of $Q_{\omega,\tt}$
	we obtain
	\[
		\Frm[\Pi^{[m]}(v_n - v)] + 
		\|\Pi^{[m]}(v_n - v)\|_\Lcyl^2
		\longrightarrow 0,\qquad n\rightarrow\infty.
	\]
	Letting 
	$u_n(r,z) =\sqrt{r}(\pi^{[m]} v_n)(r,z)$, 
	the last equation rewrites
	\[
		\|u_n-u\|_{+1,m}^2 = \frmq^{[m]}[u_n - u] + 
		\|u_n - u\|_{L^2(\Gui(\tt))}^2 	
		\arr 0,\qquad n\arr\infty.
	\]
	Since $v_n\in \cC_0^\infty(\Lay(\tt))$, we get that 
	$u_n\in\cC_0^\infty(\Gui(\tt))$	which concludes the proof.
\end{proof}

\begin{proof}[Proof of Lemma~\ref{lem:B2}]
	Let $u\in\cC_0^\infty(\Gui(\tt))$ be fixed. 
	The claim of the lemma is a consequence of the non-negativity 
	of $\frq_{0,\tt}[u]$
	\begin{equation}\label{eqn:form_dom_hardy}
		\frq_{0,\tt}[u] 
		= 
		\int_{\Gui(\tt)}|\p_r u|^2 + |\p_z u|^2 - \frac{1}{4r^2}|u|^2 
		\dd r\dd z \geq 0.
	\end{equation}
	The inequality~\eqref{eqn:form_dom_hardy}
	can be easily derived from Hardy inequality
	in the form as stated in~\cite[\S VI.4, eq. 4.6]{Kato}.
	Further, we remark that
	\begin{equation}\label{eqn:eq_norm_proof}
		\|u\|_{+1,m}^2 = 
		\frmq^{[m]}[u] + \|u\|_{L^2(\Gui(\tt))}^2
		= 
		\|u\|_{H^1(\Gui(\tt))}^2 +
		\big[(m-\omega)^2- 1/4\big]\int_{\Gui(\tt)}\frac{|u|^2}{r^2}\dd r\dd z
	\end{equation}
	Now, we distinguish the special case $m=0$ from $m\neq0$.

	\noindent\underline{\boldmath{$m=0$}.}
	In this case,~\eqref{eqn:eq_norm_proof} simplifies as
	\begin{equation}
		\|u\|_{+1,0}^2 = 
		\|u\|_{H^1(\Gui(\tt))}^2 - [1/4-\omega^2]\int_{\Gui(\tt)}\frac{|u|^2}	
		{r^2}\dd r\dd z
		\label{eqn:up_low_m0}
	\end{equation}
	Since the second term on the right-hand side of~\eqref{eqn:up_low_m0} 
	is non-positive, we immediately get the upper bound
	\[
		\|u\|_{+1,0} \leq \|u\|_{H^1(\Gui(\tt))}.
	\]
	To obtain the lower bound, we combine \eqref{eqn:up_low_m0} 
	with inequality~\eqref{eqn:form_dom_hardy}
	\[
	\begin{split}
		\|u\|_{+1,0}^2 &= \|u\|_{H^1(\Gui(\tt))}^2 -
		\Big(\frac14 - \omega^2\Big)
		\int_{\Gui(\tt)}\frac{|u|^2}{r^2}\dd r\dd z\\
		&\geq 
		\|u\|_{H^1(\Gui(\tt))}^2 - 
		(1-4\omega^2)
		\big(\|\p_r u\|_{L^2(\Gui(\tt))}^2 +\|\p_z u\|_{L^2(\Gui(\tt))}^2\big)
		\geq 
		4\omega^2 \|u\|_{H^1(\Gui(\tt))}^2.
	\end{split}
	\]
	\noindent\underline{\boldmath{$m\neq 0$}.}
	In this case the second term on the right-hand side of~\eqref{eqn:eq_norm_proof} 
	is non-negative and we get the lower bound
	\[
		\|u\|_{+1,m} \geq \|u\|_{H^1(\Gui(\tt))}.
	\]
	To get an upper bound we combine~\eqref{eqn:eq_norm_proof} 
	with~\eqref{eqn:form_dom_hardy}
	\[
	\begin{split}
		\|u\|_{+1,m}^2	
		&	= 
		\|u\|_{H^1(\Gui(\tt))}^2 
		+
		\big [(m-\omega)^2 - 1/4\big]
		 \int_{\Gui(\tt)}\frac{|u|^2}{r^2}\dd r\dd z\\
		&\leq 
		\|u\|_{H^1(\Gui(\tt))}^2 + \big[4(m-\omega)^2 - 1\big]\big(
		\|\p_r u\|_{L^2(\Gui(\tt))}^2 + \|\p_z u\|_{L^2(\Gui(\tt))}^2\big)\\
		&\leq 4(m-\omega)^2 \|u\|_{H^1(\Gui(\tt))}^2.\qedhere
	\end{split}
	\]
	%
\end{proof}

\subsection*{Acknowledgements}
D.~K. and V.~L. are supported by the project RVO61389005
and by the Czech Science Foundation (GA\v{C}R) within the project 14-06818S. 
T.~O.-B. is supported by the Basque Government through the BERC 2014-2017 program and by Spanish Ministry of Economy and Competitiveness MINECO: BCAM Severo Ochoa excellence accreditation SEV-2013-0323. He is grateful for the stimulating research stay and the hospitality of the Nuclear Physics Institute of Czech Republic in January 2016 where part of this paper was written.

\bibliographystyle{alpha}

\newcommand{\etalchar}[1]{$^{#1}$}

\end{document}